\documentclass[12pt]{amsart}

\usepackage{cite}
\usepackage[hypertexnames=false, colorlinks, linkcolor=blue, allcolors=blue]{hyperref}
\hypersetup{nesting=true,debug=true,naturalnames=true}
\usepackage[margin=1in]{geometry}
\usepackage{graphicx}
\usepackage{enumitem}
\usepackage{thm-restate}
\usepackage{cleveref}
\usepackage{overpic}
\usepackage{booktabs}
\usepackage{bm}

\usepackage{longtable}
\usepackage{multirow}
\usepackage{multicol}

\usepackage[colorinlistoftodos]{todonotes}

\newtheorem{theorem}{Theorem}[section]
\newtheorem{lemma}[theorem]{Lemma}

\newtheorem{definition}[theorem]{Definition}
\newtheorem{proposition}[theorem]{Proposition}
\newtheorem{corollary}[theorem]{Corollary}
\newtheorem{question}[theorem]{Question}

\theoremstyle{definition}
\newtheorem{example}[theorem]{Example}
\newtheorem{remark}[theorem]{Remark}

\begin{document}

\title{Negative amphichiral knots and the half-{C}onway polynomial} 

\author{Keegan Boyle}  
\address{Department of Mathematics, University of British Columbia, Canada} 
\email{kboyle@math.ubc.ca}

\author{Wenzhao Chen}
\address{Department of Mathematics, University of British Columbia, Canada}
\email{chenwzhao@math.ubc.ca}

\setcounter{section}{0}

\begin{abstract}
In 1979, Hartley and Kawauchi proved that the Conway polynomial of a strongly negative amphichiral knot factors as $f(z)f(-z)$. In this paper, we normalize the factor $f(z)$ to define the \emph{half-Conway} polynomial. First, we prove that the half-Conway polynomial satisfies an equivariant skein relation, giving the first feasible computational method, which we use to compute the half-Conway polynomial for knots with 12 or fewer crossings. This skein relation also leads to a diagrammatic interpretation of the degree-one coefficient, from which we obtain a lower bound on the equivariant unknotting number. Second, we completely characterize polynomials arising as half-Conway polynomials of knots in $S^3$, answering a problem of Hartley-Kawauchi. As a special case, we construct the first examples of non-slice strongly negative amphichiral knots with determinant one, answering a question of Manolescu. The double branched covers of these knots provide potentially non-trivial torsion elements in the homology cobordism group. 
\end{abstract}
\maketitle
\tableofcontents
\section{Introduction}
A \emph{strongly negative amphichiral} knot is a smooth oriented knot $K \subset S^3$ along with an order 2 symmetry $\rho\colon (S^3,K) \to (S^3,K)$ which reverses the orientation of $S^3$ and of $K$, and which has fixed set $S^0$; see Figure \ref{fig:4_1} for some examples. Since this symmetry reverses the orientation on $S^3$, it is considerably more difficult to study than its orientation-preserving cousins: periodic and strongly invertible knots. For example, in the concordance group these knots are all torsion so that additive concordance invariants, such as the signature, must vanish. Consequently, it was only recently  shown that strongly negative amphichiral knots can have large 4-genus \cite{Miller2020}.

\begin{figure}
  \begin{overpic}[width=400pt, grid=false]{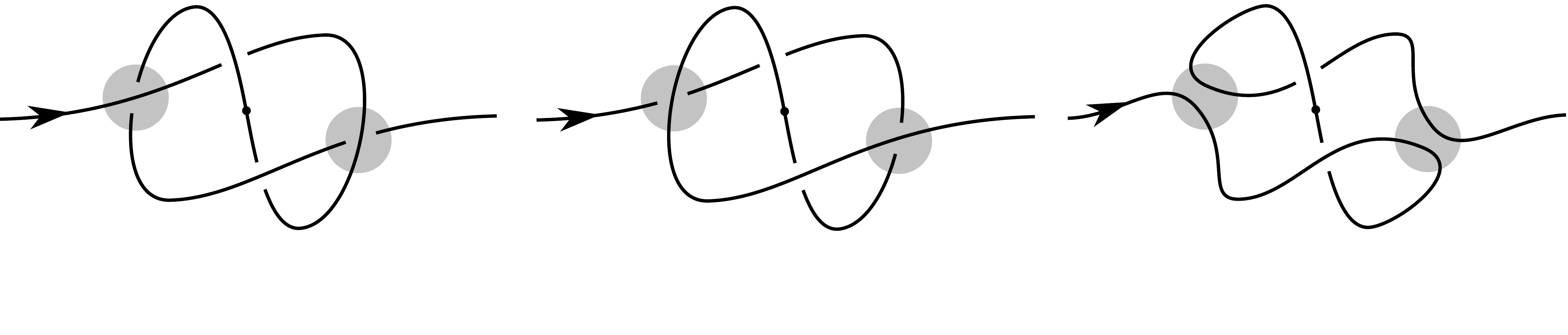}
    \put (13, 3) {$K_+$}
    \put (48, 3) {$K_-$}
    \put (82, 3) {$K_0$}
  \end{overpic}
\caption{A triple of knots related by the equivariant skein relation in Theorem \ref{thm:skeinrelation}; here $K_+$ is the figure-eight knot (left), and $K_-$ (center) and $K_0$ (right) are each the unknot. In each diagram the symmetry is point reflection across the marked point.}
\label{fig:4_1}
\end{figure}

In this paper we define and study an equivariant version of the Conway polynomial. Apart from our direct results and applications, one long-term goal of this project is to build a foundation to study equivariant knot Floer and Khovanov homology theories for strongly negative amphichiral knots. In the strongly invertible case, such homology theories have been used to show that certain slice disks are not isotopic \cite{DaiHeddenMallick} (see also \cite{Hayden} and \cite{HaydenSundberg}), and that a refinement of Khovanov homology can distinguish mutants \cite{MR4286365}. In particular, equivariant theories can provide inroads on non-equivariant problems.

\subsection{The half-Conway polynomial}

Hartley and Kawauchi showed that the Conway polynomial of a strongly negative amphichiral knot factors as $f(z)f(-z)$ for some polynomial $f(z)$ \cite[Theorem 1]{MR543095}. Here $f(z)$ corresponds to the Alexander polynomial of the non-orientable homology circle which is the quotient of the knot exterior by the amphichiral symmetry \cite{MR391110}. This factorization is recognized as an obstruction to the existence of a strongly negative amphichiral symmetry.  However, the polynomial $f(z)$ has not been studied as an invariant of the symmetry directly. Indeed, we provide the first method for computing $f(z)$ from a knot diagram.

To produce a diagrammatic computation method, we give a method to naturally choose an element of $\{f(z),f(-z),-f(z),-f(-z)\}$, which we define to be the \emph{half-Conway} polynomial $\nabla_{(K,\rho)}(z)$ of a strongly negative amphichiral knot $(K,\rho)$; see Definition \ref{def:half-Conway-normalization}. We can then show that the half-Conway polynomial satisfies an equivariant skein relation analogous to the skein relation for the Conway polynomial. Such a skein relation requires a sign associated with symmetric pairs of crossings, but there is no obvious choice; each pair consists of a positive and a negative crossing. Nonetheless, we assign a sign to such pairs (see Definition \ref{def:crossingpairsign}).


\begin{restatable}{theorem}{skeinrelation}\label{thm:skeinrelation}
The half-Conway polynomial satisfies the following equivariant skein relation:
\[
\nabla_{(K_+,\rho)}(z) - \nabla_{(K_-,\rho)}(z) = z \cdot \nabla_{(K_0,\rho)}(z),
\]
where $K_+, K_-,$ and $K_0$ are an equivariant skein triple as in Definition \ref{def:skeintriad}. 
\end{restatable}

Here $(K_+,\rho)$ is any strongly negative amphichiral knot with a positive dichromatic\footnote{See Definition \ref{def:chromatic}.} symmetric crossing pair, and $(K_-,\rho)$ and $(K_0,\rho)$ are obtained from $(K_+,\rho)$ by changing and resolving this crossing pair respectively. (See Figure \ref{fig:4_1} for an example and Section \ref{sec:equivariantskeinrelation} for a precise definition of $K_+,K_-$, and $K_0$.)

\begin{remark}
Although Hartley proved that the Conway polynomial $\nabla_K(z)$ factors as $f(z)f(-z)$ for any negative amphichiral knot \cite[Theorem 3.1]{MR593625}, Theorem \ref{thm:skeinrelation} only makes sense when the negative amphichiral symmetry is realized by a diagrammatic involution; that is for \emph{strongly} negative amphichiral knots.
\end{remark}

Following the precedent of the Conway polynomial, one may hope to prove Theorem \ref{thm:skeinrelation} by examining a $\rho$-invariant Seifert surface or a pair of Seifert surfaces exchanged by $\rho$. However, strongly negative amphichiral knots cannot bound symmetric Seifert surfaces; such a surface must contain a point-wise fixed arc but there are only two fixed points in $S^3$. Additionally, it is not clear how to extract the necessary homological information from a pair of Seifert surfaces. Instead, we use a novel argument based on symmetric surgery diagrams to piece together the skein relation. 

For our first application of Theorem \ref{thm:skeinrelation}, we compute the half-Conway polynomial for all strongly negative amphichiral knots with 12 or fewer crossings; see Section \ref{sec:computations}. Using these computations of the half-Conway polynomial, we can distinguish many strongly negative amphichiral symmetries on a given knot. (The symmetries in the following proposition can also be distinguished by \cite[Theorem 15.1]{BonahonSiebenmann}.)

\begin{restatable}{proposition}{manysymmetries}\label{prop:manysymmetries}
For any $n >0$, there exists a knot $K$ with $n$ strongly negative amphichiral symmetries, distinguished by their half-Conway polynomials.
\end{restatable}

For a second application, we relate the half-Conway polynomial to another invariant, the \emph{half-linking number} $h(K)$ defined in Section \ref{sec:half-linking-number}. The integer $h(K)$ is the sum of signs of symmetric pairs of crossings in a symmetric diagram (see Definitions \ref{def:crossingpairsign} and \ref{def:half-linking-number}), and we use a theory of symmetric Reidemeister moves (developed in Section \ref{sec:background} and Appendix \ref{app:Rmoves}) to prove that $h(K)$ is invariant under equivariant isotopy. We also show that the half-linking number provides a lower bound on the equivariant unknotting number, the minimum number of symmetric pairs of crossing changes necessary to produce the unknot; see Section \ref{sec:equivariantunknotting}. 

\begin{restatable}{theorem}{unknotting} \label{thm:unknotting}
Let $K$ be a strongly negative amphichiral knot. Then the half-linking number $h(K)$ is a lower bound on the equivariant unknotting number. That is, $\widetilde{u}(K) \geq |h(K)|$. 
\end{restatable}

Perhaps surprisingly, this bound turns out to be independent of the slice genus, and even the equivariant slice genus; see Example \ref{ex:eqslicegenusvshalflinking}.

In the following corollary of Theorem \ref{thm:skeinrelation}, we specify the relationship between the half-Conway polynomial and the half-linking number, which further implies a relationship between the half-linking number and the Arf invariant. 


\begin{restatable}{corollary}{halfConwayhalflinkingarf} \label{cor:halfConwayhalflinkingarf}
Let $(K,\rho)$ be an oriented strongly negative amphichiral knot. Then
\begin{enumerate}
\item The coefficient of $z$ in $\nabla_{(K,\rho)}(z)$ is equal to the half-linking number $h(K)$, and
\item $h(K) \equiv \textup{Arf}(K)$ (mod 2).
  \end{enumerate}
\end{restatable}

In particular, the equivalence $h(K) \equiv \textup{Arf}(K)$ (mod 2) gives a simple diagrammatic interpretation of the Arf invariant for strongly negative amphichiral knots. Corollary \ref{cor:halfConwayhalflinkingarf} also implies the invariance of the half-linking number, but we include an additional proof via SNA Reidemeister moves; these moves provide a platform for defining and studying diagrammatic invariants, and we believe that they will be an important tool for studying Questions \ref{question:diagrammatichalfConwayproof} and \ref{question:HOMFLY} below.

Finally, we construct knots with arbitrary prescribed half-Conway polynomials, solving an open problem from \cite[Remark (3)]{MR543095}. (The following theorem follows immediately from Theorem \ref{thm:realization}.)

\begin{theorem} \label{thm:Conwayrealization}
Let $f(z) \in \mathbb{Z}[z]$ such that $f(0) = 1$. Then there is an oriented strongly negative amphichiral knot $(K,\rho)$ with half-Conway polynomial $\nabla_{(K,\rho)}(z) = f(z).$
\end{theorem}

As a particular example, this theorem allows the construction of many non-slice strongly negative amphichiral knots with determinant 1 by choosing a half-Conway polynomial which fails the Fox-Milnor condition and gives determinant 1; see Figure \ref{fig:det1snak} for one example. This answers Problem 20(3) from \cite{FT}. 

\begin{corollary} \label{cor:det1snak}
There exist non-slice (strongly negative) amphichiral knots with determinant 1.
\end{corollary}

Such knots are interesting because their double branched covers represent potentially non-trivial torsion elements in the homology cobordism group. In fact, the involutive Heegaard Floer theoretic invariant called the $\iota$-complex could obstruct these double branched covers from bounding a homology 4-ball (see \cite{dai2018infinite}), although it is currently a challenge to compute these $\iota$-complexes. It is unknown whether there is torsion in the homology cobordism group; see \cite[Section 2]{MR3966804}. 

\begin{figure} 
\begin{overpic}[width=150pt, grid=false]{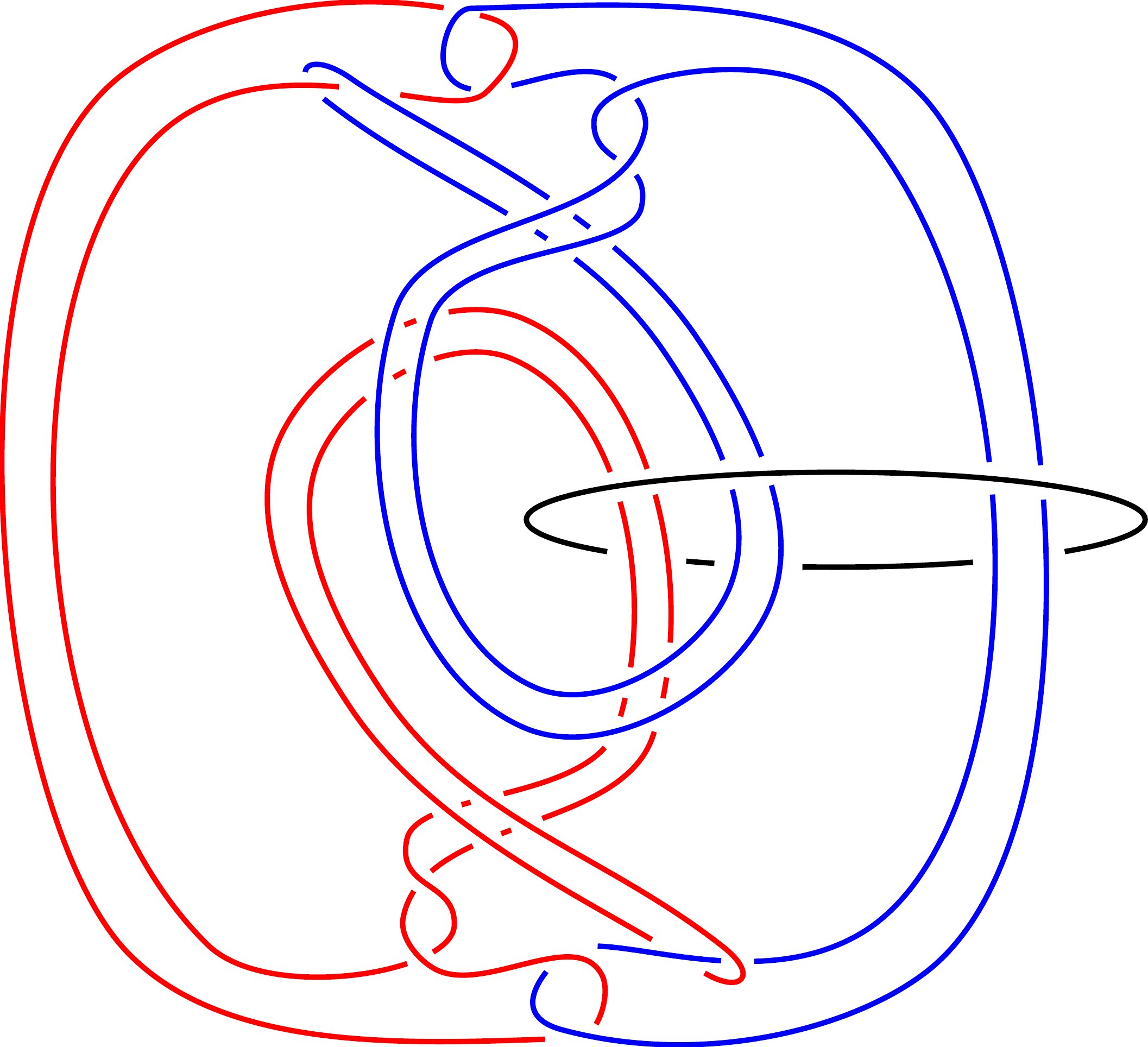}
\put (1,1) {$+1$}
\put (85,80) {$-1$}
\put (101,45) {$K$}
\end{overpic} \hspace{1 cm}
\begin{overpic}[width=150pt, grid=false]{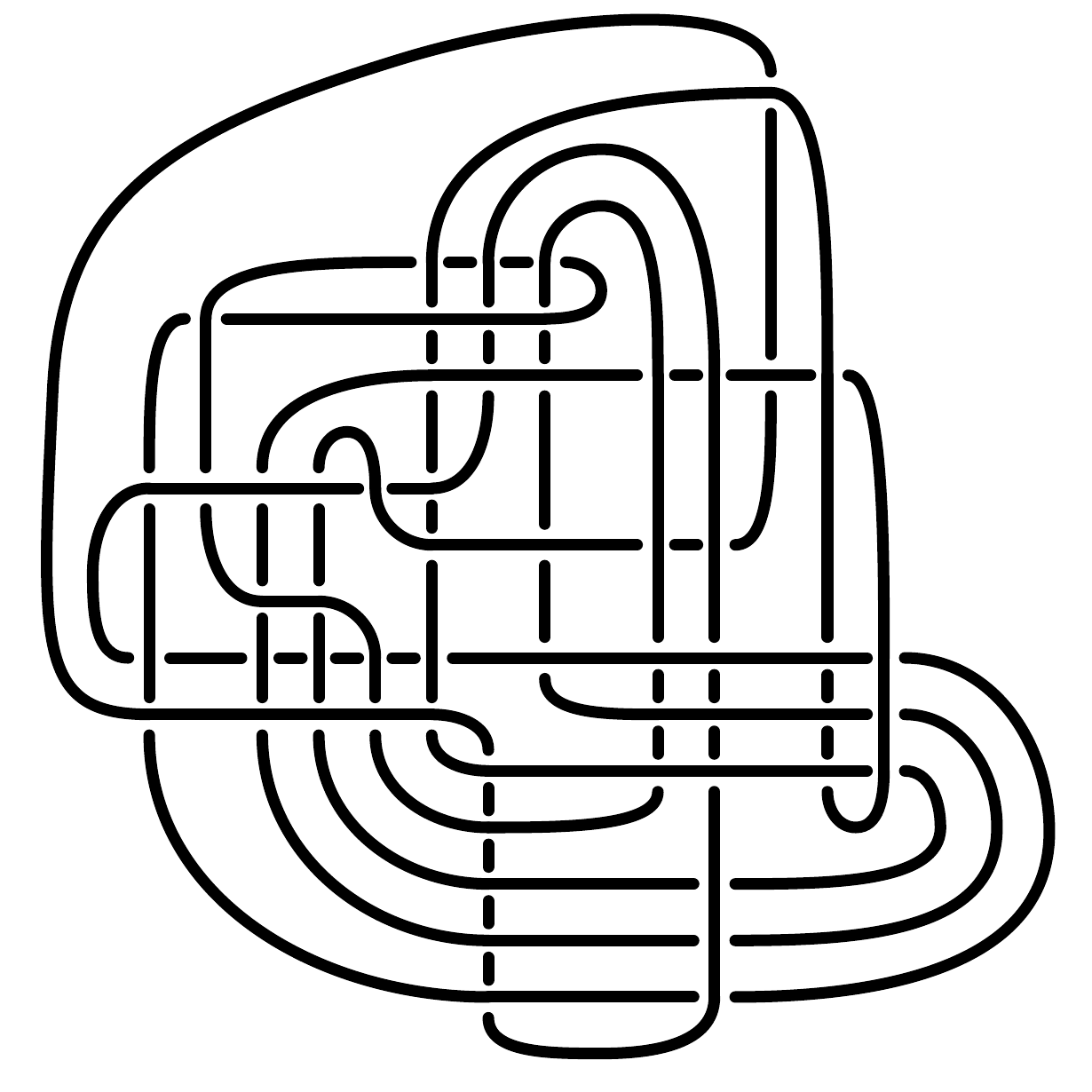}

\end{overpic}
\caption{A symmetric surgery diagram (left) and a knot diagram (right) for a non-slice strongly negative amphichiral knot $K$ with determinant 1. The diagram on the right was obtained using the KLO \cite{KLO} and SnapPy \cite{SnapPy} software.}
\label{fig:det1snak}
\end{figure}

\begin{remark}
To construct interesting integer homology spheres, Van Buskirk \cite{MR715764} constructed prime strongly negative amphichiral knots with Alexander polynomial 1, and Siebenmann and Van Buskirk \cite{MR694681} constructed prime strongly positive amphichiral knots with determinant 1. In both cases, it is unknown if the knots are non-slice. More generally, it is unknown if there are non-slice amphichiral knots with Alexander polynomial 1 \cite[Problem 21]{FT}. 
\end{remark}

\subsection{Open questions} We conclude with some open questions which we hope will inform future research in the area. 

\begin{question} \label{question:diagrammatichalfConwayproof}
Using the equivariant skein relation in Theorem \ref{thm:skeinrelation} as a definition for the half-Conway polynomial, is there a proof of equivariant isotopy invariance using symmetric Reidemeister moves?
\end{question} 

In this paper, we define the half-Conway polynomial by normalizing an Alexander polynomial associated to the quotient of the knot complement so that equivariant isotopy invariance is straightforward; in exchange the proof that it satisfies the equivariant skein relation is somewhat complicated. We attempted the proof of invariance using symmetric Reidemeister moves but were unable to check the SNA R4 move (see Theorem \ref{thm:Rmoves}). This proof method is of interest because it is related to the following question.

\begin{question}  \label{question:HOMFLY}
Is there an equivariant version of the Jones or HOMFLY-PT polynomial satisfying an equivariant skein relation?
\end{question}
It turns out that the obvious equivariant analog of the skein relation for the Jones polynomial does not define an invariant of strongly negative amphichiral knots. Nonetheless, we believe that some equivariant analog of the Jones polynomial should exist, albeit with a more complicated skein relation. Of course, one may further wish to define an equivariant Khovanov homology theory as has been done for strongly invertible knots; see \cite{MR4286365},\cite{MR3649241},\cite{MR2550464}. In a related direction, the following question may be more approachable.

\begin{question}  
Is there an equivariant version of knot Floer homology categorifying the half-Conway polynomial?
\end{question}
One reason to construct such a theory would be to generate an equivariant slice obstruction to finish the classification of equivariantly slice strongly negative amphichiral knots with 12 or fewer crossings (see \cite{BoyleIssa}). Additionally, it may answer the fundamental Question \ref{question:commutativesum} below.

\begin{question} \label{question:commutativesum}
Is the equivariant connected sum operation commutative up to equivariant isotopy? up to equivariant concordance? 
\end{question}
Here the equivariant connected sum depends on an orientation and choice of fixed point; see Definition \ref{def:connectsum}. Certainly the equivariant connected sum operation seems to be non-commutative; recently Di Prisa proved that a similar equivariant concordance group for strongly invertible knots is not abelian \cite{DiPrisa}. However, abelian invariants like the half-linking number and half-Conway polynomial cannot detect non-commutativity. 

\subsection*{Organization}

In Section \ref{sec:background} we give some background on strongly negative amphichiral knots and develop a theory of symmetric Reidemeister moves. In Section \ref{sec:half-linking-number} we introduce the half-linking number, and establish its additivity under equivariant connected sum. In Section \ref{sec:equivariantunknotting} we introduce the equivariant unknotting number and prove Theorem \ref{thm:unknotting}. In Section \ref{sec:half-Conway} we define the half-Conway polynomial, discuss its basic properties, and prove Proposition \ref{prop:manysymmetries} and Corollary \ref{cor:halfConwayhalflinkingarf}. Section \ref{sec:surgery-presentations} discusses an approach to computing the half-Conway polynomial via symmetric surgery diagrams. In Section \ref{sec:realization} we prove Theorem \ref{thm:realization} (equivalently, Theorem \ref{thm:Conwayrealization}) and Corollary \ref{cor:det1snak}. In Section \ref{sec:proofofmaintheorem} we prove Theorem \ref{thm:skeinrelation}. In Section \ref{sec:computations} we compute the half-Conway polynomial for knots with 12 or fewer crossings. In Appendix \ref{app:table} we tabulate symmetric diagrams for these knots.

\subsection*{Acknowledgments}
We would like to thank the organizers of the 2022 BIRS workshop \textit{Interactions of Gauge Theory with Contact and Symplectic Topology in Dimensions 3 and 4} where this project took shape. We would also like to thank Ahmad Issa, Robert Lipshitz, Charles Livingston, and Liam Watson for helpful comments. The second author is supported by the Pacific Institute for the Mathematical Sciences (PIMS). The research and findings may not reflect those of the institute. Finally, we would like to thank the anonymous referees for their careful reading and comments on a previous version.

\section{Background on strongly negative amphichiral knots}\label{sec:background}
We begin with some basic definitions about strongly negative amphichiral knots. 
\begin{definition}
A \emph{strongly negative amphichiral} knot is a knot $K \subset S^3$ along with a smooth involution $\rho\colon S^3 \to S^3$ with fixed-point set $S^0 \subset K$. 
\end{definition}

Note that by \cite{MR111044,MR155323}, all strongly negative amphichiral involutions on $S^3$ are conjugate in the diffeomorphism group. Thinking of $S^3$ as $\mathbb{R}^3 \cup \{\infty\}$, we will work exclusively with the representative of this conjugacy class given by the point reflection symmetry across the origin. Here the two fixed points are the origin and $\{\infty\}$.

\begin{definition}
A \emph{symmetric diagram} for a strongly negative amphichiral knot $(K,\rho)$ is a regular projection of $K$ to a $\rho$-invariant $S^2 \subset S^3$ along with under and over-crossing data. 
\end{definition}
Note that a symmetric diagram must contain both fixed points of $\rho$. We will place one of these fixed points at infinity, and the other in the center of our diagram (that we mark with a dot) so that $K$ may be viewed as a 2-ended tangle with a point reflection symmetry about its center. See Figure \ref{fig:4_1}.

\begin{definition}
A \emph{direction} on a strongly negative amphichiral knot $(K,\rho)$ is a choice of fixed point and orientation on $K$. 
\end{definition}
Diagrammatically, our convention is to place the chosen fixed point at infinity, and to orient the knot from the left to the right of the page. We will refer to the chosen fixed point as $x_{\infty}$ and the other fixed point as $x_0$. We will refer to the arc beginning at $x_{\infty}$ and ending at $x_0$ as $a_r$ (often drawn in red), and the arc beginning at $x_0$ and ending at $x_{\infty}$ as $a_b$ (often drawn in blue).

\begin{definition}
Two strongly negative amphichiral knots $(K,\rho)$ and $(K',\rho')$ are \emph{equivalent} if there is an equivariant orientation-preserving homeomorphism of pairs $\varphi\colon(S^3,K) \to (S^3,K')$. If $K$ and $K'$ are oriented, we further require that $\varphi$ preserves the orientation, and if $K$ and $K'$ are directed, we require that $\varphi$ takes the chosen fixed point for $K$ to the chosen fixed point for $K'$. 
\end{definition}

Note that this notion of equivalence is the same as an equivariant isotopy between $(K,\rho)$ and $(K',\rho)$ (see Theorem \ref{thm:Rmoves} below), but stronger than the existence of an orientation-preserving homeomorphism $\psi\colon (S^3,K) \to (S^3,K')$ with $\psi(\rho)$ isotopic to $\rho'$, the equivalence used in the mapping class group MCG$(S^3,K')$. In particular, changing the choice of direction in an equivariant connected sum (see below) may change the equivalence class of the symmetry, but not the isotopy class. 


In this paper, we will work with directed strongly negative amphichiral knots, although our invariants will be insensitive to the choice of fixed point. However, this choice is necessary in order to define an equivariant connected sum operation. 

\begin{definition} \label{def:connectsum}
The \emph{equivariant connected sum} $(K,\rho)\widetilde{\#}(K',\rho')$, or just $K \widetilde{\#}K'$, of two directed strongly negative amphichiral knots $(K,\rho)$ and $(K',\rho')$ is defined by removing a $\rho$-invariant neighborhood of $x_0$ from $(S^3,K,\rho)$ and a $\rho'$-invariant neighborhood of $x_{\infty}$ from $(S^3,K',\rho')$ and gluing the resulting 3-balls together in the way compatible with the chosen orientations on $K$ and $K'$ and with the symmetries $\rho$ and $\rho'$. The result is again a directed strongly negative amphichiral knot in $S^3$ with $x_0$ the remaining fixed point on $K'$ and $x_{\infty}$ the remaining fixed point on $K$. See Figure \ref{fig:connectsum}.
\end{definition}

\begin{remark}
We will also make use of the (non-equivariant) connected sum $K \# \overline{K}$ of a knot $K$ with its reverse mirror $\overline{K}$. Note that $K \# \overline{K}$ has an obvious strongly negative amphichiral symmetry exchanging $K$ and $\overline{K}$. If $K$ happens to be a strongly negative amphichiral knot, then $K \widetilde{\#} K$ should not be confused with $K \# \overline{K}$, which in general are two distinct symmetries on $K \# K$.
\end{remark}

\begin{figure}
\scalebox{.7}{\includegraphics{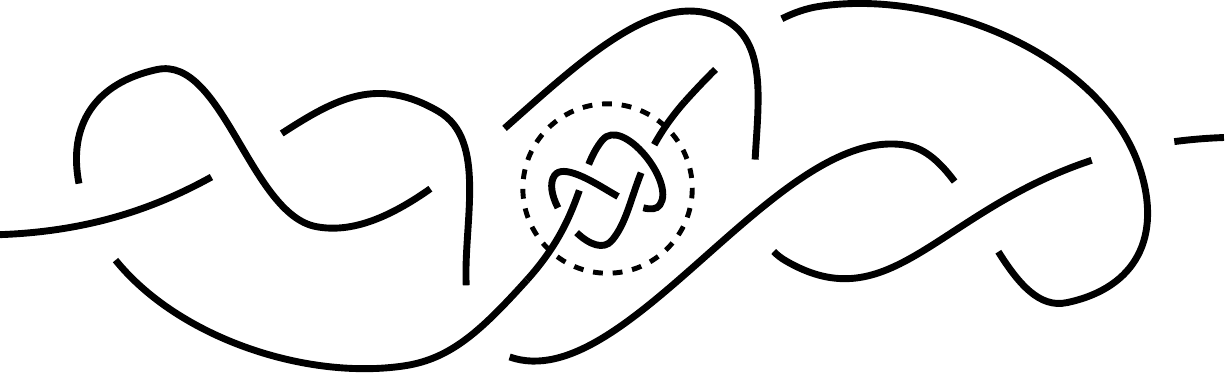}}
\caption{The equivariant connected sum of directed strongly negative amphichiral symmetries on the knots $8_9$ and $4_1$. The dotted circle indicates the connected summing sphere.}
\label{fig:connectsum}
\end{figure}

\subsection{Symmetric Reidemeister moves for strongly negative amphichiral knots}
Perhaps the simplest way to obtain invariants of strongly negative amphichiral knots is to define a theory of Reidemeister moves which are invariant under the symmetry, which we call strongly negative amphichiral (SNA) Reidemeister moves.
\begin{theorem} \label{thm:Rmoves}
Every pair of equivariant diagrams for a directed strongly negative amphichiral knot $K$ are related by a finite sequence of the following moves.
\begin{enumerate} 
  \item An equivariant planar isotopy (R0).
  \item An equivariant pair of any of the 3 standard Reidemeister moves R1, R2, and R3.
  \item The following new symmetric move R4 corresponding to pulling an equivariant pair of strands across one of the two fixed points.\\

\scalebox{.4}{\includegraphics{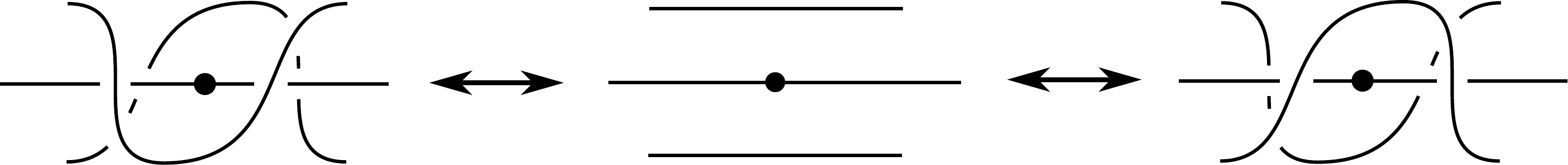}}\\
  \end{enumerate}
\end{theorem}
\begin{remark}
If $K$ is not directed, then there is an additional ``move'' which consists of redrawing the same diagram with the other fixed point at $\infty$. In part to avoid this, we work with directed strongly negative amphichiral knots throughout this paper.
\end{remark}
The proof of Theorem \ref{thm:Rmoves} is somewhat involved, and perhaps of independent interest; we provide it in Appendix \ref{app:Rmoves}.

\section{The half-linking number} \label{sec:half-linking-number}

The two fixed points on a strongly negative amphichiral knot separate it into a pair of arcs $a_r$ and $a_b$ exchanged by the involution, and a direction on the knot specifies an order on the crossings of any symmetric diagram as follows. We say that for a pair of crossings $c$ and $c'$ in $D$, $c<c'$ if we encounter $c$ before $c'$ when traversing $K$ following the orientation and starting at $x_{\infty}$. Using this structure, we associate a sign to each equivariant pair of crossings $(c,\rho(c))$. (Note that $c$ and $\rho(c)$ always have opposite signs.)

\begin{definition} \label{def:crossingpairsign}
Let $(K,\rho)$ be a directed strongly negative amphichiral knot, and fix an equivariant diagram $D$ for $K$. For an equivariant pair of crossings $(c,\rho(c))$, we define $\textup{sign}(c,\rho(c))$ as the sign of the larger of the two crossings $c$ and $\rho(c)$. In other words 
\[
\textup{sign}(c,\rho(c)) = \begin{cases}
\textup{sign}(c), &\mbox{if }c>\rho(c) \\
\textup{sign}(\rho(c)), &\mbox{if }\rho(c)>c.\\
\end{cases}
\]
\end{definition}

\begin{definition} \label{def:chromatic}
Let $c,\rho(c)$ be a pair of crossings in a strongly negative amphichiral knot diagram. If $c$ contains an arc from $a_r$ and an arc from $a_b$ then we call $(c,\rho(c))$ a \emph{dichromatic} crossing pair. If both arcs of $c$ belong to either $a_r$ or $a_b$, we call $(c,\rho(c))$ a \emph{monochromatic} crossing pair.
\end{definition}

\begin{remark}
As we will see, when $(c,\rho(c))$ is a dichromatic crossing pair, sign$(c,\rho(c))$ does not depend on the choice of fixed point in the direction of $K$, only on the orientation; see Proposition \ref{prop:half-linking-negation}. On the other hand, if $(c,\rho(c))$ is a monochromatic crossing pair, then sign$(c,\rho(c))$ does depend on the choice of fixed point.
\end{remark}

We can now define an invariant of strongly negative amphichiral knots by counting the signs of equivariant pairs of crossings involving an arc in $a_r$ and an arc in $a_b$.
\begin{definition} \label{def:half-linking-number}
Let $D$ be an equivariant diagram for a directed strongly negative amphichiral knot $(K,\rho)$. Then the \emph{half-linking number} is
\[
h(D) = \dfrac{1}{2} \sum \textup{sign}(c,\rho(c)),
\]
where the sum is over dichromatic crossing pairs $(c,\rho(c))$.
\end{definition}

\begin{example}
Consider the strongly negative amphichiral diagram $D$ for the directed strongly negative amphichiral knot $10_{88}$ shown in Figure \ref{fig:10_88}. The greater in each dichromatic crossing pair is marked with a gray circle, and since both are positive crossings we have that $h(D) = \frac{1}{2}(1+1) = 1$.
\end{example}

\begin{figure}
\scalebox{.6}{\includegraphics{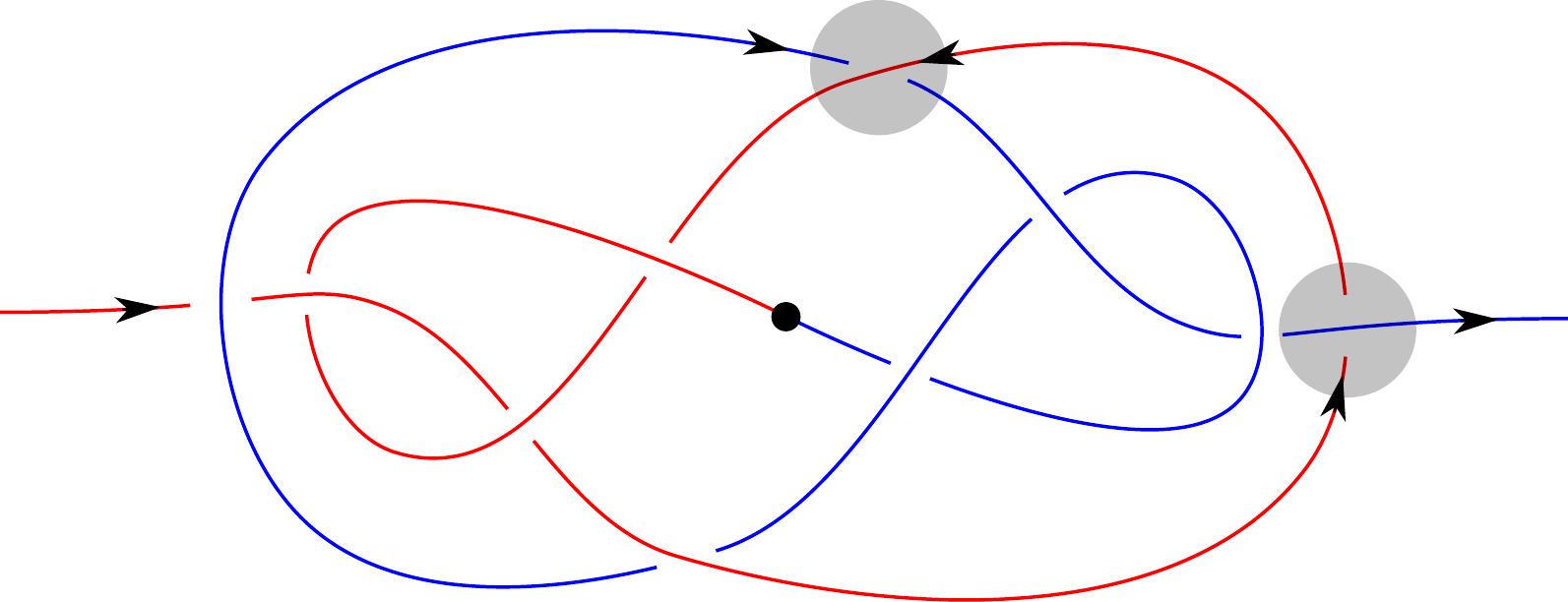}}
\caption{A strongly negative amphichiral diagram for $10_{88}$ with a choice of direction. For convenience the arcs $a_r$ and $a_b$ are colored red and blue respectively.}
\label{fig:10_88}
\end{figure}

In Theorem \ref{thm:half-linking-invariance} and Proposition \ref{prop:half-linking-negation} below, we show that the half-linking number is an invariant of an oriented strongly negative amphichiral knot. In other words, the half-linking number does not depend on the choice of equivariant diagram, or the choice of fixed point on the knot. We will then write $h(K)$ to refer to the half-linking number of the oriented strongly negative amphichiral knot $K$.

\begin{theorem} \label{thm:half-linking-invariance}
Let $D_1$ and $D_2$ be equivariant diagrams for a directed strongly negative amphichiral knot $(K,\rho)$. Then $h(D_1) = h(D_2)$. In particular, the half-linking number does not depend on the choice of equivariant diagram.
\end{theorem}

\begin{proof}
By Theorem \ref{thm:Rmoves}, $D_1$ and $D_2$ are related by a sequence of the SNA Reidemeister moves R0, R1, R2, R3, and R4, so it is enough to check that the half-linking number is unchanged under each of these moves.

After an SNA R0 move, the crossings and order are identical so that the half-linking number is unchanged.

After an SNA R1 move there is a new invariant crossing pair, but it is monochromatic so that it does not contribute to the half-linking number.

After an SNA R2 move there are two new crossing pairs, and either both are dichromatic, or neither are. In the first case, the two crossing pairs contribute to the half-linking number with opposite signs so that in both cases the half-linking number is unchanged.

For SNA R3 moves there are two cases. Either all three crossing pairs are monochromatic, in which case invariance is clear, or one crossing pair is monochromatic and the other two are dichromatic. In the latter case, the order on each of the two dichromatic crossing pairs does not change. In particular, there are the same number of crossing pairs which contribute to the half-linking number, with the same signs.

Finally after an SNA R4 move there are three new pairs of crossings. Regardless of how we orient the strands (ensuring that it is compatible with the amphichiral symmetry) and which direction we choose, exactly two of these pairs of crossings are dichromatic; one contributes $+1$ to the half-linking number and one contributes $-1$ to the half-linking number. Figure \ref{fig:R4invariance} shows the result of an SNA R4 move with one choice of orientations and direction; the other choices are similar.
\end{proof}

\begin{figure}
\begin{overpic}[width=200pt, grid = false]{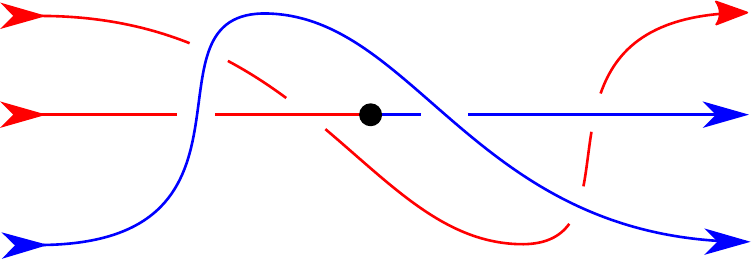}
\put (22, 32) {$+1$}
\put (18, 21) {$-1$}
\put (39, 14) {0}
\put (59, 21) {0}
\put (76,0) {$+1$}
\put (80,14) {$-1$}
\end{overpic}
\caption{The half-linking number is invariant under the R4 move. The crossing pairs are labeled with their contribution to the half-linking number.}
\label{fig:R4invariance}
\end{figure}
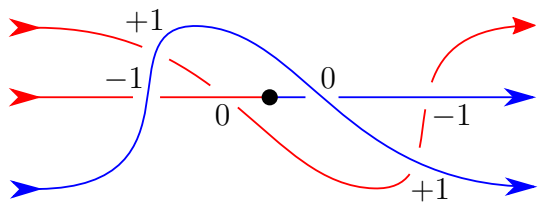

\begin{proposition} \label{prop:half-linking-negation}
Let $K$ be a directed strongly negative amphichiral knot. Let $rK$ be the reverse of $K$ and let $K^*$ be $K$ with the other choice of fixed point. Then
\begin{enumerate}
\item $h(K) = -h(rK)$, and
\item $h(K) = h(K^*)$.
\end{enumerate}
\end{proposition}
\begin{proof}
Reversing the orientation on $K$ reverses the order on the crossings so that if $c < \rho(c)$ in a diagram for $K$, then $\rho(c) < c$ in the corresponding diagram for $rK$. However, the signs of the crossings themselves do not depend on the orientation. Hence the sign of $(c,\rho(c))$ in $K$ and $rK$ are different and $h(K) = -h(rK)$.

Next, note that in diagrams for $K$ and $rK^*$, the crossings along $a_r$ appear in the opposite order so that $h(K) = -h(rK^*)$. Then by (1), $h(K^*) = -h(rK^*)$. Therefore $h(K) = h(K^*)$.
\end{proof}

\begin{proposition}
Given two directed strongly negative amphichiral knots $K$ and $K'$,
\[
h(K\widetilde{\#}K') = h(K) + h(K').
\]
\end{proposition}

\begin{proof}
This is immediate from the definitions by considering diagrams $D$ and $D'$ for $K$ and $K'$ respectively, and the corresponding diagram $D \widetilde{\#} D'$ for $K \widetilde{\#}K'$. The dichromatic crossing pairs in $D$ and $D'$ (see Definition \ref{def:half-linking-number}) correspond precisely with the dichromatic crossing pairs in $D \widetilde{\#}D'$, and since the corresponding crossing pairs have the same sign, their contributions to the half-linking number agree.
\end{proof}

\begin{remark}
One may wonder whether $h(K)$ is an equivariant concordance invariant; it is analogous to the linking number, it is additive under equivariant connected sum, and changing the direction on $K$ (as needed to define the inverse in the equivariant concordance group) negates $h(K)$. However, this is false; see the following example.
\end{remark}

\begin{example} \label{ex:eqslicegenusvshalflinking}
Consider the knot $K = 8_9$ (with the direction as shown in Figure \ref{fig:8_9}). We compute $h(K) = -2$, but $K$ is equivariantly slice (see \cite[Figure 1]{BoyleIssa}). 
\end{example}

\begin{figure}
\scalebox{.7}{\includegraphics{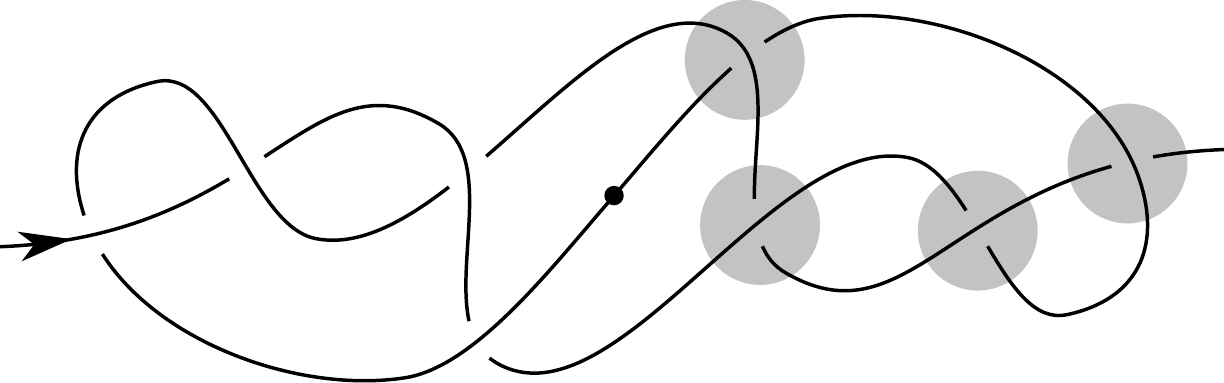}}
\caption{A strongly negative amphichiral diagram for the equivariantly slice knot $K = 8_9$ for which $h(K) = -2$. The four crossings contributing to $h(K)$ are indicated with a gray circle; they are all negative crossings.}
\label{fig:8_9}
\end{figure}

\subsection{Equivariant unknotting} \label{sec:equivariantunknotting} Next we discuss an application of the half-linking number to equivariant unknotting.
\begin{definition}
Given a strongly negative amphichiral knot $K$, the \emph{equivariant unknotting number} $\widetilde{u}(K)$ is the minimum number of equivariant pairs of crossing changes necessary to transform $K$ into the unknot. 
\end{definition}
Note that a single crossing change cannot be equivariant for a strongly negative amphichiral knot. We also have the immediate inequality $\widetilde{u}(K) \geq u(K)$, where $u(K)$ is the (usual) unknotting number of $K$. 

\begin{proposition} \label{prop:finiteunknotting}
The equivariant unknotting number $\widetilde{u}(K)$ is finite.
\end{proposition}
\begin{proof}
Given an equivariant diagram $D$ for a strongly negative amphichiral knot $K$, we will produce a finite equivariant unknotting sequence. To begin, note that the two fixed points separate $K$ into two arcs $a_r$ and $a_b$. In each dichromatic pair of crossings in $D$, either $a_r$ passes over $a_b$ in both crossings, or vice versa. First, perform a finite sequence of equivariant crossing changes to ensure $a_r$ always passes over $a_b$. This produces a new strongly negative amphichiral knot $K'$ which is a connected sum $J \# \overline{J}$. Here $\overline{J}$ is the reverse mirror of $J$ and the symmetry on $K'$ exchanges $J$ and $\overline{J}$. Now take any finite unknotting sequence $\{c_i\}$ for $J$. Then $\{\rho(c_i)\}$ is a finite unknotting sequence for $\overline{J}$. Thus $\{c_i,\rho(c_i)\}$ is a finite equivariant unknotting sequence for $K'$. 
\end{proof}

\unknotting*

\begin{proof}
We will prove this by induction on $\widetilde{u}(K)$. Clearly $\widetilde{u}(K) = 0$ if and only if $K$ is the unknot, and the half-linking number of the unknot is 0. 

Now consider an equivariant pair of crossing changes which take a knot $K_1$ with $\widetilde{u}(K_1) = n$ to a knot $K_2$ with $\widetilde{u}(K_2) = n-1$. By inductive assumption, $|h(K_2)| \leq n-1$. Taking any symmetric diagram in which this equivariant pair of crossing changes is visible, let $a_r$ and $a_b$ be the two arcs of the diagram. On one hand, the crossing changes transforming $K_1$ into $K_2$ may occur between a monochromatic crossing pair. In this case the half-linking number is unchanged so that $|h(K_1)| = |h(K_2)| \leq n - 1 \leq n$, as desired. Otherwise, the crossing pair is dichromatic so that $h(K_1)$ and $h(K_2)$ differ by $\pm 1$. In particular, $|h(K_1)| \leq |h(K_2)| + 1 \leq n -1 +1 = n$.
\end{proof}
This theorem is sharp for many knots with 12 or fewer crossings, such as in the following example. 
\begin{example}
Consider the strongly negative amphichiral symmetry $\rho$ on $12a_{1287}$ shown in Figure \ref{fig:12a_1287}. The 6 crossings in the bottom right are all positive, so that $h(12a_{1287},\rho) = 3$ and hence $\widetilde{u}(12a_{1287},\rho) \geq 3$ by Theorem \ref{thm:unknotting}. On the other hand, changing the 3 equivariant pairs of crossings so that the red arc $a_r$ always passes over the blue arc $a_b$ produces the unknot, so that $\widetilde{u}(12a_{1287},\rho) \leq 3$. It is interesting to compare this to the (non-equivariant) unknotting number for $12a_{1287}$, which is unknown (it is either 2 or 3). 
\begin{figure}
\scalebox{.4}{\includegraphics{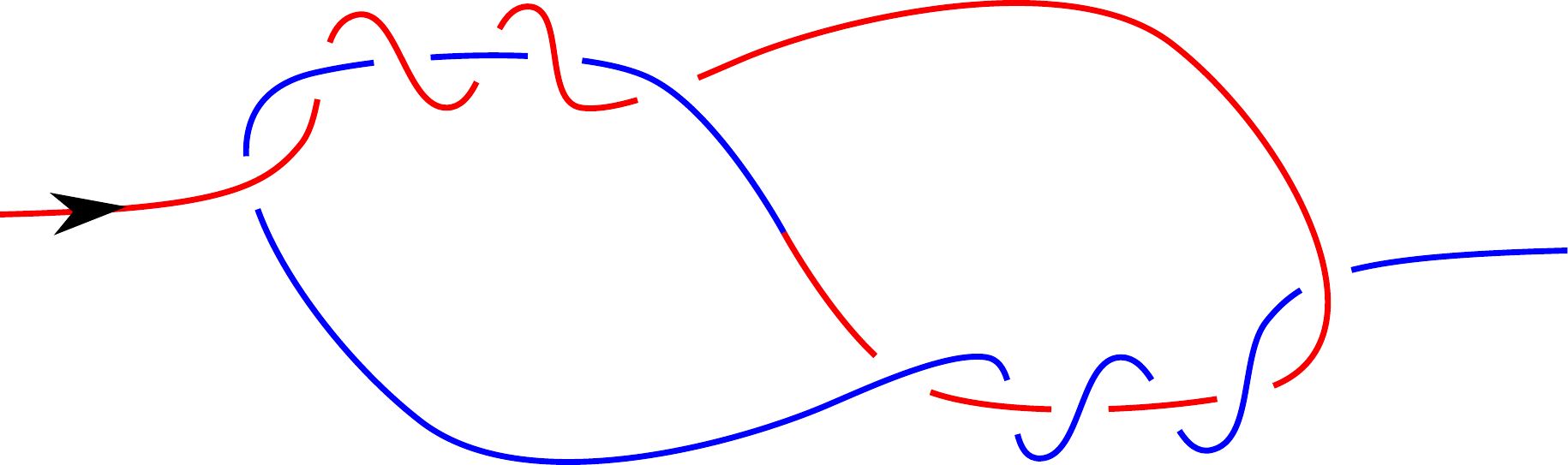}}
\caption{A strongly negative amphichiral diagram for $12a_{1287}$.}
\label{fig:12a_1287}
\end{figure}

\end{example}

\section{The half-Conway polynomial} \label{sec:half-Conway}
Let $(K,\rho)$ be an oriented strongly negative amphichiral knot. Then $\rho$ restricts to a free involution of $S^3 - \nu(K)$. The quotient $Q = (S^3 - \nu(K))/\rho$ is a non-orientable homology circle; that is, $H_*(Q;\mathbb{Z}) = H_*(S^1;\mathbb{Z})$ and $H_3(Q,\partial Q;\mathbb{Z}) = 0$ (see \cite[Lemma 1]{MR543095} for details). In particular, one can define the Alexander polynomial of $Q$ after choosing a generator of $H_1(Q;\mathbb{Z})$; see Kawauchi \cite{MR391110}. In the case of the quotient of a strongly negative amphichiral knot exterior, an oriented meridian of $K$ descends to a generator of $H_1(Q;\mathbb{Z})$ so that an orientation on $K$ specifies a generator of $H_1(Q;\mathbb{Z})$ by the right-hand rule. The Alexander polynomial of $Q$ is then the principal generator of the first elementary ideal of $H_1$ of the infinite cyclic cover, thought of as a $\mathbb{Z}[\tau,\tau^{-1}]$-module. Unlike the usual Alexander polynomial, it is important that the orientation of $K$ specifies a generator $\tau$ (as opposed to $\tau^{-1}$) since the polynomial need not be invariant under $\tau \mapsto \tau^{-1}$. 

\begin{definition} \label{def:KawauchiFactor}
The \emph{half-Alexander polynomial} $\Delta_{(K,\rho)}(\tau)$ of $(K,\rho)$ is the Alexander polynomial of $Q = (S^3 - \nu(K))/\rho$.
\end{definition}
Similarly to the usual Alexander polynomial, $\Delta_{(K,\rho)}(\tau)$ is defined only up to multiplication by a unit in $\mathbb{Z}[\tau,\tau^{-1}]$, which we indicate with the symbol $``\doteq"$. Clearly $\Delta_{(K,\rho)}(\tau)$ is an oriented equivariant isotopy invariant.

\begin{theorem}[{\cite[Theorem 1]{MR543095}}]\label{thm:Hartley-Kawauchi}
The half-Alexander polynomial $\Delta_{(K,\rho)}(\tau)$ satisfies the following.
\begin{enumerate}
\item $|\Delta_{(K,\rho)}(1)| = 1$,
\item $\Delta_{(K,\rho)}(\tau) \doteq \Delta_{(K,\rho)}(-\tau^{-1})$, and
\item $\Delta_{(K,\rho)}(\tau)\cdot \Delta_{(K,\rho)}(-\tau) \doteq \Delta_K(\tau^2)$.
  \end{enumerate}
\end{theorem}

In the following definition, we use Theorem \ref{thm:Hartley-Kawauchi} to resolve the ambiguity in $\Delta_{(K,\rho)}(\tau)$ of multiplication by a unit in $\mathbb{Z}[\tau,\tau^{-1}]$. This is analogous to the Conway normalization of the usual Alexander polynomial \cite{MR0258014} (see also \cite[Chapter 8]{MR1472978}).
\begin{definition} \label{def:half-Conway-normalization}
Let $(K,\rho)$ be an oriented strongly negative amphichiral knot. Then the \emph{half-Conway polynomial} $\nabla_{(K,\rho)}(z)$ is the unique polynomial of $z$ such that 
\begin{enumerate} 
\item $\nabla_{(K,\rho)}(\tau-\tau^{-1}) \doteq \Delta_{(K,\rho)}(\tau)$, where $\Delta_{(K,\rho)}(\tau)$ is the half-Alexander polynomial (see Definition \ref{def:KawauchiFactor}), and 
\item $\nabla_{(K,\rho)}(0) = 1$.
\end{enumerate}
\end{definition}
To see that $\nabla_{K,\rho}(z)$ is well-defined, note that Theorem \ref{thm:Hartley-Kawauchi}(2) implies that for some $a_i \in \mathbb{Z}$ and $n \geq 1$,
\[
\Delta_{(K,\rho)}(\tau) \doteq a_0 + \sum^n_{i=1} a_i(\tau^i + (-\tau^{-1})^i). 
\] 
In particular, by the fundamental theorem of symmetric polynomials, the right side can be written as a polynomial in $\tau\cdot(-\tau^{-1})$ and in $\tau - \tau^{-1}$, or just $\tau - \tau^{-1}$. Furthermore, Theorem \ref{thm:Hartley-Kawauchi}(1) allows us to pin down the overall sign of $\nabla_{K,\rho}(z)$ by specifying that $\nabla_{K,\rho}(0) = 1$. Note that $\nabla_{(K,\rho)}(z)$ is clearly an oriented equivariant isotopy invariant. 

In the second two parts of the following proposition, we show that the half-Conway polynomial $\nabla_{(K,\rho)}(z)$ has similar properties to the usual Conway polynomial $\nabla_K(z)$ (see \cite[Section 8]{MR1472978} for background about the Conway polynomial). In particular, we can describe the effect of orientation reversal (c.f. \cite[Proposition 6.12]{MR1472978}), and of connected sums. In Section \ref{sec:equivariantskeinrelation} we will see that the half-Conway also satisfies a skein relation.

\begin{proposition} \label{prop:half-Conway-properties}
Let $(K,\rho)$ be an oriented strongly negative amphichiral knot. Then the half-Conway polynomial $\nabla_{(K,\rho)}(z)$ has the following properties.
\begin{enumerate}
\item $\nabla_{(K,\rho)}(z)\cdot \nabla_{(K,\rho)}(-z) = \nabla_K(z)$.

\item $\nabla_{(K,\rho)}(z) = \nabla_{(rK,\rho)}(-z)$, where $rK$ is $K$ with the opposite orientation.
\item $\nabla_{(K \widetilde{\#} K',\rho \widetilde{\#} \rho')}(z) = \nabla_{(K,\rho)}(z)\cdot\nabla_{(K',\rho')}(z)$.
\end{enumerate}
\end{proposition}

\begin{proof}
For statement (1), recall that $\Delta_K(\tau^2) = \nabla_K(\tau - \tau^{-1})$. Now rewriting Theorem \ref{thm:Hartley-Kawauchi}(3) in terms of the Conway and half-Conway polynomial, we get
\[
\nabla_{(K,\rho)}(\tau - \tau^{-1})\cdot \nabla_{(K,\rho)}(-\tau + \tau^{-1}) = \nabla_K(\tau - \tau^{-1}),
\]
as desired. 

For statement (2), note that reversing the orientation on $K$ corresponds to exchanging $\tau$ and $\tau^{-1}$, and hence $\Delta_{(K,\rho)}(\tau) \doteq \Delta_{(rK,\rho)}(\tau^{-1})$. Then substituting for the half-Conway polynomial we have $\nabla_{(K,\rho)}(z) = \nabla_{(rK,\rho)}(-z)$.

For statement (3), let $\widetilde{Q}, \widetilde{Q}'$, and $\widetilde{Q}''$ be the infinite cyclic covers of $(S^3 - \nu(K))/\rho$, $(S^3 - \nu(K'))/\rho'$, and $(S^3 - \nu(K \# K'))/(\rho \# \rho')$ respectively. Consider the connected summing sphere $S$ for $K\#K'$. Lifting $S/\rho$ to $\widetilde{Q}''$, we have $\widetilde{S} \cong \mathbb{R} \times I$, which separates $\widetilde{Q}''$ into two pieces: one $\tau$-equivariantly homeomorphic to $\widetilde{Q}$ and one $\tau$-equivariantly homeomorphic to $\widetilde{Q}'$. Now since $\mathbb{R} \times I$ is contractible, the Mayer-Vietoris sequence applied to this decomposition of $\widetilde{Q}''$ gives that the $\mathbb{Z}[\tau,\tau^{-1}]$-module $H_1(\widetilde{Q}'')$ splits as a direct sum of $H_1(\widetilde{Q})$ and $H_1(\widetilde{Q}')$. Hence the half-Alexander polynomial is multiplicative under connected sum, and so is the half-Conway polynomial.
\end{proof}

\manysymmetries*
\begin{proof}
Let $K' = 10_{43}$ with $\rho$ the strongly negative amphichiral symmetry shown in Figure \ref{fig:10_43} and let $(rK',\rho)$ be its reverse. In Section \ref{sec:computations} we compute that $\nabla_{(K',\rho)}(z) = 1+z^2+z^3$. Note that $K'$ is reversible so that $K'$ and $rK'$ are (non-equivariantly) isotopic. For any $n \in \mathbb{N}$ we can take equivariant connected sums to obtain $(K_{i,j},\rho_{i,j}) := \widetilde{\#}^i (K',\rho) \widetilde{\#}^j(rK',\rho)$, where $i,j\geq 0$ with $i + j = n$. Each $K_{i,j}$ is isotopic to $\#^n K' = K$, but computing the half-Conway polynomials (see Proposition \ref{prop:half-Conway-properties}(2) and (3)) gives that
\[
\nabla_{(K_{i,j},\rho_{i,j})}(z) = (1+z^2+z^3)^i(1+z^2-z^3)^j
\] 
so that each of these $n+1$ symmetries is distinct.
\end{proof}

\begin{figure}
  \begin{overpic}[width=250pt, grid=false]{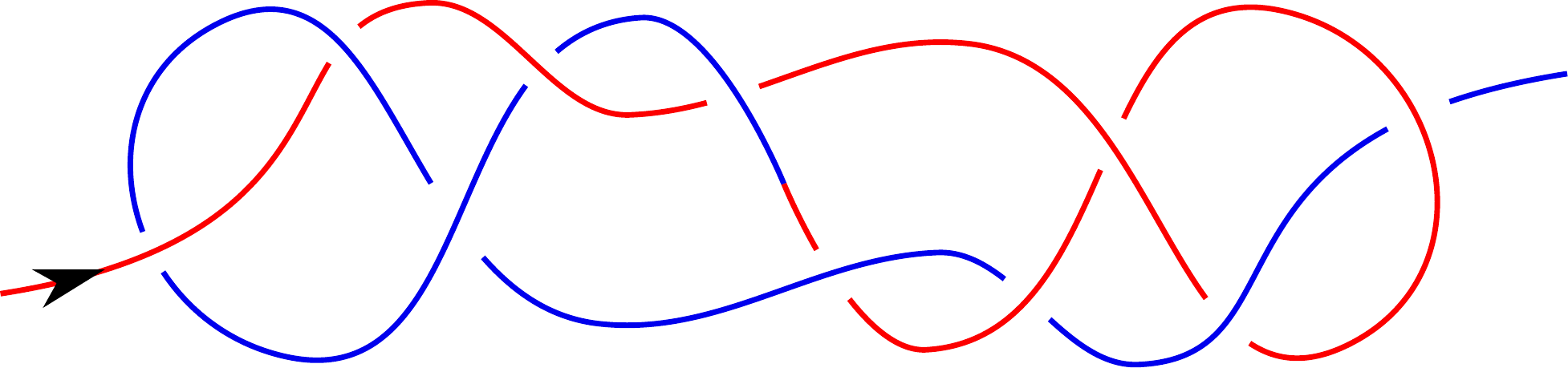}
  \end{overpic}
\caption{A directed strongly negative amphichiral symmetry on $10_{43}$ with half-Conway polynomial $1+z^2+z^3$.}
\label{fig:10_43}
\end{figure}

The following proposition will be an important first step towards computing the half-Conway polynomial for an arbitrary strongly negative amphichiral knot. 

\begin{proposition} \label{prop:Conway_connect_sum}
Let $J \subset S^3$ be any oriented knot, and let $\overline{J}$ be the reverse mirror of $J$ so that $J \# \overline{J}$ has a natural strongly negative amphichiral symmetry $\rho$. Then
\[
\nabla_{(J \# \overline{J},\rho)}(z) = \nabla_J(z),
\]
where $\nabla_J(z)$ is the usual Conway polynomial of $J$.
\end{proposition}
We postpone the proof of this proposition until the end of Section \ref{sec:surgery-presentations}, after we discuss symmetric surgery presentations.

\subsection{The equivariant skein relation} \label{sec:equivariantskeinrelation}
Before stating the main theorem, we set up some notation and conventions as follows. Recall that the fixed points separate $K$ into a pair of arcs $a_r$ and $a_b$, and recall that in Definition \ref{def:crossingpairsign} we assigned a sign to each equivariant pair of crossings. 

\begin{definition}
\label{def:skeintriad}
An \emph{equivariant skein triple} is a triple $K_+, K_-,$ and $K_0$ of strongly negative amphichiral knots such that
\begin{enumerate}
\item There are diagrams $D_+,D_-$, and $D_0$ for $K_+,K_-$, and $K_0$ respectively which are identical outside of an equivariant pair of disks.
\item Within the equivariant pair of disks, $D_+$ is a positive dichromatic crossing pair, $D_-$ is a negative dichromatic crossing pair, and $D_0$ is the unique oriented resolution of that crossing pair; see Figure \ref{fig:skeinrelation}.
\end{enumerate}
\end{definition}

\begin{figure}
  \begin{overpic}[width=250pt, grid=false]{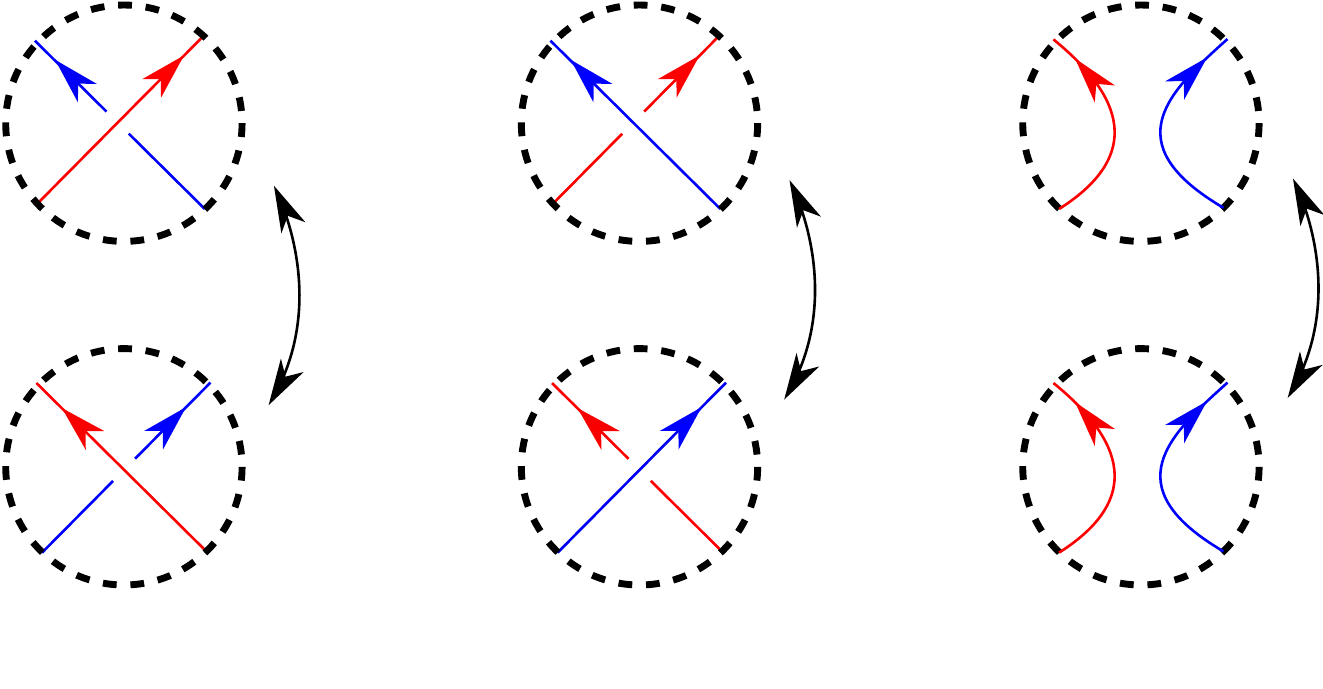}
    \put (24, 29) {$\rho$}
    \put (63, 29) {$\rho$}
    \put (101, 29) {$\rho$}
    \put (-6, 40) {$c_+$}
    \put (-12, 12) {$\rho(c_+)$}
    \put (33, 40) {$c_-$}
    \put (27, 12) {$\rho(c_-)$}
    \put (7,0) {$D_+$}
    \put (46,0) {$D_-$}
    \put (83,0) {$D_0$}
  \end{overpic}
\caption{The diagrams $D_+$, $D_-$, and $D_0$ are identical outside of the equivariant pair of disks shown. Here $c_+> \rho(c_+)$ and $c_- > \rho(c_-)$. Note that we require $(c_+,c_-)$ to be a dichromatic crossing pair; see Definition \ref{def:chromatic}.}
\label{fig:skeinrelation}
\end{figure}
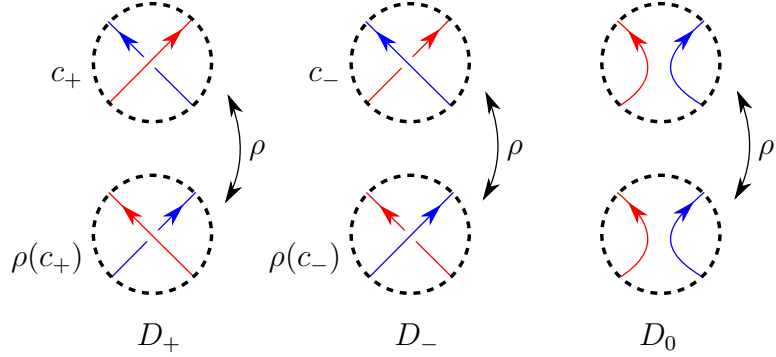

\skeinrelation*
We defer the proof of Theorem \ref{thm:skeinrelation} until Section \ref{sec:proofofmaintheorem}.
\begin{remark}
In contrast with the skein relation for the usual Conway polynomial, the resolution of a dichromatic crossing pair is a knot rather than a link. One might wish to extend this skein relation to all equivariant pairs of crossings. However, to do so requires consideration of links, which arise when resolving monochromatic crossing pairs. Since there is no clear way to assign a sign to an equivariant pair of crossings in a link, we only define the skein relation on dichromatic crossing pairs. 
\end{remark}

Although the equivariant skein relation is not defined for every pair of crossings, we can still use it to compute the half-Conway polynomial $\nabla_{(K,\rho)}(z)$ for any strongly negative amphichiral knot $K$ as follows. Iteratively apply the equivariant skein relation to obtain a skein decomposition of $K$ such that in every element in the decomposition the arc $a_r$ always passes over the arc $a_b$. In such a diagram, we can perform a small equivariant isotopy to pull $a_r$ above the plane of the diagram, and $a_b$ behind the plane of the diagram, so that the plane of the diagram decomposes the knot as $J \# \overline{J}$ for some $J$, and $\nabla_{(J\#\overline{J},\rho)}(z) = \nabla_J(z)$; see Proposition \ref{prop:Conway_connect_sum}. The half-Conway polynomials of the summands in the skein decomposition then determine the half-Conway polynomial of $K$ by Theorem \ref{thm:skeinrelation}. An example computation is worked out in Figure \ref{fig:8_17skein}.

\begin{figure}
\scalebox{.3}{\includegraphics{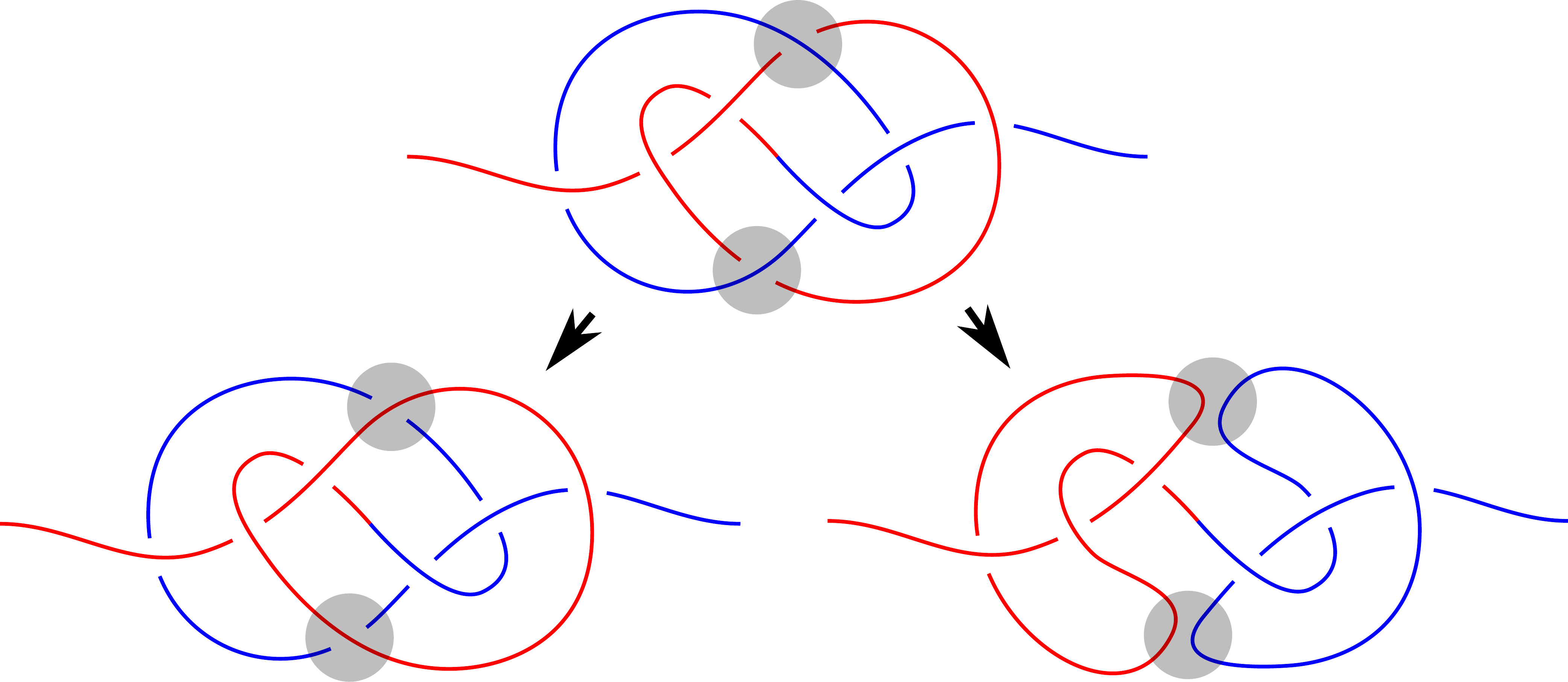}}
\caption{A strongly negative amphichiral diagram for $K_+ = 8_{17}$ (top), and the strongly negative amphichiral diagrams obtained by a crossing change ($K_-$ on the bottom left) and an oriented resolution ($K_0$ on the bottom right) on the indicated symmetric pair of crossings. Here $K_-$ is the unknot, and $K_0$ is $3_1 \# \overline{3_1}$, so that $\nabla_{(K_+,\rho)}(z) = \nabla_{(K_-,\rho)}(z) + z \cdot \nabla_{(K_0,\rho)}(z) = 1 + z(1+z^2) = 1+z+z^3$.}
\label{fig:8_17skein}
\end{figure}

In addition to providing an easy computational method, this equivariant skein relation also allows us to relate the half-Conway polynomial and the Arf invariant to the half-linking number of $K$.

\halfConwayhalflinkingarf*

\begin{proof}
We first prove claim (1). Let $a_1(K)$ denote the $z$-coefficient of $\nabla_{(K,\rho)}(z)$. First, consider strongly negative amphichiral knots of the form $J\#\overline{J}$ for some knot $J$. By definition, $h(J\#\overline{J}) = 0$, and by Proposition \ref{prop:Conway_connect_sum}, $a_1(J\#\overline{J}) = 0$ since the Conway polynomial $\nabla_{J}(z)$ is a polynomial in $z^2$. Next, note that the half-linking number $h(K)$ satisfies the relation
\begin{equation} \label{eqn:hcrossingchange}
h(K_+) - h(K_-) = 1,
\end{equation}
and Theorem \ref{thm:skeinrelation} (along with Definition \ref{def:half-Conway-normalization}) implies that 
\begin{equation}\label{eqn:a_1crossingchange}
a_1(K_+) - a_1(K_-) = 1.
\end{equation}
Now any strongly negative amphichiral knot $K$ admits a finite sequence of dichromatic crossing changes which transforms it to a knot of the form $J\#\overline{J}$, and by Equation \ref{eqn:hcrossingchange} and Equation \ref{eqn:a_1crossingchange}, both $h(K)$ and $a_1(K)$ are equal to the signed count of these crossing changes. Hence $h(K) = a_1(K)$.

For claim (2), recall that the $z^2$-coefficient of the Conway polynomial is the Arf invariant (mod $2$) \cite[Corollary 10.8]{MR712133} and by Proposition \ref{prop:half-Conway-properties}, 
\[
\nabla_{(K,\rho)}(z) \cdot \nabla_{(K,\rho)}(-z) = \nabla_K(z).\]
Letting $\nabla_{(K,\rho)}(z) = 1+ a_1z + a_2z^2 + \dots$, we have $\nabla_{(K,\rho)}(-z) = 1 - a_1z + a_2z^2 - \dots$ and hence
\[
\nabla_K(z) = 1 + (2a_2 - a_1^2)z^2 + \dots.
\]
Then since $(2a_2 - a_1^2) \equiv a_1$ (mod $2$), we have that $\textup{Arf}(K) \equiv a_1 \equiv h(K)$ (mod $2$).
\end{proof}

\section{Computing the half-Alexander polynomial with a surgery presentation} \label{sec:surgery-presentations}

In this section we use a symmetric surgery description of a strongly negative amphichiral knot to obtain an equivariant presentation for the Alexander module from which we can compute the half-Alexander polynomial. This is analogous to the non-equivariant discussion in \cite[Section 7]{MR1472978}.

We begin with some notation. Let $(K,\rho)$ be an oriented strongly negative amphichiral knot and let $X = S^3 -\nu(K)$ be the exterior of $K$ with quotient $Q = (S^3 - \nu(K))/\rho$. Let $X_{\infty}(K)$ be the infinite cyclic cover of $K$. Let $t$ be the generator of the deck transformation group specified by the orientation on $K$. Then $H_1(X_{\infty};\mathbb{Z})$ is naturally a $\mathbb{Z}[t,t^{-1}]$-module. The following proposition is implicit in the work of Kawauchi \cite{MR433474}, but we produce a proof here for convenience.

\begin{proposition}
The symmetry $\rho$ lifts to a homeomorphism $\tau\colon X_{\infty}(K) \to X_{\infty}(K)$ such that $\tau^2=t$. Furthermore, $X_{\infty}(K) \to X \to Q$ is the infinite cyclic cover of $Q$, and $\tau$ is the generator of the deck transformation group compatible with the orientation on $K$. In particular, $\tau$ is unique and the $\mathbb{Z}[\tau,\tau^{-1}]$-module $H_1(X_{\infty}(K);\mathbb{Z})$ is the Alexander module of $Q$.
\end{proposition}

\begin{proof}
Let $m$ be an oriented meridian of $K$ such that $lk(m,K)=1$; we may choose $m$ such that $\rho(m)=m$. Let $\pi_\rho$ denote the quotient map $X\rightarrow Q$. Note that $\pi_\rho$ is the orientation cover of $Q$ since it is an orientable double cover. Let $l = \pi_{\rho}(m)$ be the oriented loop in $Q$ with $({\pi_\rho})_*([m])=2[l]\in H_1(Q;\mathbb{Z})$. After choosing a basepoint $q \in l$, the covering space correspondence produces from $[l] \in \pi_1(Q,q)$ a deck transformation on $X \to Q$ and a deck transformation of $X_{\infty}(K) \to X \to Q$. By construction, the deck transformation of $X \to Q$ is $\rho$, and we call the other deck transformation $\tau\colon X_{\infty}(K) \to X_{\infty}(K)$. Clearly $\tau$ is a lift of $\rho$, and it remains to show that $\tau^2=t$. Since $\tau$ is a lift of $\rho$, we have that $\tau^2$ is a lift of $\rho^2 = id_X$, and hence that $\tau^2$ must be a deck transformation of the covering map $X_\infty(K)\rightarrow X$. Explicitly choosing a basepoint $p \in \pi^{-1}(q)$, we see that both $\tau^2$ and $t$ correspond to $[m] \in \pi_1(X,p)$ and hence $\tau^2 = t$.

To see that $X_{\infty}(K) \to X \to Q$ is the infinite cyclic cover of $Q$, note that $X_\infty(K)/\tau=(X_\infty(K)/\tau^2)/\rho=X/\rho=Q$, and in particular $\tau$ generates the (infinite cyclic) deck transformation group. It also follows from the construction in the previous paragraph that $\tau$ is compatible with the orientation of $K$.

Next we show that $\tau$ is the unique lift of $\rho$ satisfying $\tau^2=t$. Note that all lifts of $\rho$ are deck transformations of $X_\infty(K)\rightarrow Q$ and hence they are of the form $\tau^n$ for some $n\in \mathbb{Z}$. Clearly the only lift that squares to $t$ is $\tau$. The final statement, that the $\mathbb{Z}[\tau,\tau^{-1}]$-module $H_1(X_\infty(K);\mathbb{Z})$ is the Alexander module of $Q$, follows from the second part of the statement.
\end{proof}

We now describe how to compute the Alexander module of $Q$ in terms of a symmetric surgery description of $X_{\infty}(K)$. Here a symmetric surgery diagram of $K$ refers to a strongly negative amphichiral unknot $U$ along with a symmetric framed link $L$ such that surgery along $L$ gives $S^3$ and $U$ becomes $K$. We further require that the linking number of each component of $L$ with $U$ is $0$ so that the link lifts to $X_{\infty}(U)$. See for example Figures \ref{fig:4_1surgery} and \ref{fig:det1snak} (left). 

We now describe one way to obtain a symmetric surgery description. Start with a symmetric diagram $(D,\rho)$, and let $\{c_i,\rho(c_i)\}$ be a symmetric collection of crossings such that changing all of these crossings produces an unknot (see Proposition \ref{prop:finiteunknotting}). We now change the crossings $\{c_i,\rho(c_i)\}$ and introduce a pair of framed surgery circles $\{\alpha_i, \beta_i\}$; see Figure \ref{fig:4_1surgery} for the case of the figure-eight knot. The resulting unknot $U$, along with $\{\alpha_i,\beta_i\}$, produces a symmetric description of $K$. Note that $lk(U,\alpha_i) = lk(U,\beta_i) = 0$. 

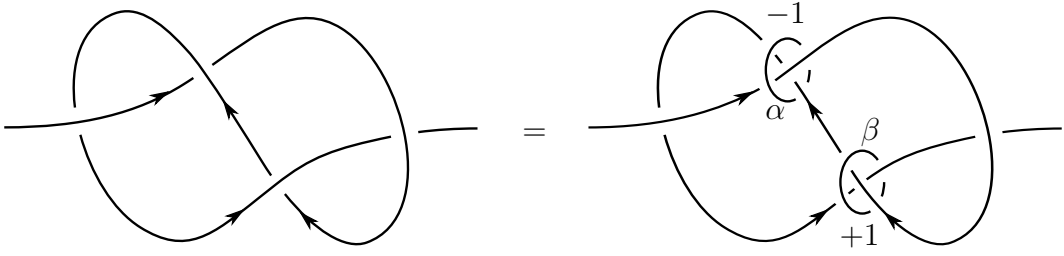
\begin{figure}
  \begin{overpic}[width=400pt, grid=false]{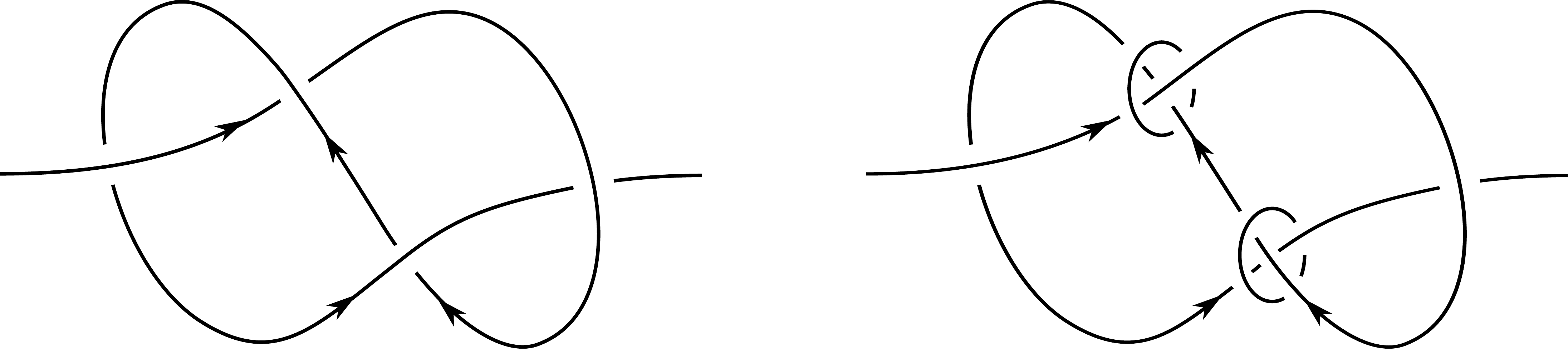}
  \put (72,21) {$-1$}
  \put (72,12) {$\alpha$}
  \put (79,0) {$+1$}
  \put (81,10) {$\beta$}
  \put (49,10) {$=$}
  \end{overpic}
\caption{An amphichiral link for which $-1$-surgery on the $\alpha$ component and the symmetric $+1$-surgery on the $\beta = \rho(\alpha)$ component produces the strongly negative amphichiral figure-eight knot.}
\label{fig:4_1surgery}
\end{figure}

We can now lift the framed link $\bigcup_{i}(\alpha_i \cup \beta_i)$ to the infinite cyclic cover of $U$ to obtain a surgery description for $X_{\infty}(K)$. Since $U$ is an unknot, the infinite cyclic cover is $D^2 \times \mathbb{R}$ which we project to $D^1 \times \mathbb{R}$; see Figure \ref{fig:4_1surgerysimplified} and \ref{fig:4_1surgerycover} for the case of the figure-eight knot. Note that $\tau$ acts as the composition of a horizontal translation with a reflection across $D^1 \times \mathbb{R}$, taking a lift of $\alpha_i$ to a lift of $\beta_i$, and a lift of $\beta_i$ to a lift of $\alpha_{i}$.

Next we write a presentation matrix for the $\mathbb{Z}[\tau,\tau^{-1}]$-module $H_1(X_{\infty}(K);\mathbb{Z})$. To do so, we choose for each $i$ a lift $\overline{\alpha}_i$ of $\alpha_i$, and an oriented meridian $m_i$ of $\overline{\alpha_i}$. Then $H_1(X_{\infty}(K);\mathbb{Z})$ is generated over $\mathbb{Z}[\tau,\tau^{-1}]$ by $\{m_i\}$. Each symmetric pair of surgery circles $(\alpha_i,\beta_i)$ then gives a relator as follows. Orient $\overline{\alpha}_i$ so that the linking with $m_i$ is $+1$ and let $l_i$ be the oriented parallel push-off of $\overline{\alpha}_i$ representing the framing. Now define 
\begin{equation}\label{eqn:matrix-entries}
f_{i,j}(\tau) = \sum_k lk(l_i,\tau^k\overline{\alpha}_j)(-\tau)^k \in \mathbb{Z}[\tau,\tau^{-1}],
\end{equation}
and the relator corresponding to $(\alpha_i,\beta_i)$ is
\[
r_i = \sum_j f_{i,j}(\tau)m_j. 
\]
The $(i,j)$th entry of the presentation matrix $M$ is then given by $f_{i,j}$. Since $M$ is a square matrix, we have $$\Delta_{(K,\rho)}(\tau) \doteq \det(M).$$ Note that $M$ has an unusual symmetry: $f_{i,j}(\tau) = f_{j,i}(-\tau^{-1})$. Indeed, using the deck transformation $\tau^{-k}$, we have
\[
lk(l_i,\tau^k\overline{\alpha}_j) = (-1)^k\cdot lk(\tau^{-k}l_i,\overline{\alpha_j}) = lk((-\tau)^{-k}\overline{\alpha}_i,l_j),
\]
where the $(-1)^k$ factor appears because $\tau$ is orientation reversing. This symmetry implies $\det(M)(\tau) = \det(M)(-\tau^{-1})$ and hence $\det(M)$ already gives rise to a symmetrized half-Alexander polynomial, but there is still an overall sign ambiguity. 

\begin{example}
We will compute the half-Alexander polynomial for the strongly negative amphichiral symmetry $\rho$ on the figure-eight knot $K$ shown on the left in Figure \ref{fig:4_1surgery}. Note that $K$ can be equivariantly unknotted with two crossing changes so that we have a symmetric surgery diagram for $K$ shown on the right in Figure \ref{fig:4_1surgery}. We can then simplify this link with an equivariant isotopy to obtain Figure \ref{fig:4_1surgerysimplified}. From here, we can lift the link to the infinite cyclic cover; see Figure \ref{fig:4_1surgerycover}. Note that the $\mathbb{Z}[\tau,\tau^{-1}]$-module $H_1(X_{\infty}(K);\mathbb{Z})$ is generated by a single element, the meridian of $\overline{\alpha}$. Furthermore, there is a single relator which we now compute. Since the framing on $\overline{\alpha}$ is $-1$, we have that $lk(l_0,\overline{\alpha}) = -1$, and also
\begin{align*}
lk(\overline{\alpha},\tau\overline{\alpha}) &= lk(\overline{\alpha},\overline{\beta}) = 1, \mbox{ and } \\
lk(\overline{\alpha},\tau^{-1}\overline{\alpha}) &= lk(\overline{\alpha},t^{-1}\overline{\beta}) = -1.
\end{align*}
When computing the linking numbers above, we can choose an arbitrary orientation on $\overline{\alpha}$, and applying $\tau$ or $\tau^{-1}$ will produce an orientation on $\overline{\beta}$ or $t^{-1}\overline{\beta}$. In Figure \ref{fig:4_1surgerycover} this corresponds to all curves being oriented clockwise, or all oriented counterclockwise. As a result we have the relator $\tau^{-1} -1 - \tau$ (see Equation \ref{eqn:matrix-entries}), and therefore the half-Alexander polynomial is 
\[
\Delta_{(K,\rho)}(\tau) \doteq -\tau^{-1}  +1 +\tau.
\]
\end{example}

\begin{figure}
\begin{overpic}[width = 220pt, grid = false]{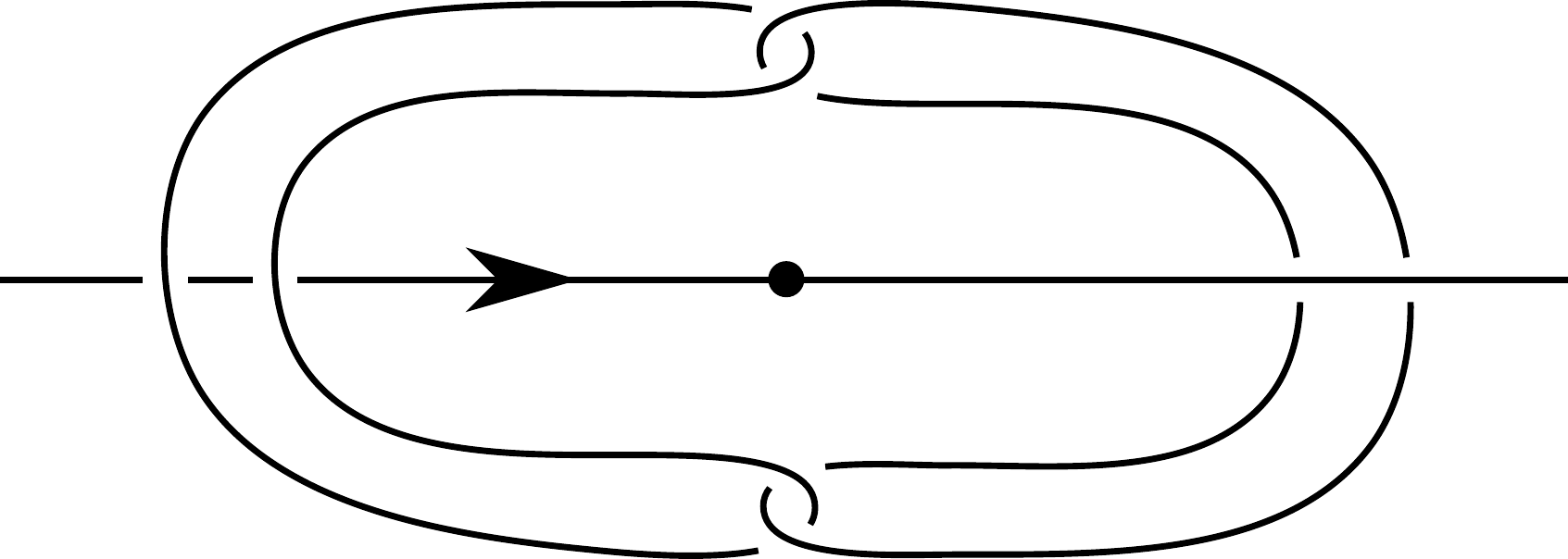}
  \put (14,2) {$\alpha$}
  \put (8,32) {$-1$}
  \put (87,2) {$\beta$}
  \put (82,32) {$+1$}
\end{overpic}
\caption{The symmetric link from Figure \ref{fig:4_1surgery}, simplified equivariantly so that the unknot $U$ has no crossings.}
\label{fig:4_1surgerysimplified}
\end{figure}

\begin{figure}
  \begin{overpic}[width=250pt, grid=false]{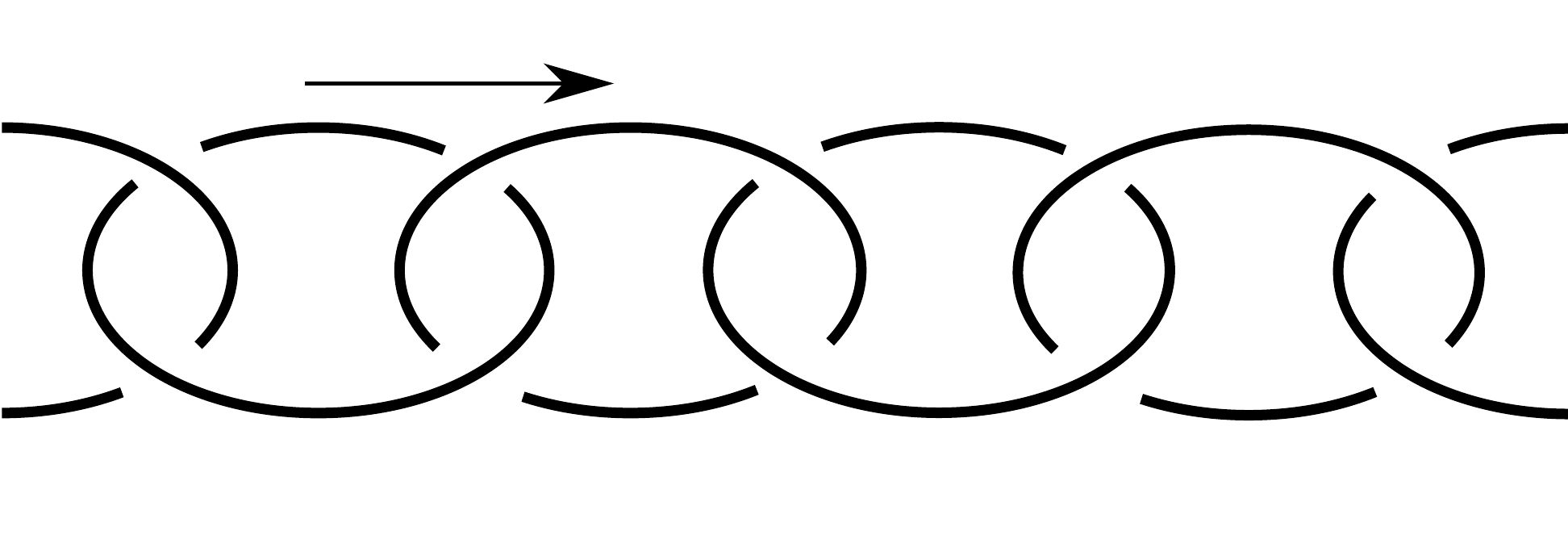}
    \put (28, 32) {$\tau$}
    \put (15,21) {$t^{-1}\overline{\beta}$}
    \put (38.5,21) {$\overline{\alpha}$}
    \put (58,21) {$\overline{\beta}$}
    \put (77.5,21) {$t\overline{\alpha}$}
    \put (100,18) {$\dots$}
    \put (-8,18) {$\dots$}

    \put (17,5) {$+1$}
    \put (37,5) {$-1$}
    \put (57,5) {$+1$}
    \put (77,5) {$-1$}
  \end{overpic}
\caption{The lifts of $\alpha$ and $\beta$ in Figure \ref{fig:4_1surgerysimplified} to the infinite cyclic cover of $U$. Here $-1$-surgery on the $t^n\overline{\alpha}$ components and $+1$-surgery on the $t^n\overline{\beta}$ components produces the infinite cyclic cover of the figure-eight knot, and $\overline{\alpha}$ and $\overline{\beta}$ are lifts from Figure \ref{fig:4_1surgerysimplified} of $\alpha$ and $\beta$ respectively. Note that $\tau^2 = t$.}
\label{fig:4_1surgerycover}
\end{figure}
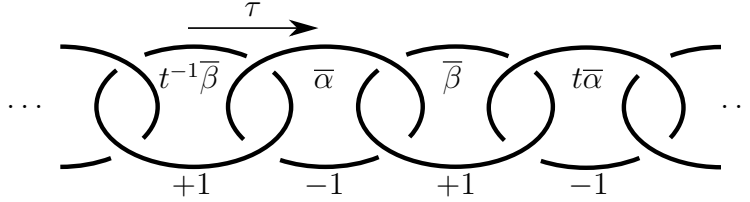
\subsection{Proof of Proposition \ref{prop:Conway_connect_sum}}
Now that we can compute the half-Alexander polynomial from a symmetric surgery presentation, we are prepared to give the proof of Proposition \ref{prop:Conway_connect_sum}.
\begin{proof}[Proof of Proposition \ref{prop:Conway_connect_sum}]
Converting this back to a statement about the half-Alexander polynomial by setting $z = \tau - \tau^{-1}$, we have the equivalent statement
\[
\Delta_{(J\#\overline{J},\rho)}(\tau) \doteq \Delta_J(\tau^2).
\]
To prove this, we will compare surgery descriptions for the infinite cyclic covers $X_{\infty}(J)$ and $X_{\infty}(J\#\overline{J})$. Let $\rho$ be the strongly negative amphichiral symmetry on $J \#\overline{J}$ stated in the proposition. Let $\Sigma$ be a connected summing sphere which is $\rho$-invariant, separating $S^3$ into two 3-balls: $B_+$ containing $J$ and $B_-$ containing $\overline{J}$. Note that $(J\#\overline{J}) \cap \Sigma$ is exactly the two fixed points of $\rho$. Start with a surgery presentation $\{\alpha_1,\alpha_2,\dots,\alpha_n; U\}$ for $J$, where each $\alpha_i$ is a framed circle and $U$ is the unknot. Removing a neighborhood of a point on $J$ where we perform the connected sum, this also gives a surgery description for $(B_+,J \cap B_+)$. Applying the symmetry $\rho$, we then have a symmetric surgery description $\{\alpha_1, \alpha_2,\dots, \alpha_n, \rho(\alpha_1), \rho(\alpha_2), \dots, \rho(\alpha_n); U\}$ for $J\#\overline{J}$.

Lifting this to the infinite cyclic cover $X_{\infty}(U)$ we get a surgery presentation of the infinite cyclic cover $X_{\infty}(J\#\overline{J})$. The lift of the connected summing sphere $\Sigma$ is $\mathbb{R} \times I$ which separates $X_{\infty}(U)$ into the infinite cyclic cover $X_{\infty}(B_+)$ of $B_+$ over $U \cap B_+$ and the infinite cyclic cover $X_{\infty}(B_-)$ of $B_-$ over $U \cap B_-$. For each $i$, let $\overline{\alpha_i}$ be a lift of $\alpha_i$ to $X_{\infty}(U)$. Now $X_{\infty}(B_+)$ along with the lifts $\displaystyle \bigsqcup_{i,j}\tau^{2j}\overline{\alpha}_i$ is a surgery presentation for $X_{\infty}(J)$, where $\tau^2 = t$ is the generator of the deck transformation group. Applying the lift $\tau$ of the symmetry $\rho$, we have that $X_{\infty}(B_-)$ along with the lifts $\displaystyle \bigsqcup_{i,j} \tau^{2j+1}\overline{\alpha}_i$ of the surgery curves $\rho(\alpha_i)$ is a surgery presentation for $X_{\infty}(\overline{J})$. We then compute $\Delta_J(t) \doteq \det M$, where $M$ is the $n\times n$ matrix whose $(i,j)$ entry is
\begin{equation} \label{eqn:MentriesConwayconnectsum}
\sum_k lk(\overline{\alpha}_i,t^k\overline{\alpha}_j)\cdot t^k \in \mathbb{Z}[t,t^{-1}].
\end{equation}
Here $lk(\overline{\alpha}_i,\overline{\alpha}_i)$ is understood to be the linking of $\overline{\alpha}_i$ with its appropriately framed longitude. On the other hand, the half-Alexander polynomial $\Delta_{(J\#\overline{J},\rho)}(\tau) \doteq \det N$, where $N$ is the $n\times n$ matrix whose the $(i,j)$ entry is
\[
\sum_k lk(\overline{\alpha}_i,\tau^k\overline{\alpha}_j)\cdot (-\tau)^k \in \mathbb{Z}[\tau,\tau^{-1}].
\]
Note that $lk(\overline{\alpha}_i,\tau^{2m+1}\overline{\alpha}_j) = 0$ for all $m$ since $\overline{\alpha}_i$ and $\tau^{2m+1}\overline{\alpha}_j$ are separated by the lift of the connected summing sphere $\Sigma$. Therefore the $(i,j)$ entry of $N$ is 
\begin{equation} \label{eqn:NentriesConwayconnectsum}
\sum_k lk(\overline{\alpha}_i,(\tau^2)^k\overline{\alpha}_j)\cdot (\tau^2)^k.
\end{equation}
Comparing Equations \ref{eqn:MentriesConwayconnectsum} and \ref{eqn:NentriesConwayconnectsum} we see that $\Delta_{J}(\tau^2) \doteq \det M \big{|}_{t = \tau^2} \doteq \det N \doteq \Delta_{(J\#\overline{J},\rho)}(\tau)$.
\end{proof}

\section{Realizations of half-Conway polynomials} \label{sec:realization}

In this section we characterize which polynomials can be realized as the half-Alexander polynomial of a strongly negative amphichiral knot. Work of Kawauchi \cite[Theorem 1.6]{MR391110} implies that any polynomial $f(\tau) \in \mathbb{Z}[\tau,\tau^{-1}]$ with $f(\tau) \doteq f(-\tau^{-1})$ and $|f(1)| = 1$ is the half-Alexander polynomial of a strongly negative amphichiral knot in an integer homology $3$-sphere. However, it was unknown if these polynomials can be realized by knots in $S^3$ (see \cite[Remark (3)]{MR543095}). We show that this is the case by directly constructing such knots using an equivariant version of the construction in \cite{MR180964}. The precise statement follows.

\begin{theorem} \label{thm:realization}
Let $\displaystyle f(\tau) = \sum_{i=-N}^N a_i\tau^i \in \mathbb{Z}[\tau,\tau^{-1}]$ such that 
\begin{enumerate}
\item $a_i = (-1)^i a_{-i}$, and
\item $\displaystyle \sum_i a_i = 1$.
\end{enumerate}
Then there is a strongly negative amphichiral knot $(K,\rho)$ such that $\Delta_{(K,\rho)}(\tau) \doteq f(\tau)$.
\end{theorem}

\begin{proof}
We directly construct a knot $(K,\rho)$ with $\Delta_{(K,\rho)}(\tau) \doteq f(\tau)$ via a symmetric surgery description as follows. Let $b_i = a_{2i}$ for $1 \leq i \leq k = \lfloor \frac{N}{2} \rfloor$. Likewise, let $c_i = -a_{2i-1}$ for $1 \leq i \leq l = \lfloor\frac{N+1}{2}\rfloor$. Then the knot $(K,\rho)$ is shown in Figure \ref{fig:halfconwayrealization}. First, observe that $K$ is a strongly negative amphichiral knot in $S^3$, since the blue and red surgery curves form a $\rho$-invariant 2-component unlink. Indeed, ignoring $K$, the $b_i$ crossing boxes each untwist and then the $c_i$ crossing boxes cancel in pairs. 

Now let $\alpha$ be the $+1$-surgery curve and $\beta$ be the $-1$-surgery curve, and choose a lift $\overline{\alpha}$ of $\alpha$ to the infinite cyclic cover of the unknot. Then the $i$th coefficient of $\Delta_{(K,\rho)}(\tau)$ is $lk(\overline{\alpha},\tau^i\overline{\alpha}),$ for $i \neq 0$, since $H_1(X_{\infty};\mathbb{Z})$ is $\mathbb{Z}[\tau,\tau^{-1}]$ modulo a single relation (see Equation \ref{eqn:matrix-entries}). By examining a surgery diagram for the infinite cyclic cover of $K$, we can see that $lk(\overline{\alpha},\tau^{2i}\overline{\alpha}) = b_i$ and that $lk(\overline{\alpha},\tau^{2i-1}\overline{\alpha}) = c_i$. Specifically, observe that the only crossings between $\overline{\alpha}$ and $\tau^{2i}\overline{\alpha}$ are the lifts of the crossings in the twisting region labeled $-b_i$. Indeed, the two arcs of $\alpha$ disjoint from the $-b_i$ twisting region link the knot $i$ and $-i$ times, and $\tau^{2i}$ is the $i$th power of the deck transformation. Note that the crossings in the $-b_i$ regions are positive crossings since the two arcs are oriented in opposite horizontal directions. (This computation is identical to the non-equivariant case in \cite[Section 5]{MR180964}; see in particular Figures 1 and 3 of \cite{MR180964}.)

Similarly, the only crossings between $\overline{\alpha}$ and $\tau^{2i-1}\overline{\alpha}$ are the lifts of the crossings in the twisting region labeled $-c_i$. See Figures \ref{fig:4_1surgerysimplified} and \ref{fig:4_1surgerycover} for the case where $c_1 = 1$ and all other coefficients are $0$. Generalizing this, the arc of $\alpha$ between the $-c_1$ and $-c_i$ twisting regions links the knot $i-1$ times so that the $-c_i$ twisting region contributes to the linking number between $\overline{\alpha}$ and $\tau^{2i-1}\overline{\alpha}$. Finally, the constant term of $\Delta_{(K,\rho)}(\tau)$ is forced by condition (2). As a result, we have that $\Delta_{(K,\rho)}(\tau) \doteq f(\tau)$.
\end{proof}

\begin{figure}
  \begin{overpic}[width=400pt, grid=false]{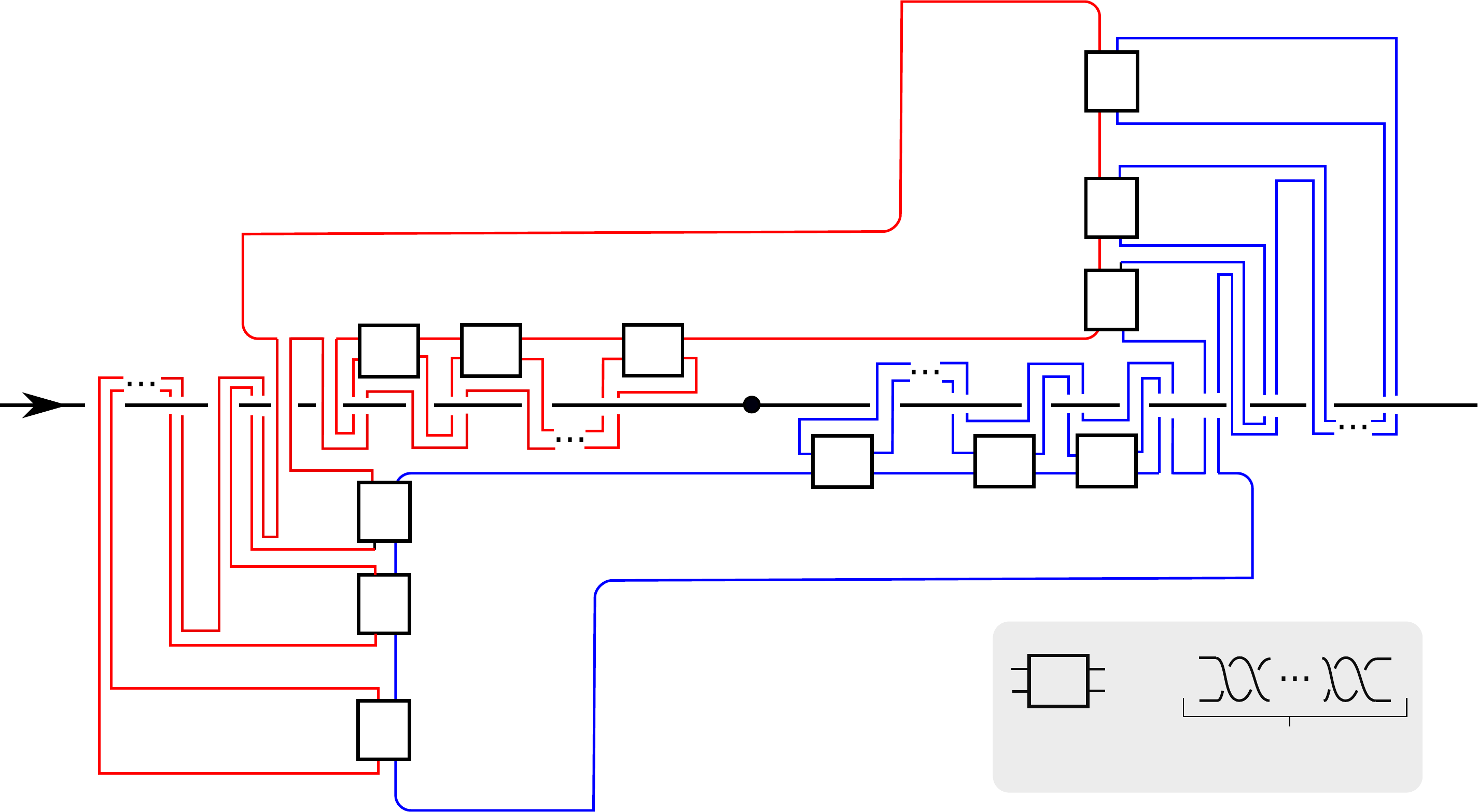}
  \put (70.7,8.3) {$n$}
  \put (80,3) {$n$ full twists}
  \put (76.5,8.3) {$=$}
  \put (41,41) {$+1$}
  \put (51,13) {$-1$}

  \put (0,30) {$K$}

  \put (25,30.8) {\tiny-$b_1$}
  \put (32,30.8) {\tiny-$b_2$}
  \put (43,30.8) {\tiny-$b_k$}

  \put (74,23.3) {\tiny$b_1$}
  \put (67.1,23.3) {\tiny$b_2$}
  \put (56,23.3) {\tiny$b_k$}

  \put (25,5.1) {\tiny-$c_l$}
  \put (24.8,13.7) {\tiny-$c_2$}
  \put (24.9,20) {\tiny-$c_1$}

  \put (74.5, 34.3) {\tiny $c_1$}
  \put (74.5,40.5) {\tiny $c_2$}
  \put (74.5, 49) {\tiny$c_l$}
  \end{overpic}
\caption{A symmetric surgery description of a strongly negative amphichiral knot $K$ which has a prescribed half-Alexander polynomial.}
\label{fig:halfconwayrealization}

\end{figure}
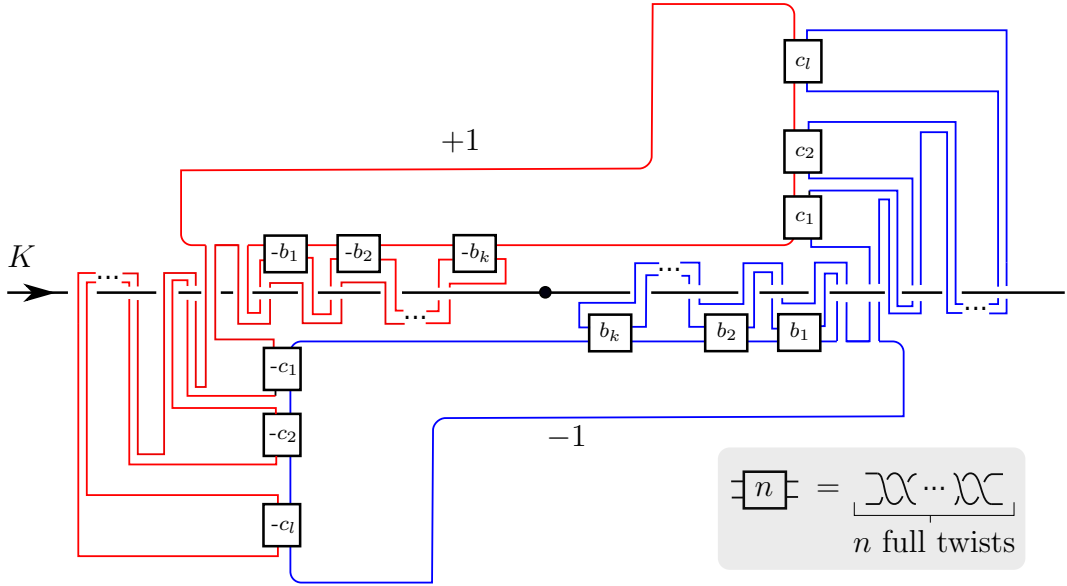

Using Theorem \ref{thm:realization}, we can construct strongly negative amphichiral knots with a specific Alexander polynomial. In particular, we can construct strongly negative amphichiral knots for which the determinant is 1 and simultaneously the Fox-Milnor condition obstructs sliceness.

\begin{proof}[Proof of Corollary \ref{cor:det1snak}]
Let $K$ be a strongly negative amphichiral knot with half-Alexander polynomial $n\tau^{-3} + n\tau^{-1} + 1 - n\tau - n\tau^3$ for either $n = 1$ or $n \geq 3$ prime. The case $n = 1$ is shown in Figure \ref{fig:det1snak}. Let $\Delta_K(t)$ be the Alexander polynomial of $K$. Then by Theorem \ref{thm:Hartley-Kawauchi}, 
\begin{align*}
\Delta_K(\tau^2) &=(n\tau^{-3} +n\tau^{-1} + 1 -n\tau - n\tau^3)(- n\tau^{-3} - n\tau^{-1} + 1 +n\tau + n\tau^3),
\end{align*}
so that
\[
\Delta_K(t) = -n^2t^{-3} - 2n^2t^{-2} + n^2t^{-1} + (1+4n^2) + n^2t -2n^2t^2-n^2t^3.
\]
Plugging in $t = -1$, we have that the determinant is $1$, and we will use the Fox-Milnor condition \cite{MR211392} to see that $K$ is not slice by showing that $\Delta_K(t)$ does not factor as $f(t)f(t^{-1})$ for any Laurent polynomial $f(t)$. For $n = 1$, this is apparent since $\Delta_K(t)$ is irreducible over $\mathbb{Q}$. For $n \geq 3$ prime, we will argue by contradiction. Let $f(t) = a_0 + a_1t + a_2t^2 + a_3t^3$, and suppose that $f(t)f(t^{-1}) = \Delta_K(t)$. That is
\[
(a_0 + a_1t + a_2t^2 + a_3t^3)(a_0 + a_1t^{-1} + a_2t^{-2} + a_3t^{-3}) = \Delta_K(t).
\]
Comparing coefficients, we have the system of equations
\begin{align*}
a_0a_3 &= -n^2, \\
a_0a_2 + a_1a_3 &= -2n^2, \\
a_0a_1 + a_1a_2 + a_2a_3 &= n^2,  \\
a_0^2 + a_1^2 + a_2^2 + a_3^2 &= 4n^2 + 1.
\end{align*}
Since $n$ is prime, the first equation implies that $a_0,a_3 \in \{\pm 1,\pm n, \pm n^2\}$. However, the last equation immediately rules out $a_0 = \pm n^2$ or $a_3 = \pm n^2$ since $n \geq 3$. We may then assume without loss of generality that $a_0 = n$ and $a_3 = -n$. In this case the second, third, and fourth equations reduce to
\begin{align*}
a_2 - a_1 &= -2n, \\
a_1a_2 &= -n^2, \\
a_1^2 + a_2^2 &= 1+2n^2,
\end{align*}
which has no solutions.
\end{proof}

There are no strongly negative amphichiral knots with determinant 1 which have 12 or fewer crossings, so that this knot is necessarily somewhat complicated. Moreover, this is the first example of a non-slice strongly negative amphichiral knot with determinant 1, answering Problem 20.3 in \cite{FT}. Such knots are of interest because their double branched covers potentially represent non-trivial torsion elements in the homology cobordism group; non-trivial torsion elements in this group are currently not known to exist (see \cite[Section 2]{MR3966804}). 

\section{Proof of Theorem \ref{thm:skeinrelation}} \label{sec:proofofmaintheorem}

Let $K_+$ and $K_-$ be strongly negative amphichiral knots related by a dichromatic equivariant crossing change $(c_+,\rho(c_+)) \to (c_-,\rho(c_-))$ such that the pair of crossings $(c_+,\rho(c_+))$ in $K_+$ is positive (see Definition \ref{def:crossingpairsign}). We can take an equivariant unknotting sequence for $K_-$ outside of the neighborhood of $(c_-,\rho(c_-))$. This gives us a symmetric surgery description for $K_-$ from the unknot $U$, which determines a symmetric surgery description of $X_{\infty}(K_-)$ in $X_{\infty}(U) = D^2 \times \mathbb{R}$; see Section \ref{sec:surgery-presentations}. Let $M_-$ be the presentation matrix of $H_1(X_{\infty}(K_-);\mathbb{Z})$ induced by this surgery. The unknotting sequence for $K_-$ gives an unknotting sequence for $K_+$ by adding in the additional crossing change $(c_+,\rho(c_+)) \to (c_-,\rho(c_-))$. The surgery presentation for $K_+$ is then obtained from the surgery presentation for $K_-$ by adding an equivariant pair of surgery circles; see Figure \ref{fig:surgery_skein_relation}. Let $M_+$ be the corresponding presentation matrix for $H_1(X_{\infty}(K_+);\mathbb{Z})$. Finally, replacing the $\pm1$-framed circles corresponding to $(c_+,\rho(c_+)) \to (c_-,\rho(c_-))$ with $0$-framed circles, we obtain a new framed link in $S^3 - \nu(U)$ which lifts to a framed link $\widetilde{L}_0$ in $X_{\infty}(U)$. Let $X_{\infty}^0$ be the manifold obtained by surgery on $\widetilde{L}_0$ and let $M_0$ be the corresponding presentation matrix for the $\mathbb{Z}[\tau,\tau^{-1}]$-module $H_1(X_{\infty}^0;\mathbb{Z})$. 

\begin{figure}
  \begin{overpic}[width=250pt, grid=false]{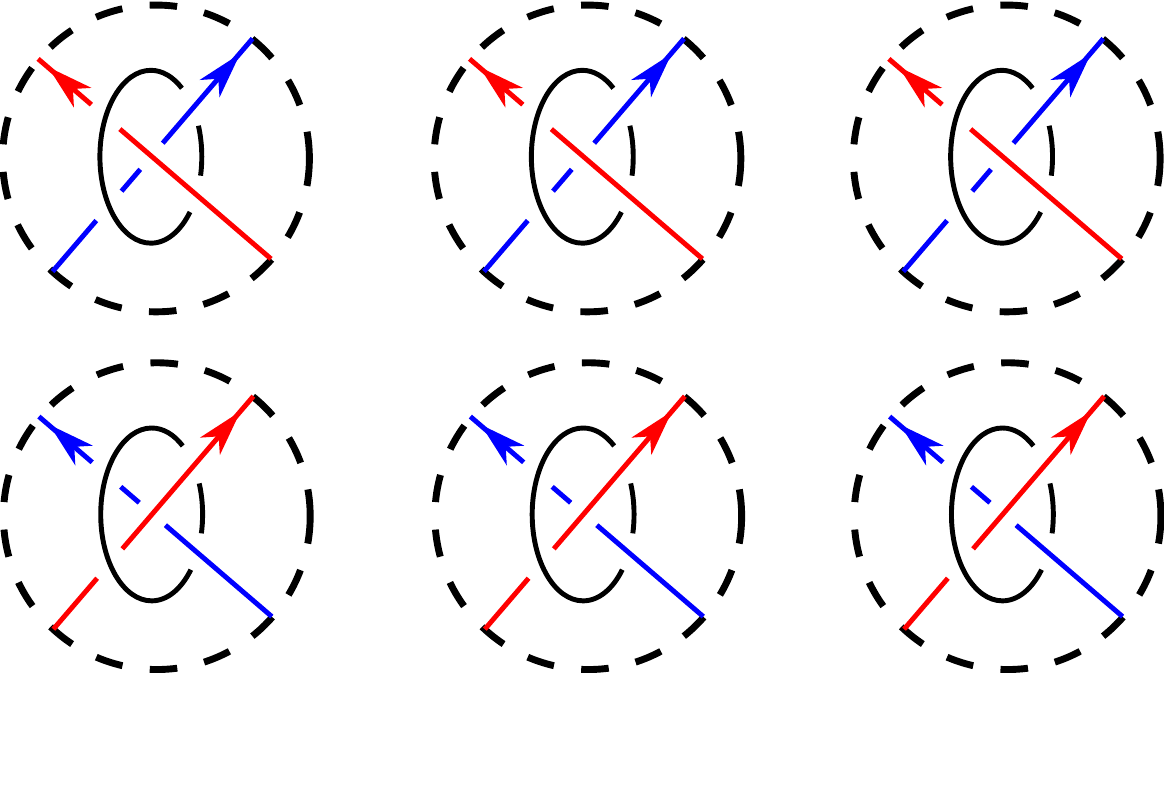}
    \put (9.5, 62.5) {$+1$}
    \put (9.5, 32) {$-1$}
    \put (47.5, 62.5) {$\infty$}
    \put (47.5, 32) {$\infty$}
    \put (85, 62.5) {$0$}
    \put (85, 32) {$0$}
    \put (10, 3) {$K_+$}
    \put (47.5, 3) {$K_-$}
    \put (85, 3) {$K_0'$}
  \end{overpic}
\caption{The local surgeries needed to obtain $K_+,K_-$, and $K_0'$. Note that $K_0'$ is distinct from the resolution $K_0$ in Figure \ref{fig:skeinrelation}. The crossings in the top row come after the crossings in the bottom row, so that the sign of each pair is the sign of the top crossing; see Definition \ref{def:crossingpairsign}.}
\label{fig:surgery_skein_relation}
\end{figure}
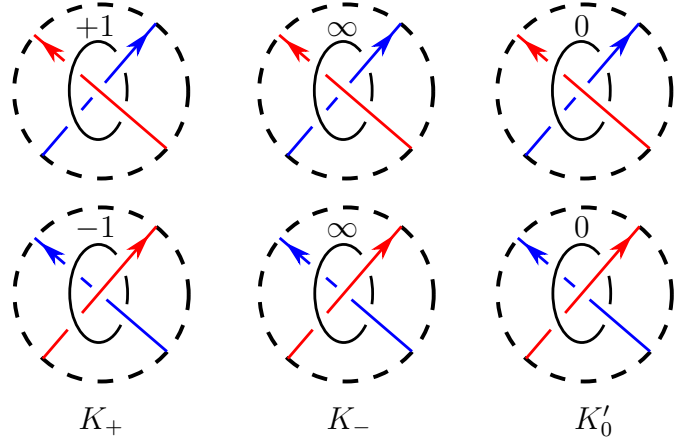

\begin{remark}
Note that the $\mathbb{Z}[\tau,\tau^{-1}]$-module $H_1(X_{\infty}^0;\mathbb{Z})$ is the Alexander module of the strongly negative amphichiral knot $K_0'$ in $(S^1 \times S^2) \# (S^1 \times S^2)$ obtained from $K_-$ by $0$-surgery on the pair of circles on the right in Figure \ref{fig:surgery_skein_relation}.
\end{remark}
The following proposition (which depends on Lemmas \ref{lemma:fminush}, \ref{lemma:vectorvalues}, and \ref{lemma:hderivative} below) is our first step towards proving Theorem \ref{thm:realization}.
\begin{proposition} \label{prop:matrix-skein-relation}
Let $M_+, M_-$, and $M_0$ be as above. Then
\begin{enumerate}
\item $\det(M_+) - \det(M_-) = \det(M_0)$,
\item $\det(M_+)\big{|}_{\tau = 1} = \det(M_-)\big{|}_{\tau = 1}$, and 
\item letting $z = \tau - \tau^{-1}$,
\[
\frac{\partial \det(M_0)}{\partial z}\big{|}_{z = 0} = \frac{1}{2}\cdot\frac{\partial \det(M_0)}{\partial \tau}\big{|}_{\tau = 1} =  \det(M_-)\big{|}_{\tau = 1}.
\]
\end{enumerate}
\end{proposition}
\begin{proof}
We start by writing the matrices $M_+$ and $M_0$ in terms of $M_-$. Recall that the entries of a presentation matrix are linking numbers between lifts of surgery curves (see equation (\ref{eqn:matrix-entries})). Hence $M_+$ and $M_0$ are each obtained from $M_-$ by adding a row and a column containing linking numbers involving lifts of the extra equivariant pair of surgery circles as shown in Figure \ref{fig:surgery_skein_relation}. Furthermore, the added surgery circles for $K_0'$ and $K_+$ differ only by their framings, so the off-diagonal entries in $M_+$ and $M_0$ are identical. Concretely, let $n = \operatorname{rank}(M_-)$. Then there are Laurent polynomials $f(\tau), h(\tau)$, and a $n$-tuple of Laurent polynomials $\vec{g}(\tau) = (g_1(\tau),g_2(\tau),\dots,g_n(\tau))$ such that 
\begin{equation}\label{eqn:M+andM-}
M_+ = 
\begin{bmatrix}
M_- & \vec{g}(\tau)\\
\vec{g}^T(-\tau^{-1})& f(\tau) 
\end{bmatrix}, \mbox{ and }
M_0 = \begin{bmatrix}
M_- & \vec{g}(\tau)\\
\vec{g}^T(-\tau^{-1})& h(\tau) 
\end{bmatrix}.
\end{equation}
By Lemma \ref{lemma:fminush} below, $f(\tau) = h(\tau) + 1$ and we compute 
\begin{align*}
\det(M_+) - \det(M_-) &= \det\begin{bmatrix}
M_- & \vec{g}(\tau)\\
\vec{g}^T(-\tau^{-1})& f(\tau) 
\end{bmatrix} - 
\det\begin{bmatrix}
M_- & 0\\
0 & 1
\end{bmatrix}\\
&=\det\begin{bmatrix}
M_- & \vec{g}(\tau)\\
\vec{g}^T(-\tau^{-1})& f(\tau) 
\end{bmatrix} - 
\det\begin{bmatrix}
M_- & \vec{g}(\tau)\\
0 & 1
\end{bmatrix}\\
&=\det\begin{bmatrix}
M_- & \vec{g}(\tau)\\
\vec{g}^T(-\tau^{-1})& f(\tau)-1 
\end{bmatrix}\\
 &= \det(M_0),
\end{align*}
proving Proposition \ref{prop:matrix-skein-relation} (1). Next, we will show that $\det(M_+)\big{|}_{\tau = 1} = \det(M_-)\big{|}_{\tau = 1}$. By Lemma \ref{lemma:vectorvalues} below, $\vec{g}(1) = 0 = \vec{g}^T(-1)$ and $f(1) = 1$ so that
\[
\det(M_+)\big{|}_{\tau = 1} = \det \begin{bmatrix}
M_-\big{|}_{\tau = 1} & \vec{g}(1)\\
\vec{g}^T(-1) & f(1)
\end{bmatrix} = 
\det\begin{bmatrix}
M_-\big{|}_{\tau = 1} & 0\\
0 & 1
\end{bmatrix} = 
\det(M_-)\big{|}_{\tau = 1}.
\]
For the final claim, we have 
\[
\frac{\partial \det(M_0)}{\partial z}\big{|}_{z = 0} = \frac{1}{2}\cdot\frac{\partial \det(M_0)}{\partial \tau}\big{|}_{\tau = 1}
\]
by the chain rule. Let $\vec{r}_i$ be the $i$th row of $M_-$. Then we compute 
\[
\frac{1}{2} \cdot \frac{\partial \det{M_0}}{\partial \tau}\big{|}_{\tau = 1} = \frac{1}{2}\bigg{(}\det \begin{bmatrix}
M_-\big{|}_{\tau = 1}&\vec{g}(1)\\
(\vec{g}^T)'(-1)&h'(1)
\end{bmatrix} + \sum_{i=1}^n H_i\bigg{)},
\]
where
\[
H_i = \det\begin{bmatrix}
\vec{r}_1(1) & g_1(1)\\
\vec{r}_2(1) & g_2(1)\\
\vdots & \\
(\vec{r}_i)'(1) & g_i'(1)\\
\vdots & \\
\vec{r}_n(1) & g_n(1)\\
\vec{g}^T(-1) & h(1)
\end{bmatrix} = \det\begin{bmatrix}
\vec{r}_1(1) & g_1(1)\\
\vec{r}_2(1) & g_2(1)\\
\vdots & \\
(\vec{r}_i)'(1) & g_i'(1)\\
\vdots & \\
\vec{r}_n(1) & g_n(1)\\
0 & 0
\end{bmatrix} = 0,
\]
in which the bottom row is 0 by Lemma \ref{lemma:vectorvalues}. Hence
\begin{align*}
\frac{1}{2} \cdot\frac{\partial \det{M_0}}{\partial \tau}\big{|}_{\tau = 1} &= \frac{1}{2}\det \begin{bmatrix}
M_-\big{|}_{\tau = 1}&\vec{g}(1)\\
(\vec{g}^T)'(-1)&h'(1)
\end{bmatrix} \\
&=\frac{1}{2} \det\begin{bmatrix}
M_-\big{|}_{\tau = 1}&0\\
(\vec{g}^T)'(-1)&2
\end{bmatrix} = \det(M_-)\big{|}_{\tau =1},
\end{align*}
where in the second equality we have $\vec{g}(1) = 0$ by Lemma \ref{lemma:vectorvalues} and $h'(1) = 2$ by Lemma \ref{lemma:hderivative}, also below.
\end{proof}

\begin{lemma} \label{lemma:fminush}
In Equation (\ref{eqn:M+andM-}), $f(\tau) = h(\tau) + 1$.
\end{lemma}
\begin{proof}
Let $\alpha$ be the $+1$-surgery curve for $K_+$ in Figure \ref{fig:surgery_skein_relation} and let $l$ be the $+1$-framed longitude of $\alpha$. Similarly, let $\alpha'$ be the $0$-surgery curve for $K_0'$ and let $l'$ be the $0$-framed longitude of $\alpha'$. Additionally, choose a lift $\overline{\alpha}$ of $\alpha$ in $X_{\infty}(K_+)$ and a push-off $\overline{l}$ of $\overline{\alpha}$ that lifts $l$, and choose lifts $\overline{\alpha}'$ and $\overline{l}'$ in $X_{\infty}^0$ similarly. Then 
\[
f(\tau) = \sum_i lk(\overline{l},\tau^i\overline{\alpha})(-\tau)^i, \mbox{ and } h(\tau) = \sum_i lk(\overline{l}',\tau^i\overline{\alpha}')(-\tau)^i.
\]
Only the framings on $\alpha$ and $\alpha'$ differ between these two sums, and the framing only affects the $i = 0$ term. Hence the non-constant terms in $f(\tau)$ and $h(\tau)$ are identical. To compare the constant terms, note that 
\[
\sum_i lk(\overline{l},\tau^{2i}\overline{\alpha})= lk(\alpha,l) = 1, \mbox{ and } \sum_i lk(\overline{l'},\tau^{2i}\overline{\alpha}') = lk(\alpha',l') = 0,
\]
where $\tau^{2i} = t^i$ are the usual deck translations. We can then compute the difference 
\begin{align*}
f(\tau) - h(\tau) &= lk(\overline{l},\overline{\alpha}) - lk(\overline{l}',\overline{\alpha}') \\
 &=\sum_i lk(\overline{l},\tau^{2i}\overline{\alpha}) - \sum_i lk(\overline{l'},\tau^{2i}\overline{\alpha}') \\
 &= 1- 0=1,
\end{align*}
where the first two equalities both use the fact that the non-constant terms in $f(\tau)$ and $h(\tau)$ are equal. Hence $f(\tau) = h(\tau) + 1$.
\end{proof}

\begin{lemma} \label{lemma:vectorvalues}
In Equation (\ref{eqn:M+andM-}), $\vec{g}(1) = \vec{g}(-1) = 0$ and $f(1) = 1$. In particular, $h(1) = 0$.
\end{lemma}
\begin{proof}
Let $\alpha, l, \overline{\alpha},$ and $\overline{l}$ be defined as in Lemma \ref{lemma:fminush}. Then
\[
f(1) = \sum_i lk(\overline{l},\tau^i\overline{\alpha})(-1)^i.
\]
Furthermore, for all $j \in \mathbb{Z}$ we have
\[
lk(\overline{l},\tau^{2j+1}\overline{\alpha}) = -lk(\tau^{-2j-1}\overline{l},\overline{\alpha}) = -lk(\tau^{-2j-1}\overline{\alpha},\overline{l})
\]
so that the odd terms in the previous equation cancel, giving us
\[
f(1) = \sum_i lk(\overline{l},\tau^{2i}\overline{\alpha}) = lk(l,\alpha) = 1.
\]
Then $h(1) = 0$ since $f(\tau) = h(\tau) + 1$ (Lemma \ref{lemma:fminush}). We will now show that $g_j(1) = 0$ for $1\leq j \leq n$. Let $\alpha_j$ and $\rho(\alpha_j)$ be the equivariant pair of surgery curves corresponding to the $j$th row of $M_-$. Let $\overline{\alpha}_j$ be a lift of $\alpha_j$ to $X_{\infty}(K_-)$. Then
\begin{align*}
g_j(1) = \sum_i lk(\overline{\alpha},\tau^i\overline{\alpha}_j)(-1)^i &= \sum_{i}lk(\overline{\alpha},\tau^{2i}\overline{\alpha}_j) - \sum_{i}lk(\overline{\alpha},\tau^{2i+1}\overline{\alpha}_j)\\
&= lk(\alpha,\alpha_j) - lk(\alpha,\rho(\alpha_j)) = 0-0 = 0.
\end{align*}
Here $lk(\alpha,\alpha_j) = lk(\alpha,\rho(\alpha_j)) = 0$ since the surgery curves are all pairwise unlinked. The argument that $g_j(-1) = 0$ is similar. Therefore $\vec{g}(1) = \vec{g}(-1) = 0$.
\end{proof}

\begin{lemma} \label{lemma:hderivative}
Let $h$ be defined as in Equation \ref{eqn:M+andM-}. Then
\[
\frac{\partial h}{\partial z}\big{|}_{z = 0} = \frac{1}{2} \cdot \frac{\partial h}{\partial \tau}\big{|}_{\tau = 1} = 1.
\]
\end{lemma}
\begin{proof}
Using the notation from Lemma \ref{lemma:fminush}, we have
\[
h(\tau) = \sum_i lk(\overline{l}',\tau^i\overline{\alpha}')(-\tau)^i,
\]
so that
\[
\frac{1}{2}h'(1) = \frac{1}{2}\sum_i i\cdot lk(\overline{l}',\tau^i\overline{\alpha}')(-1)^i.
\]
By the symmetry, $lk(\overline{l}',\tau^{2i}\overline{\alpha}') = lk(\overline{l}',\tau^{-2i}\overline{\alpha}')$ for all $i \in \mathbb{Z}$ so that the terms corresponding to index $2i$ and $-2i$ cancel each other for all $i$, and we are left with 
\[
\frac{1}{2}h'(1) = -\dfrac{1}{2}\sum_i (2i+1)lk(\overline{l}',\tau^{2i+1}\overline{\alpha}') = -\dfrac{1}{2}\sum_i (2i+1)lk(\overline{\alpha}',\tau^{2i+1}\overline{\alpha}').
\]
It remains to check that 
\begin{equation} \label{eqn:oddlinking}
-\dfrac{1}{2}\sum_i (2i+1)lk(\overline{\alpha}',\tau^{2i+1}\overline{\alpha}') = 1.
\end{equation}
We begin by putting the surgery curves $\alpha'$ and $\rho(\alpha)'$ in a particular position relative to the unknot $U$ as follows. Pull $U$ straight and let $a_r$ and $a_b$ be the two arcs of $U$ separated by the fixed points of $\rho$. The surgery curve $\alpha'$ (see Figure \ref{fig:surgery_skein_relation}) bounds a disk which intersects $U$ in two points, one on $a_r$ and one on $a_b$. This disk is isotopic, under the isotopy which pulls $U$ straight, to the band sum of a pair of meridional disks, $D_1$ around $a_r$ and $D_2$ around $a_b$. Let $b$ be the band connecting $D_1$ and $D_2$. Equivariantly, let $D_1' = \rho(D_1)$ and $D_2' = \rho(D_2)$ so that $\rho(\alpha')$ bounds a disk which is a band sum of $D_1'$ and $D_2'$. See Figure \ref{fig:horizontaldiskform}. Note that $b$ and $\rho(b)$ may be knotted and linked, but are disjoint from $U$ and the interiors of $D_1,D_2,D_1'$, and $D_2'$. 

\begin{figure}
\begin{overpic}[width=300pt, grid=false]{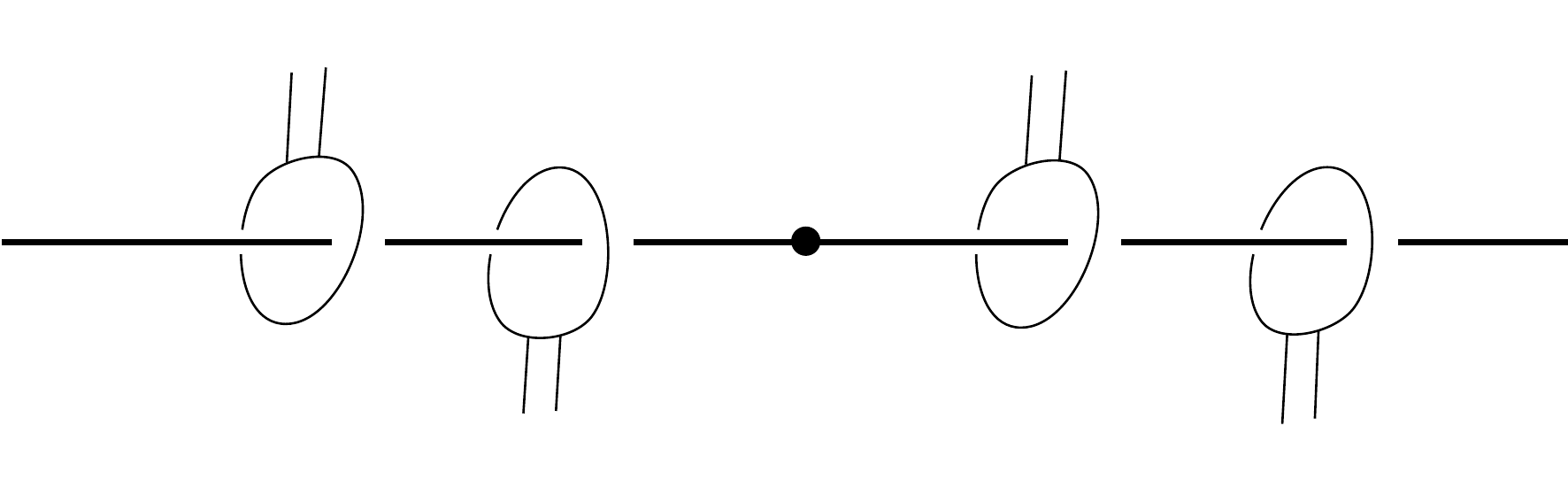}
  \put (16, 5) {$D_1'$}
  \put (34, 22) {$D_2$}
  \put (63, 5) {$D_2'$}
  \put (82, 22) {$D_1$}
  \put (11,22) {$\rho(b)$}
  \put (19,27) {$\vdots$}
  \put (58,22) {$\rho(b)$}
  \put (66,27) {$\vdots$}
  \put (37,5) {$b$}
  \put (34, 0) {$\vdots$}
  \put (85,5) {$b$}
  \put (82.4, 0) {$\vdots$}
  \put (0,10) {$U$}
\end{overpic}
\caption{The $0$-framed symmetric surgery curves $\alpha' = \partial(D_1 \cup D_2\cup b)$ and $\rho(\alpha') = \partial(D_1' \cup D_2' \cup \rho(b))$ for $K_0'$ are isotopic to this standard position in a neighborhood of $U$.}
\label{fig:horizontaldiskform}
\end{figure}

\begin{figure}
\begin{overpic}[width=180pt, grid=false]{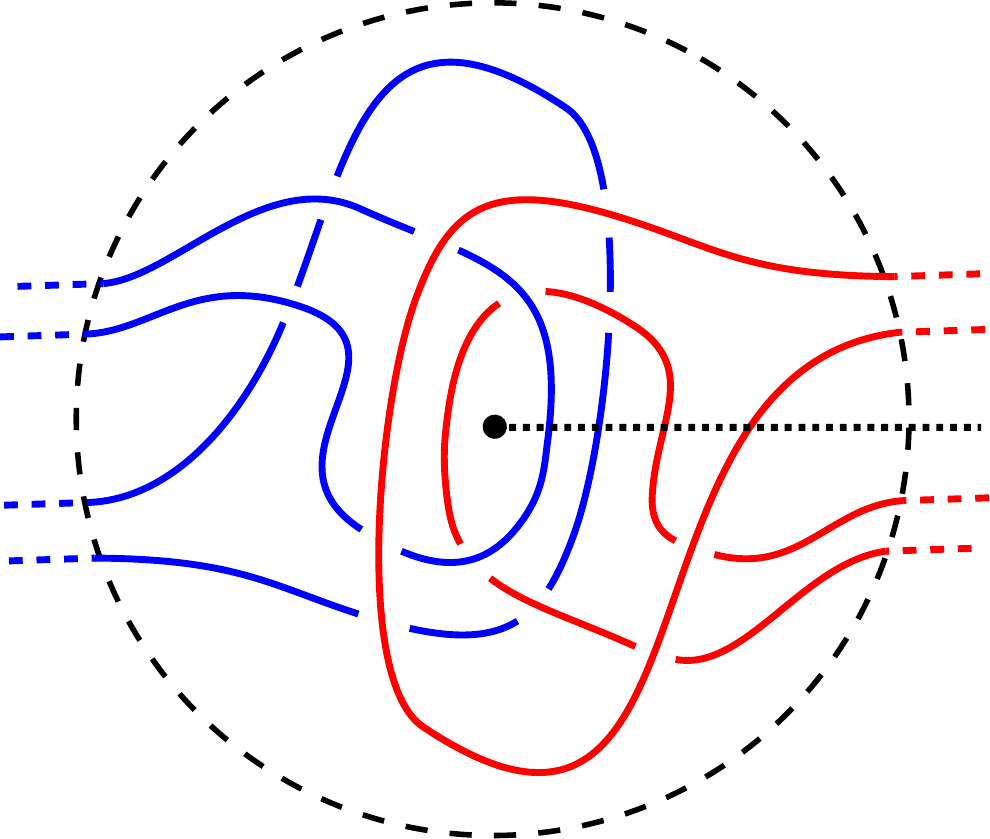}
  \put (14, 7) {$N$}
  \put (95, 46) {$\gamma$}
  \put (95, 35.5) {$\delta$}
  \put (94, 59) {$\delta$}
  \put (93, 24) {$\gamma$}
  \put (76,17) {$\eta_2$}
  \put (75, 61) {$\eta_1$}
\end{overpic}
\caption{Figure \ref{fig:horizontaldiskform} rotated so that $U$ is perpendicular to the diagram. For clarity, $\alpha'$ has been colored red and $\rho(\alpha')$ has been colored blue. The intersection $\alpha' \cap N$ consists of two components: $\eta_1$ and $\eta_2$. The arcs $\delta$ and $\gamma$ are the edges of the band $b$, so that $\alpha' = \eta_1 \cup \eta_2 \cup \delta \cup \gamma.$}
\label{fig:verticaldiskform}
\end{figure}

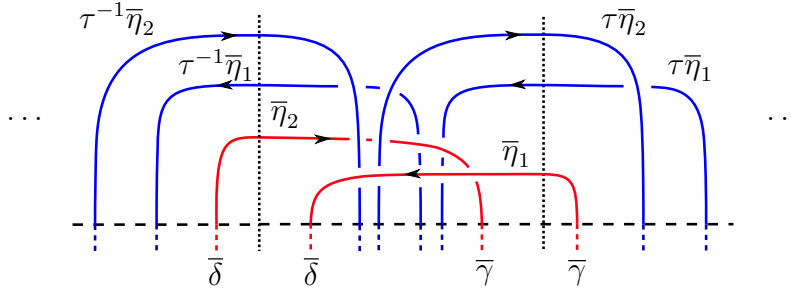
\begin{figure}
\begin{overpic}[width=250pt, grid=false]{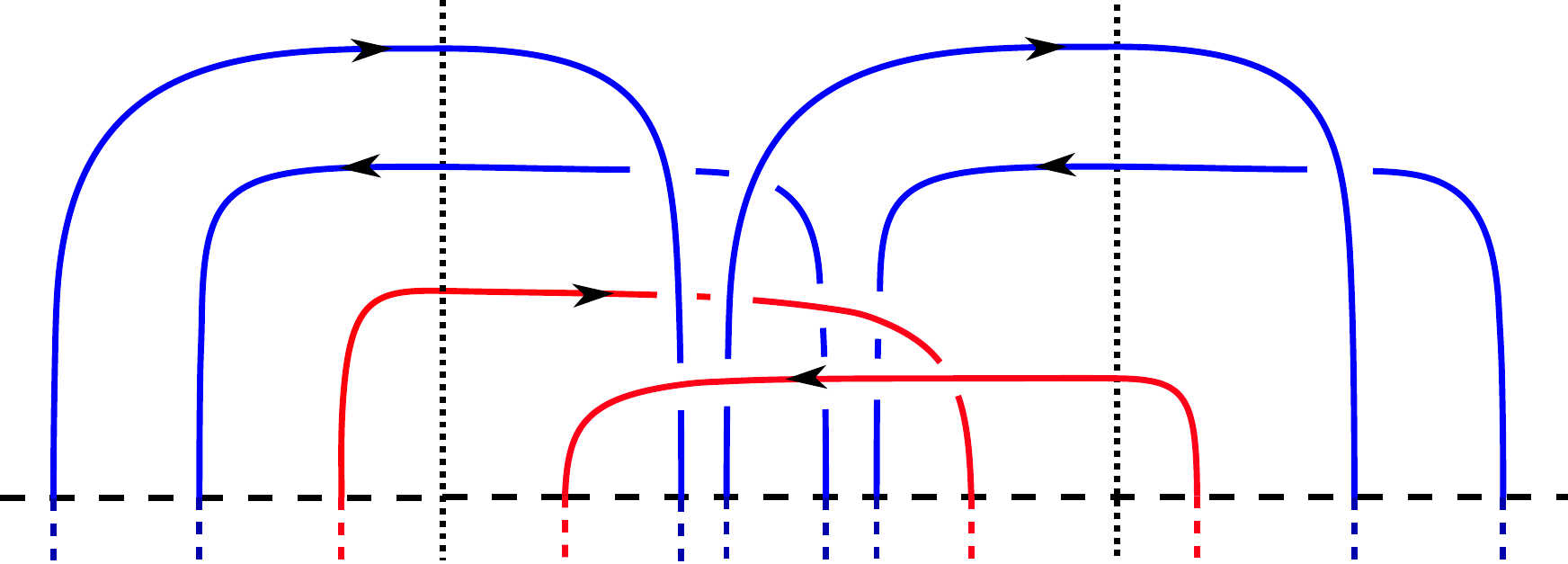}
  \put (-10, 20) {$\dots$}
  \put (105, 20) {$\dots$}
  \put (30,20) {$\overline{\eta}_2$}
  \put (80,34) {$\tau \overline{\eta}_2$}
  \put (90,27) {$\tau \overline{\eta}_1$}
  \put (1,34) {$\tau^{-1} \overline{\eta}_2$}
  \put (16, 27) {$\tau^{-1}\overline{\eta}_1$}
  \put (65,14) {$\overline{\eta}_1$}
  \put (35,-5) {$\overline{\delta}$}
  \put (20.5,-5) {$\overline{\delta}$}
  \put (75,-4) {$\overline{\gamma}$}
  \put (61,-4) {$\overline{\gamma}$}
\end{overpic}
\caption{The lift of Figure \ref{fig:verticaldiskform} to $X_{\infty}(U)$.}
\label{fig:verticaldiskformlift}
\end{figure}

We can then choose a new projection for Figure \ref{fig:horizontaldiskform} where $U$ is perpendicular to the diagram, and where a neighborhood $N$ of the projection of $U$ containing $D_1,D_2,D_1',$ and $D_2'$ is disjoint from the interiors of $b$ and $\rho(b)$; see Figure \ref{fig:verticaldiskform}. Let $\gamma$ and $\delta$ be the arcs on the boundary of $b$ disjoint from $N$ as indicated. Then lift Figure \ref{fig:verticaldiskform} to the infinite cyclic cover $X_{\infty}(U)$ as shown in Figure \ref{fig:verticaldiskformlift}. Here $\overline{\delta}$ and $\overline{\gamma}$ are lifts of $\delta$ and $\gamma$ respectively, and $\overline{\eta}_1$ and $\overline{\eta_2}$ are lifts of boundary pieces of $D_1$ and $D_2$ respectively so that $\overline{\alpha}' = \overline{\delta} \cup \overline{\gamma} \cup \overline{\eta}_1 \cup \overline{\eta_2}$ as shown. 

We can now compute Equation \ref{eqn:oddlinking} directly by counting signs of crossings between $\overline{\alpha}'$ and $\bigcup_i \tau^{2i+1}\overline{\alpha}'$ in Figure \ref{fig:verticaldiskformlift}. First, consider $\overline{\eta}_1$ and $\overline{\eta}_2$. We directly see that crossings involving these arcs contribute $\tau - \tau^{-1}$ to $h(\tau)$: a $\tau$ contribution from the crossings with $\tau\overline{\eta}_2$ and a $-\tau^{-1}$ contribution from the crossings with $\tau^{-1}\overline{\eta}_2$. This contributes $-\frac{1}{2}\cdot (-2) = 1$ to the left hand side of Equation \ref{eqn:oddlinking}. It remains to show that the crossings involving $\overline{\delta}$ and $\overline{\gamma}$ contribute 0. To see this, observe that the crossings between $\overline{\gamma}$ or $\overline{\delta}$ and $\bigcup_i \tau^{2i+1}\overline{\alpha}'$ occur as shown in Figure \ref{fig:bandcrossing}, where the pair of strands in each band are related by a $\tau^2$ (or equivalently, $t$) shift since $\eta_i$ loops around $K$ once. Indeed, each band crossing in $S^3$ consists of 4 crossings. Since we are only interested in crossings involving $\tau^{2i+1}\overline{\alpha}'$, we only consider band crossings between $b$ and $\rho(b)$. Such band crossings lift to exactly two band crossings in $X_{\infty}(U)$ which involve $\overline{\gamma}$ or $\overline{\delta}$. In these two band crossings there are 4 total crossings involving $\overline{\delta}$ or $\overline{\gamma}$. These contribute
\[
\pm\frac{1}{2}\big{(}\tau^{2n-1} - \tau^{2n+1} + \tau^{2n+3} - \tau^{2n+1}\big{)}
\] 
to $h(\tau)$, where the sign depends on the orientation of the band crossing and the $\frac{1}{2}$ is part of the usual linking number formula. This results in a contribution of 0 to the left hand side of Equation \ref{eqn:oddlinking}.
\end{proof}

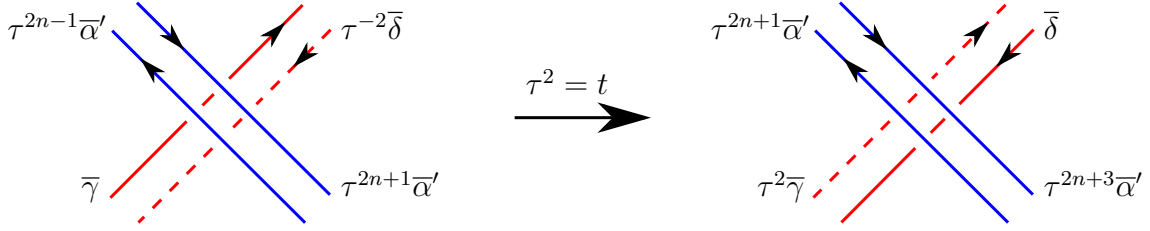
\begin{figure}
\begin{overpic}[width=350pt, grid=false]{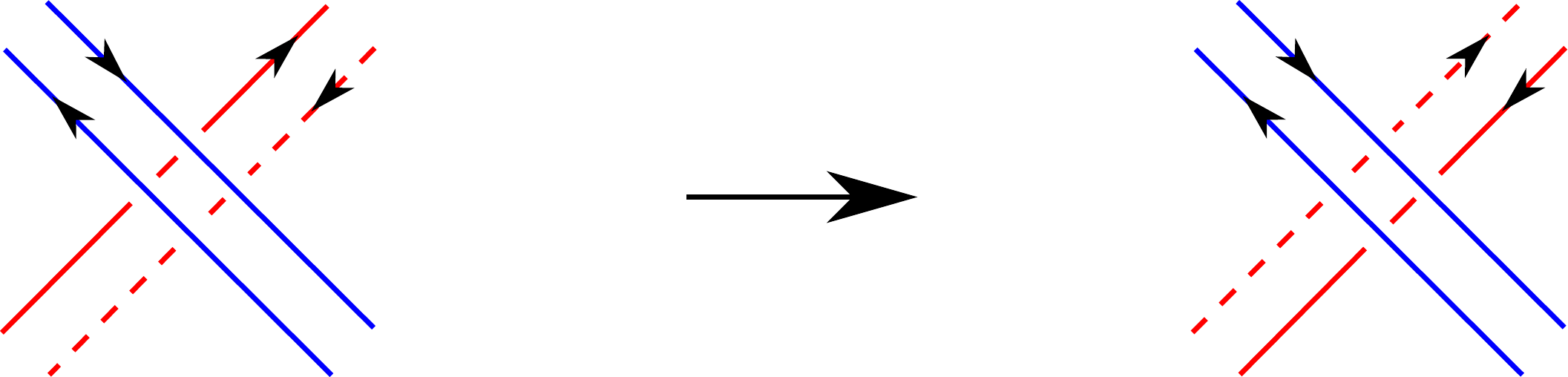}
\put (45,14) {$\tau^2 = t$}
\put (-3,3) {$\overline{\gamma}$}
\put (25,20) {$\tau^{-2}\overline{\delta}$}
\put (25, 3) {$\tau^{2n+1}\overline{\alpha}'$}
\put (-11, 20) {$\tau^{2n-1}\overline{\alpha}'$}

\put (70,3) {$\tau^2\overline{\gamma}$}
\put (101,20) {$\overline{\delta}$}
\put (101, 3) {$\tau^{2n+3}\overline{\alpha}'$}
\put (65, 20) {$\tau^{2n+1}\overline{\alpha}'$}
\end{overpic}
\caption{The two lifts of a band crossing to $X_{\infty}(U)$ involving $\overline{\gamma}$ and $\overline{\delta}$. Note that the two strands in each band are in distinct components of the lift of $\alpha$; one is a $\tau^2$ (or equivalently, $t$) shift of the other.}
\label{fig:bandcrossing}
\end{figure}

We need one additional lemma before proving Theorem \ref{thm:skeinrelation}. 

\begin{lemma} \label{lemma:M0equivalence}
For $M_0$ and $K_0$ as above, $\det(M_0) = \lambda \cdot (\tau - \tau^{-1}) \cdot \Delta_{(K_0,\rho)}(\tau)$, for some $\lambda \in \mathbb{Q}$.
\end{lemma}
\begin{proof}
Recall that $\det(M_0)$ is the Alexander polynomial of $K_0'$ (see Figure \ref{fig:surgery_skein_relation}). Let $B_0$ be the ball shown on the top right in Figure \ref{fig:skeinrelation}, and let $B_0'$ be the corresponding ball for $K_0'$ shown on the top right in Figure \ref{fig:surgery_skein_relation}. Take their pre-images $\overline{B}_0$ and $\overline{B_0}'$ in the infinite cyclic covers of $K_0$ and $K_0'$ respectively, and let $E_{\infty} = (X_{\infty}(K_0) - (\overline{B}_0 \cup \tau \overline{B}_0))$. Note there is a natural identification between $E_{\infty}$ and $(X_{\infty}(K_0') - (\overline{B}_0' \cup \tau \overline{B}_0'))$. We can decompose 
\begin{align*}
X_{\infty}(K_0) &= E_{\infty} \cup (\overline{B}_0\cup \tau\overline{B}_0), \mbox{ and }\\
X_{\infty}(K_0') &= E_{\infty} \cup (\overline{B}'_0\cup \tau\overline{B}'_0).
\end{align*}

\begin{figure}
\begin{overpic}[width=100pt, grid=false]{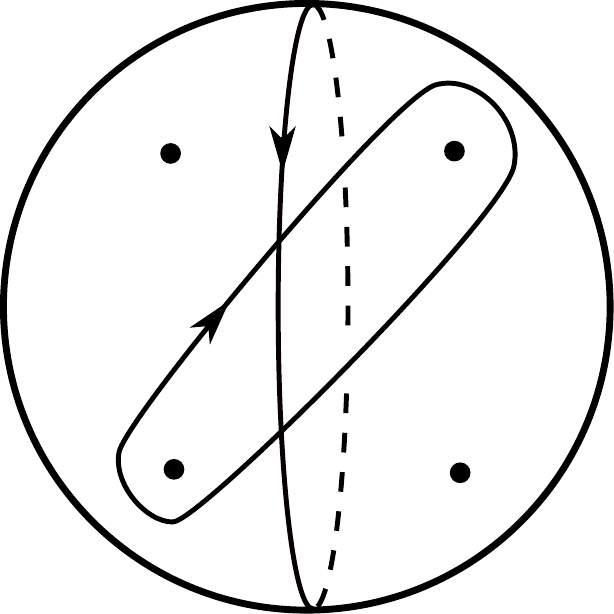}
\put (30,80) {$a$}
\put (66,68) {$c$}
\put (30,27) {$b$}
\put (77,27) {$d$}
\put (55,8) {$m_1$}
\put (70,46) {$m_2$}
\end{overpic} \hspace{1cm}
\begin{overpic}[width=100pt, grid=false]{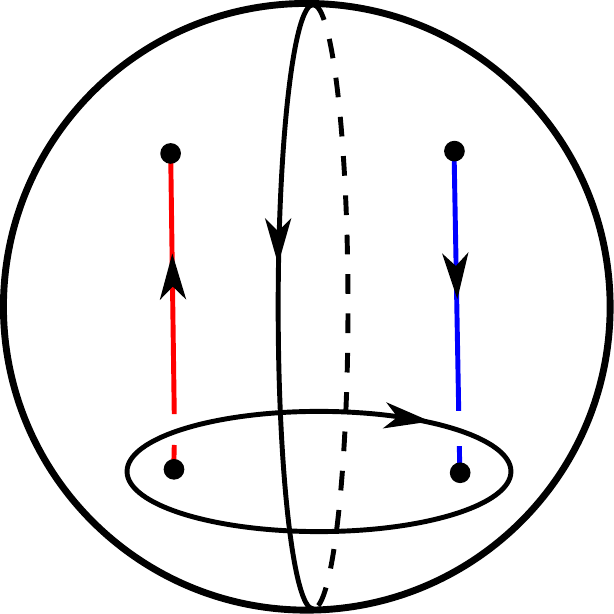}
\put (30,68) {$a$}
\put (64,68) {$d$}
\put (33,20) {$b$}
\put (65,20) {$c$}
\put (31,85) {$m_1$}
\put (9,30) {$m_2$}
\end{overpic} \hspace{1cm}
\begin{overpic}[width=100pt, grid=false]{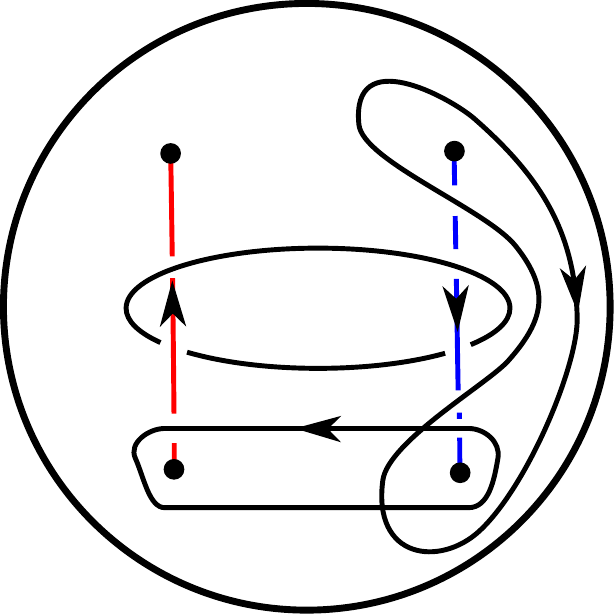}
\put (30,68) {$a$}
\put (66,75) {$b$}
\put (33,20) {$d$}
\put (65,20) {$c$}
\put (42,85) {$m_2$}
\put (40,9) {$m_1$}
\put (50,50) {$0$}
\end{overpic}
\caption{The left figure shows curves $m_1$ and $m_2$ in the 4-punctured sphere $S^2 - \{a,b,c,d\} \subset \partial B_0 = \partial B_0'$. The center figure shows $m_1$ and $m_2$ on the boundary of $B_0$ after an isotopy which rotates the blue arc upside down. The right figure shows $m_1$ and $m_2$ on the boundary of $B_0'$ after an isotopy which uncrosses the red and blue arcs.}
\label{fig:K0generators}
\end{figure}
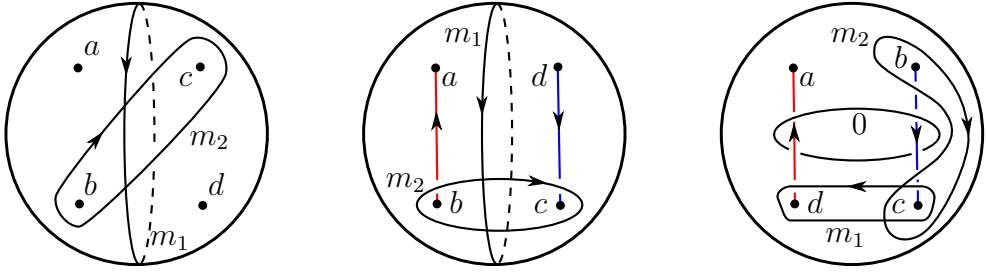

Let $\Lambda_{\mathbb{Q}} = \mathbb{Q}[\tau,\tau^{-1}]$. We can compare the $\Lambda_{\mathbb{Q}}$-modules $H_1(X_{\infty}(K_0);\mathbb{Q})$ and $H_1(X_{\infty}(K_0');\mathbb{Q})$ using the Mayer-Vietoris sequence on this $\tau$-equivariant decomposition. We use $\mathbb{Q}$ coefficients to simplify the linear algebra since $\Lambda_{\mathbb{Q}}$ is a principal ideal domain. Let $\overline{S}$ be the preimage of $\partial B_0$ in the infinite cyclic cover (either $X_{\infty}(K_0)$ or $X_{\infty}(K_0')$). Note that $\overline{S}$ and $\tau\overline{S}$ are disjoint. We have the Mayer-Vietoris sequences as follow:
\begin{align*}
H_1(\overline{S} \cup \tau\overline{S};\mathbb{Q}) \overset{i_1\oplus i_2}{\xrightarrow{\hspace*{1cm}}} H_1(E_{\infty};\mathbb{Q}) \oplus H_1(\overline{B}_0 \cup \tau \overline{B}_0;\mathbb{Q}) \to H_1(X_{\infty}(K_0);\mathbb{Q}) \to 0&, \mbox{ and } \\
H_1(\overline{S} \cup \tau\overline{S};\mathbb{Q}) \overset{i_1\oplus i_2'}{\xrightarrow{\hspace*{1cm}}} H_1(E_{\infty};\mathbb{Q}) \oplus H_1(\overline{B}_0' \cup \tau \overline{B}_0';\mathbb{Q}) \to H_1(X_{\infty}(K_0');\mathbb{Q}) \to 0&.
\end{align*}
In particular,  $H_1(X_{\infty}(K_0);\mathbb{Q})$ and $H_1(X_{\infty}(K_0');\mathbb{Q})$ are the cokernels of $i_1 \oplus i_2$ and $i_1 \oplus i_2'$ respectively. We can then obtain presentation matrices for $H_1(X_{\infty}(K_0);\mathbb{Q})$ and $H_1(X_{\infty}(K_0');\mathbb{Q})$ as follows. Since $\Lambda_{\mathbb{Q}}$ is a principal ideal domain, $H_1(E_{\infty};\mathbb{Q})$ can be written as 
\[
\Lambda_{\mathbb{Q}}/f_1 \oplus \Lambda_{\mathbb{Q}}/f_2 \oplus \cdots \oplus\Lambda_{\mathbb{Q}}/f_n
\]
for some $f_1,f_2,\dots,f_n \in \Lambda_{\mathbb{Q}}$. Therefore, $H_1(E_{\infty};\mathbb{Q})$ admits a square presentation matrix $A_E$: the diagonal matrix with entries $f_1,f_2,\dots,f_n$. Note that $\overline{S}$ is homeomorphic to two copies of $\mathbb{R} \times [0,1]$ with infinitely many tubes connecting them, as shown on the left in Figure \ref{fig:ladders}. Here $\tau^2$ acts by shifting the tubes one place to the right, and $\tau$ exchanges $\overline{S}$ and $\tau\overline{S}$. Let $\overline{m}_1$ and $\overline{m}_2$ be lifts of $m_1$ and $m_2$ from Figure \ref{fig:K0generators} as shown. Then $H_1(\overline{S} \cup \tau\overline{S};\mathbb{Q}) \cong \Lambda_{\mathbb{Q}}\langle\overline{m}_1,\overline{m}_2\rangle$. Similarly, $\overline{B}_0$ is homeomorphic to an infinite ladder: two copies of $D^2 \times \mathbb{R}$ connected by infinitely many solid tubes; see Figure \ref{fig:ladders}. As with $\overline{S}$, $\tau^2$ acts by shifting the rungs of the ladder one place to the right, so that $H_1(\overline{B}_0 \cup \tau\overline{B}_0;\mathbb{Q}) \cong \Lambda_{\mathbb{Q}}[\overline{m}_2]$. Now we obtain a presentation matrix for $H_1(X_{\infty}(K_0);\mathbb{Q})$ by enlarging $A_E$:
\[
A = \begin{bmatrix}
0& 1&0\\
\vec{v}_1&\vec{v}_2 &A_E
\end{bmatrix}.
\]
Here $\vec{v}_1$ and $\vec{v}_2$ are the coefficients of $i_1([\overline{m}_1])$ and $i_1([\overline{m}_2])$ respectively. Furthermore, the (1,1)-entry is 0 since $i_2([\overline{m}_1]) = 0$, and the (1,2)-entry is $1$ since $i_2([\overline{m}_2]) = [\overline{m}_2]$.

 Similarly, $\overline{B}_0'$ can be obtained from an infinite ladder by performing surgeries along the lifts of the $0$-surgery curves, each of which encircles the gap between an adjacent pair of rungs; see Figure \ref{fig:ladders}. Then $H_1(\overline{B}_0' \cup \tau \overline{B}_0';\mathbb{Q}) \cong \Lambda_{\mathbb{Q}}[\overline{\mu}]$, where $\overline{\mu}$ is the meridian of one of the surgery curves; see Figure \ref{fig:ladders}. Here $\overline{m}_1$ is isotopic to one of the surgery curves, and $\overline{m}_2$ is a meridian of a rung as shown. In particular, $i_2'([\overline{m}_1]) = 0$ and $i_2'([\overline{m}_2]) = (1-\tau^2)[\mu]$. Therefore $H_1(X_{\infty}(K_0');\mathbb{Q})$ has a presentation matrix
\[
A' = \begin{bmatrix}
0& (1-\tau^2) &0\\
\vec{v}_1&\vec{v}_2 &A_E
\end{bmatrix}.
\]

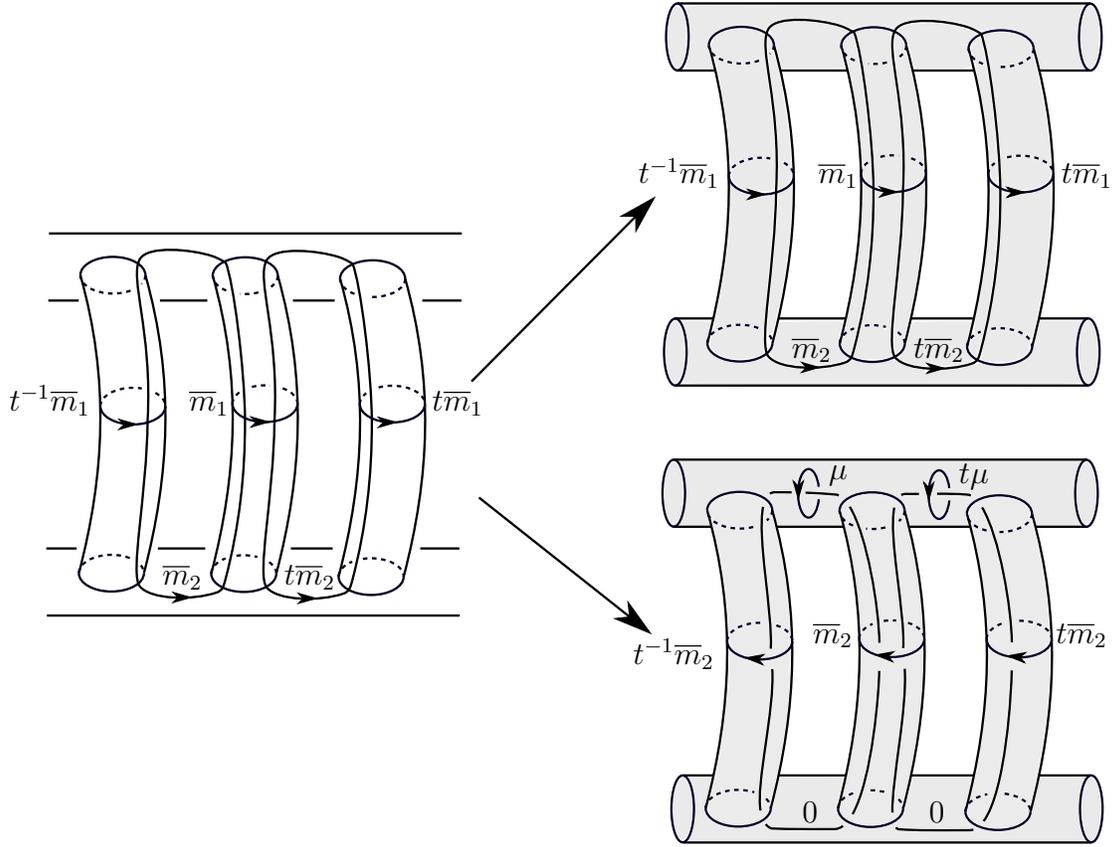
\begin{figure}
\begin{overpic}[width=400pt, grid=false]{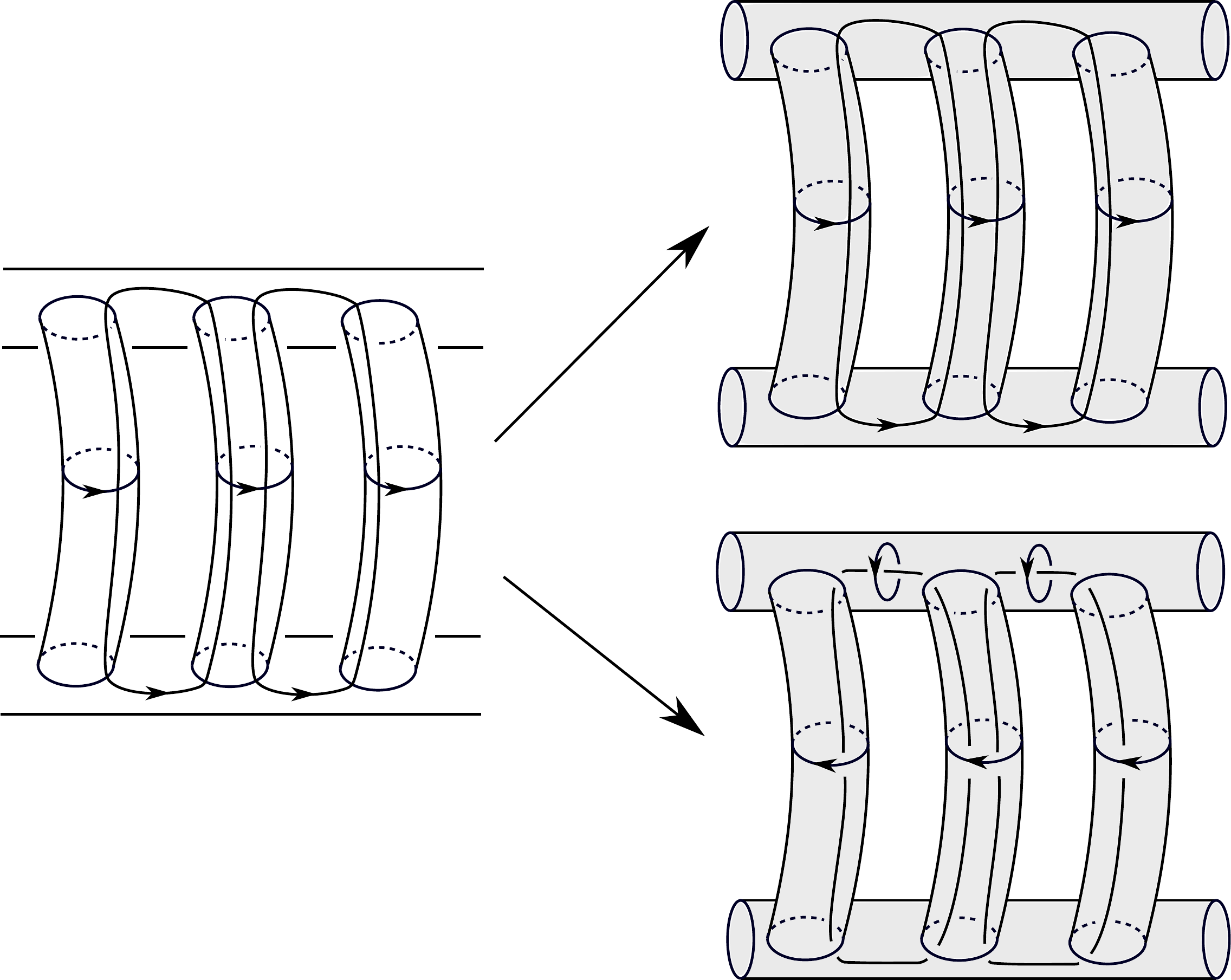}
\put (13.5,41) {$\overline{m}_1$}
\put (36.5,41) {$t\overline{m}_1$}
\put (-3.5, 41) {$t^{-1}\overline{m}_1$}
\put (11,24.5) {$\overline{m}_2$}
\put (22.5,24.5) {$t \overline{m}_2$}

\put (73,62.5) {$\overline{m}_1$}
\put (96,62.5) {$t\overline{m}_1$}
\put (56, 62.5) {$t^{-1}\overline{m}_1$}
\put (70.5,46) {$\overline{m}_2$}
\put (82,46) {$t \overline{m}_2$}

\put (72.5,19) {$\overline{m}_2$}
\put (95.5,19) {$t\overline{m}_2$}
\put (55.5, 17) {$t^{-1}\overline{m}_2$}
\put (71.5,2) {$0$}
\put (83.5,2) {$0$}
\put (74,34.5) {$\mu$}
\put (86.3,34.1) {$t\mu$}
\end{overpic}
\caption{One component of each of $\overline{B}_0$ (top right), and $\overline{B}_0'$ (bottom right). On the left is the pre-image $\overline{S}$ of $\partial B_0$ in the infinite cyclic cover, $X_{\infty}(K_0)$ or $X_{\infty}(K_0')$. In the bottom right, the $0$-framing refers to the lift of the $0$-framing in Figure \ref{fig:K0generators}.}
\label{fig:ladders}
\end{figure}

We now compare the first elementary ideals of $H_1(X_{\infty}(K_0);\mathbb{Q})$ and $H_1(X_{\infty}(K_0');\mathbb{Q})$ from these matrices. The first elementary ideal $\mathcal{E}$ of $H_1(X_{\infty}(K_0);\mathbb{Q})$ is generated by the $(n+1) \times (n+1)$ minors of $A$, and the first elementary ideal $\mathcal{E}'$ of $H_1(X_{\infty}(K_0');\mathbb{Q})$ is generated by the $(n+1) \times (n+1)$ minors of $A'$. Since $A$ and $A'$ are identical apart from the $(1,2)$-entry, we can see that the minors of $A'$ are precisely $(1-\tau^2)$ times the corresponding minors of $A$. In particular, a generator for $\mathcal{E}'$ is $(1-\tau^2)$ times a generator of $\mathcal{E}$, up to multiplication by a unit in $\Lambda_{\mathbb{Q}}$. 

Note that $\det(M_0) = \Delta_{(K_0',\rho)}(\tau)$ generates $\mathcal{E}'$ and that $\Delta_{(K_0,\rho)}(\tau)$ generates $\mathcal{E}$. Therefore, $\det(M_0)$ and $(\tau-\tau^{-1})\Delta_{(K_0,\rho)}(\tau)$ are equal up to multiplication by a unit in $\Lambda_{\mathbb{Q}}$. Furthermore, since both $\det(M_0)$ and $(\tau-\tau^{-1})\Delta_{(K_0,\rho)}(\tau)$ are symmetric under $\tau \to -\tau^{-1}$, they are in fact equal up to multiplication by a unit in $\mathbb{Q}$.
\end{proof}

\begin{proof}[Proof of Theorem \ref{thm:skeinrelation}]
Let $\varepsilon = \det(M_+)\big{|}_{\tau = 1}$, which is $1$ or $-1$ by Theorem \ref{thm:Hartley-Kawauchi}. Note that by Proposition \ref{prop:matrix-skein-relation} (2), $\varepsilon = \det(M_-)\big{|}_{\tau = 1}$ as well. By construction, $\det(M_\pm) \doteq \Delta_{(K_{\pm},\rho)}(\tau)$. Then since $\det(M_{\pm})$ is symmetric and $\nabla_{(K_{\pm},\rho)}(z)$ is normalized to have constant term 1, letting $z = \tau - \tau^{-1}$ gives us 
\[
\det(M_+) =  \varepsilon\cdot\nabla_{(K_+,\rho)}(z) \mbox{ and } \det(M_-) = \varepsilon\cdot\nabla_{(K_-,\rho)}(z).
\]
By Lemma \ref{lemma:M0equivalence},
\[
\det(M_0) = \lambda(\tau - \tau^{-1})\nabla_{(K_0,\rho)}(\tau - \tau^{-1}) = \lambda z \nabla_{(K_0,\rho)}(z).
\] 
To determine $\lambda$, note that Proposition \ref{prop:matrix-skein-relation} (3) gives that the $z$-coefficient of $\det(M_0)$ is $\varepsilon$. Furthermore, since $\nabla_{(K_0,\rho)}$ is normalized to have constant term $1$, $\lambda = \varepsilon$ whence we have
\[
\det(M_0) = \varepsilon \cdot z\nabla_{(K_0,\rho)}(z).
\]
Then by Proposition \ref{prop:matrix-skein-relation} (1), $\det(M_+) - \det(M_-) = \det(M_0)$ so that
\[
\varepsilon\cdot\nabla_{(K_+,\rho)}(z) - \varepsilon\cdot\nabla_{(K_-,\rho)}(z) = \varepsilon \cdot z\nabla_{(K_0,\rho)}(z).
\]
Dividing by $\varepsilon$, we have 
\[
\nabla_{(K_+,\rho)}(z) - \nabla_{(K_-,\rho)}(z) = z\nabla_{(K_0,\rho)}(z),
\]
as desired.
\end{proof}

\section{The half-Conway polynomials for knots with 12 or fewer crossings} \label{sec:computations}
In this section we list the values of the half-Conway polynomial for all strongly negative amphichiral knots with 12 or fewer crossings. Since the choice of orientation can replace $z$ with $-z$, we normalize so that the lowest order odd power of $z$ has a positive coefficient. For compactness we list the coefficients ordered by degree. For example, $[1,0,-2]$ represents $1-2z^2$.\\

Most of these polynomials can be computed from only the Conway polynomial (when it factors uniquely), see Proposition \ref{prop:half-Conway-properties}(1). Eleven of the remaining polynomials can be computed with the additional data of the half-linking number (indicated by bold typeface in the table); see Corollary \ref{cor:halfConwayhalflinkingarf}(1), and Example \ref{example:12a_435}. The remaining example $12a_{1152}$ (boxed in the table) requires some extra effort using Theorem \ref{thm:skeinrelation}; see Example \ref{example:skeincomputation}. In Appendix \ref{app:table} we also provide symmetric diagrams for these knots, which we used to compute the half-linking number.

\begin{tabular}{l|l||l|l||l|l}
$K$ & \multicolumn{1}{c}{$\nabla_{(K,\rho)}(z)$} &\multicolumn{1}{c}{} &\multicolumn{1}{c}{} &\multicolumn{1}{c}{} &\multicolumn{1}{c}{}  \\\midrule\midrule
$4_1$ & $[1,1]$& $6_3$ & $[1,1,1]$& $8_3$ &$[1,2]$\\
$8_9$ &$[1,2,1,1]$ &$8_{12}$ &$[1,1,-1]$ &$8_{17}$&$[1,1,0,1]$\\
$8_{18}$&$[1,1,1,1]$&$10_{17}$&$[1,2,3,1,1]$ &$10_{33}$ &$[1,2,2]$ \\
$10_{37}$ &$[1,1,2]$&$10_{43}$ & $[1,0,1,1]$&$10_{45}$&$[1,0,-1,1]$\\
$10_{79}$ &$[1,1,3,1,1]$&$10_{81}$ &$[1,1,2,1]$& $10_{88}$ &$[1,1,0,-1]$\\
$10_{99}$ &$[1,0,2,0,1]$&$10_{109}$ &$[1,1,2,0,1]$& $10_{115}$&$[1,1,1,-1]$\\
$10_{118}$ &$[1,2,2,1,1]$&$10_{123}$ &$[1,0,-1,0,-1]$& $12a_{4}$ &$[1,1,-1,-1,-1]$\\
$12a_{58}$ &$[1,1,1,1,1]$& $12a_{125}$ &$[1,1,-2,1]$& $12a_{268}$&$[1,1,2,2,1]$\\
$12a_{273}$ &$[1,1,1,2]$& $12a_{341}$ &$[1,1,0,0,-1]$& $\bm{12a_{435}}$ &$[1,2,2,2,1]$\\
$12a_{458}$ &$[1,0,0,1,-1]$& $12a_{462}$ &$[1,1,-1,1,-1]$&$12a_{465}$&$[1,2,0,0,-1]$\\
$\bm{12a_{471}}$ &$[1,1,-2]$& $12a_{477}$ &$[1,2,-1,-1]$& $12a_{499}$ &$[1,0,1,1,1]$\\
$\bm{12a_{506}}$ &$[1,0,-1,1,-1]$& $12a_{510}$ &$[1,2,2,2]$& $12a_{627}$&$[1,1,0,0,-1]$\\
$12a_{819}$ &$[1,2,3,3,1,1]$& $12a_{821}$ &$[1,2,2,2]$& $\bm{12a_{868}}$ &$[1,0,-1,2,-1]$\\
$12a_{887}$ &$[1,0,0,1,-1]$& $12a_{890}$ &$[1,1,1,2]$& $12a_{906}$&$[1,1,-2,2,-1]$\\
$12a_{960}$ &$[1,0,-1,2]$& $\bm{12a_{990}}$ &$[1,2,2,2,1]$& $12a_{1008}$ &$[1,1,0,2]$\\
$12a_{1019}$ &$[1,0,1,0,-1]$& $12a_{1039}$ &$[1,2,-1,1,-1]$& $\bm{12a_{1102}}$&$[1,2,0,0,1]$\\
$12a_{1105}$ &$[1,0,0,0,1]$& $12a_{1123}$ &$[1,1,0,1,-1]$& $\bm{12a_{1124}}$ &$[1,3,1,-1]$\\
$12a_{1127}$ &$[1,2,-2]$& $\boxed{12a_{1152}}$ &$[1,1,0,-1,-1]$& $12a_{1167}$&$[1,2,1,2,1]$\\
$12a_{1188}$ &$[1,0,0,1,1]$& $12a_{1202}$ &$[1,0,-3]$& $12a_{1209}$ &$[1,1,2,3,1,1]$\\
$12a_{1211}$ &$[1,0,1,2,0,1]$& $12a_{1218}$ &$[1,1,1,3,0,1]$& $12a_{1225}$&$[1,2,2,3,1,1]$\\
$\bm{12a_{1229}}$&$[1,1,-1,3,-1,1]$& $12a_{1249}$ &$[1,0,0,2,0,1]$& $\bm{12a_{1251}}$ &$[1,2,3,2]$\\
$12a_{1254}$ &$[1,2,0,3,0,1]$& $12a_{1260}$ &$[1,2,1,3,0,1]$& $12a_{1267}$&$[1,2,0,2]$\\
$12a_{1269}$&$[1,2,-1,2]$& $12a_{1273}$ &$[1,3,3,4,1,1]$& $12a_{1275}$ &$[1,3,2,2]$\\
$12a_{1280}$ &$[1,2,0,1,-1]$& $12a_{1281}$ &$[1,3,1,2]$& $12a_{1287}$&$[1,3]$\\
$12a_{1288}$ &$[1,3,2,4,1,1]$& $12n_{356}$ &$[1,1,-1,1]$& $\bm{12n_{462}}$ &$[1,2,1]$\\
$12n_{706}$ &$[1,0,-2,0,-1]$& $\bm{12n_{873}}$ &$[1,3,2,1,1]$& &\\
\end{tabular}

\begin{example} \label{example:12a_435}
Consider the strongly negative amphichiral knot $(K,\rho)$ shown in Figure \ref{fig:12a435}. Here the Conway polynomial is $\nabla_K(z) = 1-2z^4+z^8 = (1-z)^2(1+z)^2(1+z^2)^2$ so that the half-Conway polynomial is either $(1-z)(1+z)(1+z^2) = 1-z^4$ or $(1+z)^2(1+z^2) = 1+2z+2z^2+2z^3+z^4$ by Proposition \ref{prop:half-Conway-properties}. However, from Figure \ref{fig:12a435} we compute that the half-linking number is $h(K) = 2$ so that $\nabla_{(K,\rho)}(z) = 1
+2z+2z^2+2z^3+z^4$ by Corollary \ref{cor:halfConwayhalflinkingarf}.
\end{example}

\begin{figure}[htb!]
\begin{overpic}[width=200pt, grid=false]{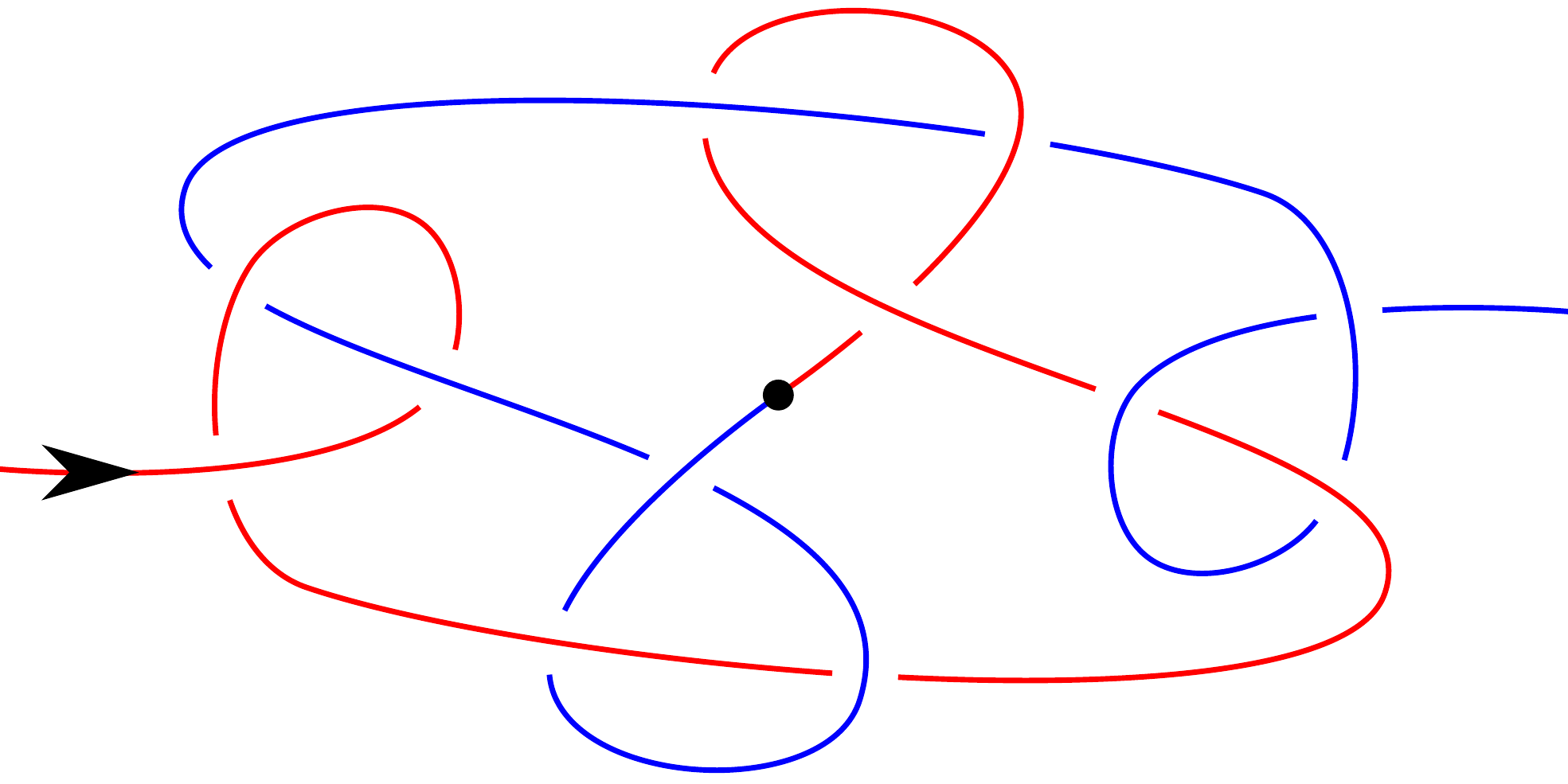}
\end{overpic}
\caption{A strongly negative amphichiral symmetry on $12a_{435}$.}
\label{fig:12a435}
\end{figure}

\begin{figure}
\begin{overpic}[width=400pt, grid=false]{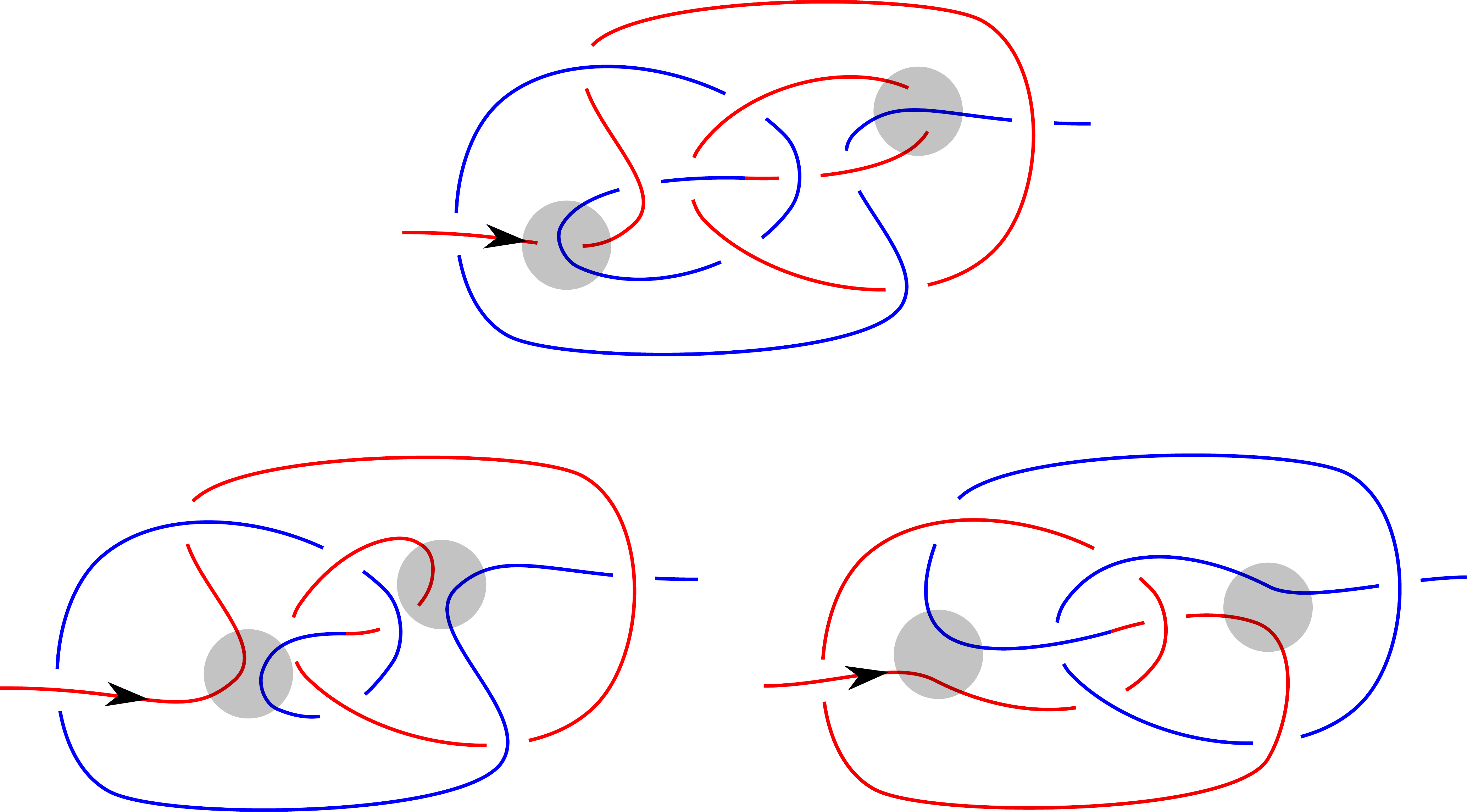}
\end{overpic}
\caption{Strongly negative amphichiral symmetries on $12a_{1152}$ (top), $4_1\#4_1$ (bottom left), and $8_{18}$ (bottom right) which are related by an equivariant skein relation on the regions indicated in gray. Note that we performed an SNA R2 move to obtain the diagram for $4_1\#4_1$ and an SNA R1 move to obtain the diagram for $8_{18}$.}
\label{fig:12a1152}
\end{figure}

\begin{example} \label{example:skeincomputation}
Let $K = 12a_{1152}$ and let $\rho$ be the strongly negative amphichiral symmetry on $K$ with $h(K) = 1$ as shown in Figure \ref{fig:12a1152}. We know that $\nabla_K(z) = 1-z^2-z^6+z^8$. This polynomial factors in multiple ways, so that the half-Conway polynomial is not determined by the Conway polynomial. Indeed, the possibilities are $1+z-z^3-z^4, 1+3z+4z^2+3z^3+z^4,$ or $1+z+z^3+z^4$. Noting that the half-linking number (which can readily be computed from Figure \ref{fig:12a1152}) is $1$, we still cannot determine whether the half-Conway polynomial is $1+z-z^3-z^4$ or $1+z+z^3+z^4$. However, we can use the equivariant skein relation (see Theorem \ref{thm:skeinrelation}) on the indicated pair of crossings. Changing the indicated crossings in the top diagram in Figure \ref{fig:12a1152} gives $(K',\rho') \widetilde{\#} (K',\rho')$ (Figure \ref{fig:12a1152}, bottom left), where $(K',\rho')$ is the figure-eight knot with $h(K',\rho') = 1$. On the other hand, resolving these crossings gives $(K'',\rho'')$ with $K'' = 8_{18}$ and $h(K'',\rho'') = 1$ (Figure \ref{fig:12a1152}, bottom right). For both $(K',\rho')$ and $(K'',\rho'')$ we can compute the half-Conway polynomial by using only the Conway polynomial and the half-linking number. In particular, $\nabla_{(K'',\rho'')}(z) = 1+z+z^2+z^3$ and $\nabla_{(K',\rho')\widetilde{\#}(K',\rho')}(z) = (\nabla_{(K',\rho')}(z))^2 = 1+2z+z^2$. Now Theorem \ref{thm:skeinrelation} gives that 
\[
\nabla_{(K',\rho')\widetilde{\#}(K',\rho')}(z) - \nabla_{(K,\rho)}(z) = z\nabla_{(K'',\rho'')}(z),
\]
and hence $\nabla_{(K,\rho)}(z) = 1+z-z^3-z^4$.

\end{example}

\appendix
\section{Proof of Theorem \ref{thm:Rmoves}}
\label{app:Rmoves}
In this appendix, we prove Theorem \ref{thm:Rmoves}, which follows immediately from Theorem \ref{thm:Rmovesappendix}. The proof works with long knots in $\mathbb{R}P^2 \times I$, which naturally arise as quotient objects of strongly negative amphichiral knots. Roughly, we will produce triangle moves relating equivalent long knots which then lift to symmetric triangle moves relating strongly negative amphichiral knots. The projections of these triangle moves decompose as Reidemeister moves. We begin by setting up some definitions.
\label{sec:SNAmoves}
\begin{definition}
Let $(F,*)$ be a pointed surface. A (smooth, tame, or polygonal) \emph{long knot} $K$ in $F \times \mathbb{R}$, is a (smooth, tame, or polygonal) embedding $K:\mathbb{R} \to F \times \mathbb{R}$ such that there is a constant $C(K) > 0$ and for all $|x| > C(K)$, $K(x) = *$. 
\end{definition}
For convenience, we take $D^2$ to be the unit disc and set the basepoint of $D^2$ (and its quotient $\mathbb{R}P^2$) to be the origin when dealing with long knots in $D^2 \times \mathbb{R}$ or $\mathbb{R}P^2 \times \mathbb{R}$. 

\begin{definition}
Let $F$ be a surface. A map $f:F \times \mathbb{R} \to F \times \mathbb{R}$ is $\emph{end-fixing}$ if there is a constant $C(f) > 0 $ so that for all $x \in \mathbb{R}$ with $|x| > C(f)$, $f(p,x) = (p,x)$ for all $p \in F$.
\end{definition}

\begin{definition}
	Two long knots $K_1$ and $K_2$ in $F\times \mathbb{R}$ are ambient isotopic if there exists a family of end-fixing homeomorphisms $h_t:F\times \mathbb{R}\rightarrow F\times \mathbb{R}$ such that $h_0=id$ and $h_1(K_1)=K_2$.
\end{definition}

\begin{remark}
We will be interested in the isotopy classes of these end-fixing homeomorphisms and in particular the mapping class group of end-fixing self-homeomorphisms of $\mathbb{R}P^2 \times \mathbb{R}$, which is equivalent to $\mathcal{M}(\mathbb{R}P^2\times I, \partial)$, the mapping class group of $\mathbb{R}P^2 \times I$ fixing the boundary.
\end{remark}
Given two elements $f,g \in \mathcal{M}(\mathbb{R}P^2\times I, \partial)$, in addition to their composition $f \circ g$, we can also concatenate them to get a map $(\mathbb{R}P^2 \times [0,2] ,\partial) \to (\mathbb{R}P^2\times [0,2],\partial)$ given by $f$ on $(\mathbb{R}P^2\times [0,1], \partial)$ and $g$ on $(\mathbb{R}P^2\times [1,2], \partial)$. Rescaling $[0,2]$ back to $[0,1]$ gives us a map $f*g:(\mathbb{R}P^2\times I, \partial) \to (\mathbb{R}P^2\times I, \partial)$. We will use this concatenation operation as the group operation in $\mathcal{M}(\mathbb{R}P^2\times I, \partial)$, as justified by the following lemma.

\begin{lemma}
The homeomorphisms $f \circ g$ and $f*g$ represent the same class in $\mathcal{M}(\mathbb{R}P^2\times I, \partial)$.
\end{lemma}
\begin{proof}
By squashing $f$ vertically along the $I$ direction, we see that $f$ is isotopic to a homeomorphism $f'$ which fixes $\mathbb{R}P^2 \times [0,0.5]$, and similarly $g$ is isotopic to a homeomorphism which fixes $\mathbb{R}P^2 \times [0.5,1]$. Then composing $f'$ with $g'$ is exactly $f*g$.
\end{proof}
We can now compute $\mathcal{M}(\mathbb{R}P^2\times I, \partial)$.
\begin{theorem}\label{Theorem: MCGofRP2Cylinder}
 Then $\mathcal{M}(\mathbb{R}P^2\times I, \partial)\cong \mathbb{Z}$, and a generator is represented by the homeomorphism $T_1$ shown in Figure \ref{fig:MCGgenerator}.
\end{theorem}
\begin{proof}
We begin by identifying $\mathbb{R}P^2\times I$ with the quotient of $D^2 \times I$ which identifies antipodal points of the boundary of each $D^2 \times \{*\}$. Let $C$ be the cylinder which is the quotient of $\partial D^2 \times I$ in $\mathbb{R}P^2 \times I$. Now consider the vertical annulus $A_0 \subset \mathbb{R}P^2 \times I$ defined as the quotient of the product $\ell \times I$, where $\ell$ is a diameter of $D^2$. Furthermore, for $i \in \mathbb{Z}$, let $A_i$ be the image of $A_0$ under the homeomorphism $T_i := (T_1)^i$. Note that the $A_i$'s can be distinguished by their intersections with $C$; each $A_i$ intersects $C$ in a curve which wraps $i$ times around $C$.

We will show that for any homeomorphism $h\colon(\mathbb{R}P^2\times I, \partial) \to (\mathbb{R}P^2\times I, \partial)$, $h(A_0)$ is ambient isotopic to some $A_i$ fixing $\partial(\mathbb{R}P^2\times I)$. Note that the boundary circles of $A_0$ intersect $C$ in two points, and are fixed by $h$, so that $\partial (h(A_0) \cap C)$ is still two points and hence the intersection $h(A_0) \cap C$, after a small isotopy to make the intersection transverse, contains a single arc $\gamma$. Now $h(A_0) \cap C$ may also contain some closed loops. Note that these closed loops are disjoint from the arc $\gamma$. In particular, they bound disks both in $C$ and in $h(A_0)$. Working from an innermost closed loop, the union of these disks is an embedded sphere. Then since $\mathbb{R}P^2\times I$ is irreducible, this sphere bounds a ball, and we can isotope $h(A_0)$ to remove the innermost intersection circle with $C$. (Note that the only prime reducible 3-manifolds are $S^2$-bundles over $S^1$, see for example \cite[Lemma 3.13]{MR2098385}, so that $\mathbb{R}P^2\times I$ is irreducible since it is prime.) Repeating this process, we may assume that $h(A_0) \cap C = \gamma$. Now in $C$, keeping the boundary fixed, the arc $\gamma$ is isotopic to $C \cap A_i$ for exactly one $i \in \mathbb{Z}$. Additionally, $h(A_0) - C$ and $A_i - C$ are ambient isotopic disks in $(\mathbb{R}P^2 \times I - C) \cong (D^2 \times I)$, fixing the boundary. Hence after an ambient isotopy, we may assume that $h(A_0) = A_i$. Given this, we will show that $h$ is isotopic to $T_i$. To see this, note that $h \circ T_{-i} (A_0) = A_0$ so that $h \circ T_{-i}$ produces a homeomorphism of $(\mathbb{R}P^2 \times I) - A_0 = B^3$, fixing the boundary. Then since $\mathcal{M}(B^3,\partial) = 0$, we conclude that $h \circ T_{-i}$ is isotopic to the identity and so $h$ is isotopic to $T_i$. Finally, note that for $i \neq 0$, $(A_i \cap C,\partial)$ is not isotopic to $(A_0 \cap C,\partial)$ so that the $T_1$ has infinite order in $\mathcal{M}(\mathbb{R}P^2\times I, \partial)$.
\end{proof}

\begin{figure}
  \begin{overpic}[width=200pt, grid=false]{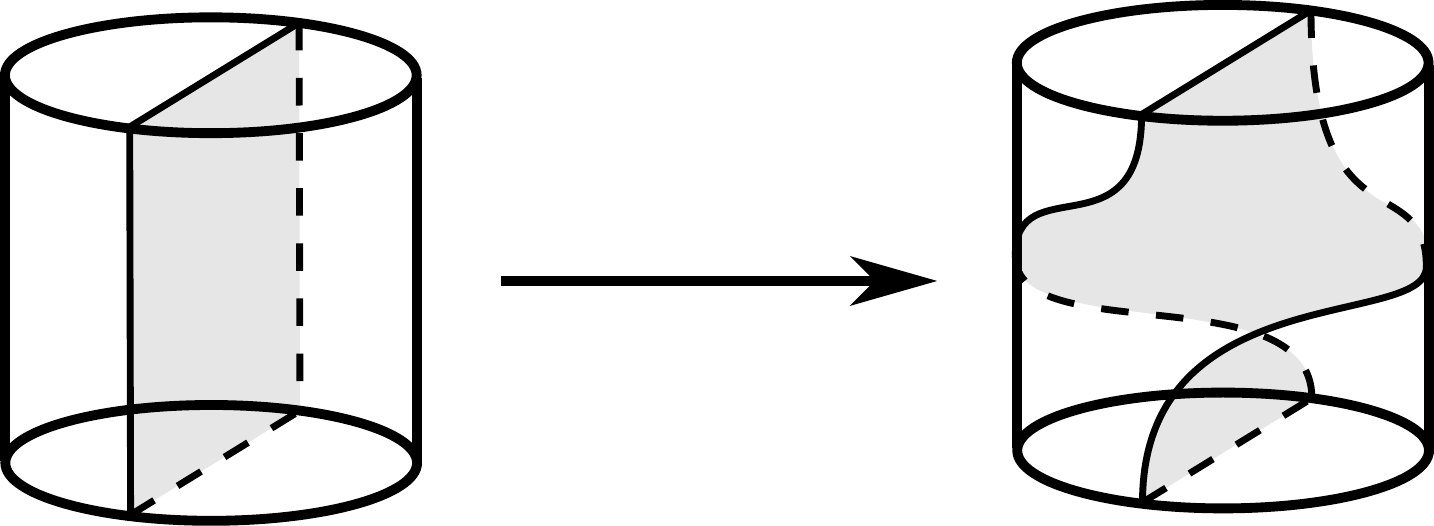}
    \put (47,20) {$T_1$}
  \end{overpic}
\caption{The homeomorphism $T_1:\mathbb{R}P^2 \times I \to \mathbb{R}P^2 \times I$ defined by $T_1(r,\theta,y) = (r, \theta + \pi y,y)$, where $\mathbb{R}P^2$ is the quotient of $D^2(r,\theta)$ by the antipodal map on the boundary. The map $T_1$ is a generator of $\mathcal{M}(\mathbb{R}P^2\times I, \partial) \cong \mathbb{Z}$.}
\label{fig:MCGgenerator}
\end{figure}

\begin{proposition}\label{Proposition:HomoemorphicLongKnotsAreIsotopic}
	Let $\mathbb{K}_1$ and $\mathbb{K}_2$ be two long knots in $\mathbb{R}P^2\times \mathbb{R}$. 
	If there exists an end-fixing homeomorphism $h\colon \mathbb{R}P^2\times \mathbb{R}\rightarrow \mathbb{R}P^2\times \mathbb{R}$ so that $h(\mathbb{K}_1)=\mathbb{K}_2$, then $\mathbb{K}_1$ and $\mathbb{K}_2$ are ambient isotopic.  
\end{proposition}
\begin{proof}
As pointed out earlier, the mapping class group of end-fixing homeomorphisms of $\mathbb{R}P^2\times \mathbb{R}$	can be identified with $\mathcal{M}(\mathbb{R}P^2\times I,\partial)$.
By Theorem \ref{Theorem: MCGofRP2Cylinder}, $h$ is isotopic to $T_i$ for some $i\in \mathbb{Z}$. Let $C >> 0$ be a constant so that both $\mathbb{K}_1$ and $\mathbb{K}_2$ restrict to $0\times[C,\infty)\subset \mathbb{R}P^2\times [C,\infty)$ and so that $h$ fixes $\mathbb{R}P^2\times[C,\infty)$. We define $t_{-i}\colon \mathbb{R}P^2\times \mathbb{R}\rightarrow \mathbb{R}P^2\times \mathbb{R}$ by $$t_{-i}(r,\theta,y)=
\begin{cases}
(r,\theta-i\pi(y-C), y)\ \ y\in[C,C+1],\\
(r,\theta, y)\ \ \text{otherwise.} 
\end{cases}$$ 
Note  $h\circ t_{-i}(\mathbb{K}_1)=\mathbb{K}_2$. Moreover, since $t_i$ is isotopic to $T_{-i}$, the map $h\circ t_{-i}$ is isotopic to the identity map through end-fixing homeomorphisms $H_t$, $t\in[0,1]$; $H_t$ provides an ambient isotopy connecting $\mathbb{K}_1$ and $\mathbb{K}_2$.   
\end{proof}

\begin{definition}
Two long knots $\mathbb{K}_1, \mathbb{K}_2\subset \mathbb{R}P^2\times \mathbb{R}$ are said to be \textbf{narrowly ambient isotopic} if there exists a neighborhood $N$ of $\mathbb{K}_1$ such that $N\cong D^2\times \mathbb{R}$ and an ambient isotopy $H_t: N\rightarrow N$ that fixes $\partial N$ and deforms $\mathbb{K}_1$ to $\mathbb{K}_2$. 
\end{definition}
In other words, two long knots in $\mathbb{R}P^2\times \mathbb{R}$ are narrowly ambient isotopic as long knots if they are ambient isotopic as long knots in some $D^2\times \mathbb{R}$ embedded in $\mathbb{R}P^2\times \mathbb{R}$.   

\begin{definition}
	Two long knots $\mathbb{K}_1, \mathbb{K}_2\subset \mathbb{R}P^2\times \mathbb{R}$ are said to be \textbf{narrowly equivalent} if there exists finitely many long knots $\mathbb{J}_1,\mathbb{J}_2,\ldots,\mathbb{J}_m$ such that $\mathbb{J}_1=\mathbb{K}_1$, $\mathbb{J}_m=\mathbb{K}_2$, and $\mathbb{J}_i$ is narrowly ambient isotopic to $\mathbb{J}_{i+1}$ for $i=1,2,\ldots m-1$. 
\end{definition}

\begin{lemma} \label{lemma:narrowequivalence}
	Let $\mathbb{K}_1$ and $\mathbb{K}_2$ be two long knots in $\mathbb{R}P^2\times \mathbb{R}$. Suppose $\mathbb{K}_1$ and $\mathbb{K}_2$ are ambient isotopic, then they are narrowly equivalent. 
\end{lemma}
\begin{proof}
Up to small isotopy and perturbations, we may assume $\mathbb{K}_1$ and $\mathbb{K}_2$ are smooth and that there is a smooth ambient isotopy $H_t$ that deforms $\mathbb{K}_1$ to $\mathbb{K}_2$. Let $\Phi\colon\mathbb{R}\times I\rightarrow \mathbb{R}P^2\times \mathbb{R}$ denote the trace of the ambient isotopy. Specifically, if $\mathbb{K}_1$ is represented by the embedding $f\colon \mathbb{R}\rightarrow \mathbb{R}P^2\times \mathbb{R}$ then $\Phi(x, t)=H_t(f(x))$. Then there exist $0=a_0<a_1<\cdots<a_m=1$ such that the long knots $\mathbb{J}_i:=\Phi(-,a_i)$ admit neighborhoods $N_i\cong D^2\times \mathbb{R}$ and such that $\Phi(x, t)\in N_i$ for $t\in[a_i,a_{i+1}]$, $i=1,2,\ldots,m-1$. Next, for $i=1,2,\ldots,m-1$, we show $\mathbb{J}_i$ and $\mathbb{J}_{i+1}$ are narrowly ambient isotopic by constructing an ambient isotopy $H^{(i)}_t\colon N_i\rightarrow N_i$, $t\in[0,1]$. We will define $H^{(i)}_t$ to be the time-$t$ flow map of a time-dependent vector field $X_t$ supported in a compact subset of $N_i$ away from the boundary. More specifically, we define $X_t(\Phi(x,(1-t)a_i+ta_{i+1}))=\frac{1}{a_{i+1}-a_i}\frac{\partial\Phi}{\partial t}$ and extend it smoothly to a $t$-dependent vector field compactly supported away from $\partial N_i$. By construction, $H^{(i)}_t$ fixes $\partial N_i$ and deforms $\mathbb{J}_i$ to $\mathbb{J}_{i+1}$.    
\end{proof}

\begin{definition}
	Let $K_1$ and $K_2$ be polygonal knots in a PL manifold. The knots $K_1$ and $K_2$ are related by a \emph{triangle move} if there is a (possibly degenerate) PL 2-simplex with two edges on $K_1$ and one edge on $K_2$, or vice versa, and $K_1$ and $K_2$ are identical away from this simplex. The knots $K_1$ and $K_2$ are \emph{$\Delta$-equivalent} if they are related by a finite sequence of triangle moves. 
\end{definition}
\begin{lemma} \label{lemma:deltaequivalence}
	Let $\mathcal{K}_1$ and $\mathcal{K}_2$ be two polygonal long knots in $D^2\times \mathbb{R}$. If there exists an end-fixing PL homeomorphism $f\colon D^2\times \mathbb{R}\rightarrow D^2\times \mathbb{R}$ so that $f(\mathcal{K}_1)=\mathcal{K}_2$, then $\mathcal{K}_1$ and $\mathcal{K}_2$ are $\Delta$-equivalent.   
\end{lemma}
\begin{proof}
We may assume the PL homeomorphism $f$ is isotopic to the identity map through end-fixing homeomorphisms; if not, using a similar argument as in the proof of Proposition \ref{Proposition:HomoemorphicLongKnotsAreIsotopic}, we may replace $f$ by its composition with an appropriate twisting map that fixes both $\mathcal{K}_1$ and $\mathcal{K}_2$. (Here, we used the fact that $\mathcal{M}(D^2 \times I, D^2 \times \{0,1\})\cong \mathbb{Z}$, where the mapping classes are represented by twisting maps; we omit the proof and remark that this can be seen by a similar argument as in the proof of Theorem \ref{Theorem: MCGofRP2Cylinder}.) In particular, now we may isotope $f$ in the complement of $\nu(\mathcal{K}_1)\cup\nu(\mathcal{K}_2)$ so that it fixes a collar neighborhood of $\partial(D^2\times \mathbb{R})$, where $\nu(\mathcal{K}_i)$ denotes a regular neighborhood of $\mathcal{K}_i$. Let $D_r$ denote the disc $\{(x,y)|\ x^2+y^2\leq r^2\}$. Up to reparametrization, we may assume $\mathcal{K}_1\cup\mathcal{K}_2\subset D_{1/5}\times \mathbb{R}$ and $f$ fixes the complement of $D_{1/2}\times \mathbb{R}$.
	
Let $C\gg 0$ so that both $\mathcal{K}_1$ and $\mathcal{K}_2$ restricts to $\{0\}\times(-\infty,C]\cup\{0\}\times[C,\infty)$ and so that $f$ fixes $\mathbb{R}P^2\times \{y\}$ for $|y| \geq C$. Let $\ell$ be an arc obtained by shifting $\mathcal{K}_1\cap D_{1/5}\times [-C,C]$ horizontally so that $\ell$ is contained in the neighborhood of $\partial(D^2 \times \mathbb{R})$ which is fixed by $f$. Let $\mathcal{K}_1'$ be the knot obtained by joining the upper end of $\ell$ to the lower end of $\{0\}\times[C,\infty)$ and joining the lower end of $\ell$ to the upper end of $\{0\}\times(-\infty,C]$. 

We now claim that both $\mathcal{K}_1$ and $\mathcal{K}_2$ are $\Delta$-equivalent to $\mathcal{K}_1'$. To see this, we can perform triangle moves along the horizontal shift between $\mathcal{K}_1$ and $\mathcal{K}_1'$. Applying $f$ then produces triangle moves relating $\mathcal{K}_2 = f(\mathcal{K}_1)$ and $\mathcal{K}_1' = f(\mathcal{K}_1')$. Composing these gives the $\Delta$-equivalence between $\mathcal{K}_1$ and $\mathcal{K}_2$.
\end{proof}

\begin{theorem} \label{thm:Rmovesappendix}
Let $(K_1,\rho)$ and $(K_2,\rho)$ be equivalent strongly negative amphichiral knots in $S^3$. Then any symmetric diagram for $K_1$ is related to any symmetric diagram for $K_2$ by a finite sequence of SNA Reidemeister moves. 
\end{theorem}
\begin{proof}
Throughout the proof, we will fix a generic $\rho$-invariant $S^2 \subset S^3$ containing all knot diagrams. We may assume, by applying a small isotopy if necessary, that $K_1$ and $K_2$ coincide in a neighborhood of each fixed point in $S^3$. This small isotopy projects to a planar isotopy of the symmetric diagram. Now removing neighborhoods of the two fixed points, we can consider the quotient of $S^3$ by the amphichiral symmetry, which is $\mathbb{R}P^2 \times \mathbb{R}$. The images of $K_1$ and $K_2$ are (oriented) long knots $\mathbb{K}_1$ and $\mathbb{K}_2$, and the equivariant homeomorphism between $K_1$ and $K_2$ descends to an end-fixing homeomorphism between $\mathbb{K}_1$ and $\mathbb{K}_2$. In fact, by Proposition \ref{Proposition:HomoemorphicLongKnotsAreIsotopic}, there is an end-fixing ambient isotopy between $\mathbb{K}_1$ and $\mathbb{K}_2$. Then by Lemma \ref{lemma:narrowequivalence}, $\mathbb{K}_1$ and $\mathbb{K}_2$ are narrowly equivalent through a sequence of long knots $\mathbb{J}_1$,\dots,$\mathbb{J}_m$. It suffices to show that the symmetric diagrams for the strongly negative amphichiral knots corresponding to $\mathbb{J}_i$ and $\mathbb{J}_{i+1}$ (which are narrowly ambient isotopic) are related by SNA Reidemeister moves. For convenience we will just assume that $\mathbb{K}_1$ and $\mathbb{K}_2$ are narrowly ambient isotopic in some $D^2 \times \mathbb{R} \subset \mathbb{R}P^2 \times \mathbb{R}$. This narrow ambient isotopy produces a homeomorphism $f\colon D^2 \times \mathbb{R} \to D^2 \times \mathbb{R}$ with $f(\mathbb{K}_1) = \mathbb{K}_2$. Furthermore, by applying small isotopies which induce planar isotopies of the corresponding symmetric diagrams, we will assume that $\mathbb{K}_1$ and $\mathbb{K}_2$ are polygonal.

We will use $f$ to get a PL homeomorphism $D^2 \times \mathbb{R} \to D^2 \times \mathbb{R}$ sending $\mathbb{K}_1$ to $\mathbb{K}_2$. (The proof is a straightforward modification to that of the corresponding statement for knots in $S^3$; see for example \cite[Theorem 6.1]{MR61822} or \cite[Theorem A.3]{MR1417494}.) By \cite[Theorem 2]{MR48805}, we can find a homeomorphism $f'$ which is an approximation of $f$ and which is piecewise-linear on the exterior $(D^2 \times \mathbb{R}) - \nu(\mathbb{K}_1)$ of $\mathbb{K}_1$, where $\nu(\mathbb{K}_1)$ is a PL-regular neighborhood of $\mathbb{K}_1$. A priori $f'(\nu(\mathbb{K}_1))$ may not be a PL-regular neighborhood of $\mathbb{K}_2$, but it is at least a PL submanifold of $D^2 \times \mathbb{R}$, since its boundary $f'(\partial \nu(\mathbb{K}_1))$ is a PL submanifold. We can then find a PL-regular neighborhood $\nu(\mathbb{K}_2)$ inside $f'(\nu(\mathbb{K}_1))$, and there is a PL strip $\mathbb{R} \times [0,1]$ interpolating between the image of a longitude $\mathbb{L}_1$ of $\mathbb{K}_1$ in $f'(\partial \nu(\mathbb{K}_1))$ and a longitude $\mathbb{L}_2$ of $\mathbb{K}_2$ in $\partial \nu(\mathbb{K'})$, by \cite[Theorem 3]{MR61377}. This strip gives a PL isotopy between $\mathbb{L}_2$ and $f'(\mathbb{L}_1)$. We also have PL isotopies between $\mathbb{K}_1$ and $\mathbb{L}_1$, and between $\mathbb{K}_2$ and $\mathbb{L}_2$. In particular, there is a composite PL homeomorphism: 
\[
(D^2 \times \mathbb{R},\mathbb{K}_1) \to (D^2 \times \mathbb{R},\mathbb{L}_1) \to (D^2 \times \mathbb{R},f'(\mathbb{L}_1)) \to (D^2 \times \mathbb{R},\mathbb{L}_2) \to (D^2 \times \mathbb{R},\mathbb{K}_2).
\]

Now by Lemma \ref{lemma:deltaequivalence}, we have a $\Delta$-equivalence between $\mathbb{K}_1$ and $\mathbb{K}_2$. Lifting the triangle moves from this equivalence to $S^3$ produces equivariant pairs of triangle moves relating $K_1$ and $K_2$. Let $\Delta$ and $\rho(\Delta)$ be an equivariant pair of triangles, and consider the projections of $\Delta$ and $\rho(\Delta)$ to the $S^2$ of the diagram. If the projections of $\Delta$ and $\rho(\Delta)$ are disjoint, then we are reduced to an equivariant isotopy or an equivariant pair of R1, R2, or R3 moves. If instead $\Delta$ and $\rho(\Delta)$ have overlapping projections, then (possibly after replacing $\Delta$ with a smaller triangle in a subdivision) the union of their projections is a $\rho$-invariant contractible region and hence contains exactly one of the two fixed points. The projections of $\Delta$ and $\rho(\Delta)$ can then be subdivided into an equivariant sequence of R1,R2, and R3 moves, and one of the two R4 moves as shown in the statement of Theorem \ref{thm:Rmoves}.
\end{proof}

\section{Table of strongly negative amphichiral knots}
\label{app:table}
For convenience we provide a table of symmetric diagrams for all strongly negative amphichiral prime knots with 12 or fewer crossings. This table was produced by taking the list of all negative amphichiral knots with 12 or fewer crossings (see for example \cite{knotinfo}) and adjusting the diagrams to be symmetric.

\begin{longtable}{c||c||c}
 $4_1$&$6_3$&$8_3$\\*
 \multirow{4}{*}{\scalebox{.2}{\includegraphics{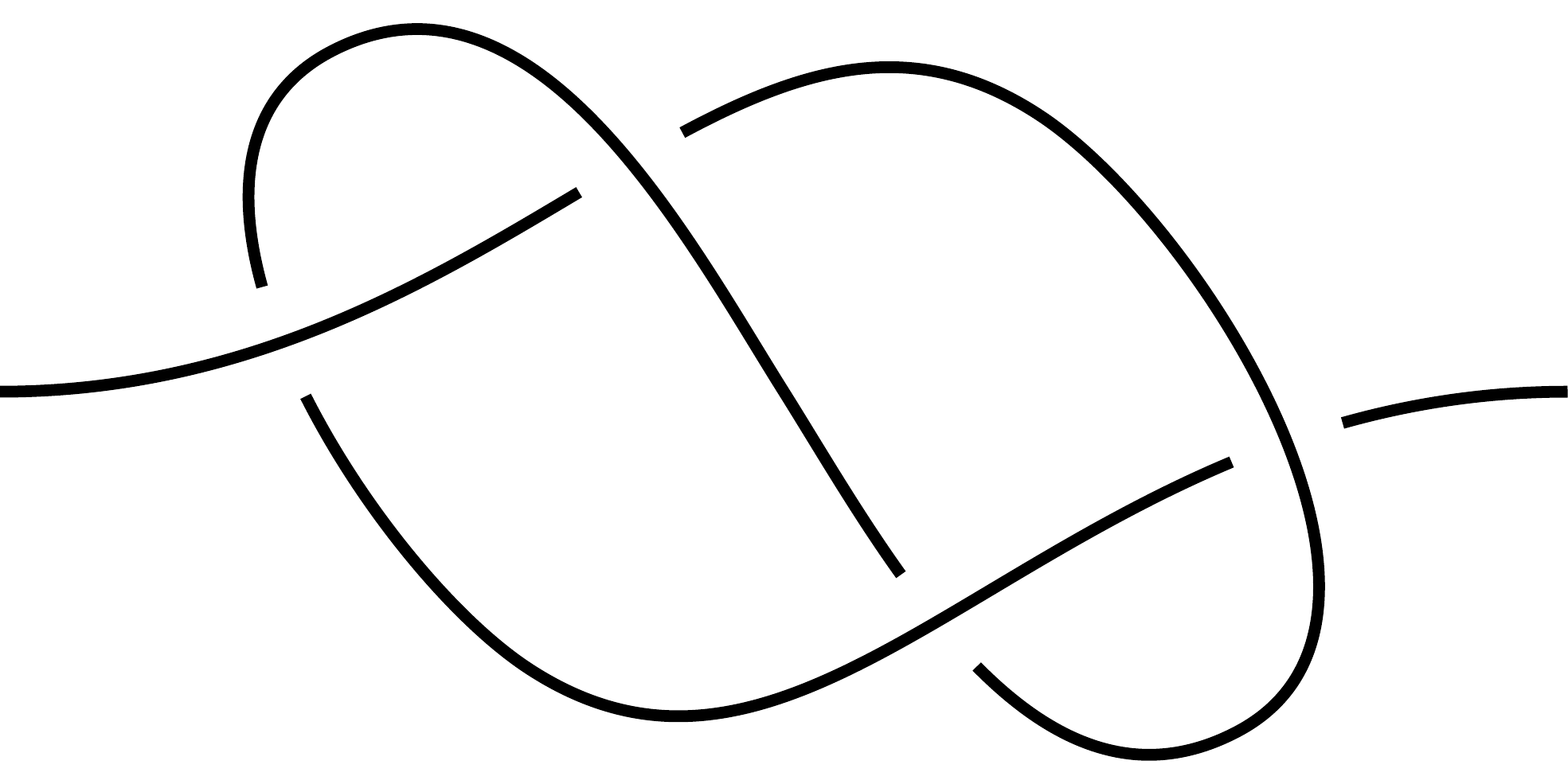}}}&\multirow{4}{*}{\scalebox{.2}{\includegraphics{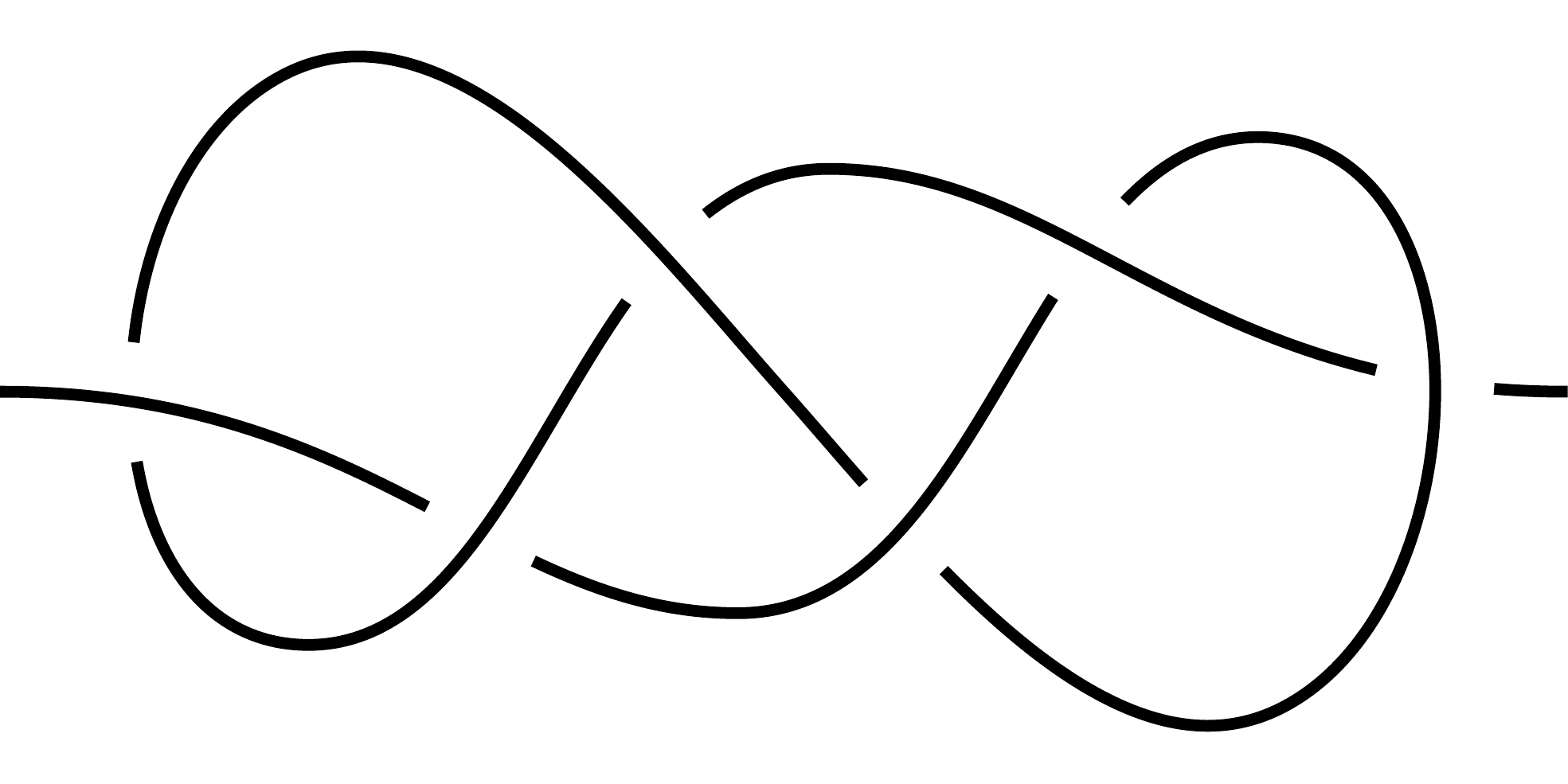}}}&
 \multirow{4}{*}{\scalebox{.2}{\includegraphics{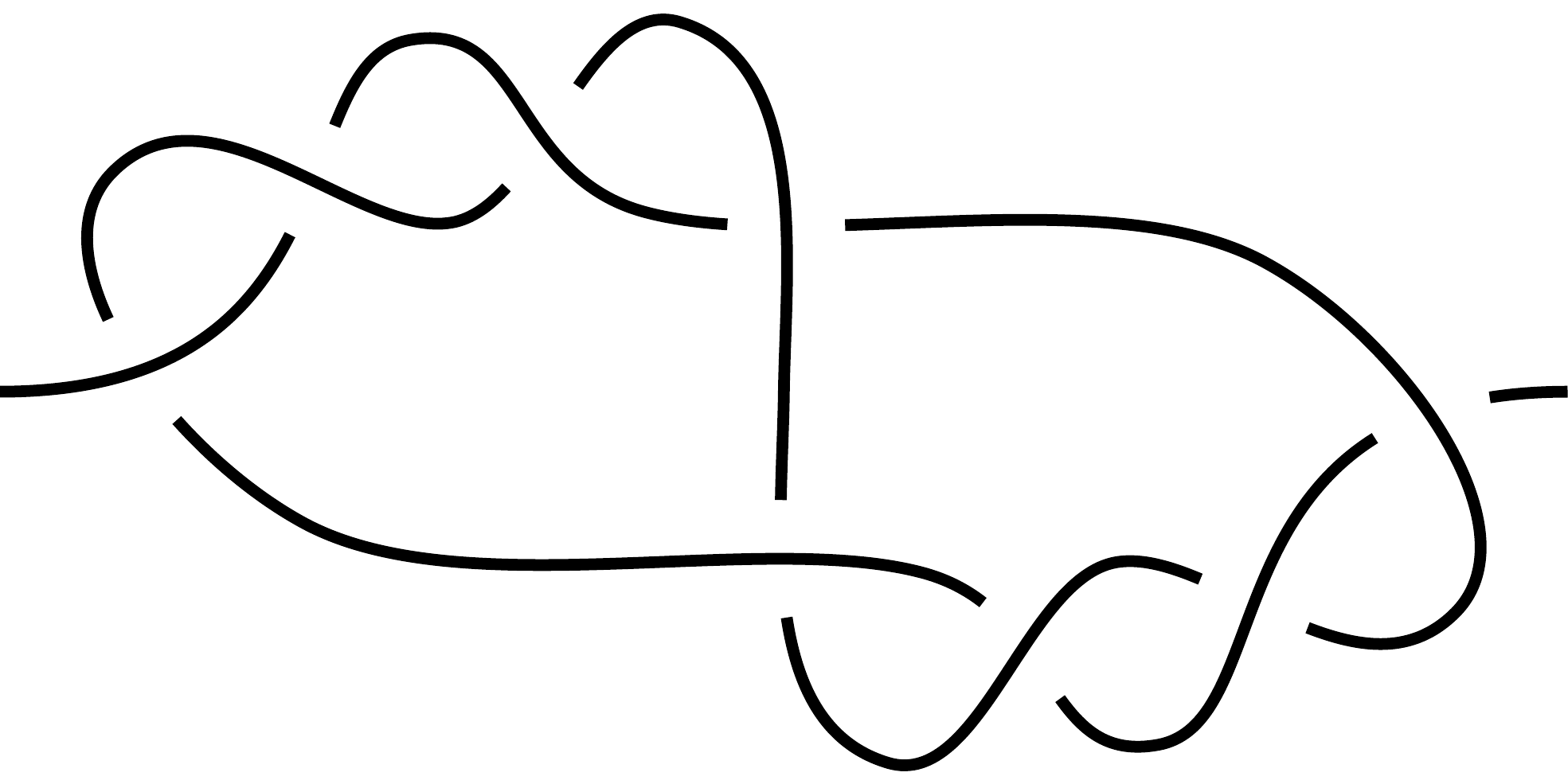}}}\\*
 &&\\*
 &&\\*
 &&\\*
 &&\\\hline\hline

 $8_9$&$8_{12}$&$8_{17}$\\*
 \multirow{4}{*}{\scalebox{.2}{\includegraphics{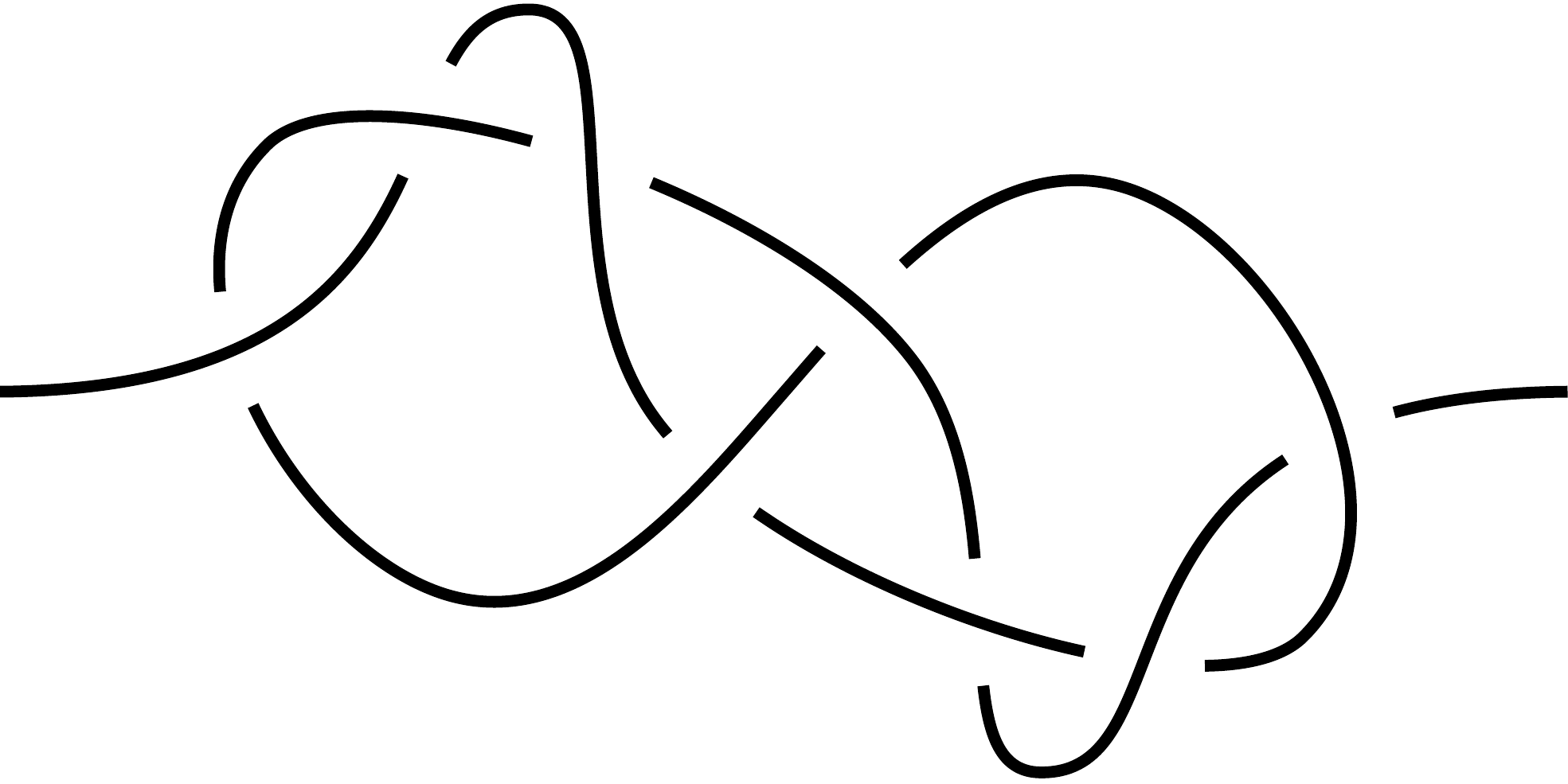}}}&\multirow{4}{*}{\scalebox{.2}{\includegraphics{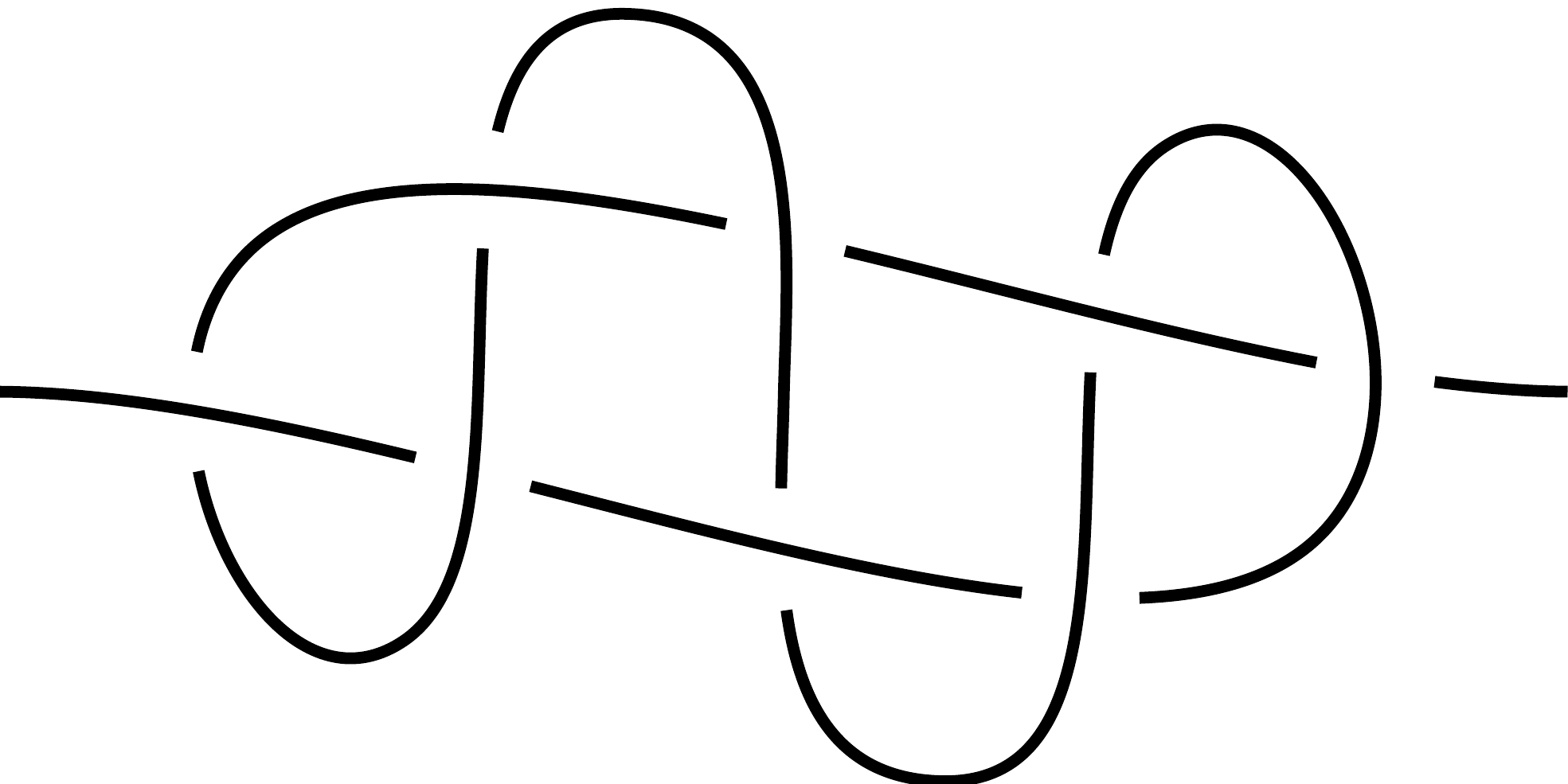}}}&
 \multirow{4}{*}{\scalebox{.2}{\includegraphics{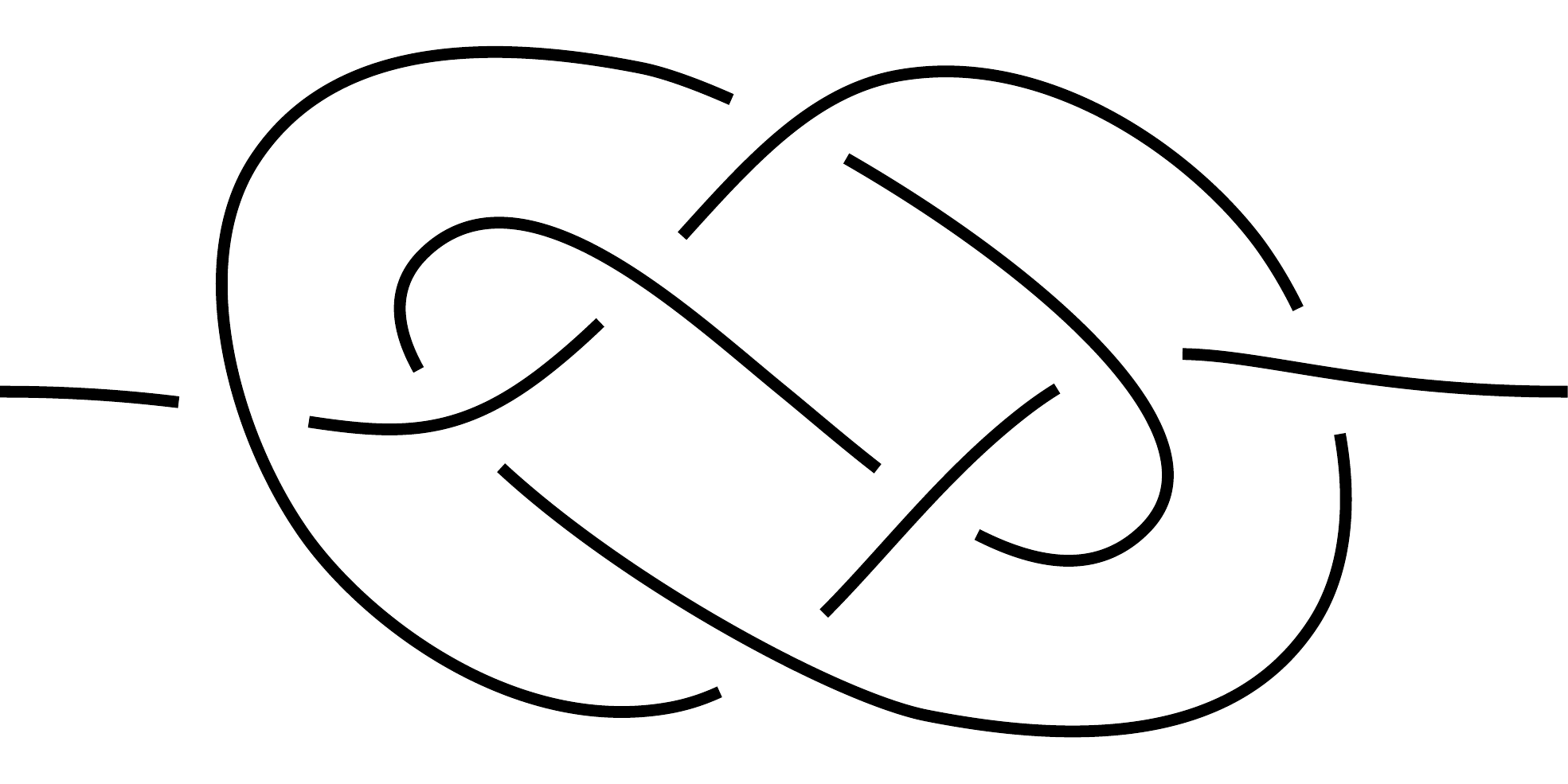}}}\\*
 &&\\*
 &&\\*
 &&\\*
 &&\\\hline\hline

 $8_{18}$&$10_{17}$&$10_{33}$\\*
 \multirow{4}{*}{\scalebox{.2}{\includegraphics{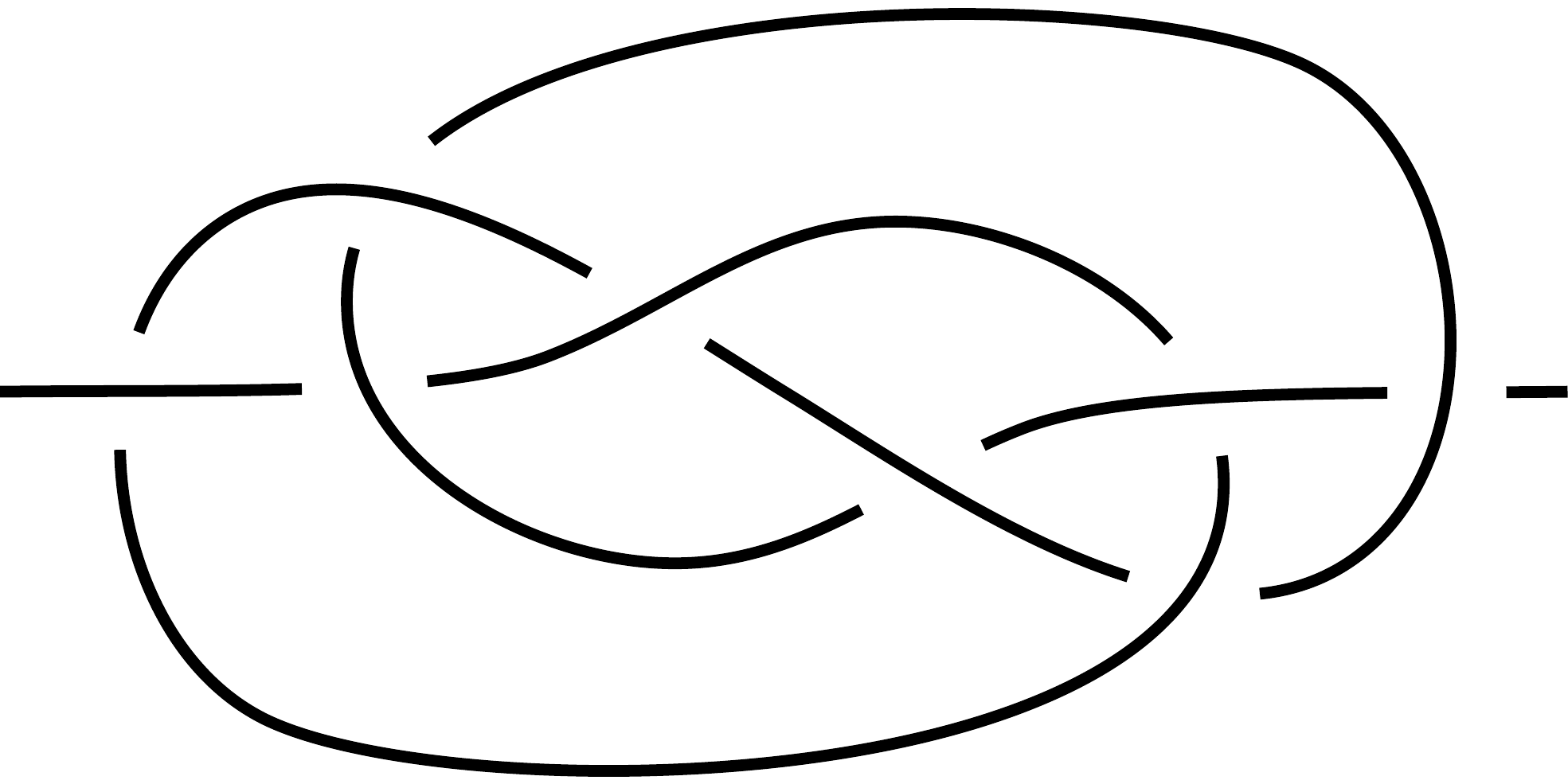}}}&\multirow{4}{*}{\scalebox{.2}{\includegraphics{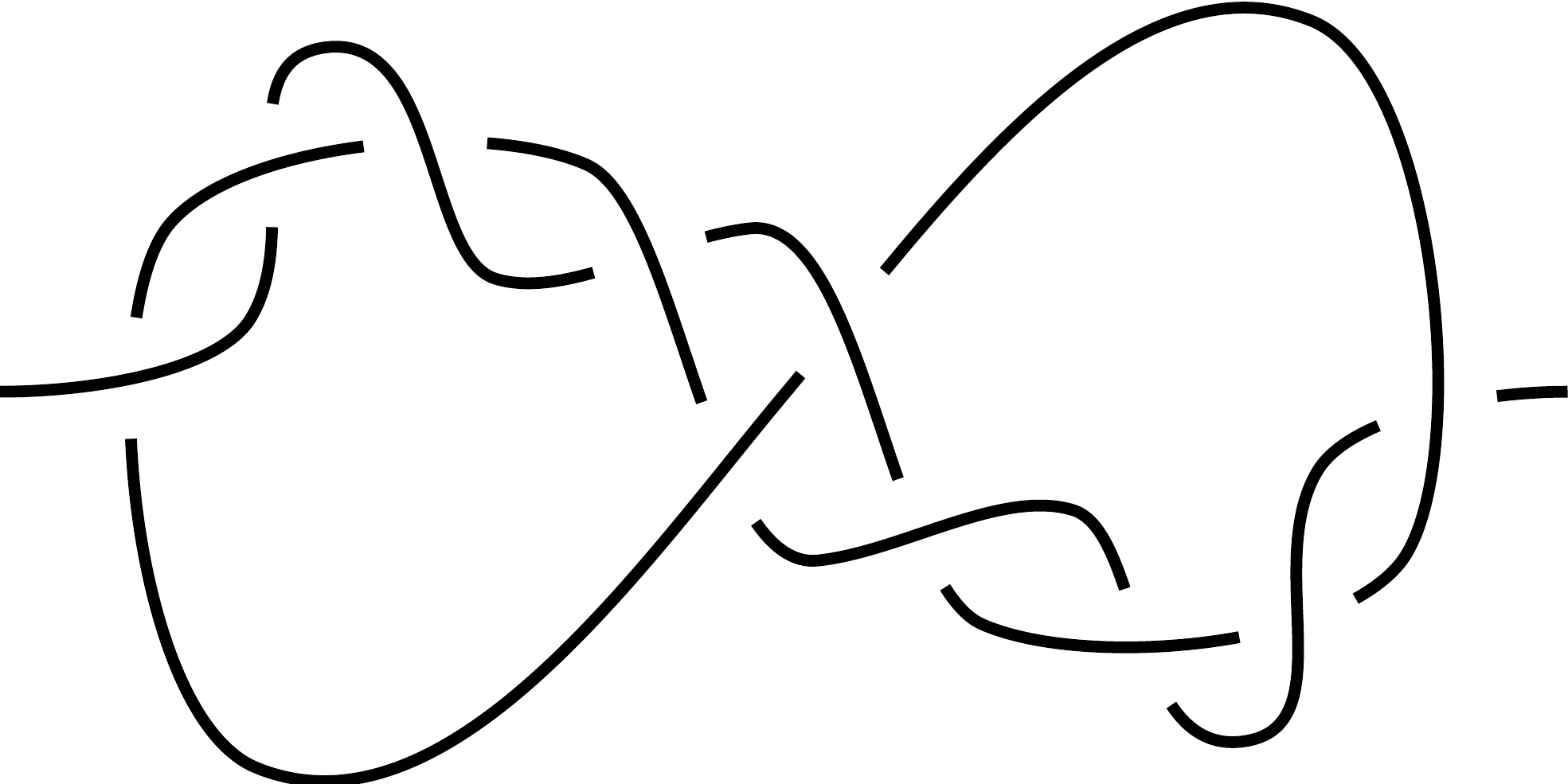}}}&
 \multirow{4}{*}{\scalebox{.2}{\includegraphics{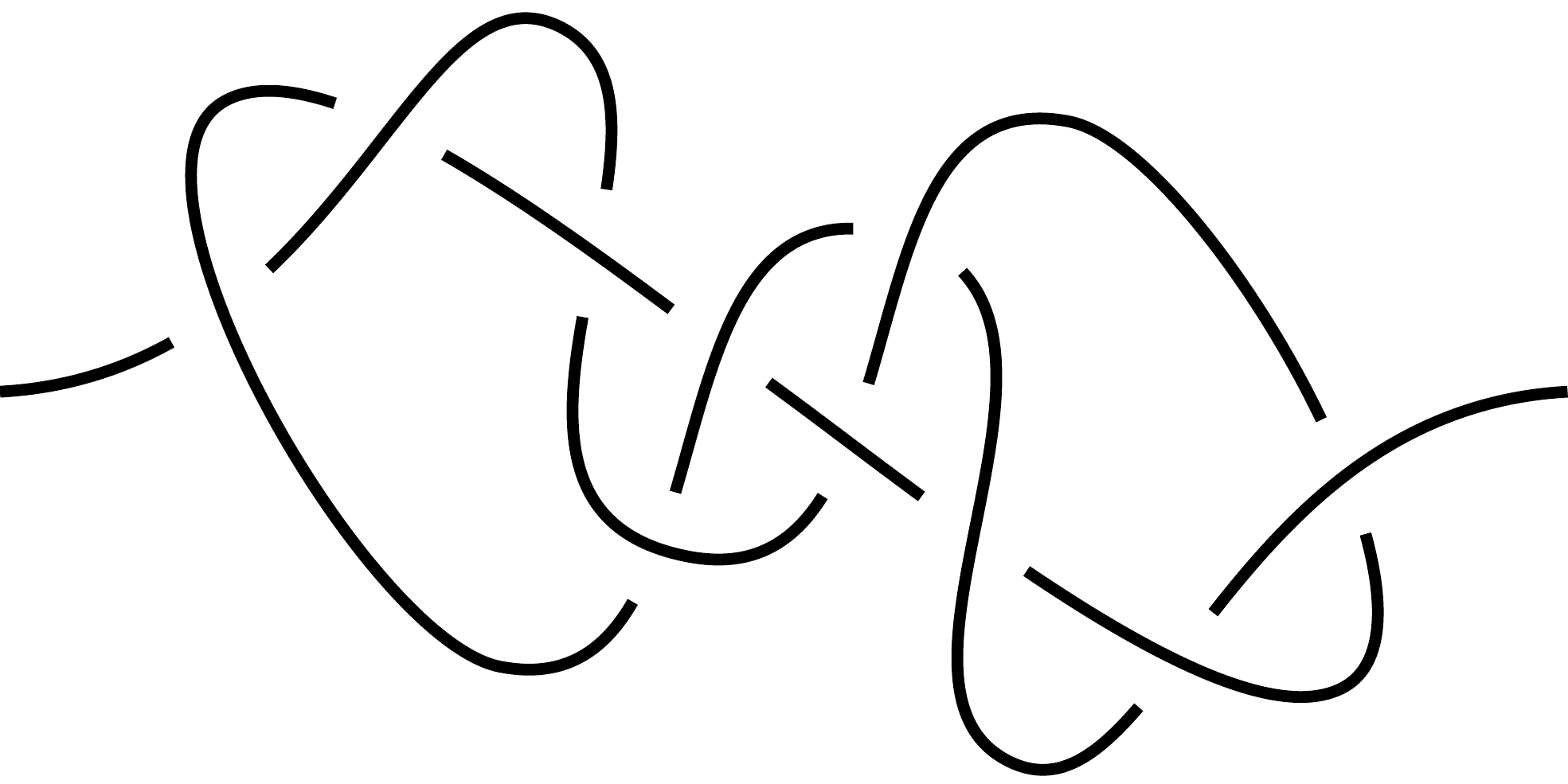}}}\\*
 &&\\*
 &&\\*
 &&\\*
 &&\\\hline\hline

 $10_{37}$&$10_{43}$&$10_{45}$\\*
 \multirow{4}{*}{\scalebox{.2}{\includegraphics{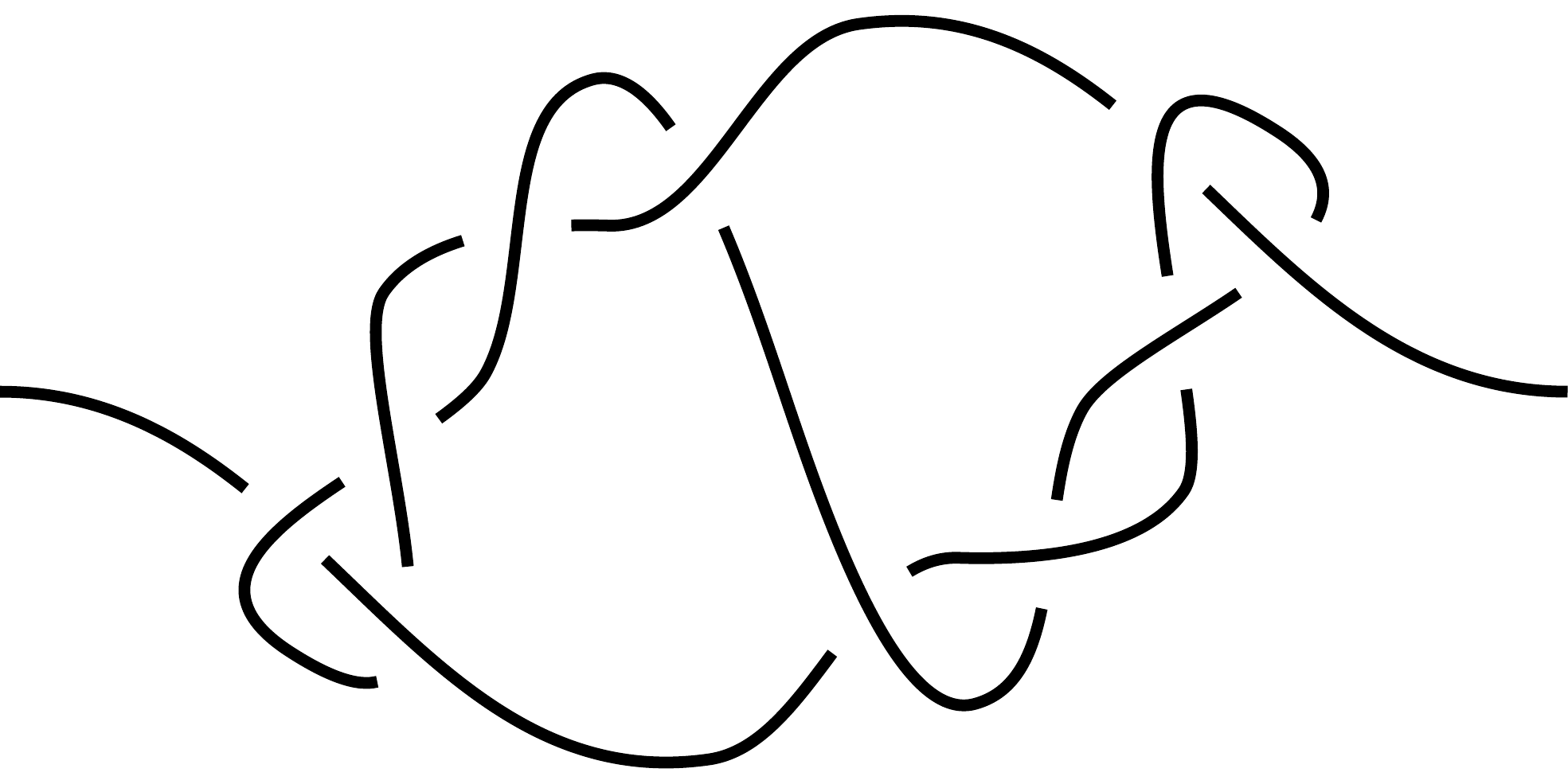}}}&\multirow{4}{*}{\scalebox{.2}{\includegraphics{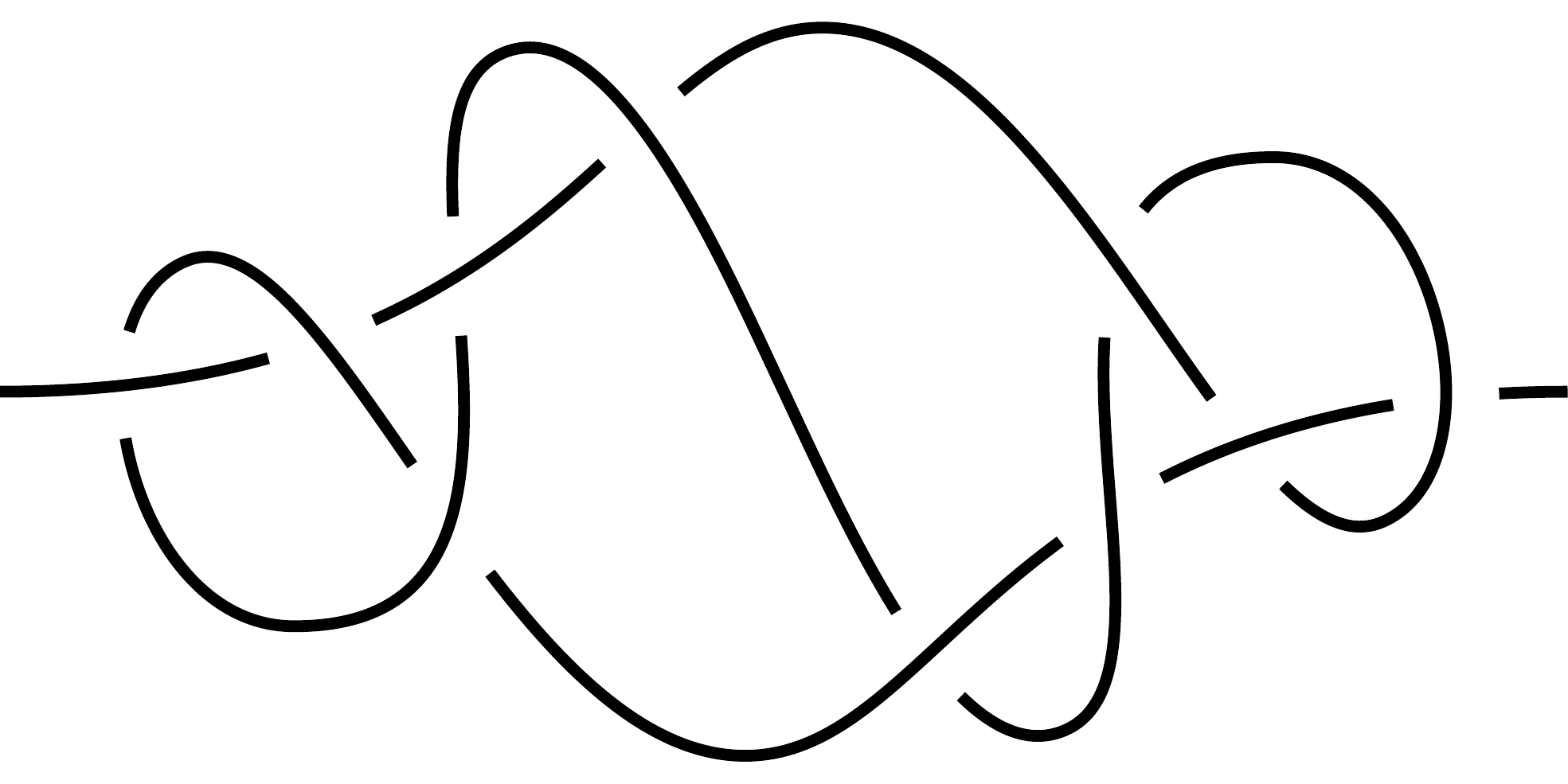}}}&
 \multirow{4}{*}{\scalebox{.2}{\includegraphics{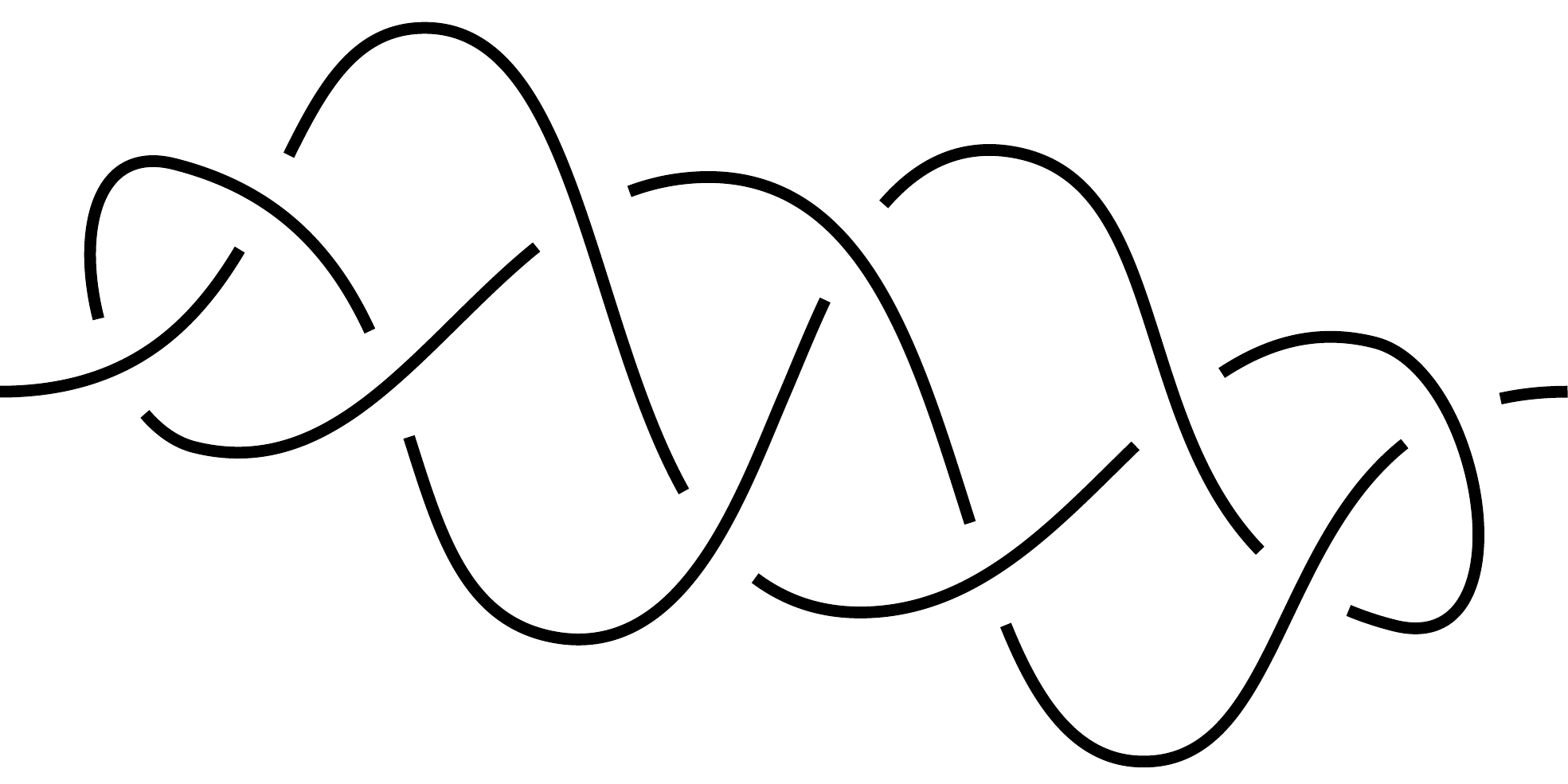}}}\\*
 &&\\*
 &&\\*
 &&\\*
 &&\\\hline\hline

 $10_{79}$&$10_{81}$&$10_{88}$\\*
 \multirow{4}{*}{\scalebox{.2}{\includegraphics{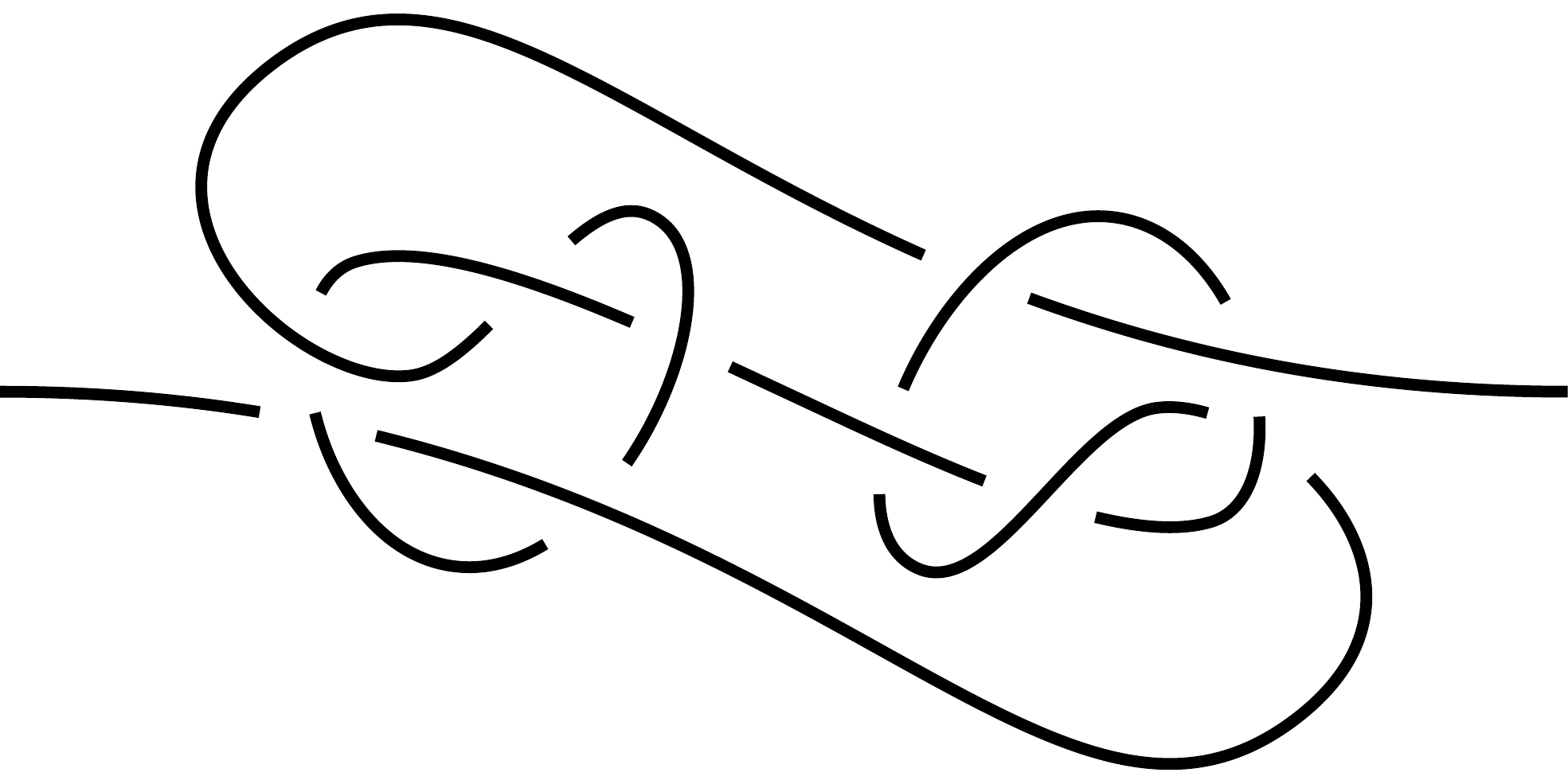}}}&\multirow{4}{*}{\scalebox{.2}{\includegraphics{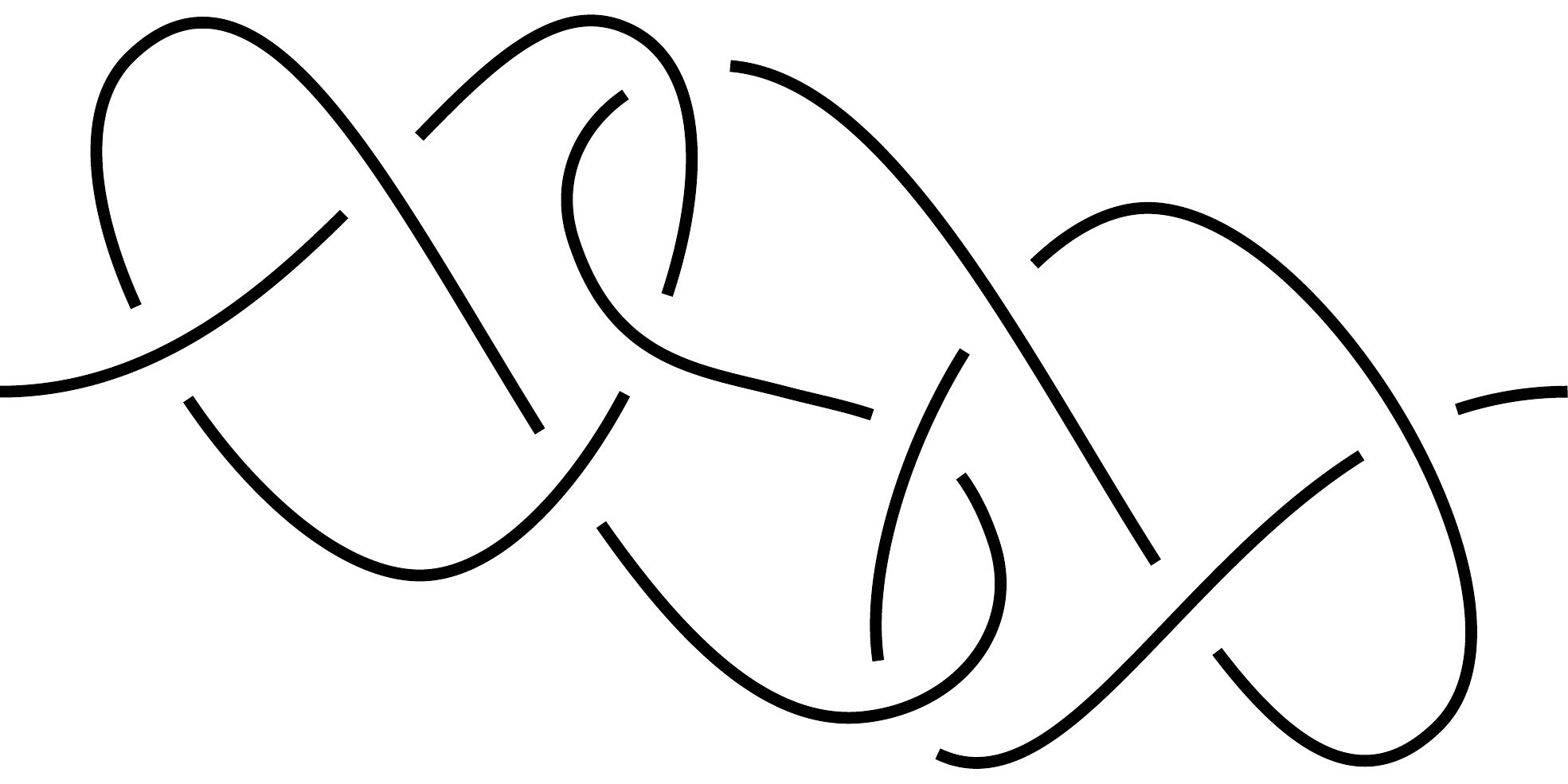}}}&
 \multirow{4}{*}{\scalebox{.2}{\includegraphics{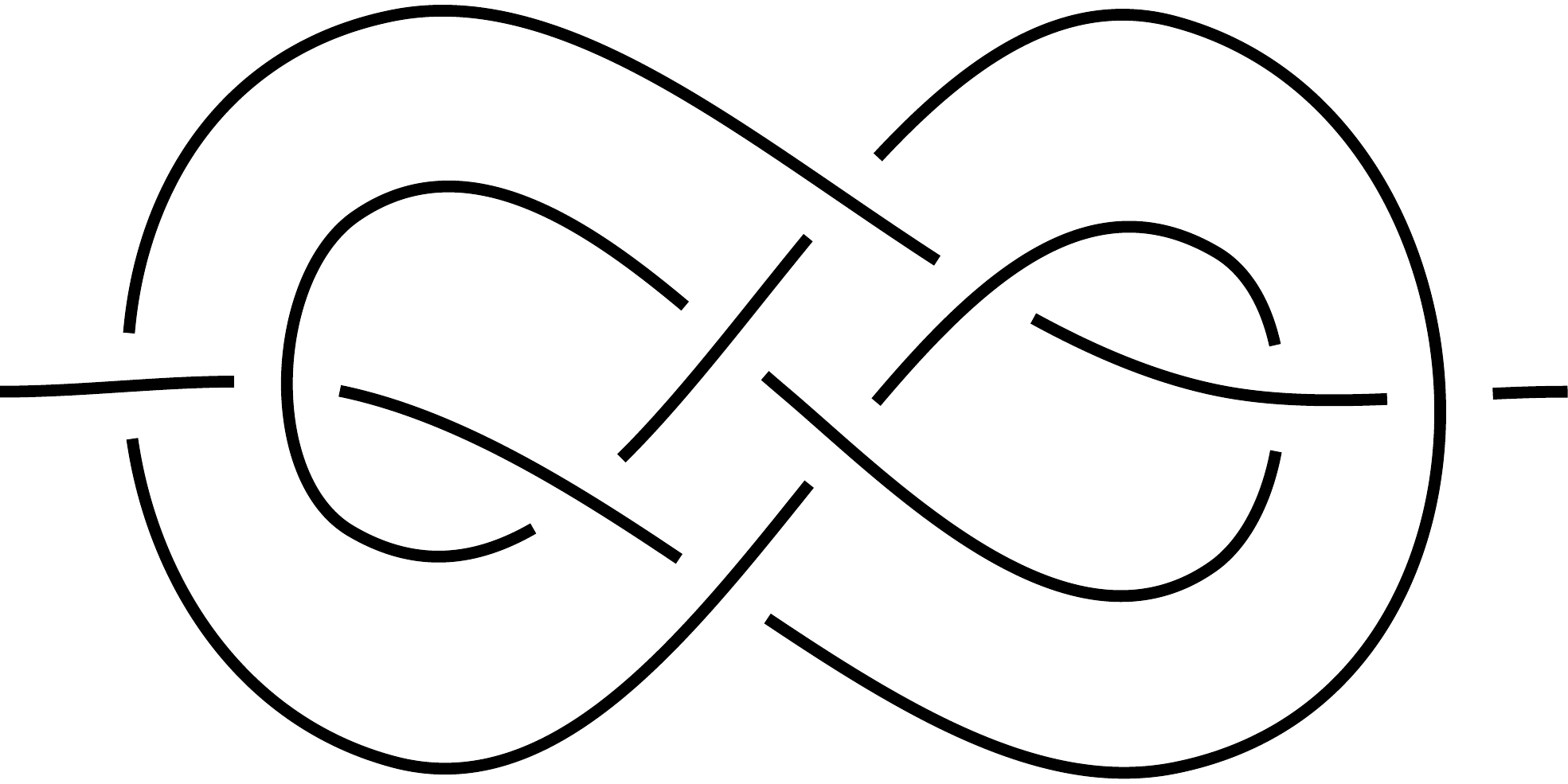}}}\\*
 &&\\*
 &&\\*
 &&\\*
 &&\\\hline\hline

 $10_{99}$&$10_{109}$&$10_{115}$\\*
 \multirow{4}{*}{\scalebox{.2}{\includegraphics{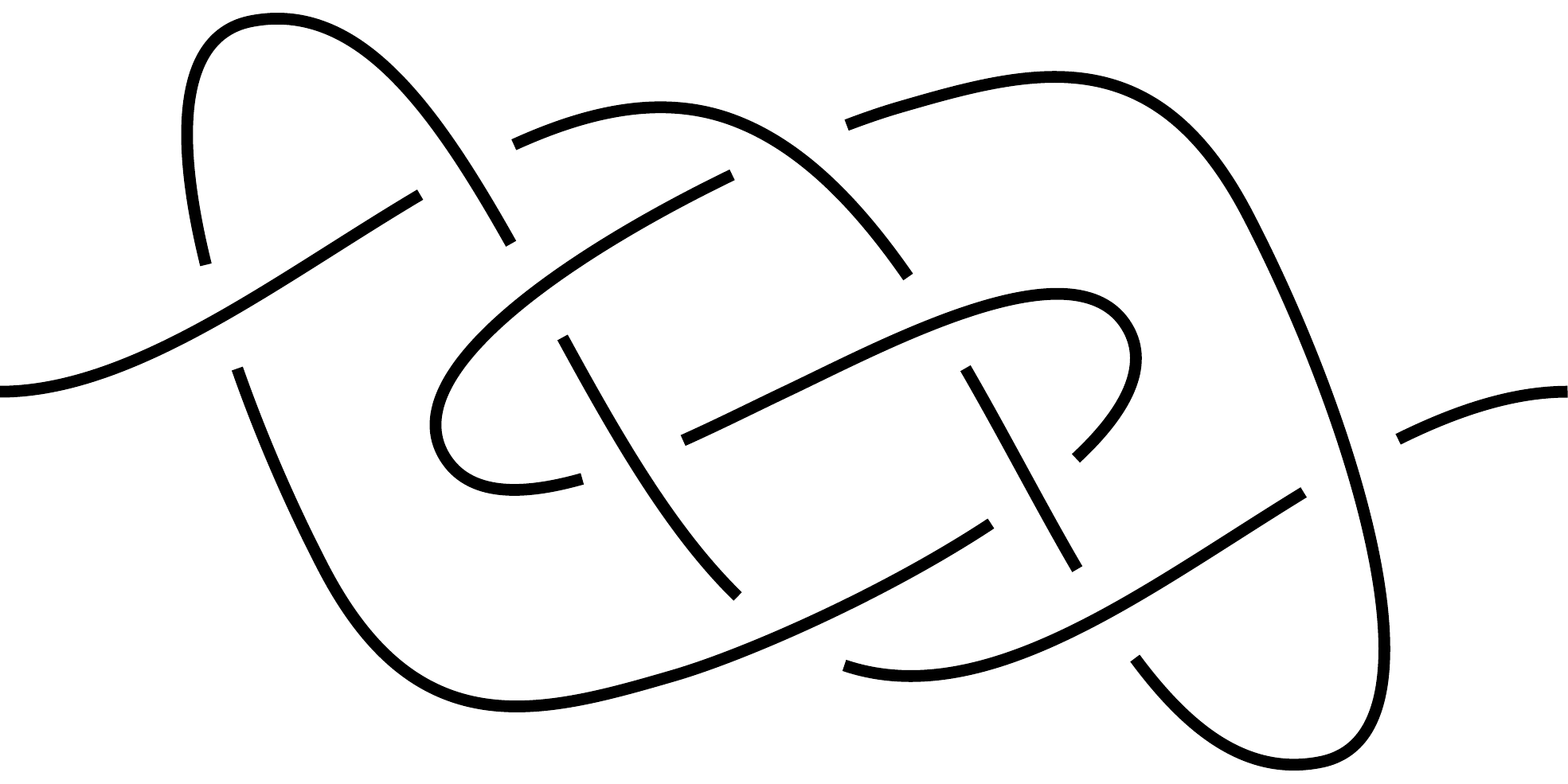}}}&\multirow{4}{*}{\scalebox{.2}{\includegraphics{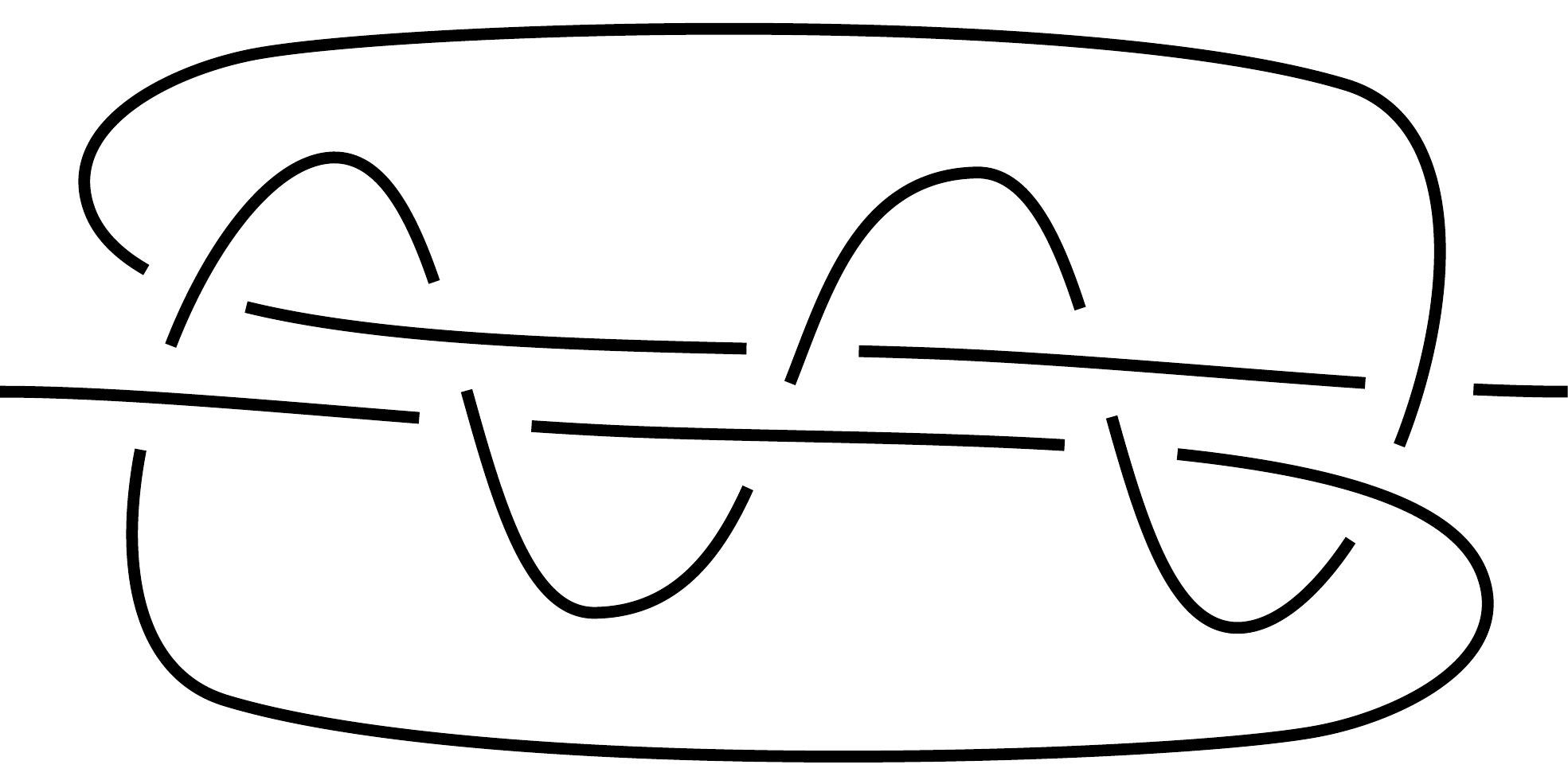}}}&
 \multirow{4}{*}{\scalebox{.2}{\includegraphics{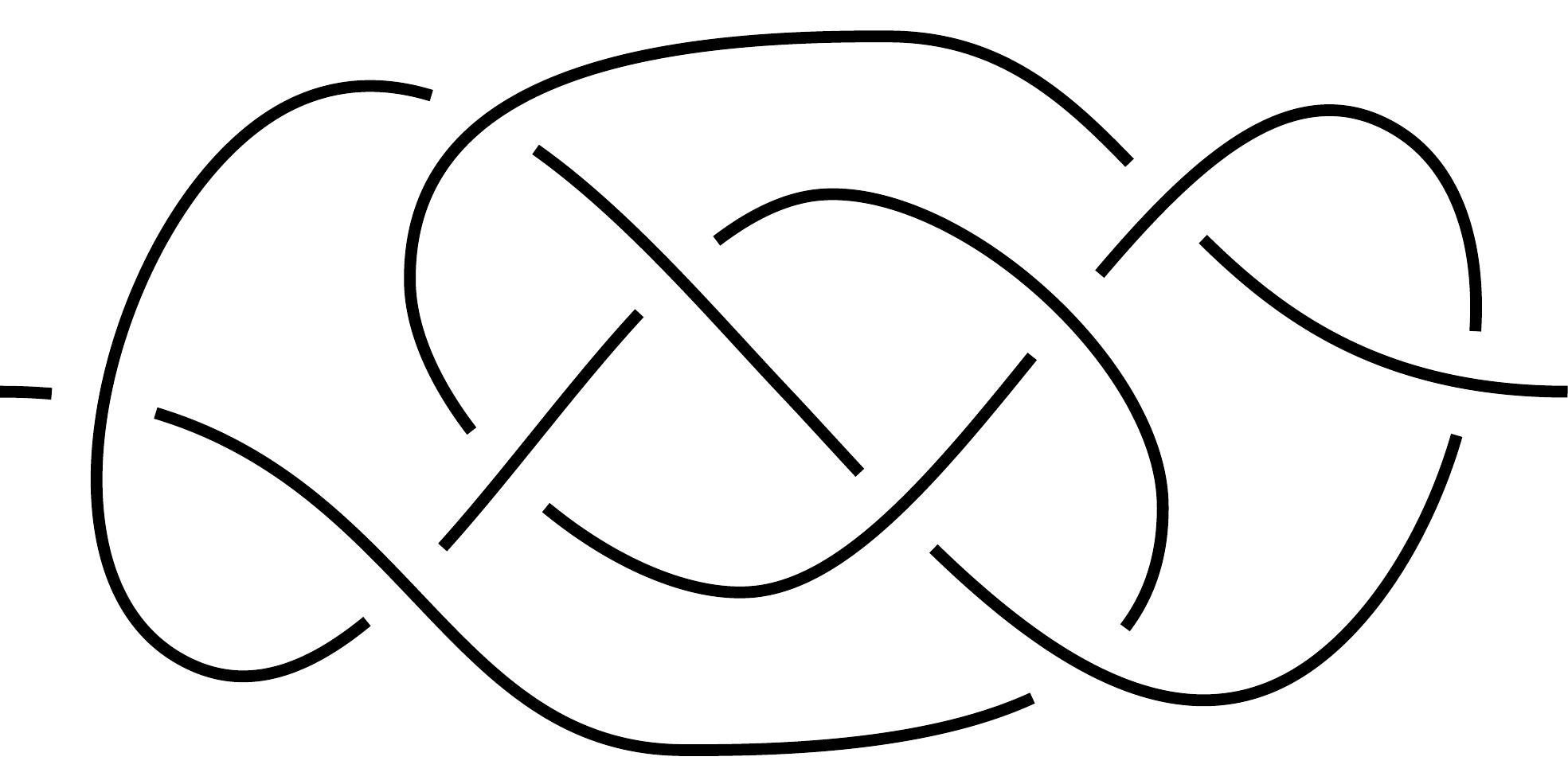}}}\\*
 &&\\*
 &&\\*
 &&\\*
 &&\\\hline\hline

 $10_{118}$&$10_{123}$&$12a_4$\\*
 \multirow{4}{*}{\scalebox{.2}{\includegraphics{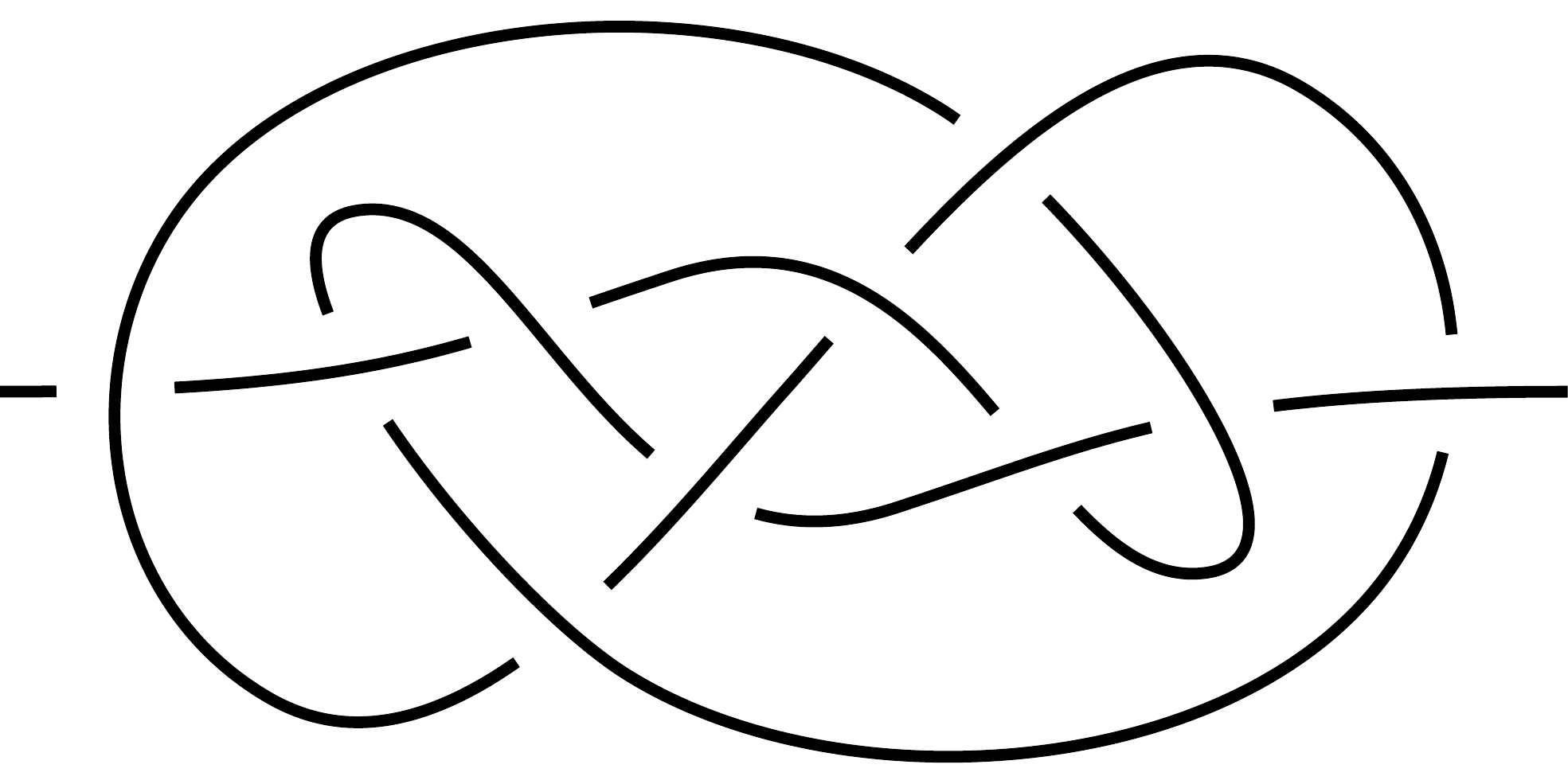}}}&\multirow{4}{*}{\scalebox{.2}{\includegraphics{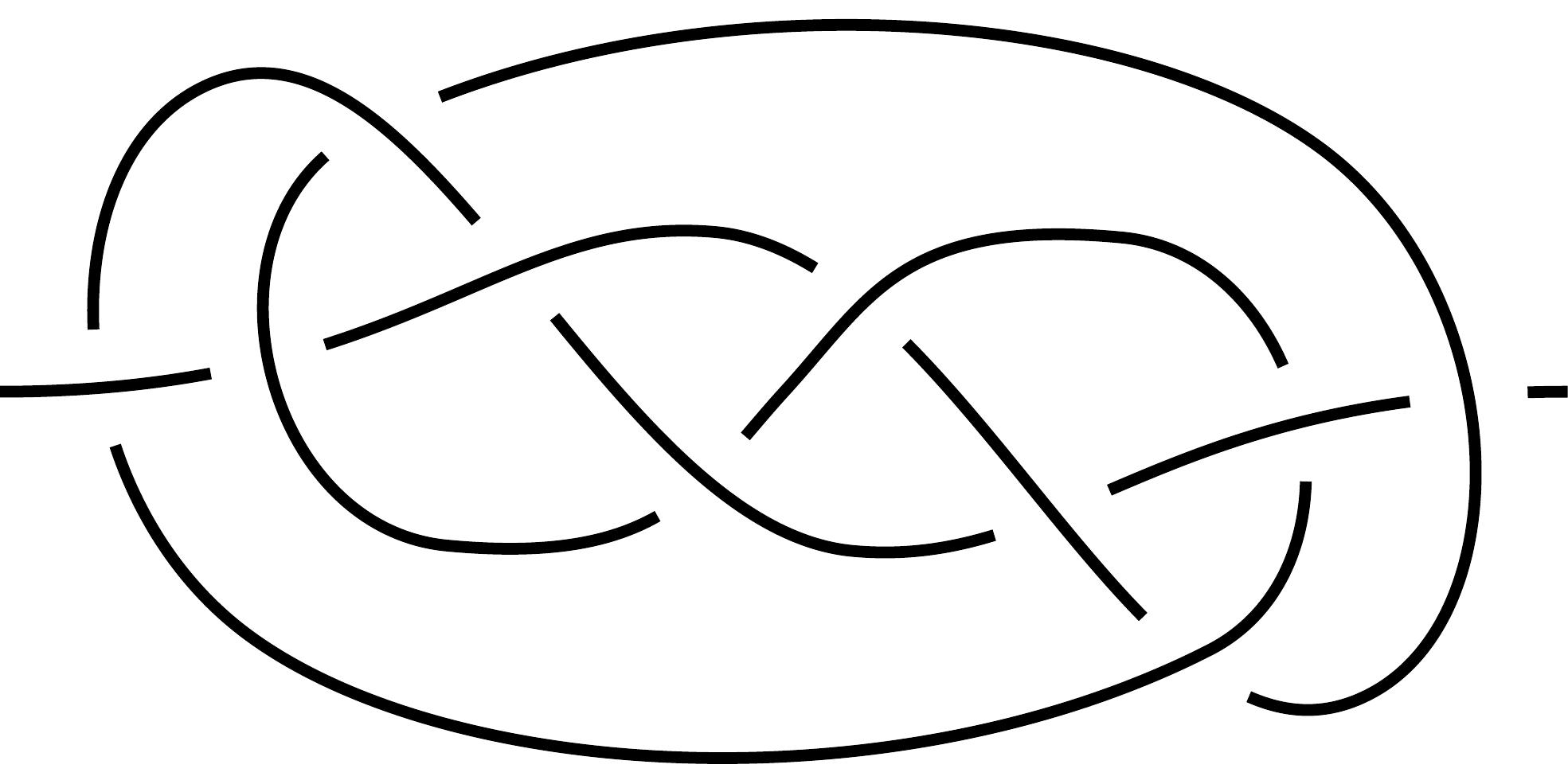}}}&
 \multirow{4}{*}{\scalebox{.2}{\includegraphics{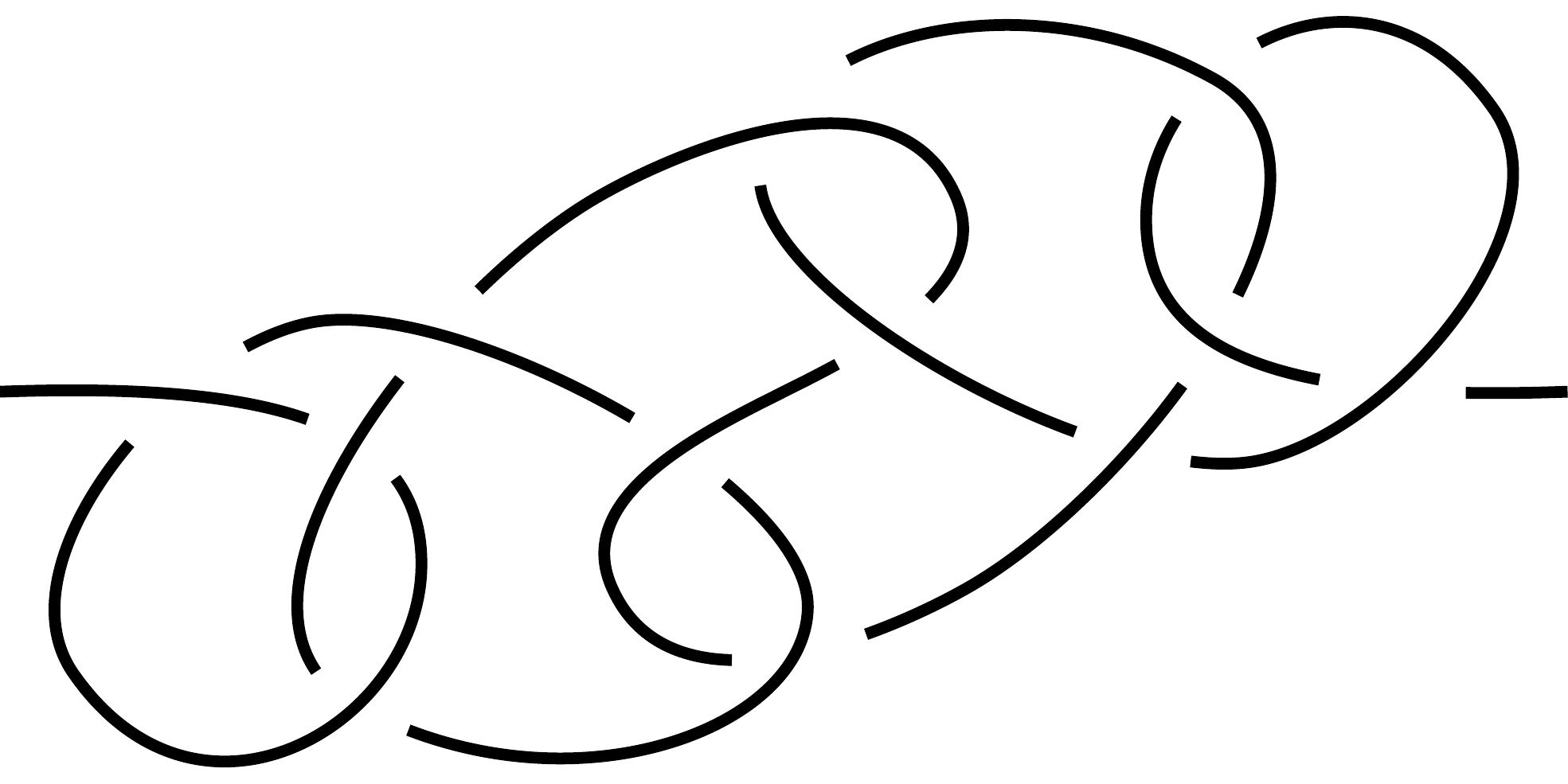}}}\\*
 &&\\*
 &&\\*
 &&\\*
 &&\\\hline\hline

 $12a_{58}$&$12a_{125}$&$12a_{268}$\\*
 \multirow{4}{*}{\scalebox{.2}{\includegraphics{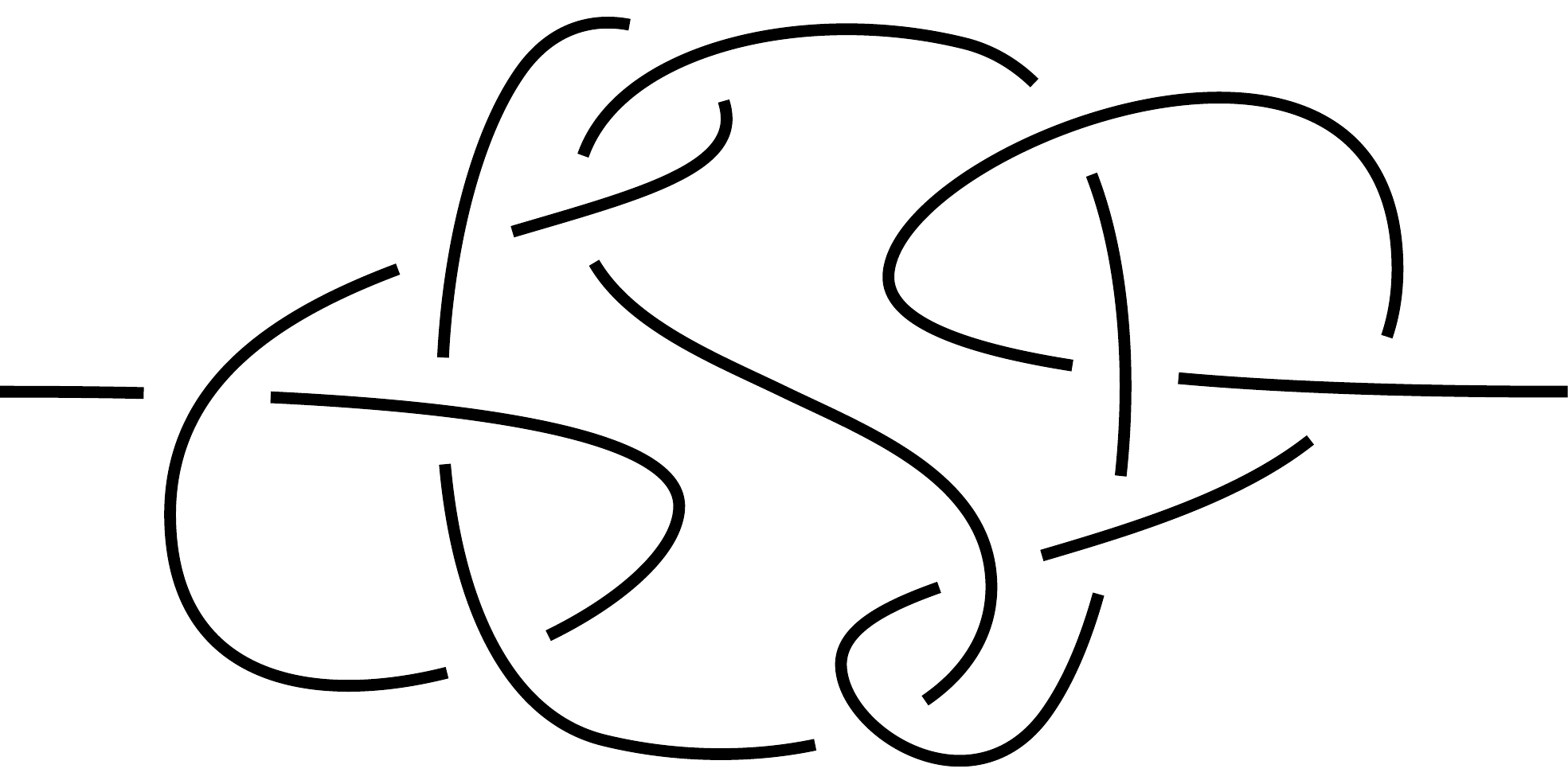}}}&\multirow{4}{*}{\scalebox{.2}{\includegraphics{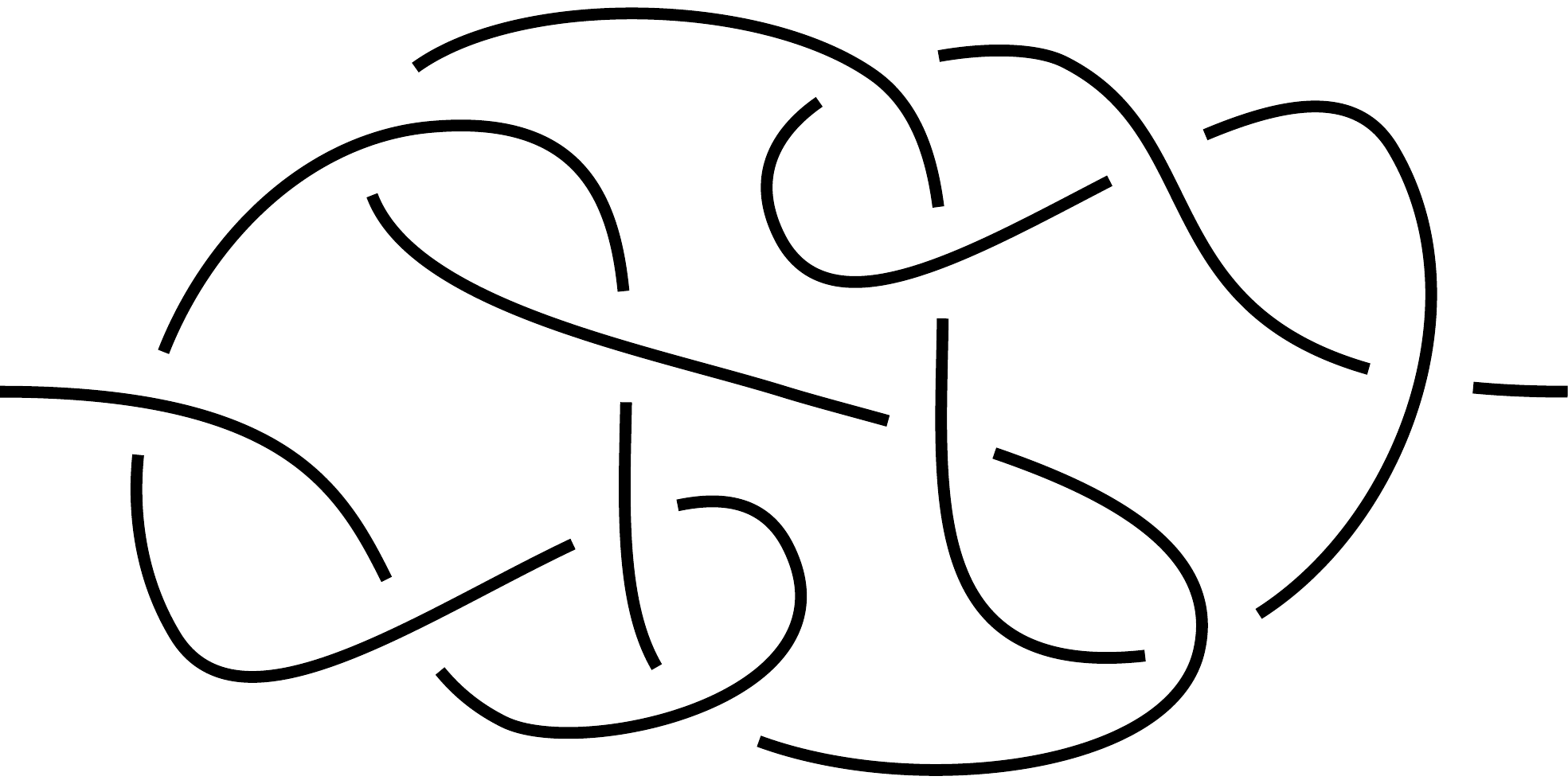}}}&
 \multirow{4}{*}{\scalebox{.2}{\includegraphics{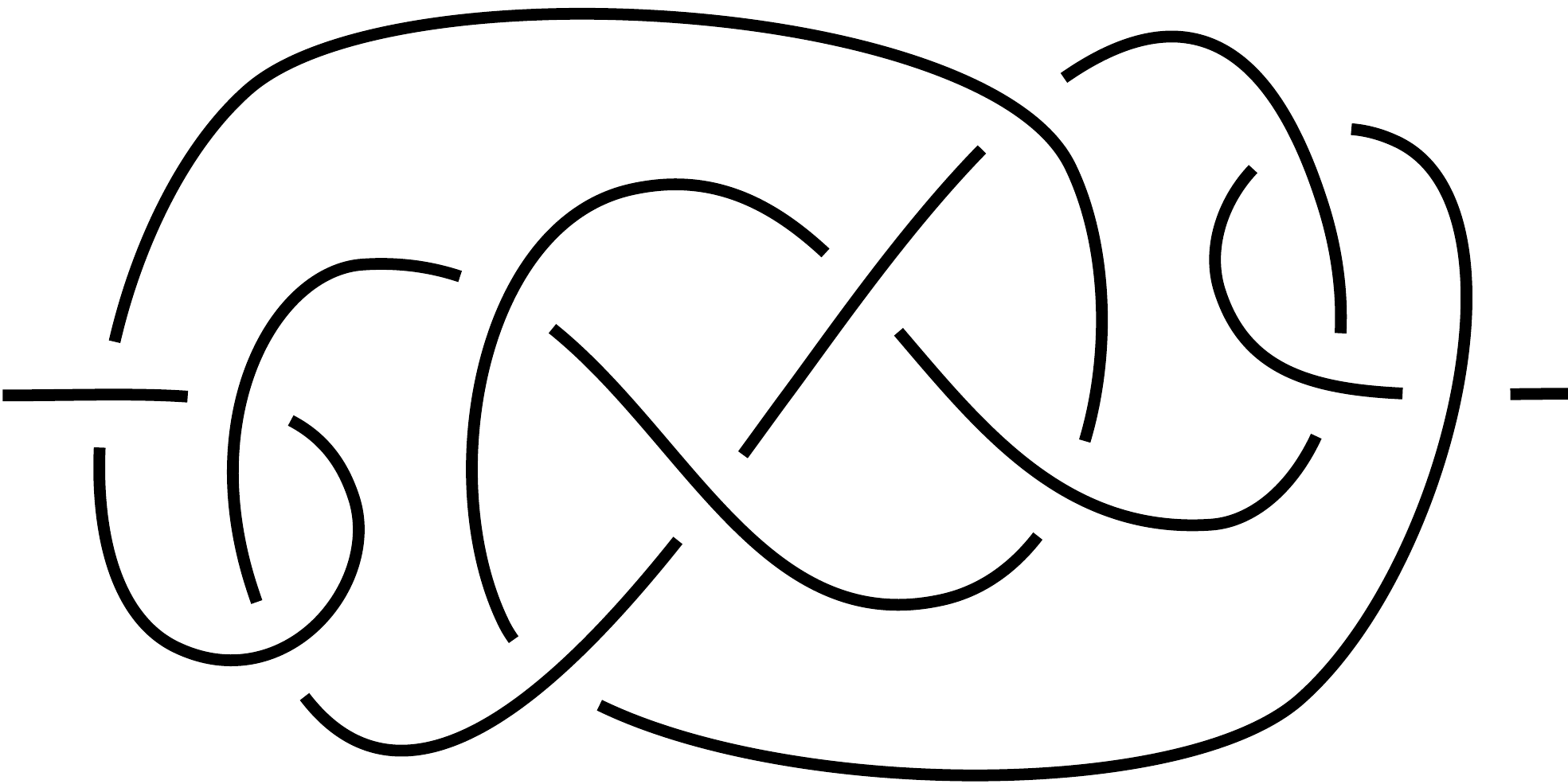}}}\\*
 &&\\*
 &&\\*
 &&\\*
 &&\\\hline\hline

 $12a_{273}$&$12a_{341}$&$12a_{435}$\\*
 \multirow{4}{*}{\scalebox{.2}{\includegraphics{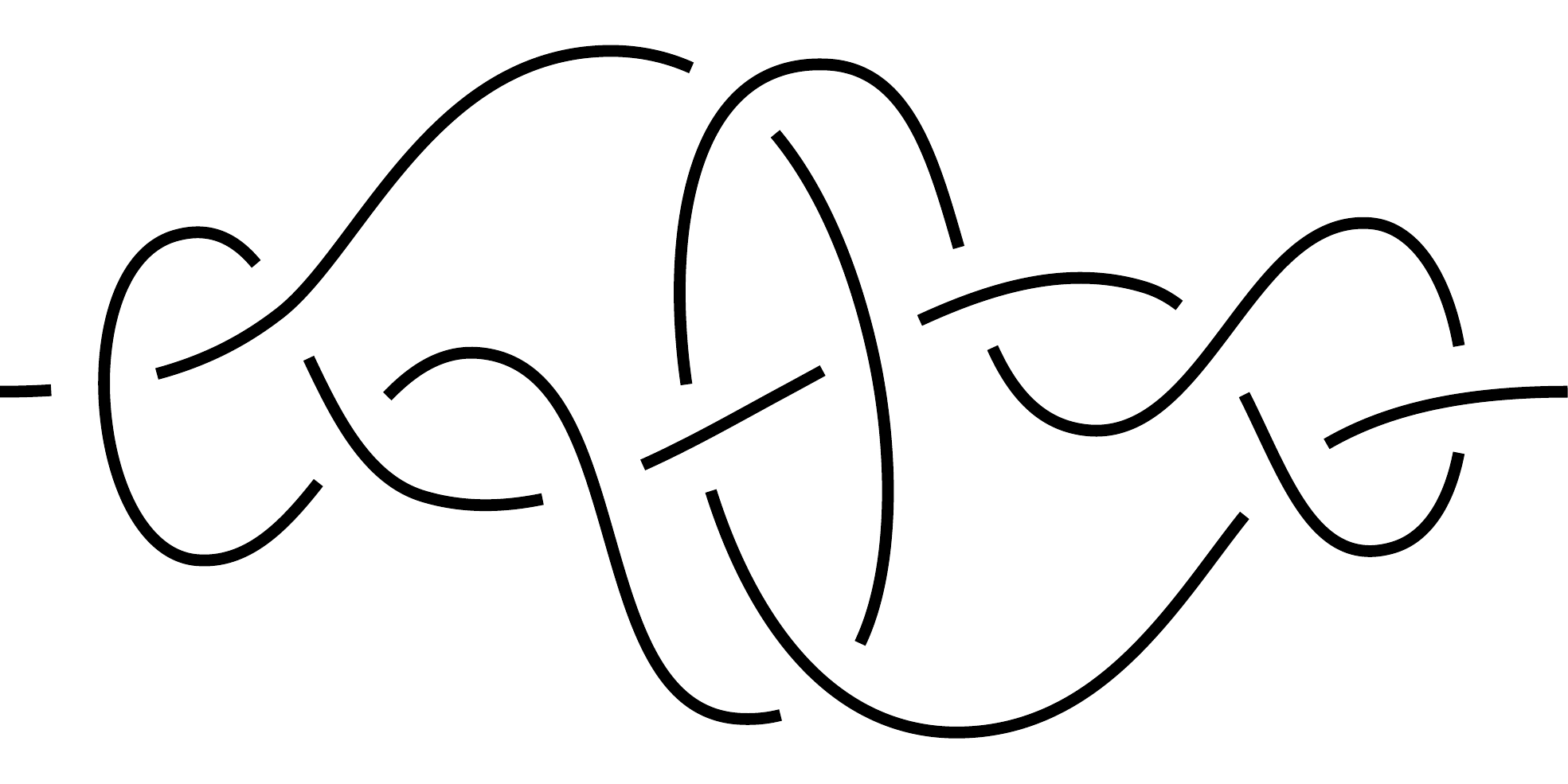}}}&\multirow{4}{*}{\scalebox{.2}{\includegraphics{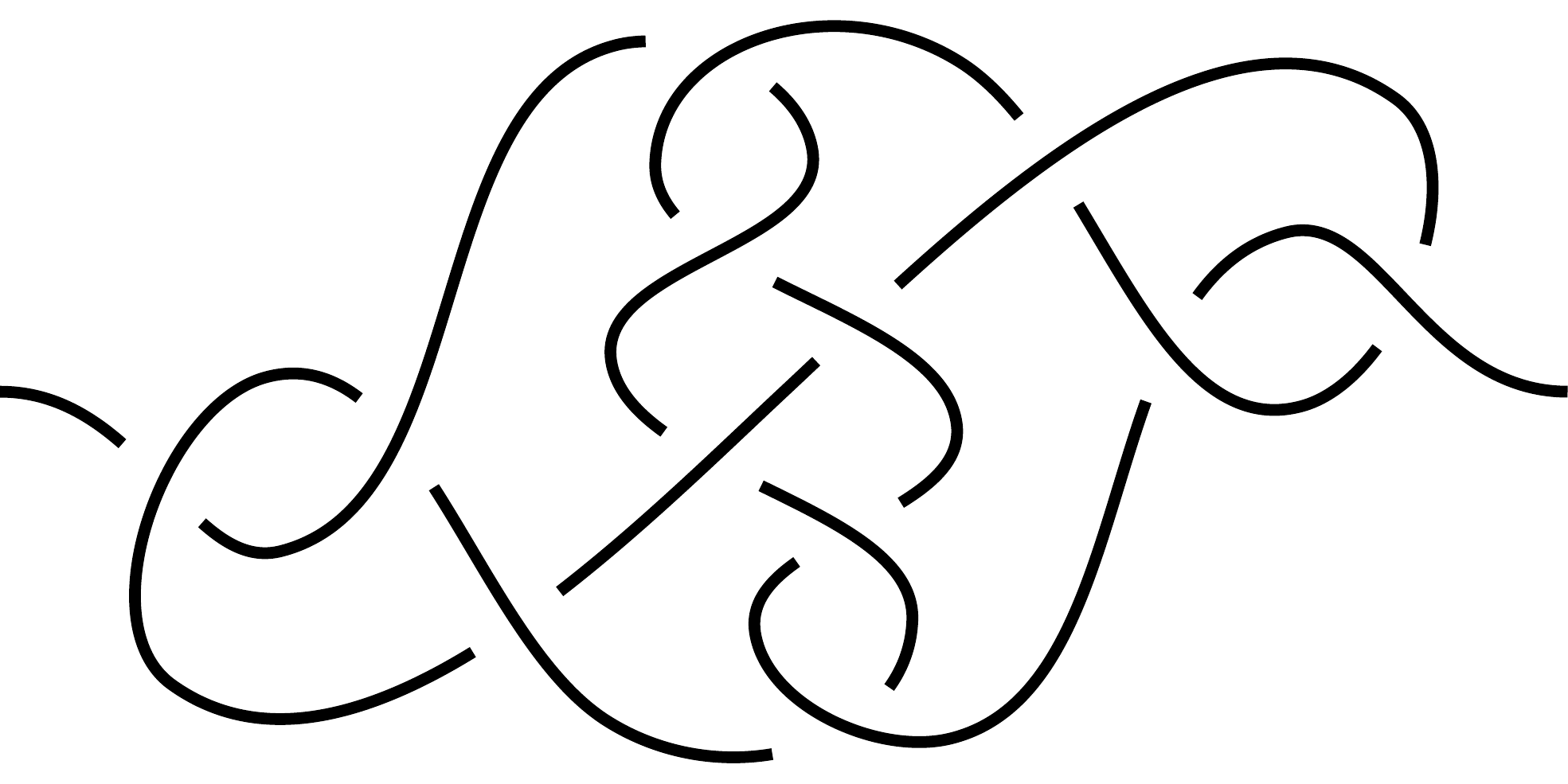}}}&
 \multirow{4}{*}{\scalebox{.2}{\includegraphics{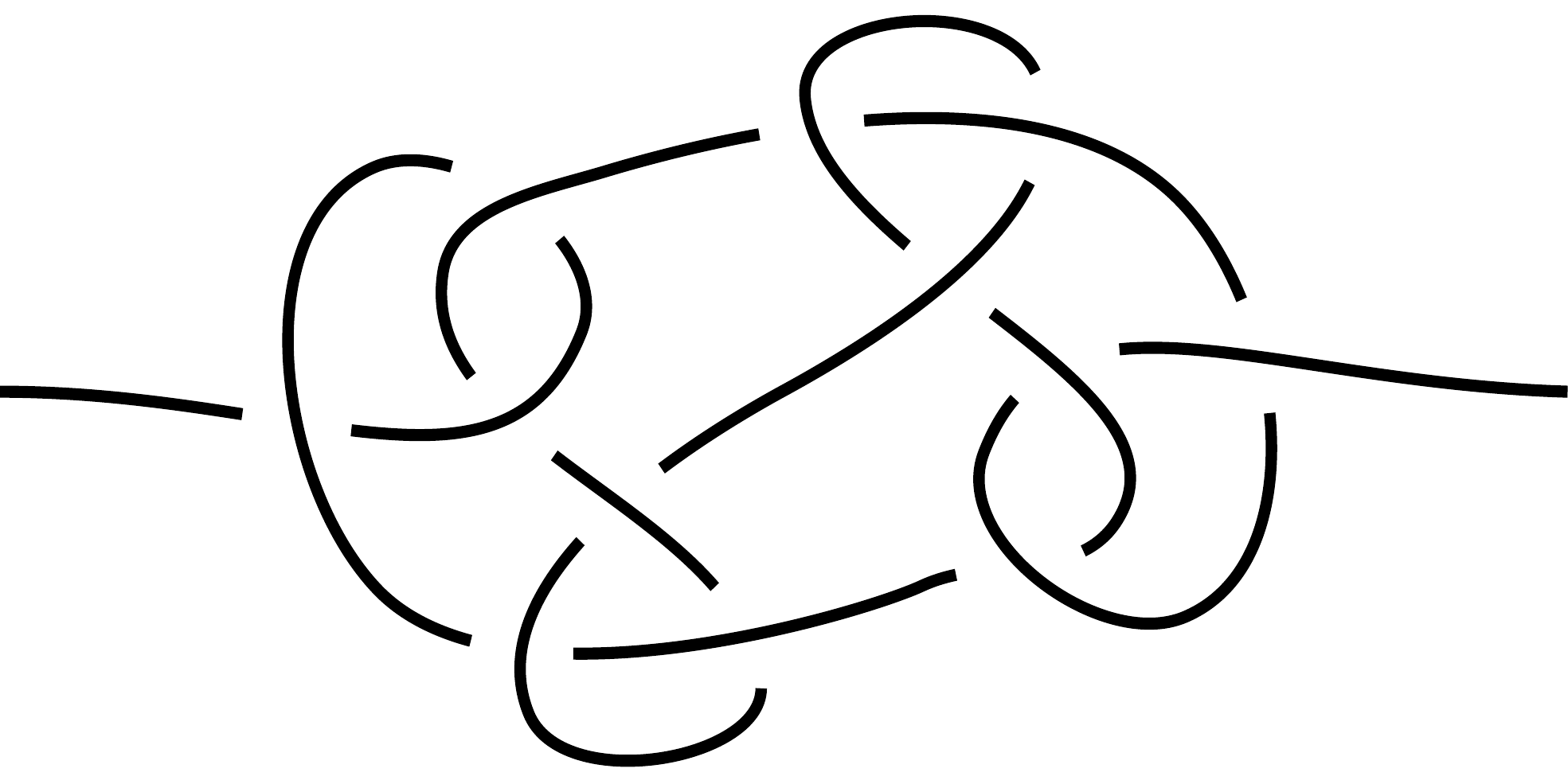}}}\\*
 &&\\*
 &&\\*
 &&\\*
 &&\\\hline\hline  

 $12a_{458}$&$12a_{462}$&$12a_{465}$\\*
 \multirow{4}{*}{\scalebox{.2}{\includegraphics{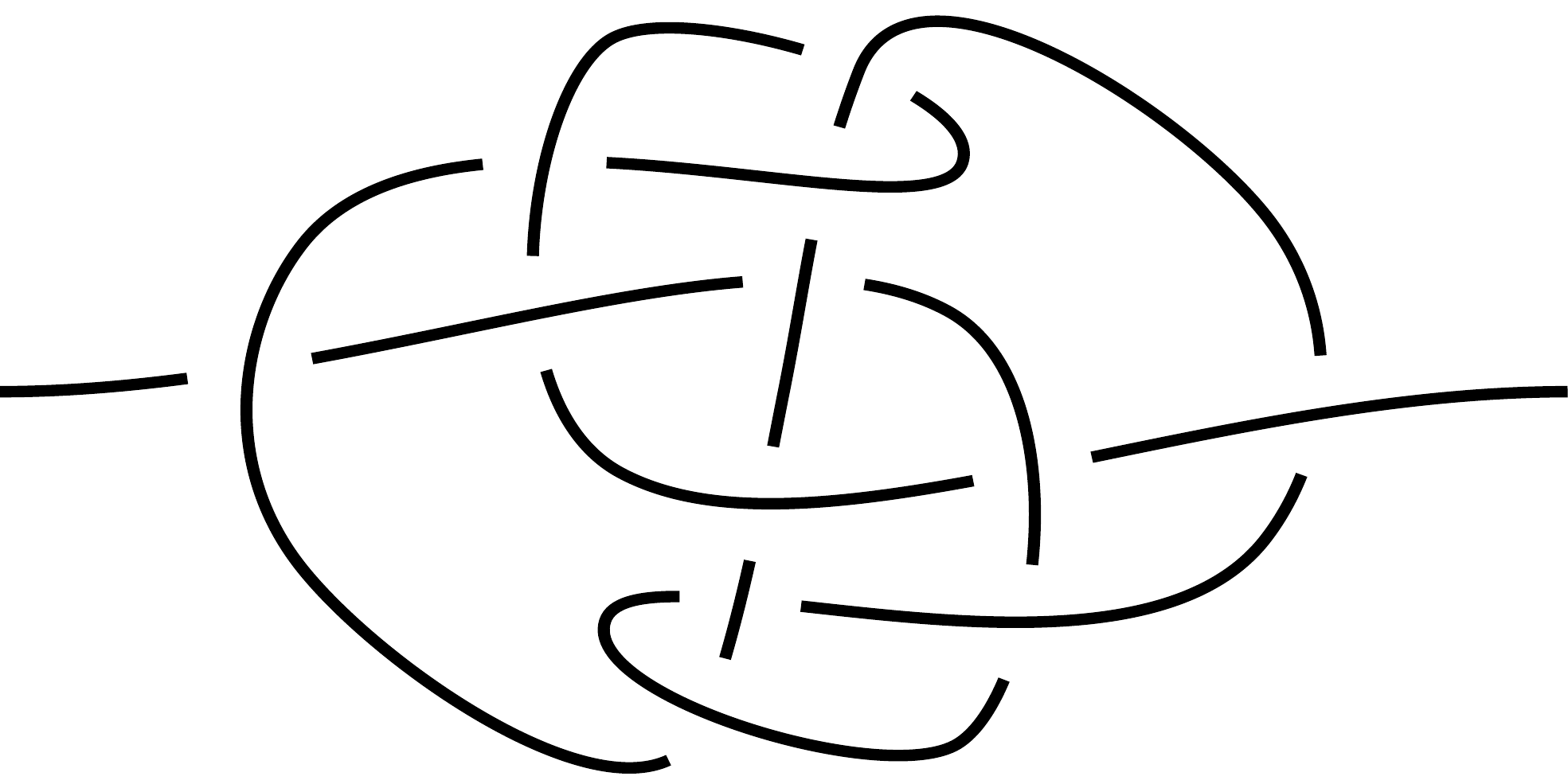}}}&\multirow{4}{*}{\scalebox{.2}{\includegraphics{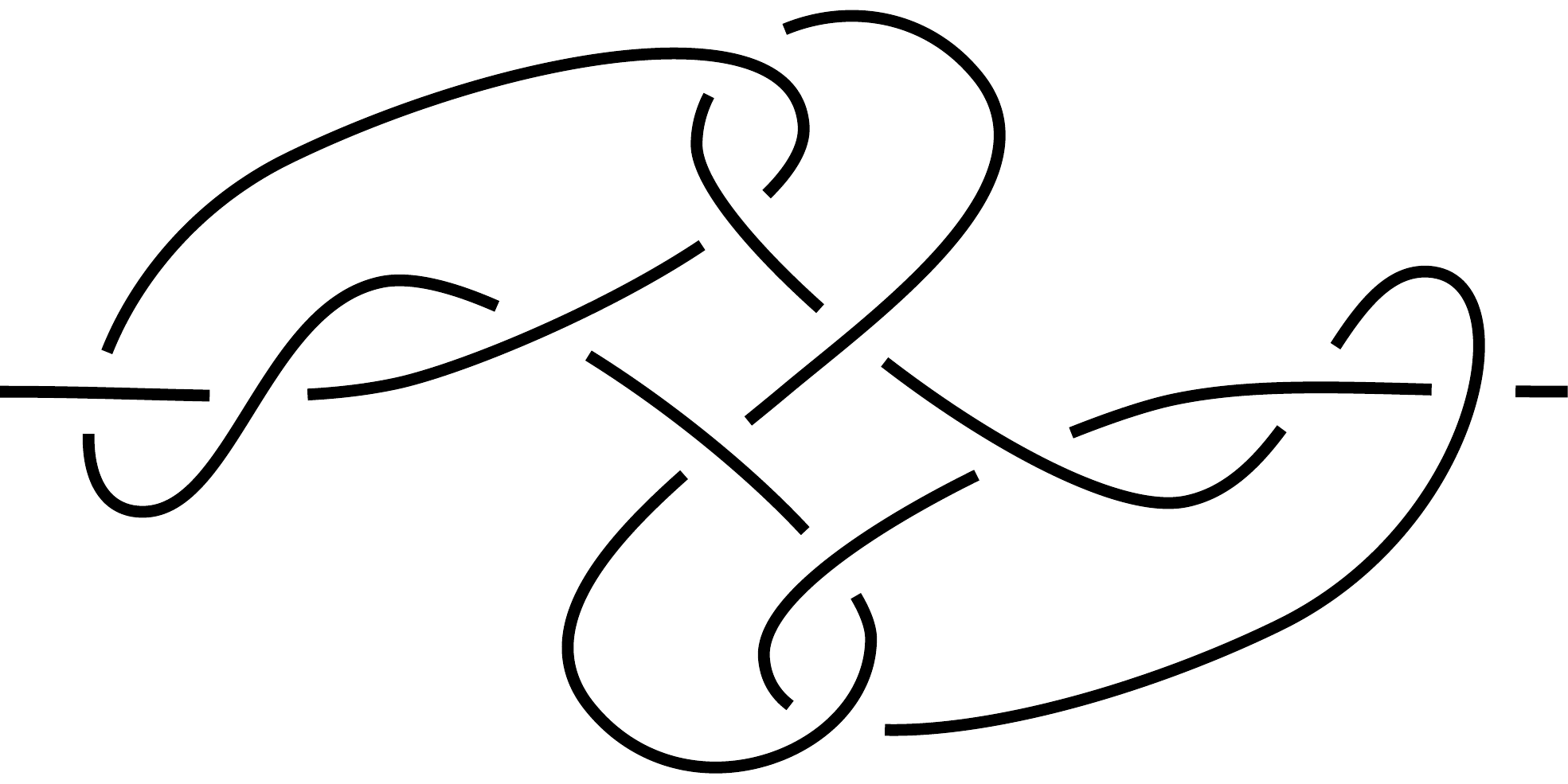}}}&
 \multirow{4}{*}{\scalebox{.2}{\includegraphics{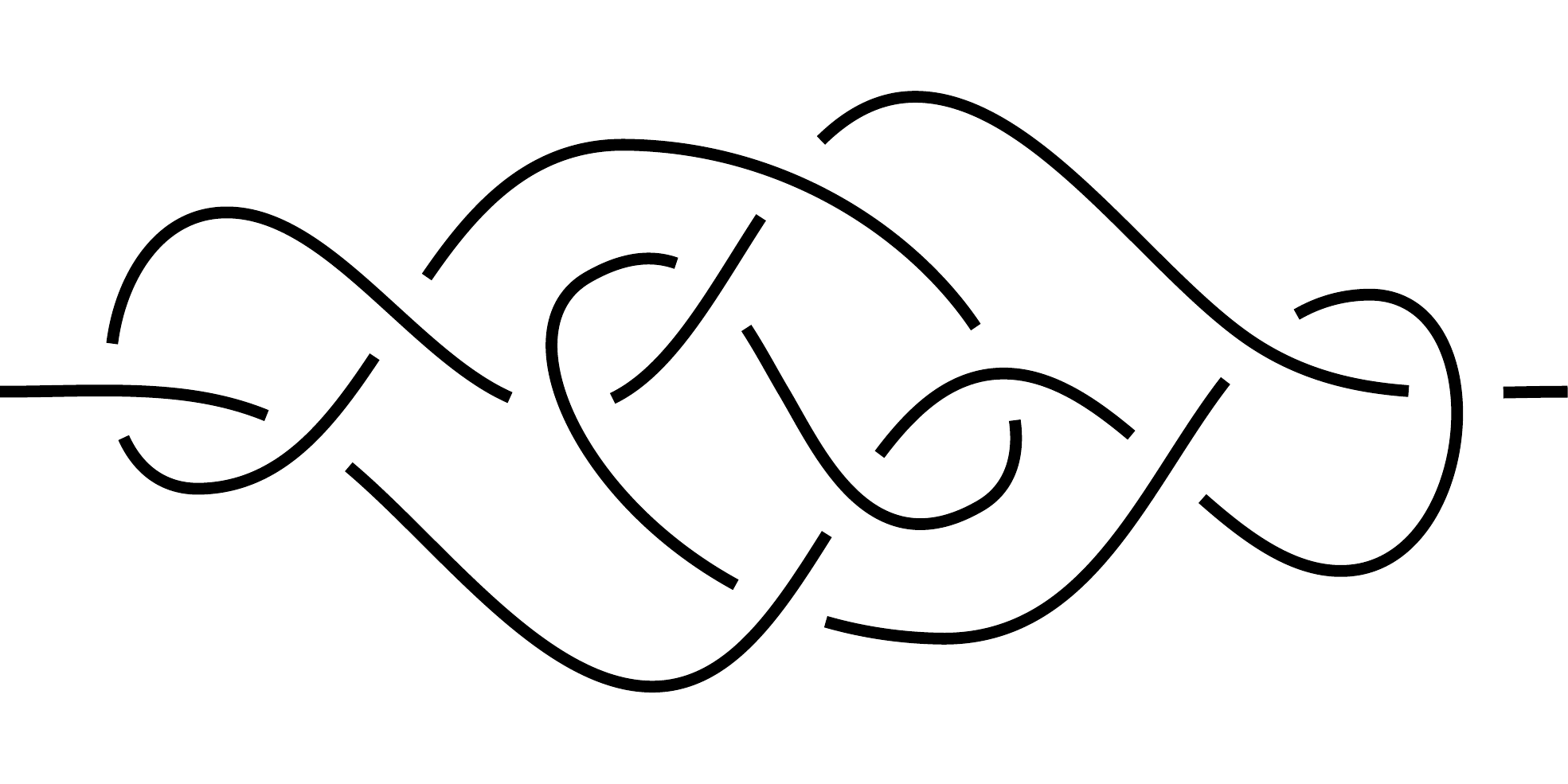}}}\\*
 &&\\*
 &&\\*
 &&\\*
 &&\\\hline\hline 

 $12a_{471}$&$12a_{477}$&$12a_{499}$\\*
 \multirow{4}{*}{\scalebox{.2}{\includegraphics{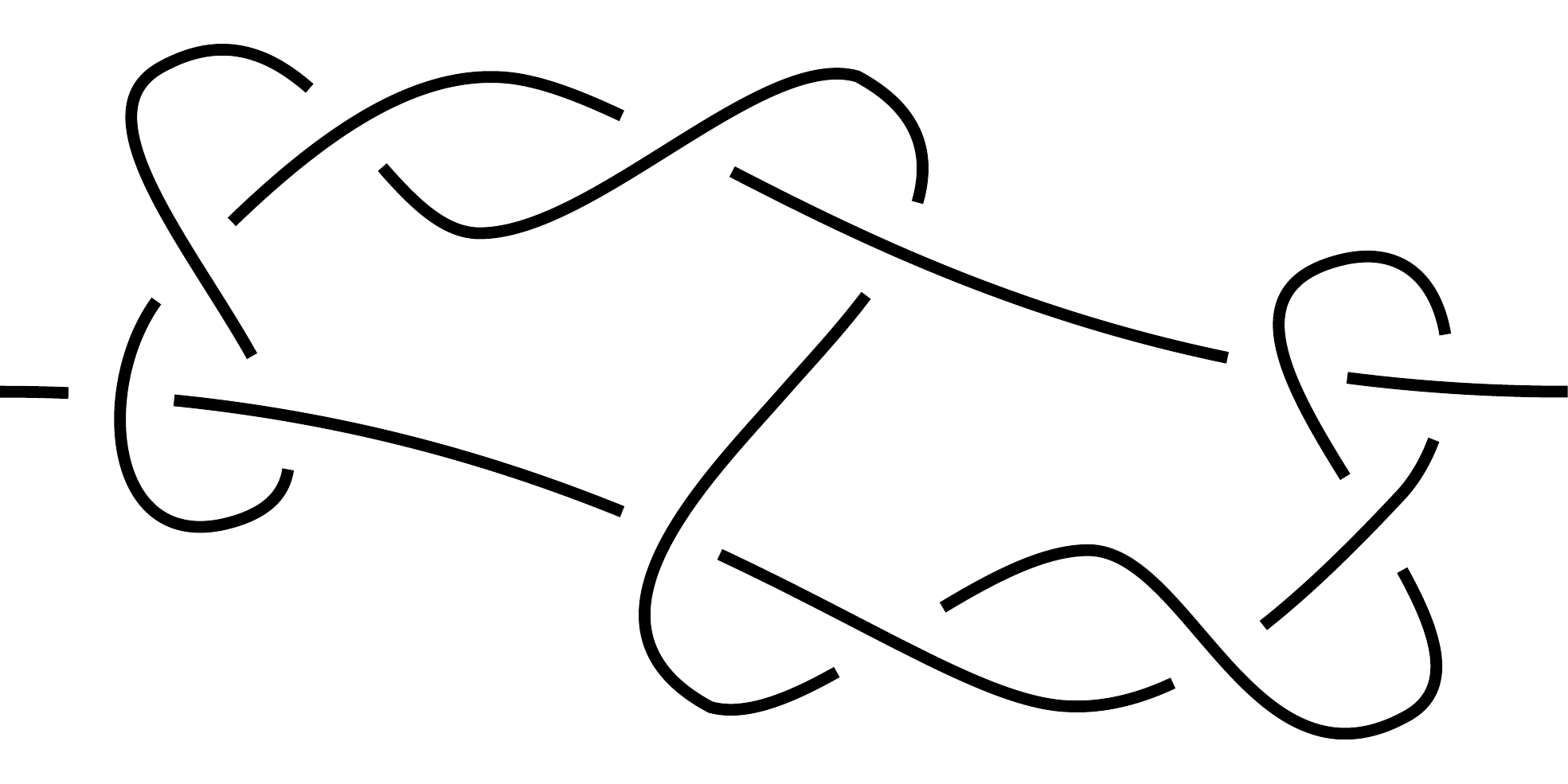}}}&\multirow{4}{*}{\scalebox{.2}{\includegraphics{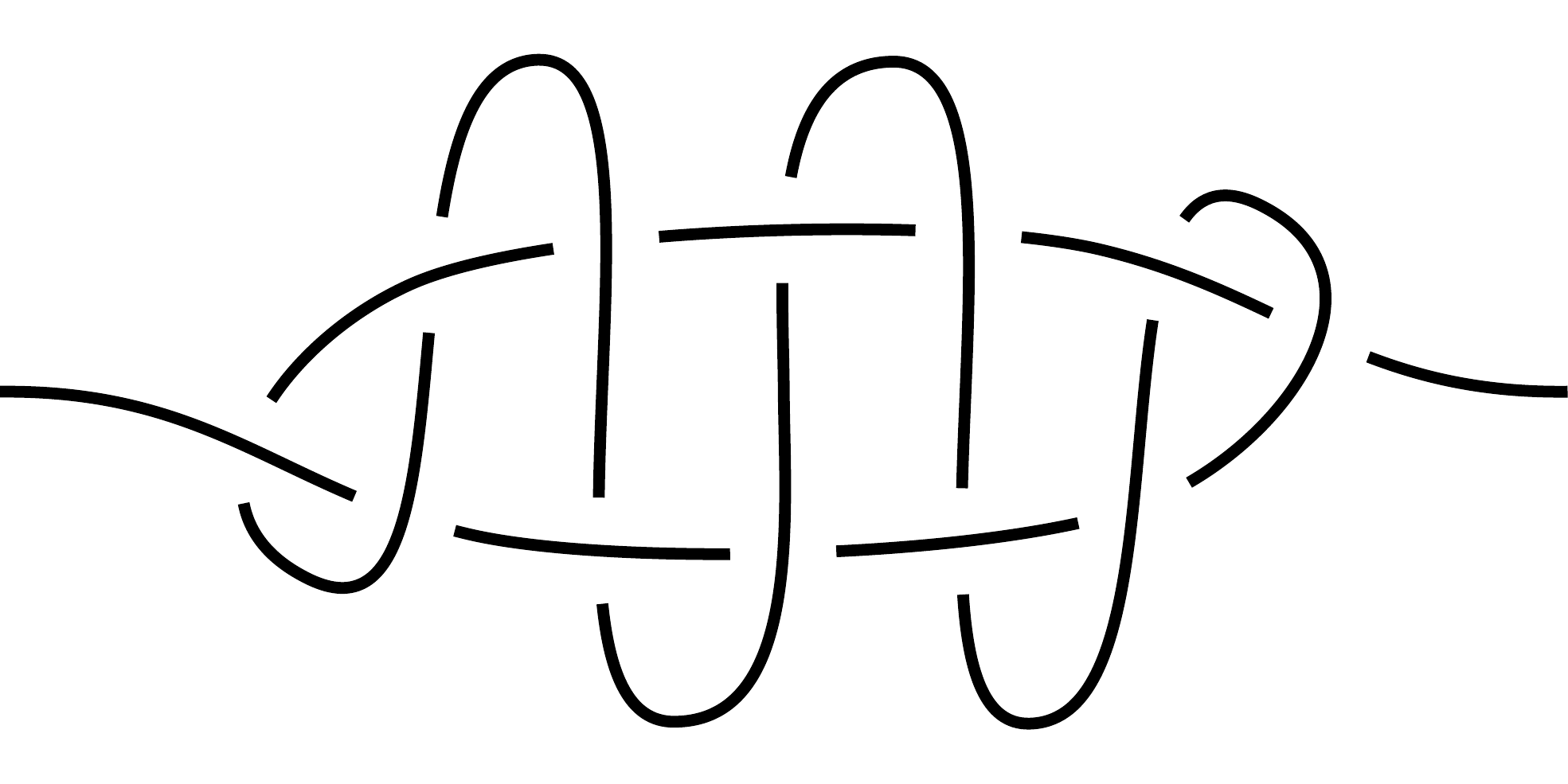}}}&
 \multirow{4}{*}{\scalebox{.2}{\includegraphics{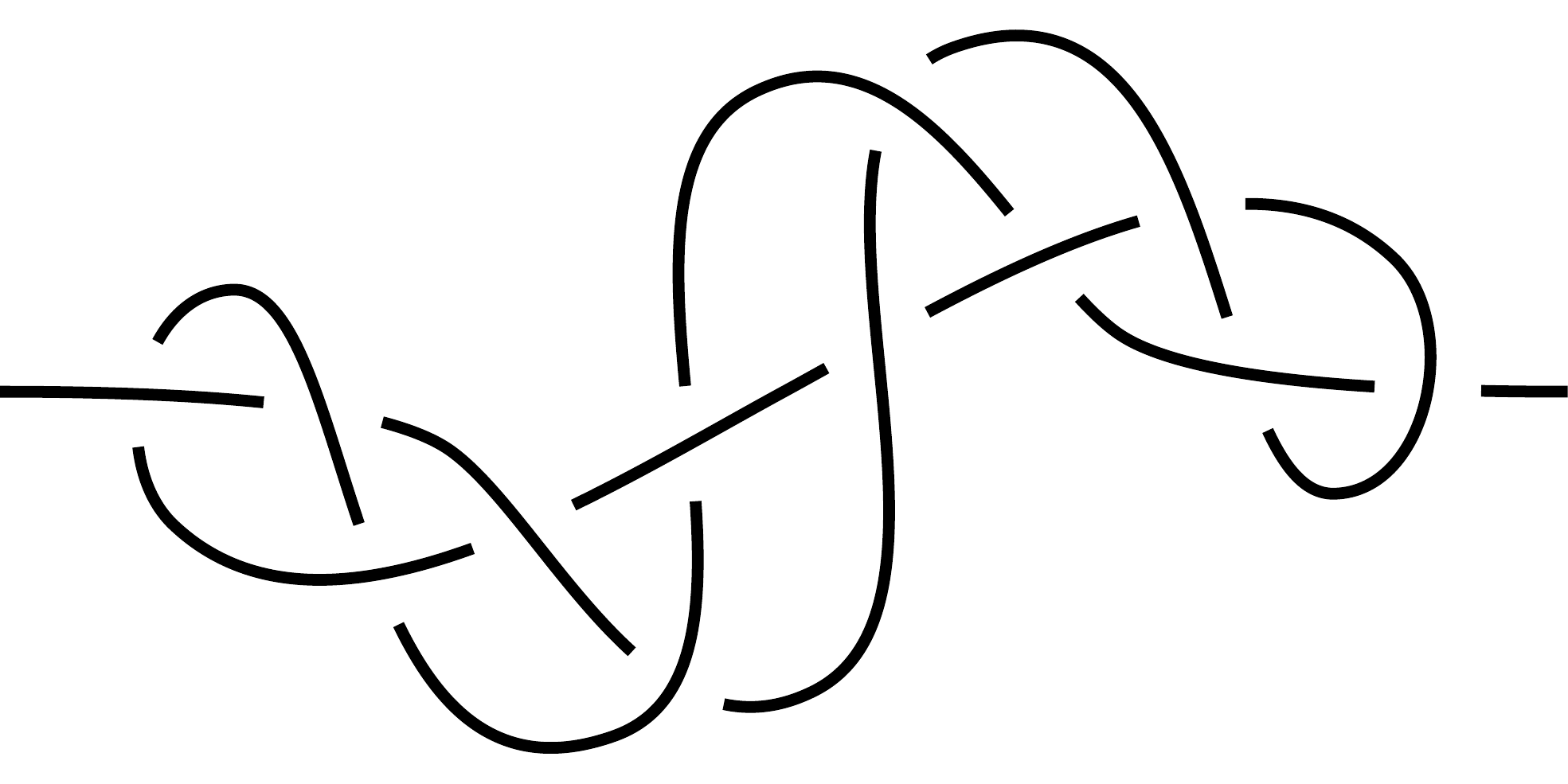}}}\\*
 &&\\*
 &&\\*
 &&\\*
 &&\\\hline\hline 

 $12a_{506}$&$12a_{510}$&$12a_{627}$\\*
 \multirow{4}{*}{\scalebox{.2}{\includegraphics{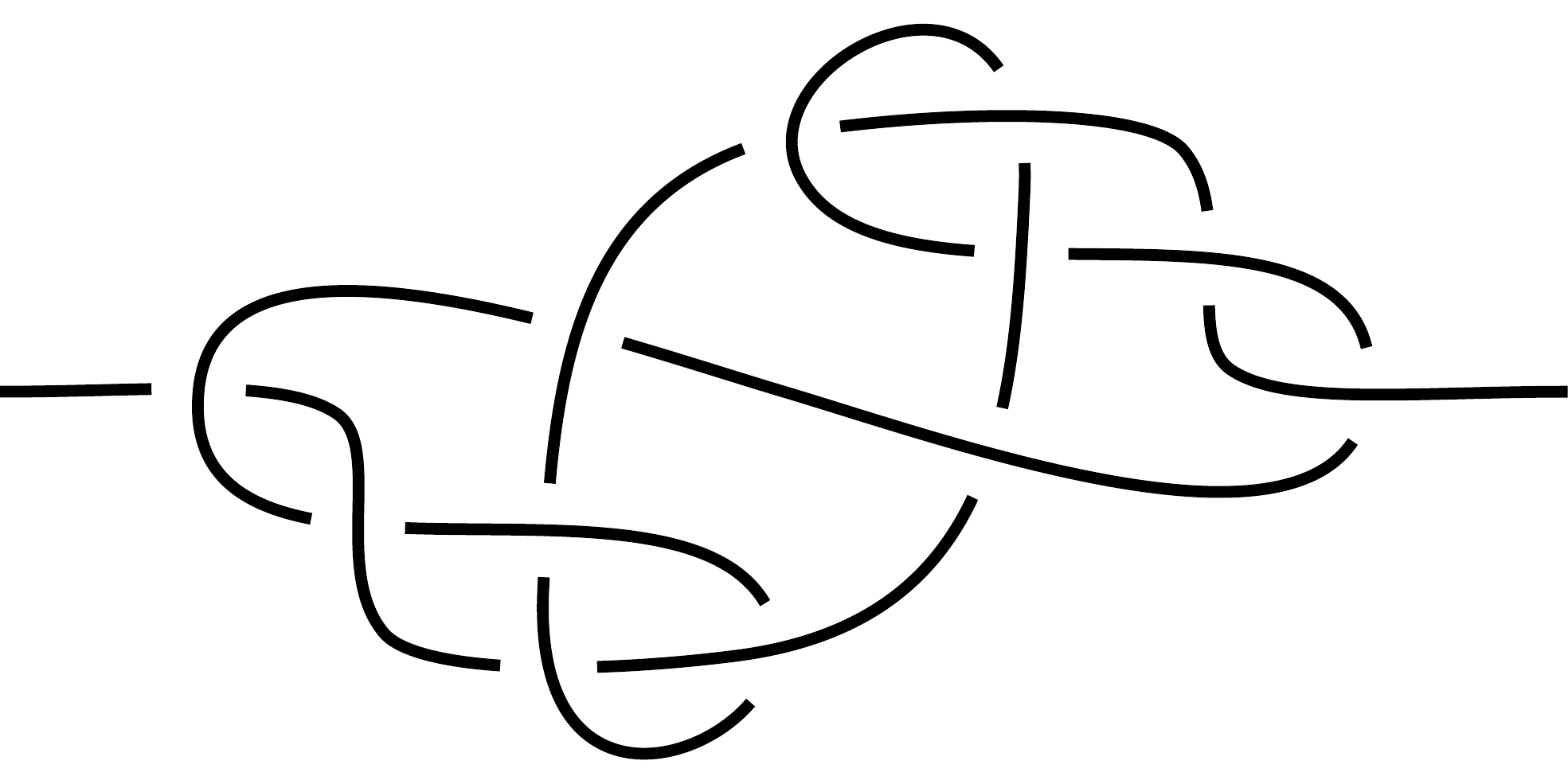}}}&\multirow{4}{*}{\scalebox{.2}{\includegraphics{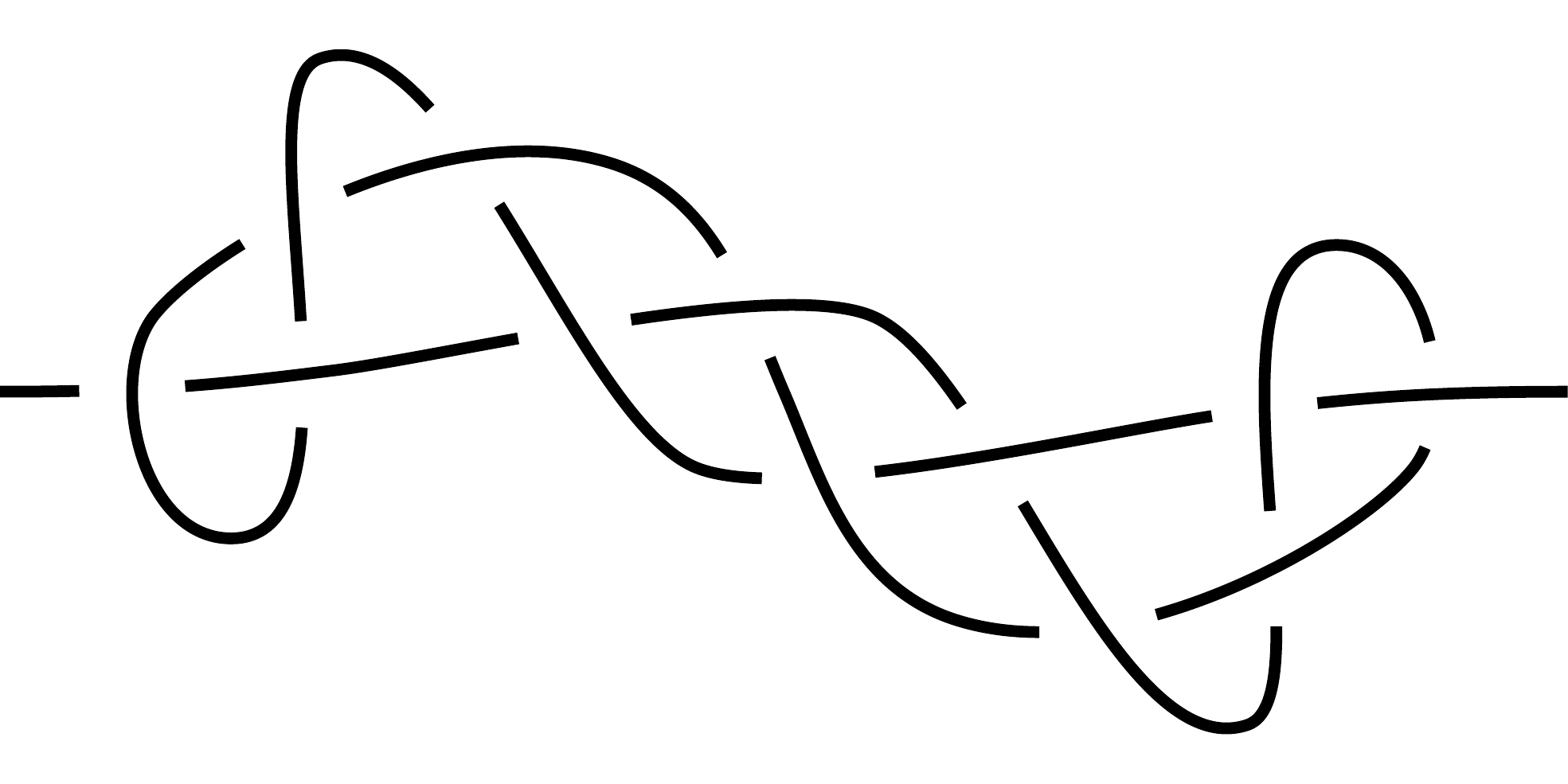}}}&
 \multirow{4}{*}{\scalebox{.2}{\includegraphics{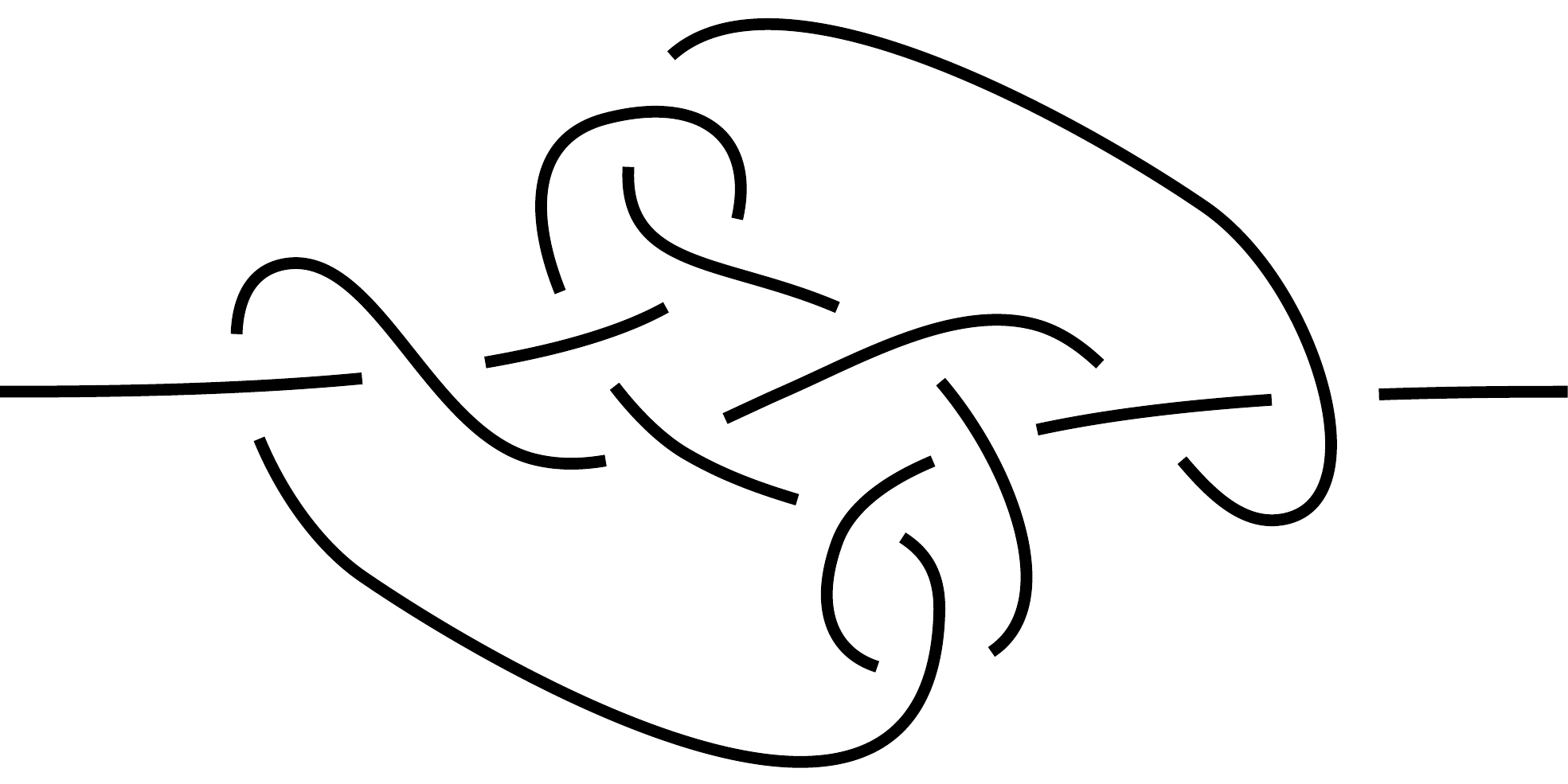}}}\\*
 &&\\*
 &&\\*
 &&\\*
 &&\\\hline\hline 

 $12a_{819}$&$12a_{821}$&$12a_{868}$\\*
 \multirow{4}{*}{\scalebox{.2}{\includegraphics{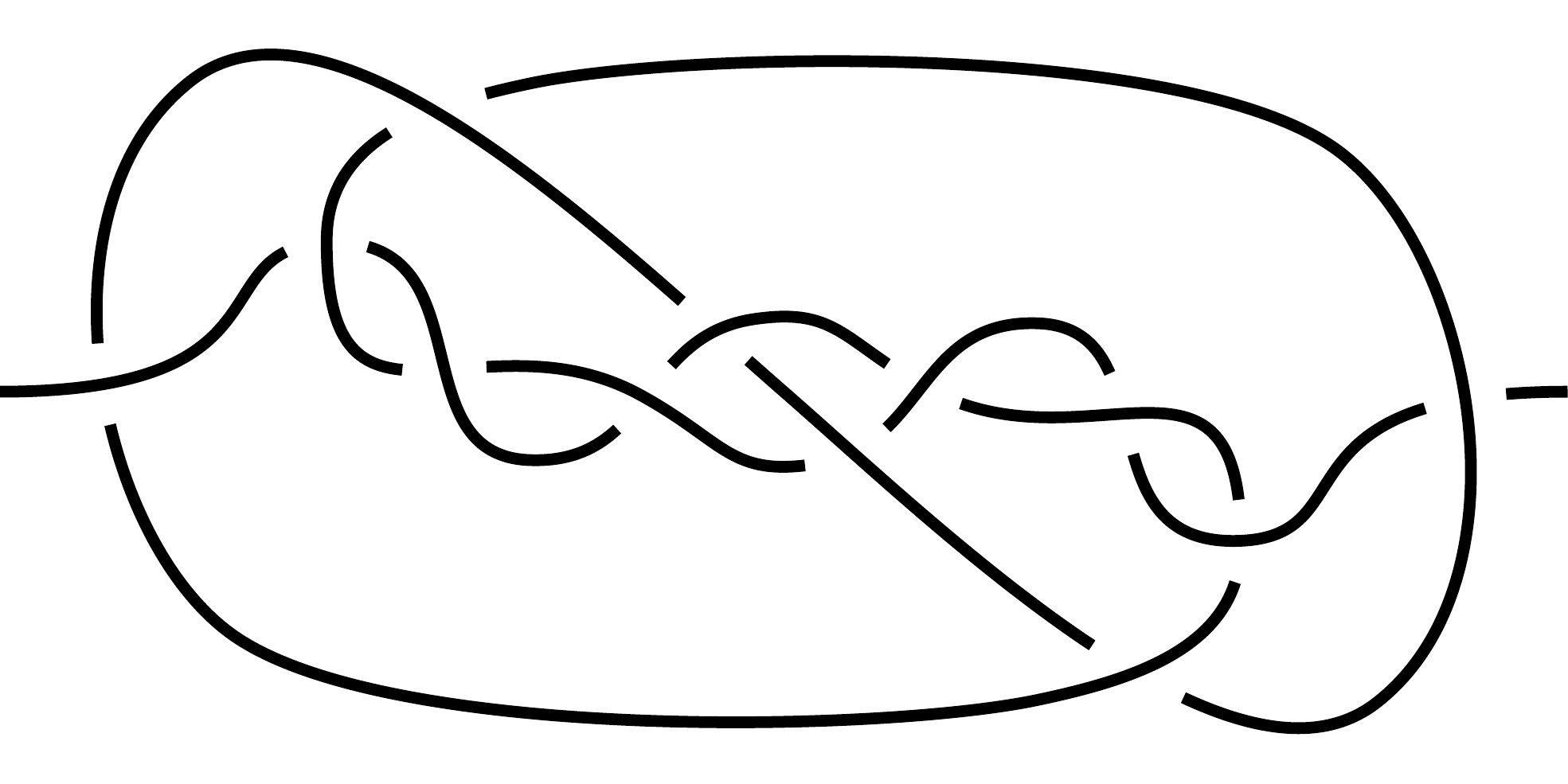}}}&\multirow{4}{*}{\scalebox{.2}{\includegraphics{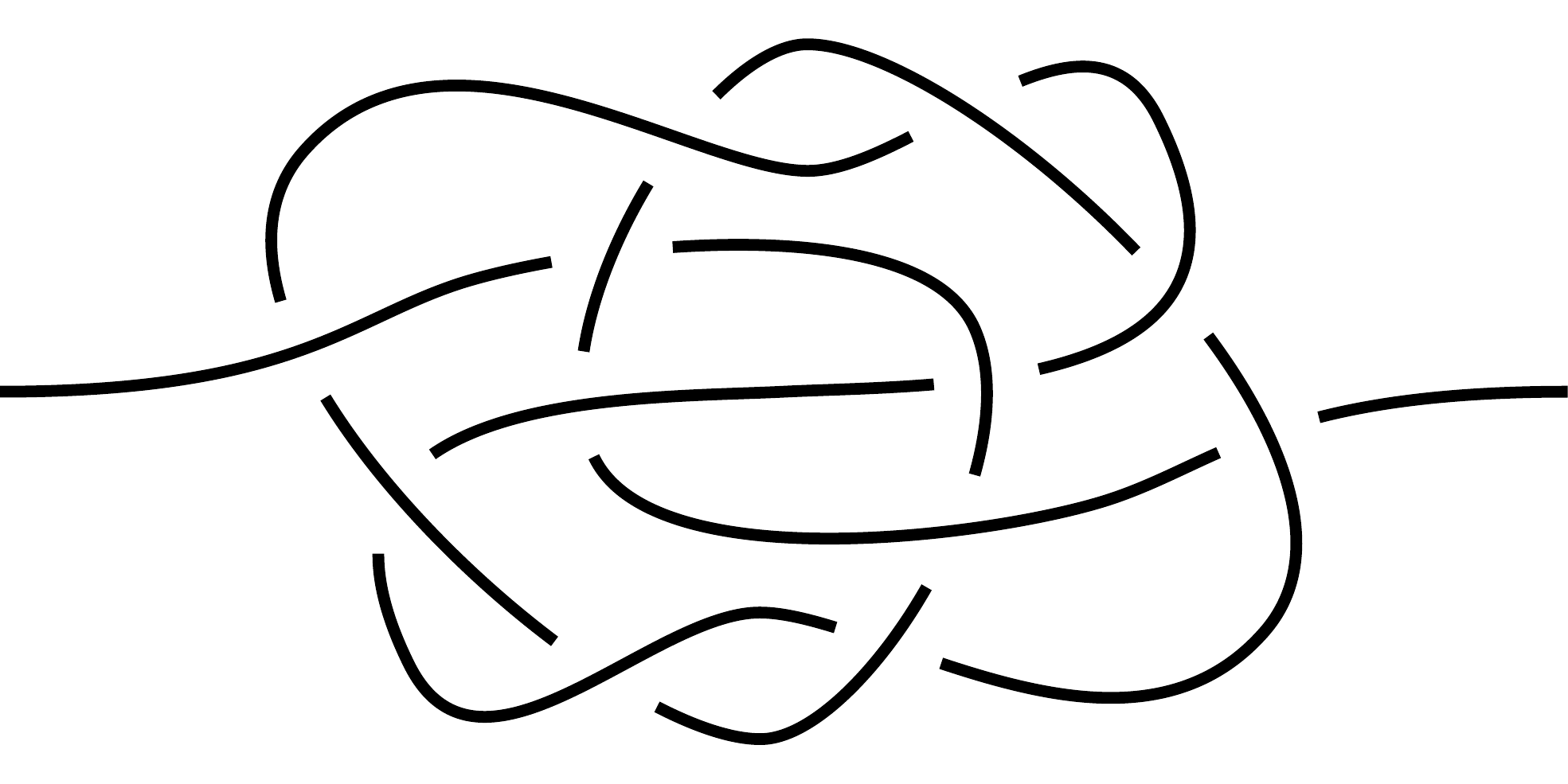}}}&
 \multirow{4}{*}{\scalebox{.2}{\includegraphics{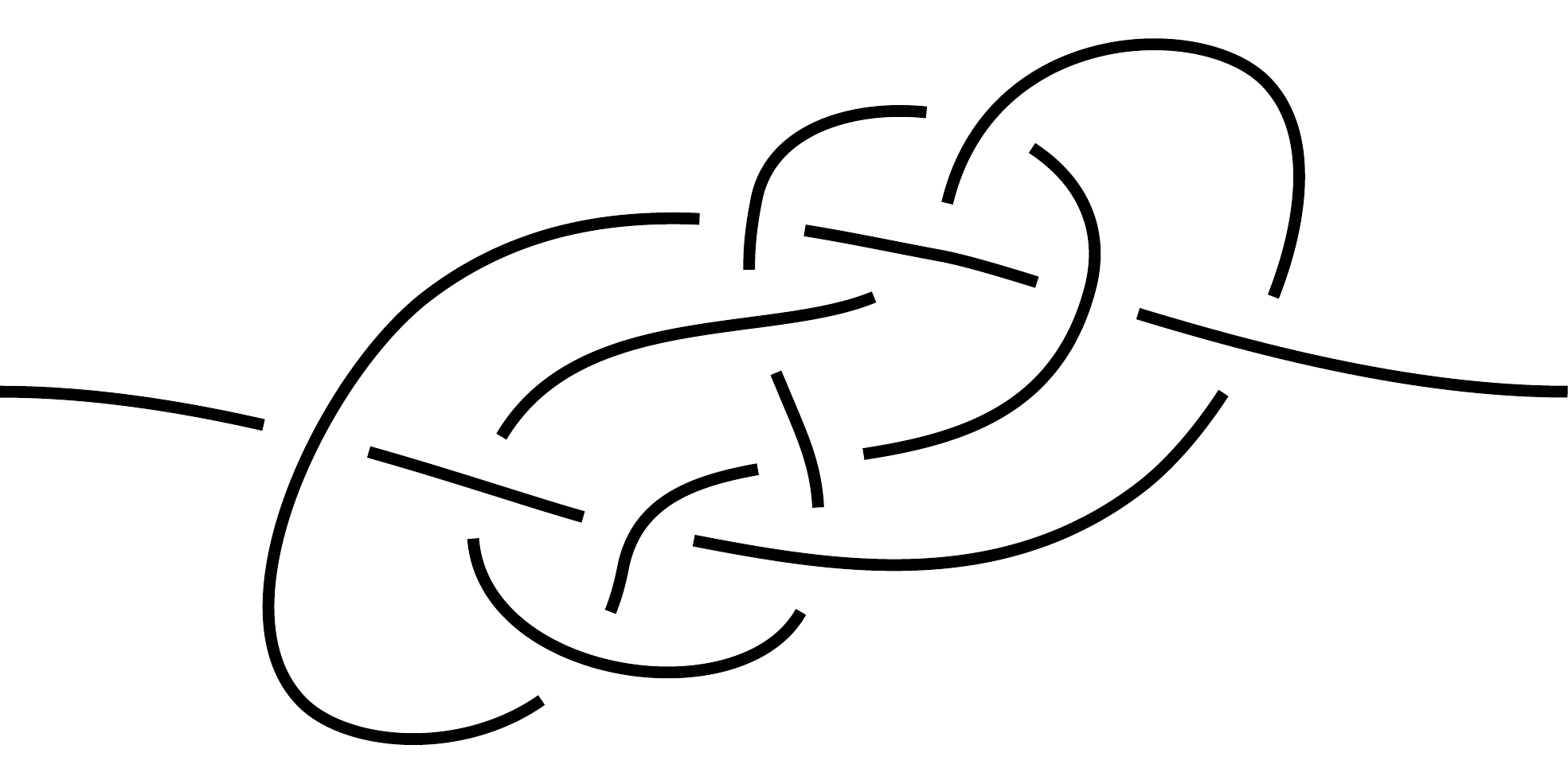}}}\\*
 &&\\*
 &&\\*
 &&\\*
 &&\\\hline\hline 

 $12a_{887}$&$12a_{890}$&$12a_{906}$\\*
 \multirow{4}{*}{\scalebox{.2}{\includegraphics{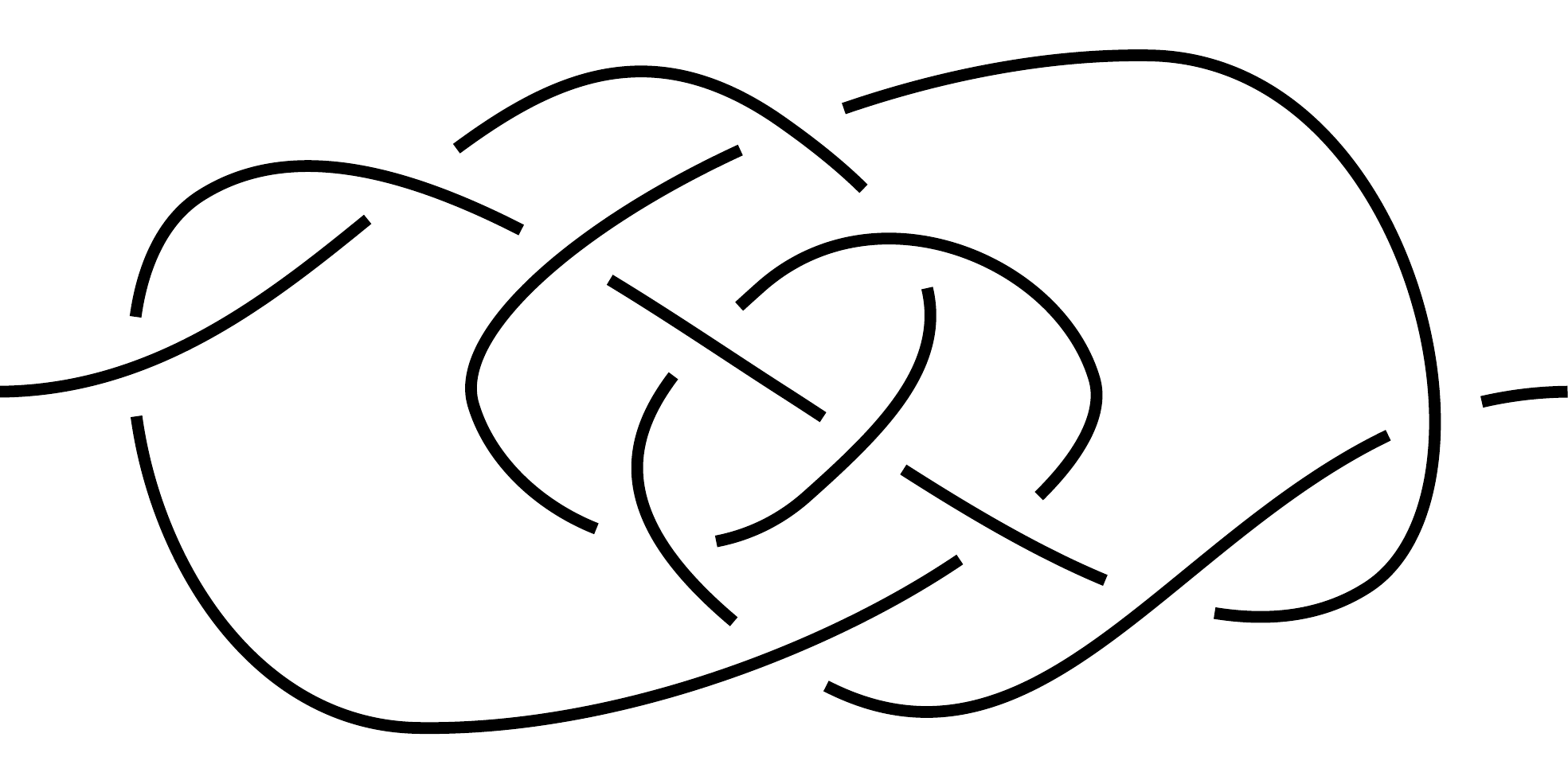}}}&\multirow{4}{*}{\scalebox{.2}{\includegraphics{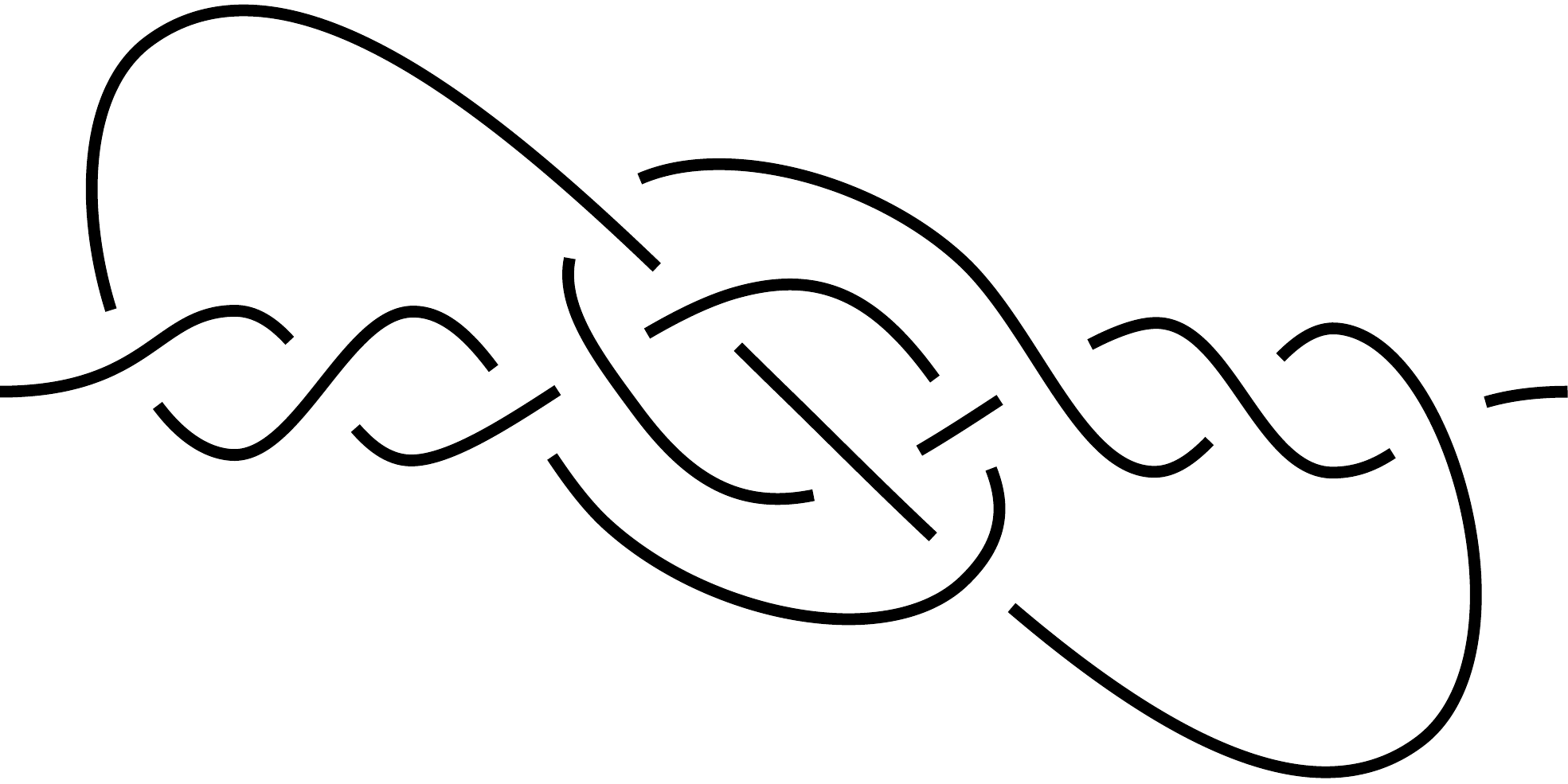}}}&
 \multirow{4}{*}{\scalebox{.2}{\includegraphics{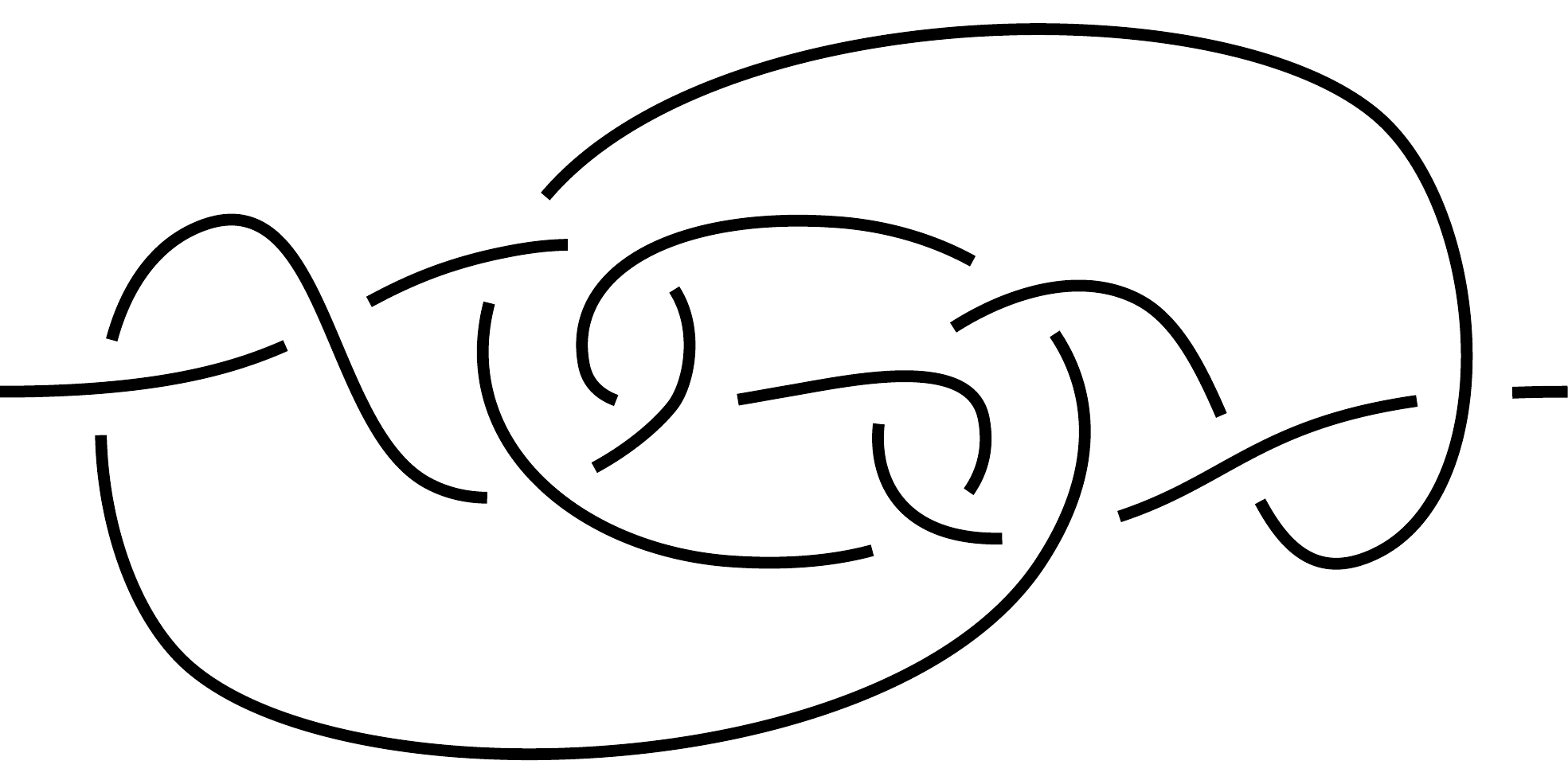}}}\\*
 &&\\*
 &&\\*
 &&\\*
 &&\\\hline\hline 

 $12a_{960}$&$12a_{990}$&$12a_{1008}$\\*
 \multirow{4}{*}{\scalebox{.2}{\includegraphics{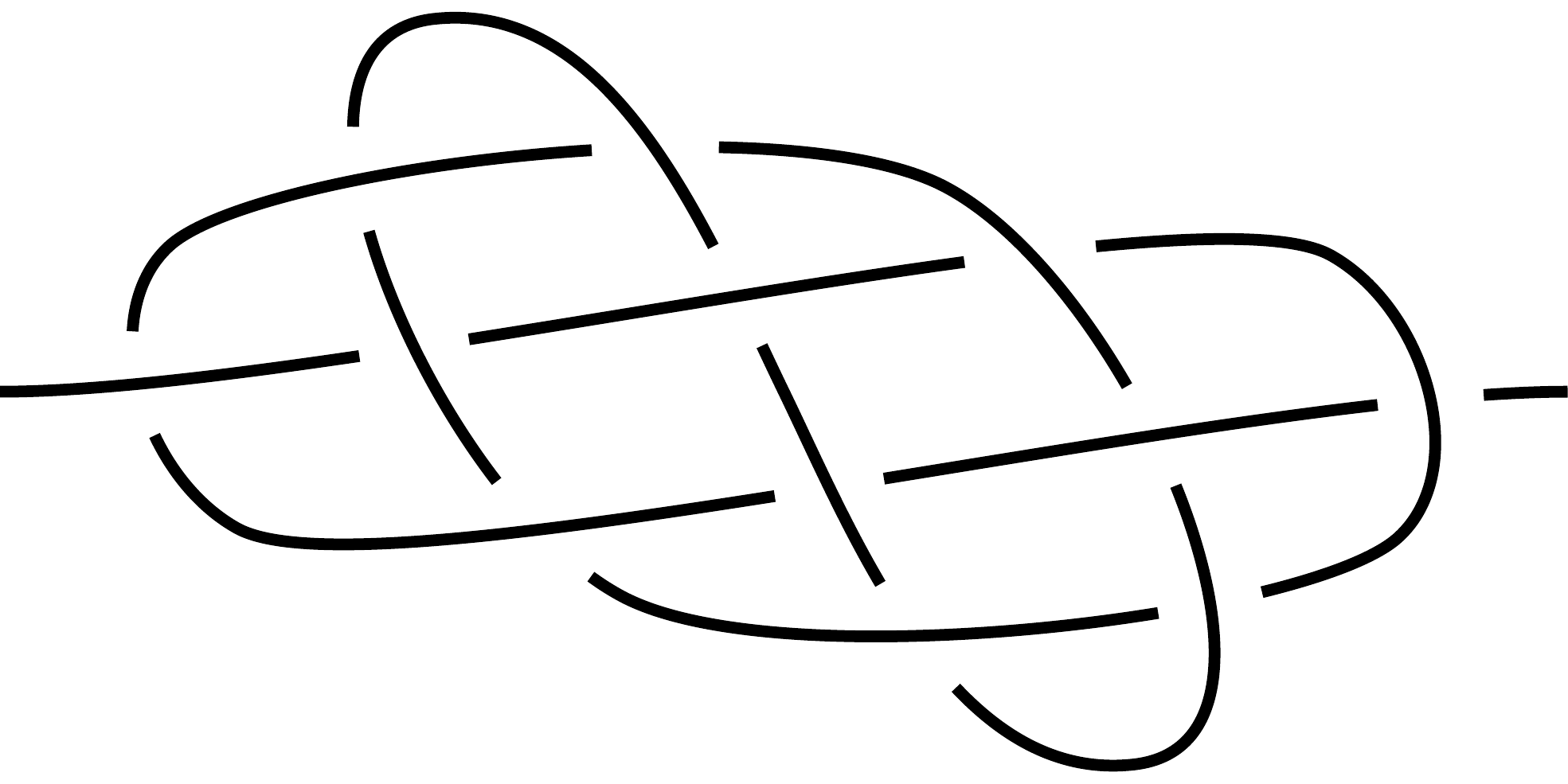}}}&\multirow{4}{*}{\scalebox{.2}{\includegraphics{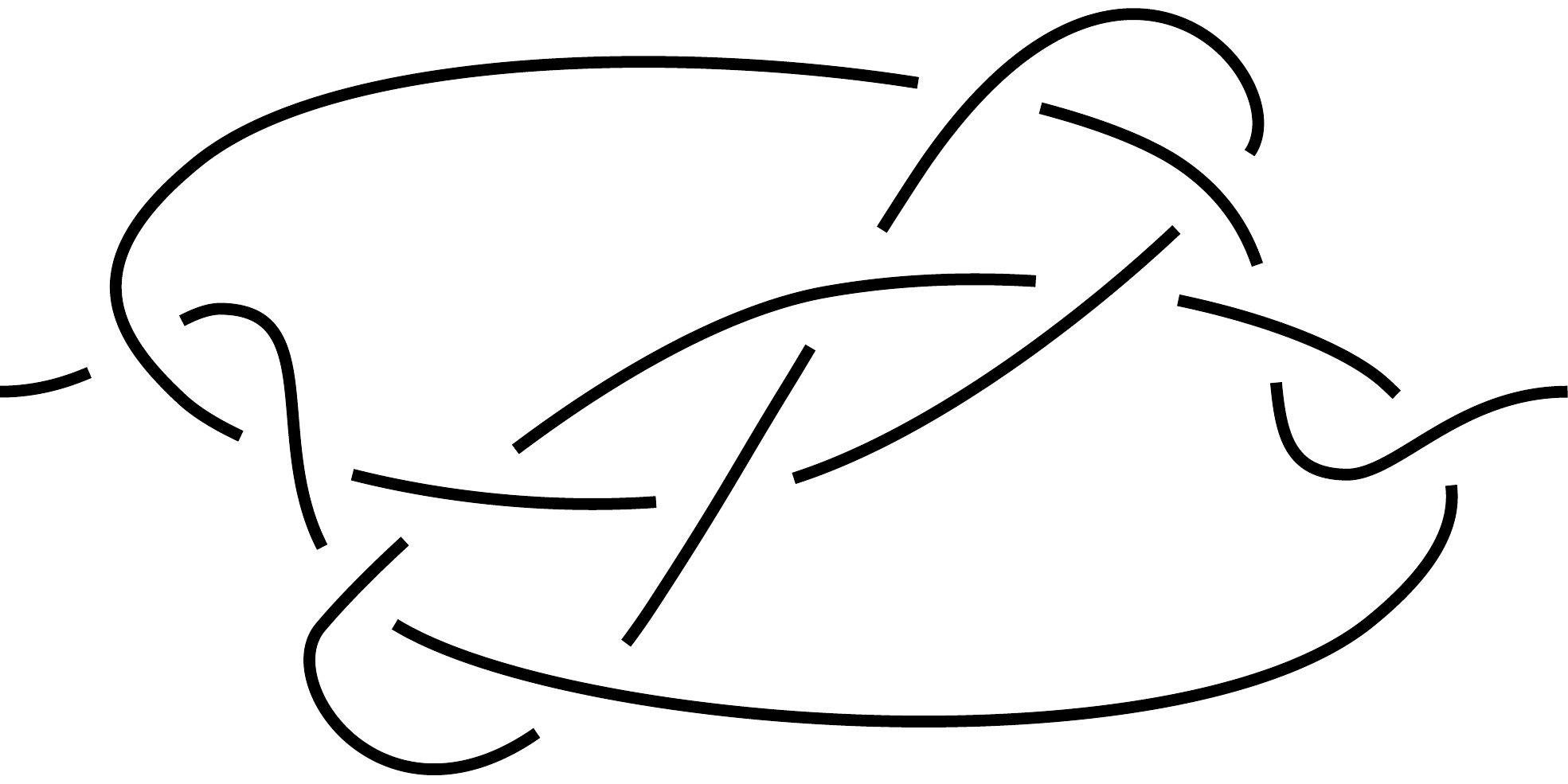}}}&
 \multirow{4}{*}{\scalebox{.2}{\includegraphics{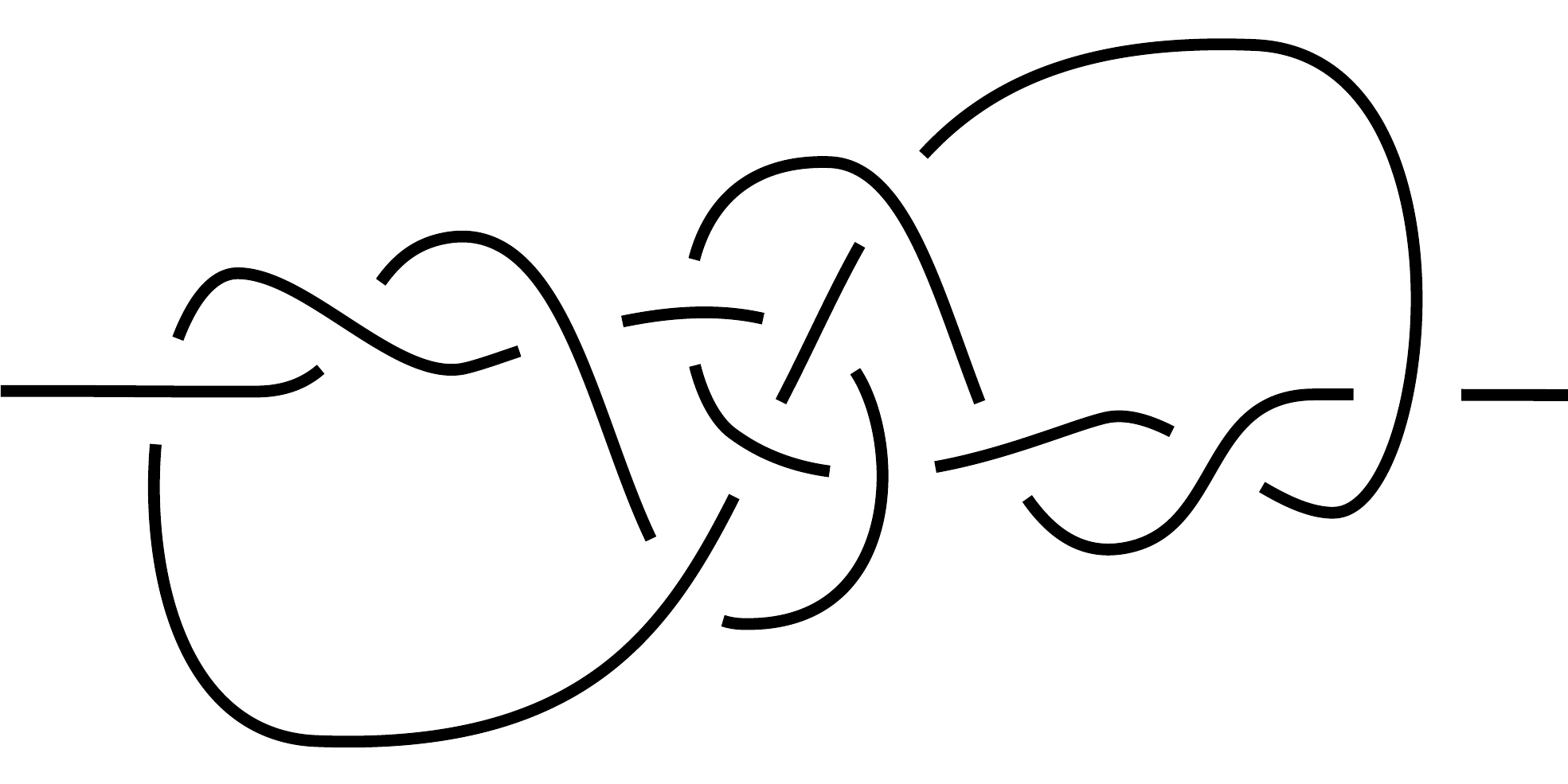}}}\\*
 &&\\*
 &&\\*
 &&\\*
 &&\\\hline\hline 

 $12a_{1019}$&$12a_{1039}$&$12a_{1102}$\\*
 \multirow{4}{*}{\scalebox{.2}{\includegraphics{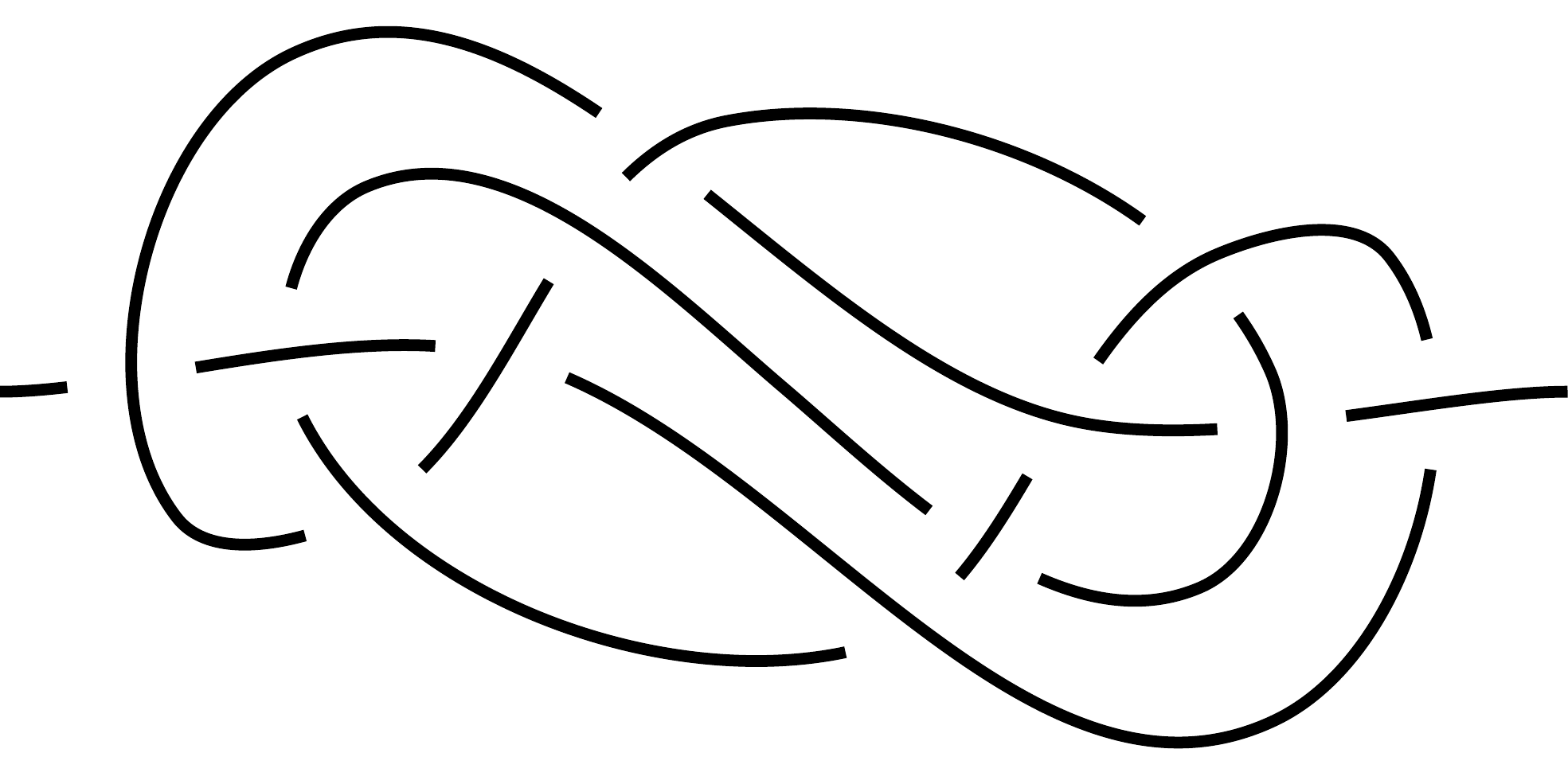}}}&\multirow{4}{*}{\scalebox{.2}{\includegraphics{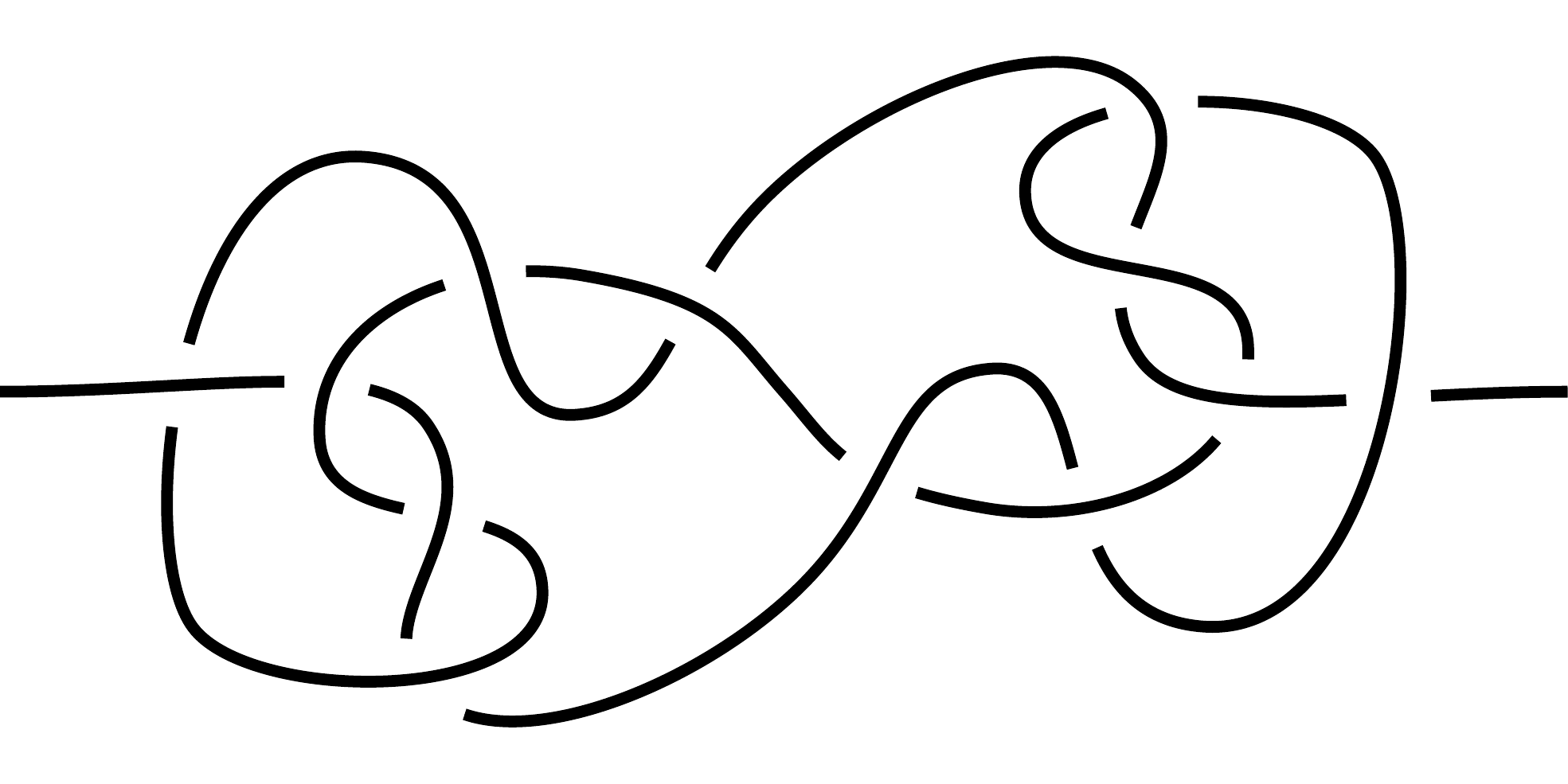}}}&
 \multirow{4}{*}{\scalebox{.2}{\includegraphics{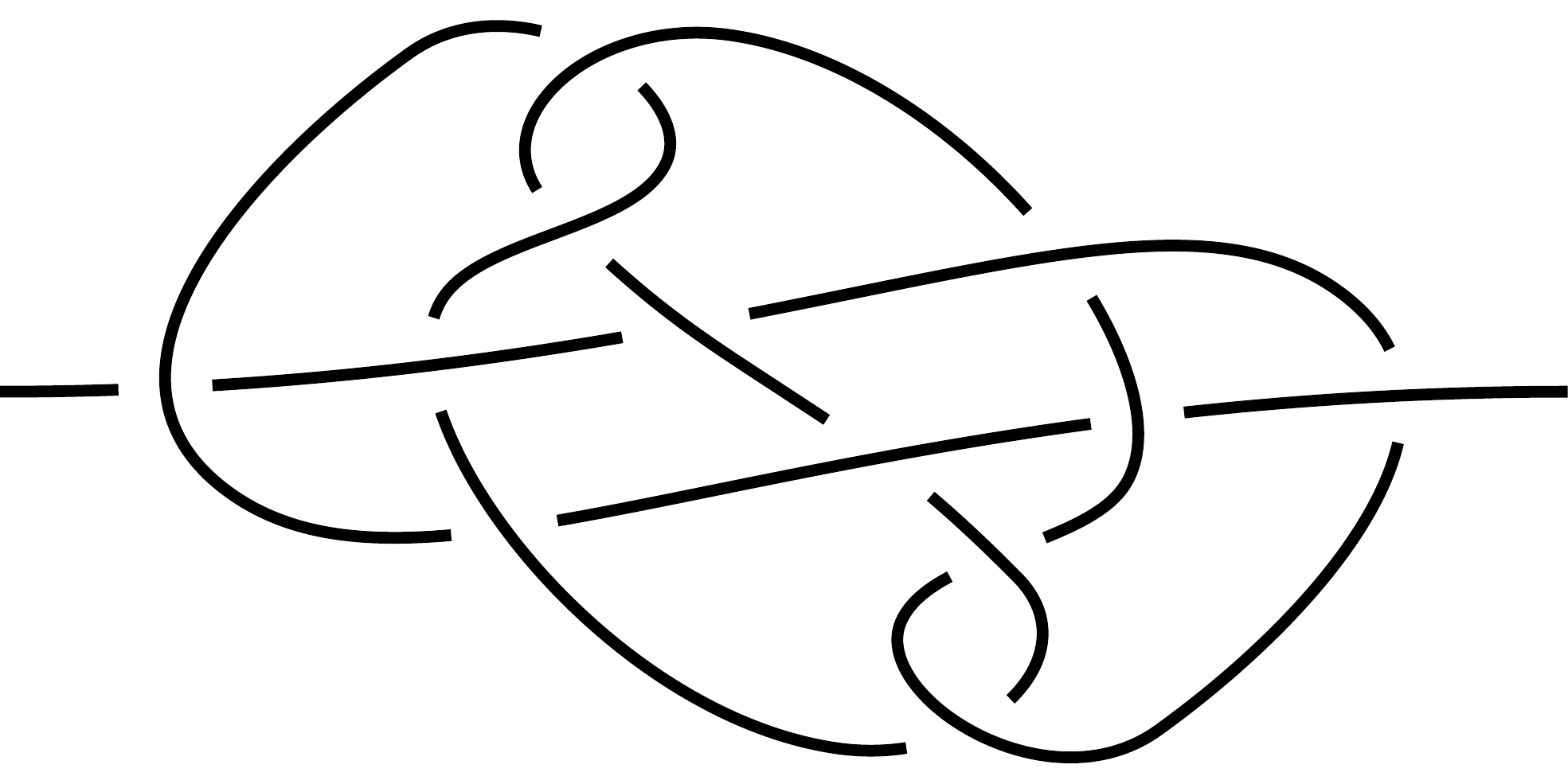}}}\\*
 &&\\*
 &&\\*
 &&\\*
 &&\\\hline\hline 

 $12a_{1105}$&$12a_{1123}$&$12a_{1124}$\\*
 \multirow{4}{*}{\scalebox{.2}{\includegraphics{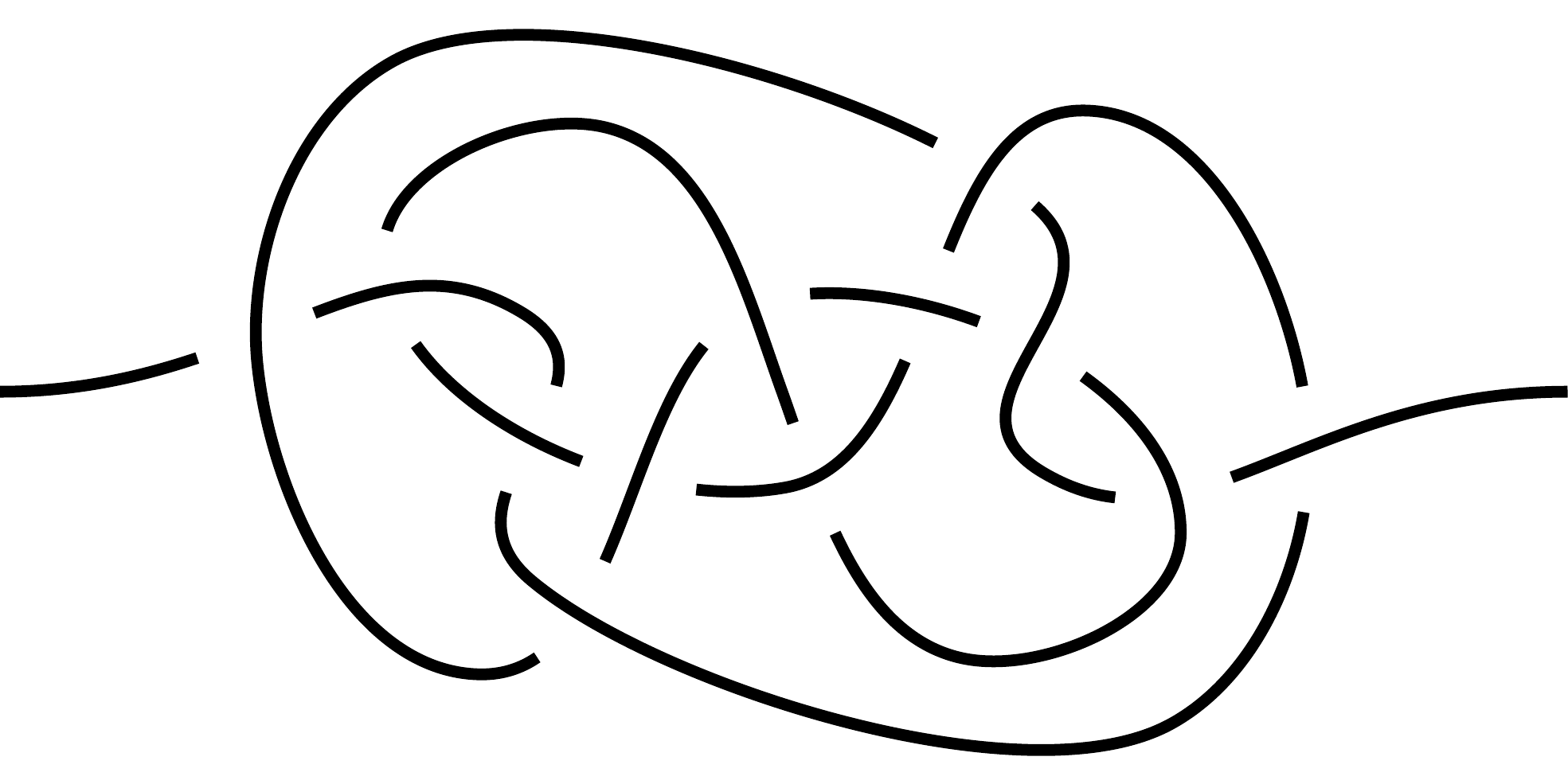}}}&\multirow{4}{*}{\scalebox{.2}{\includegraphics{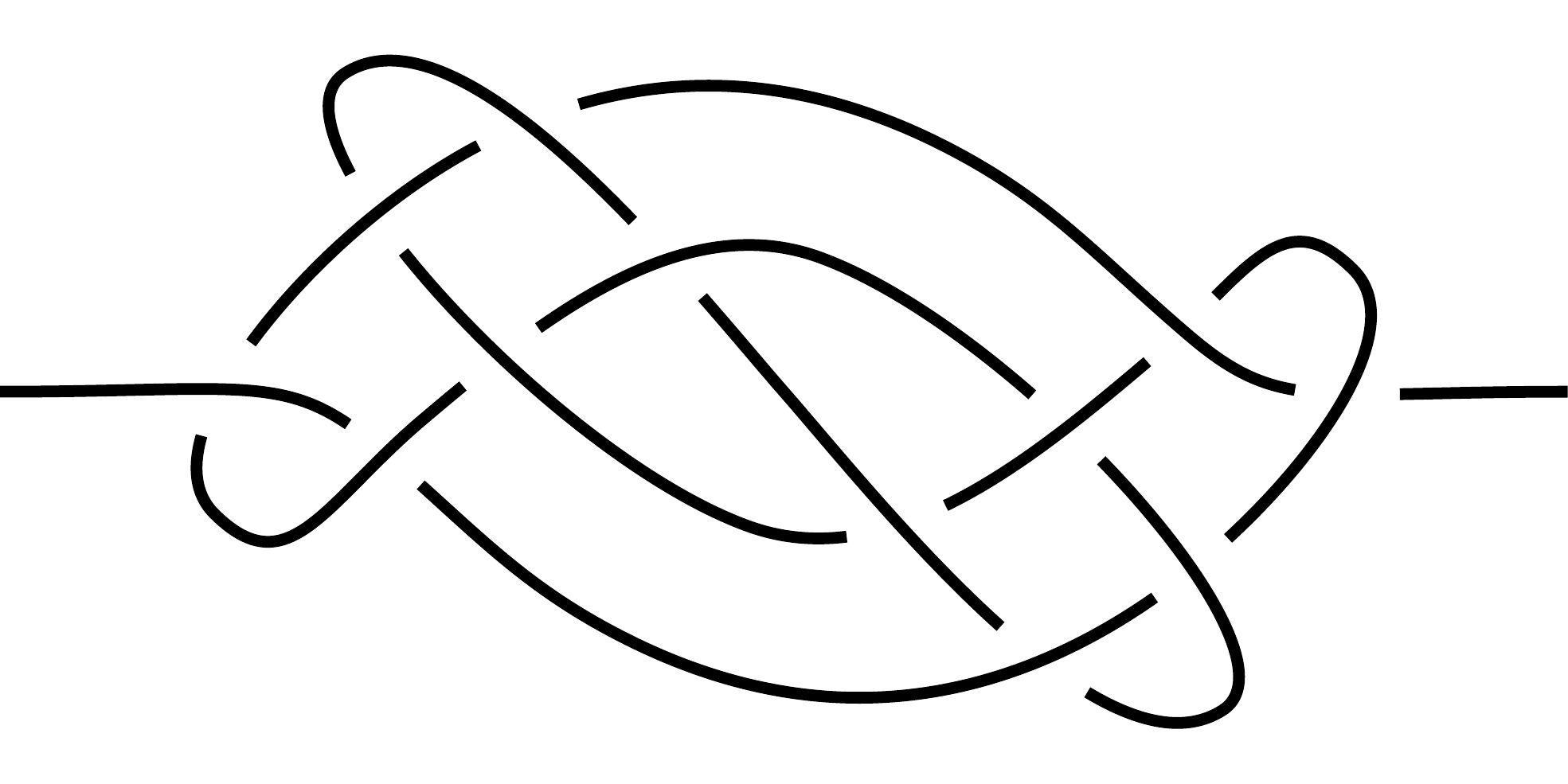}}}&
 \multirow{4}{*}{\scalebox{.2}{\includegraphics{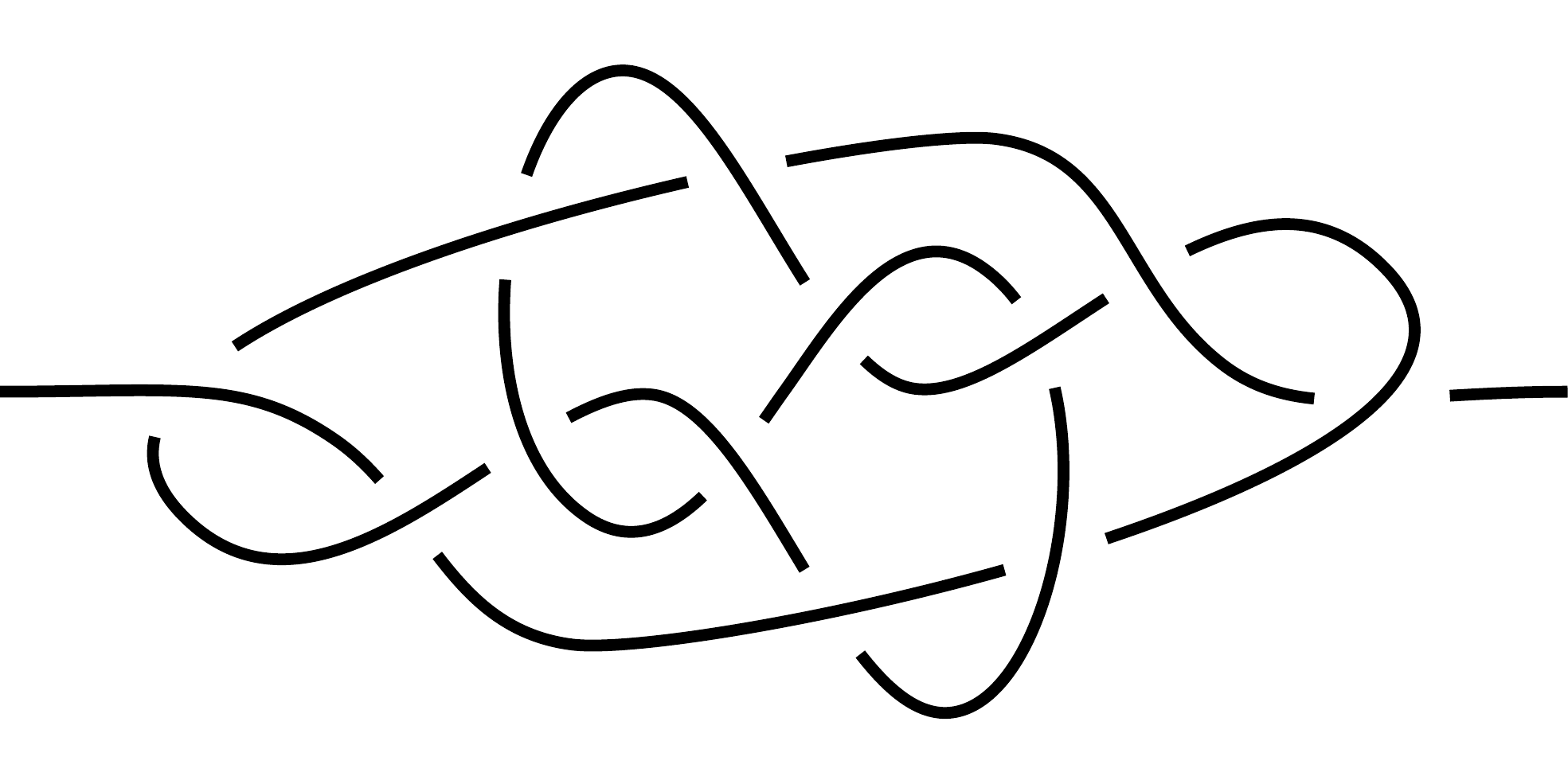}}}\\*
 &&\\*
 &&\\*
 &&\\*
 &&\\\hline\hline

$12a_{1127}$&$12a_{1152}$&$12a_{1167}$\\*
 \multirow{4}{*}{\scalebox{.2}{\includegraphics{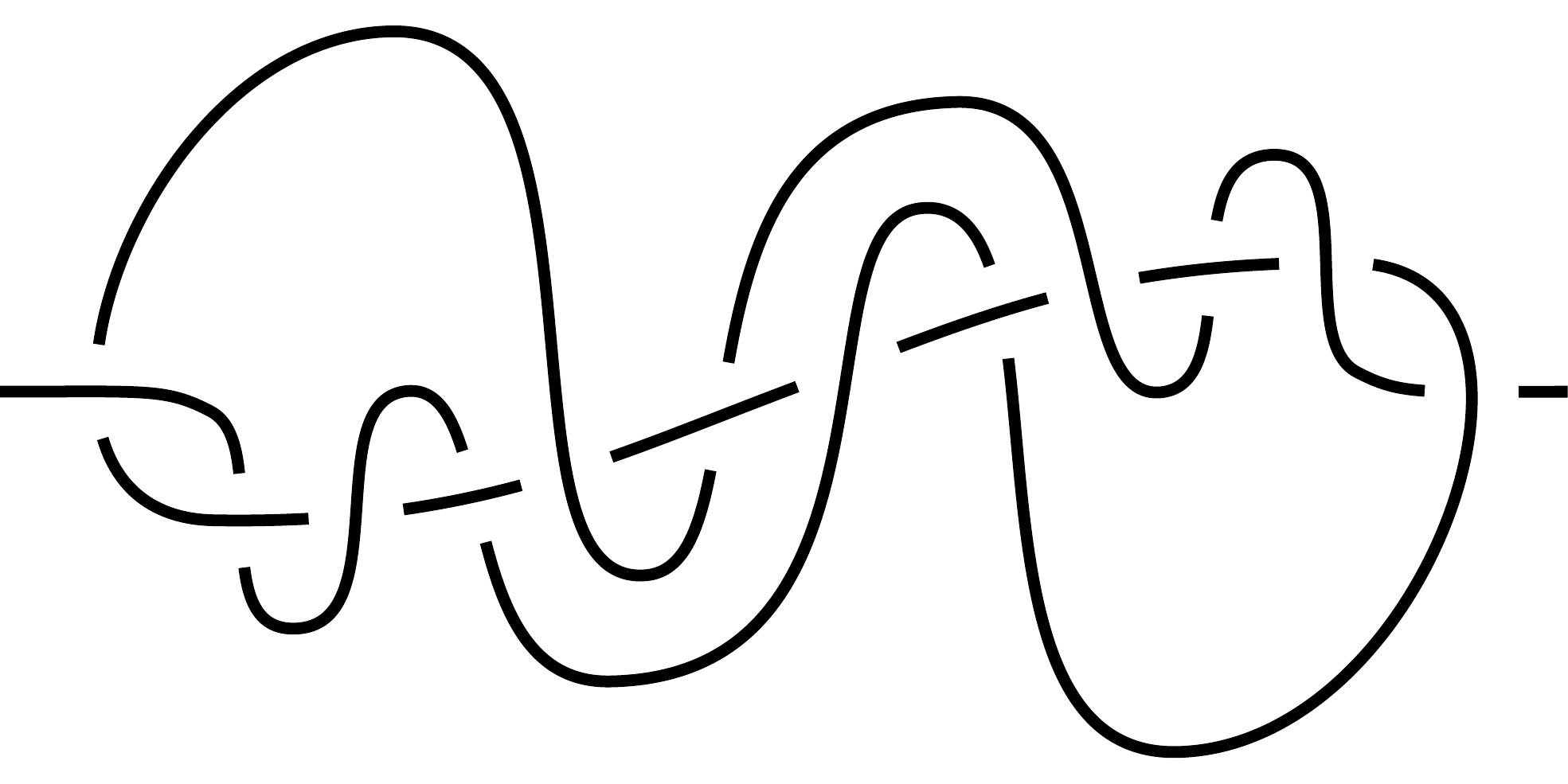}}}&\multirow{4}{*}{\scalebox{.2}{\includegraphics{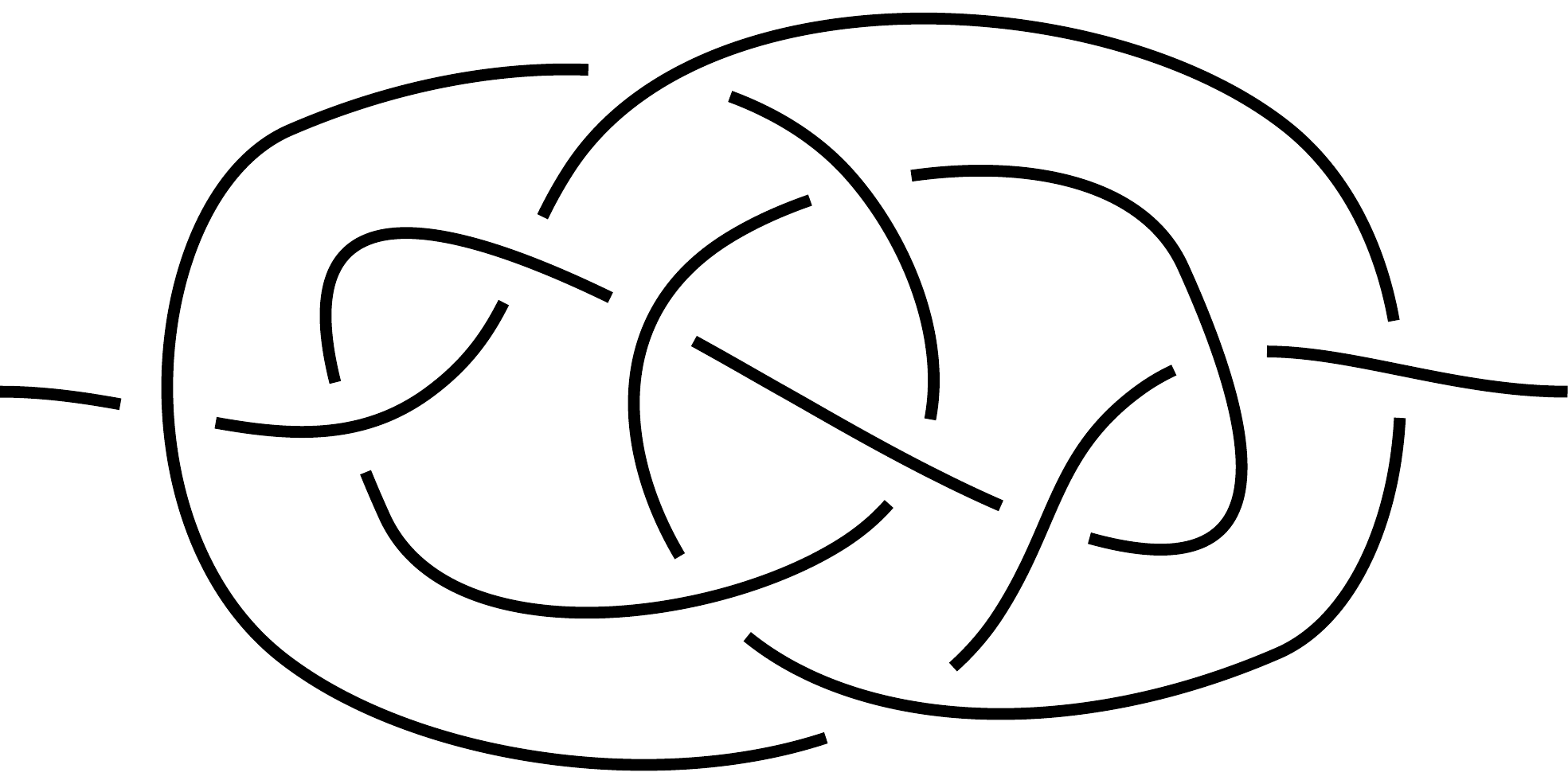}}}&
 \multirow{4}{*}{\scalebox{.2}{\includegraphics{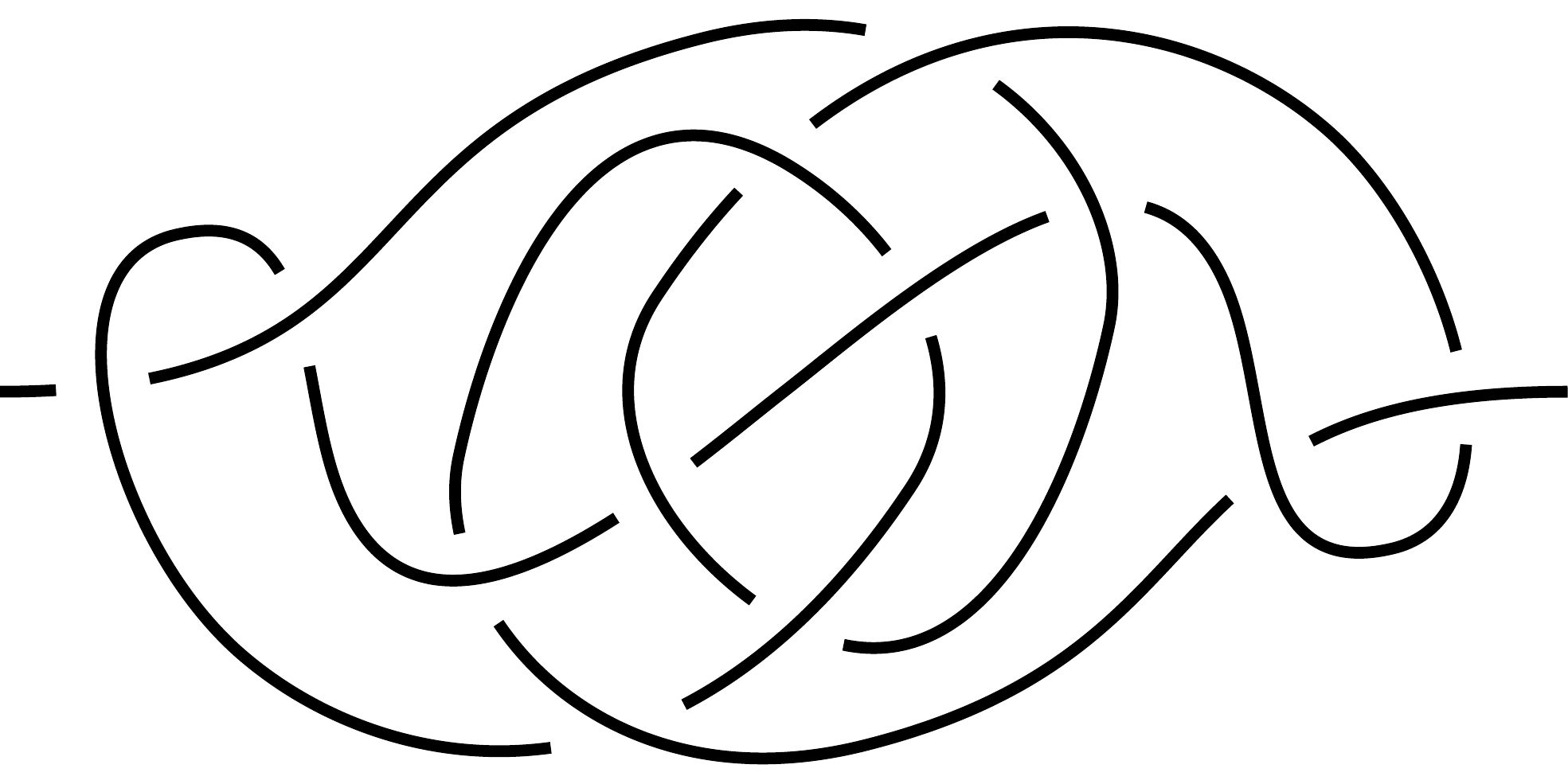}}}\\*
 &&\\*
 &&\\*
 &&\\*
 &&\\\hline\hline

$12a_{1188}$&$12a_{1202}$&$12a_{1209}$\\*
 \multirow{4}{*}{\scalebox{.2}{\includegraphics{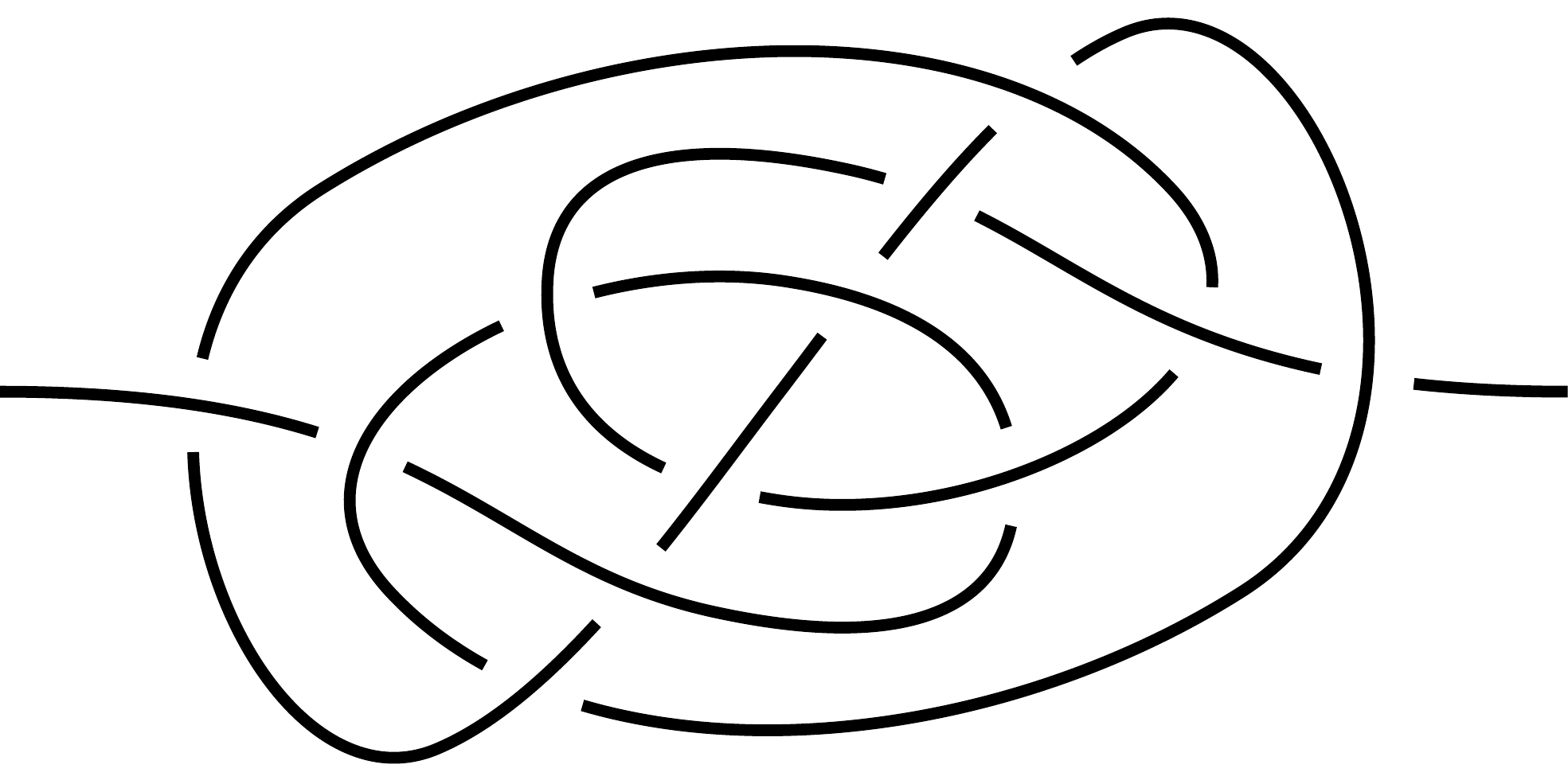}}}&\multirow{4}{*}{\scalebox{.2}{\includegraphics{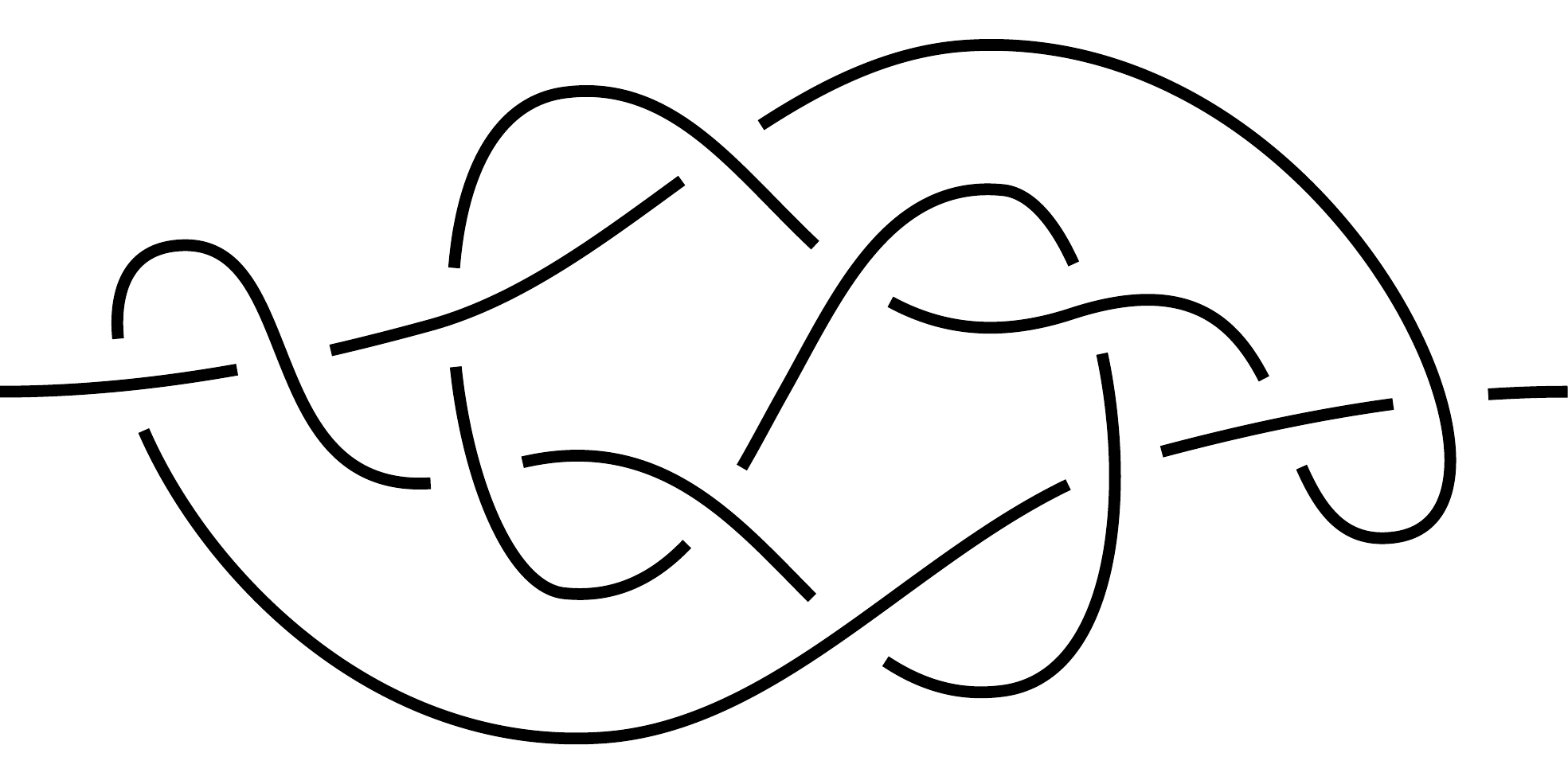}}}&
 \multirow{4}{*}{\scalebox{.2}{\includegraphics{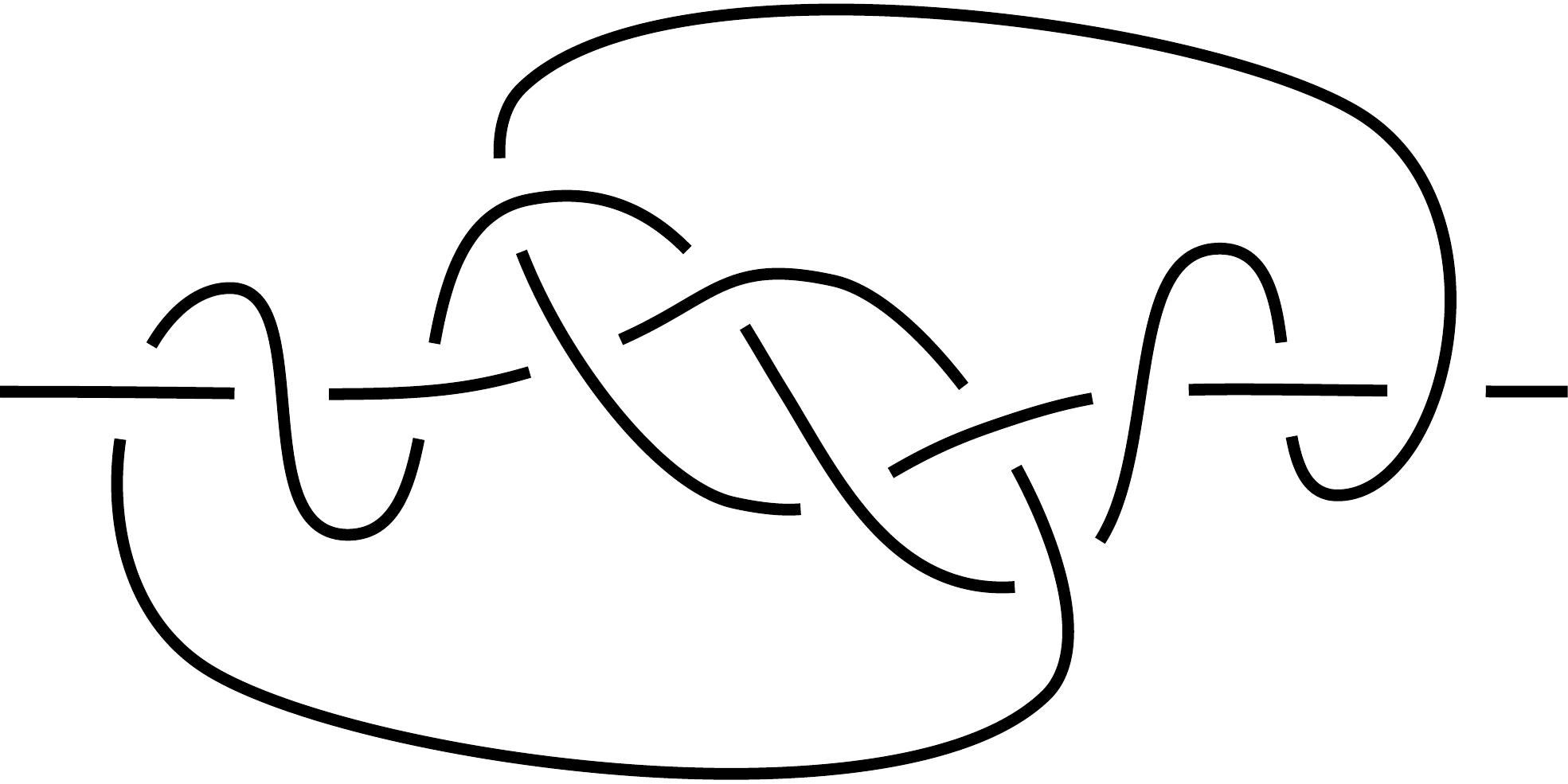}}}\\*
 &&\\*
 &&\\*
 &&\\*
 &&\\\hline\hline

$12a_{1211}$&$12a_{1218}$&$12a_{1225}$\\*
 \multirow{4}{*}{\scalebox{.2}{\includegraphics{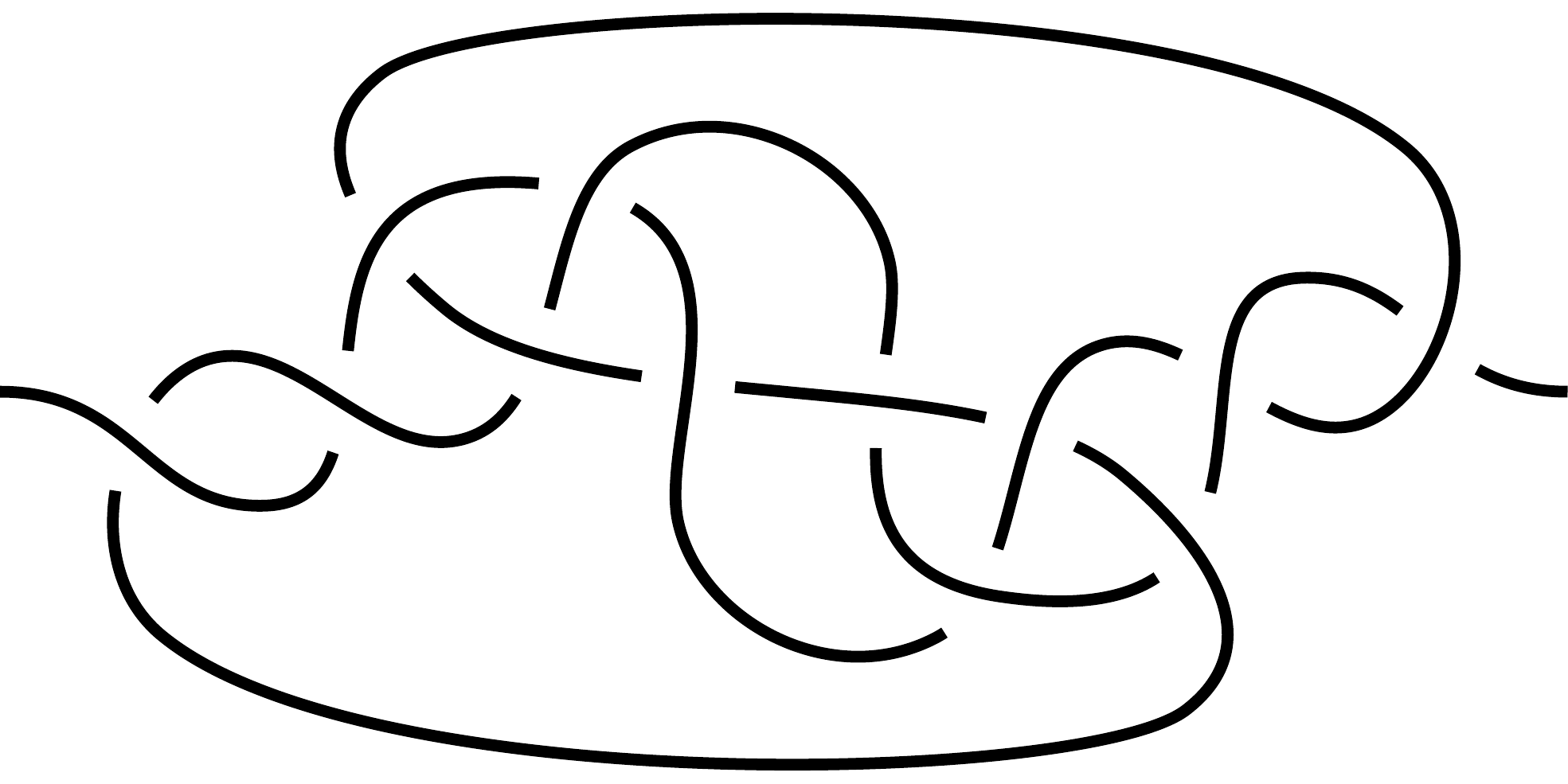}}}&\multirow{4}{*}{\scalebox{.2}{\includegraphics{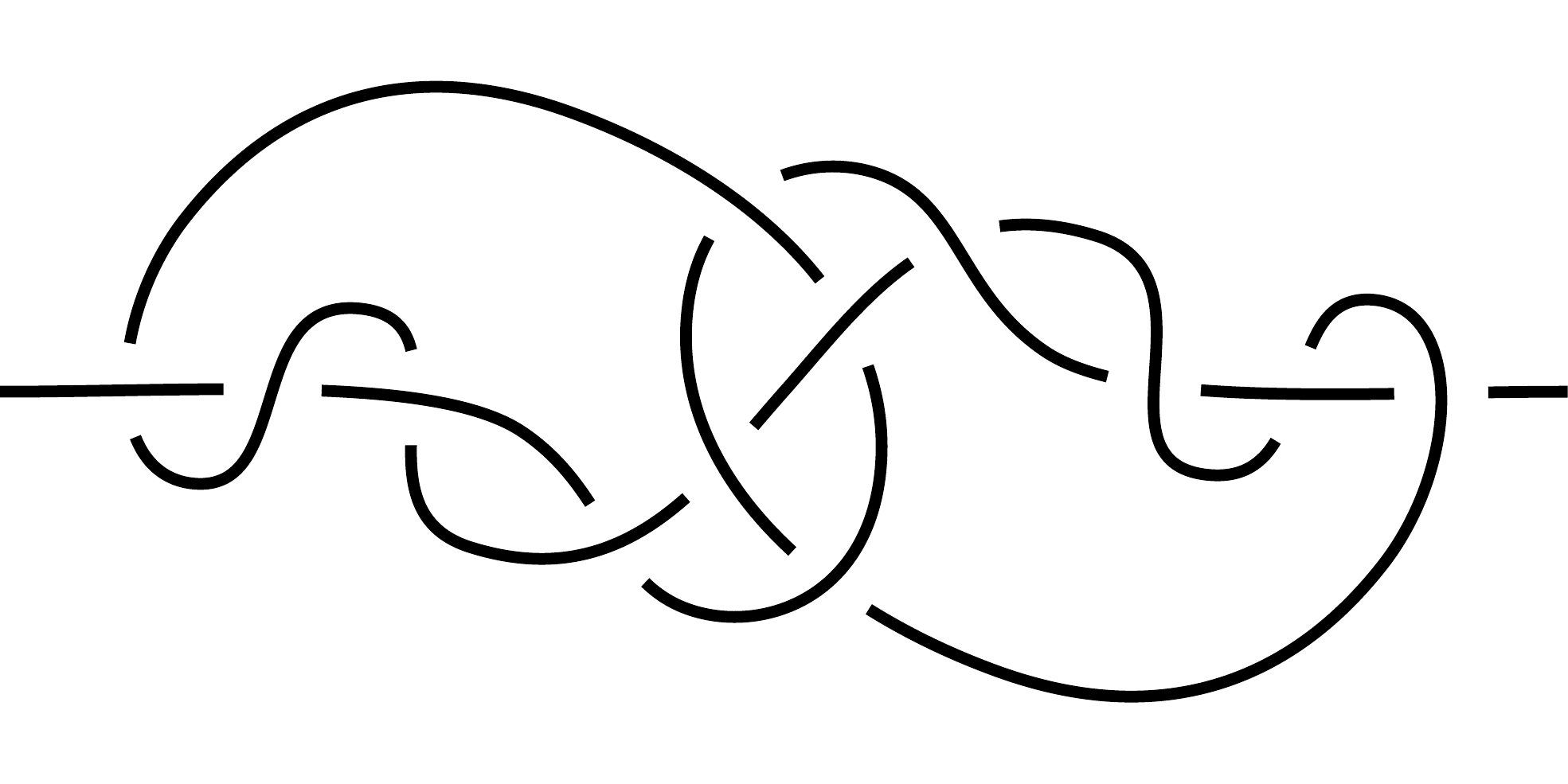}}}&
 \multirow{4}{*}{\scalebox{.2}{\includegraphics{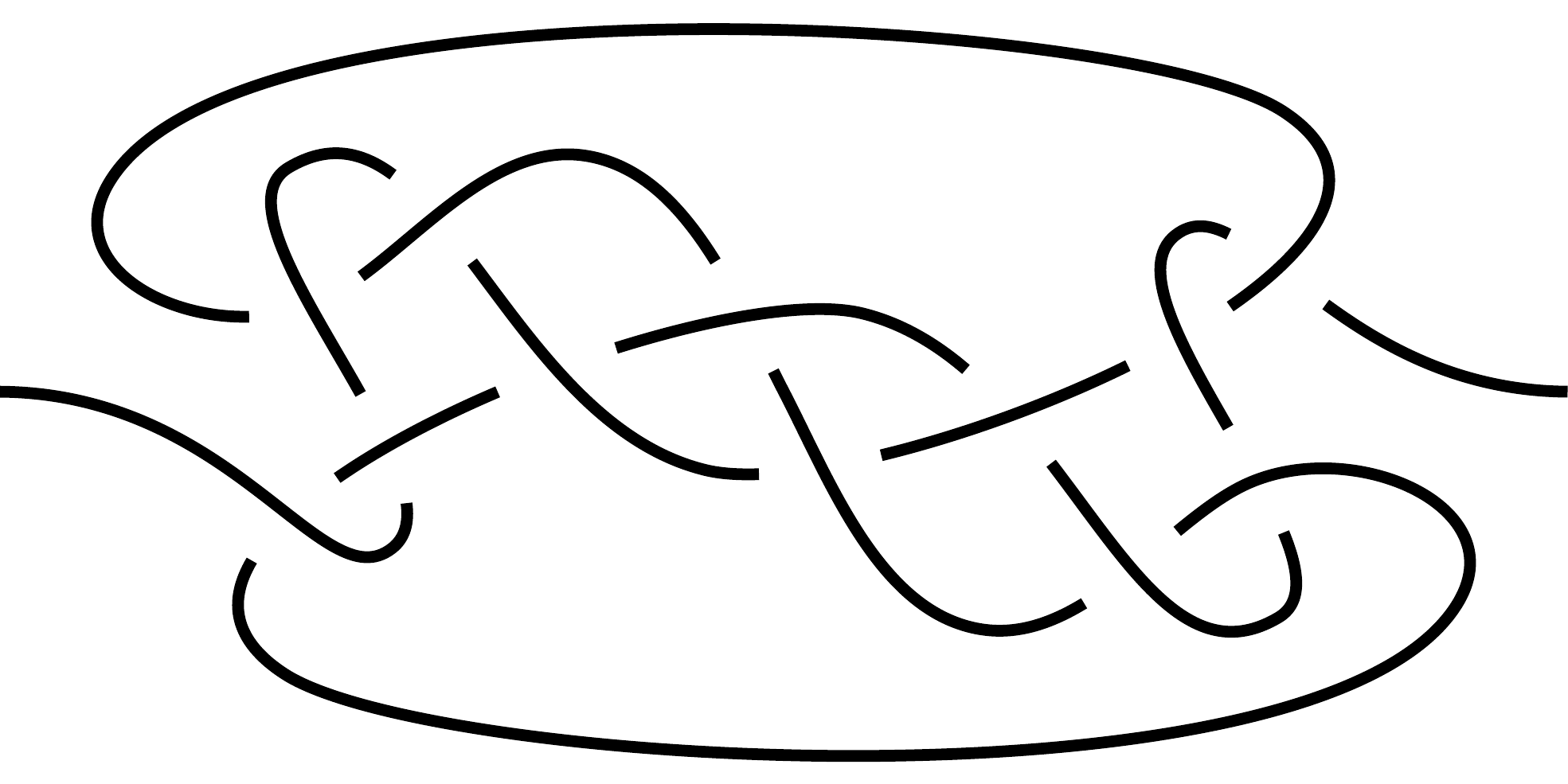}}}\\*
 &&\\*
 &&\\*
 &&\\*
 &&\\\hline\hline

$12a_{1229}$&$12a_{1249}$&$12a_{1251}$\\*
 \multirow{4}{*}{\scalebox{.2}{\includegraphics{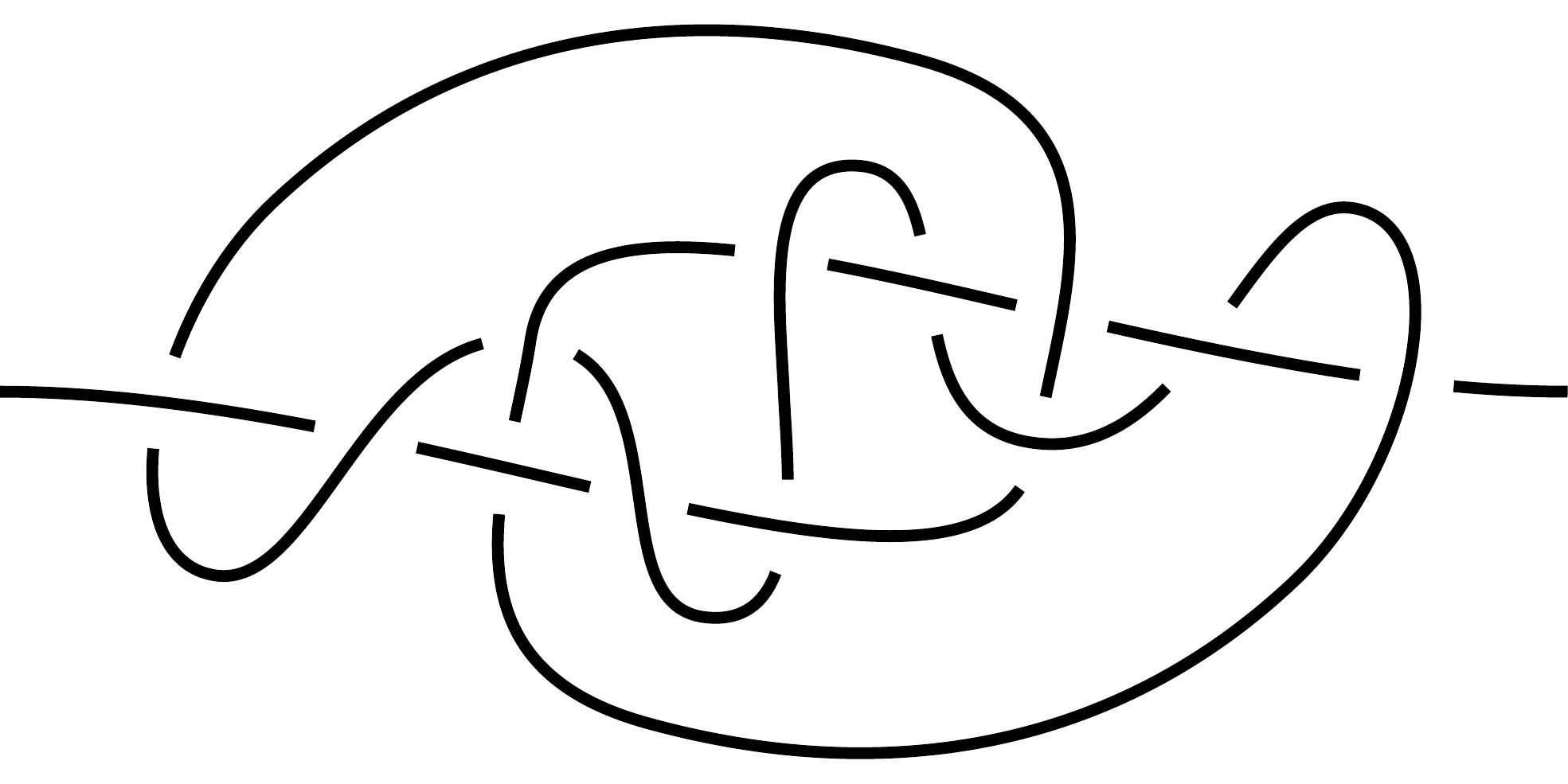}}}&\multirow{4}{*}{\scalebox{.2}{\includegraphics{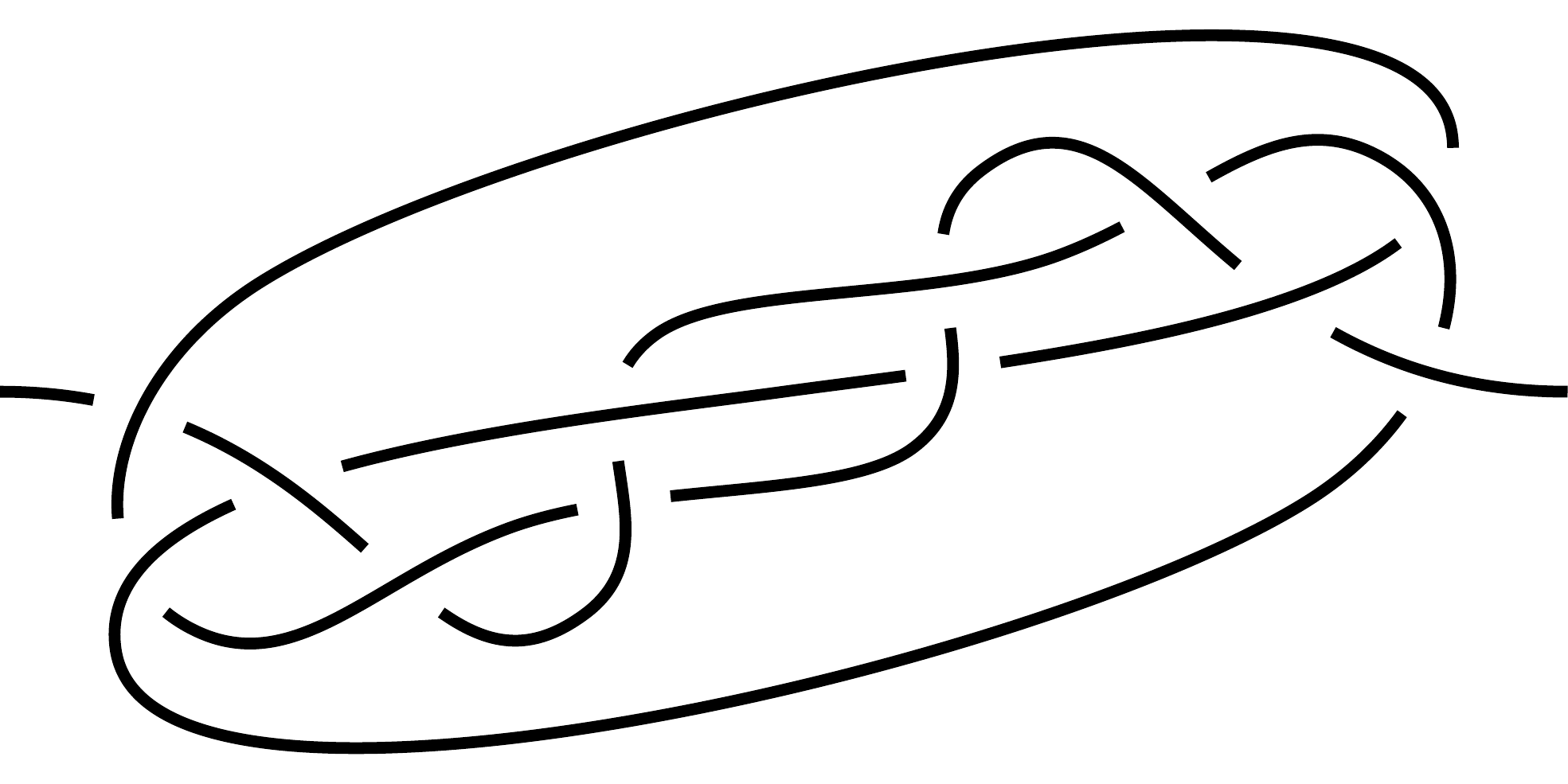}}}&
 \multirow{4}{*}{\scalebox{.2}{\includegraphics{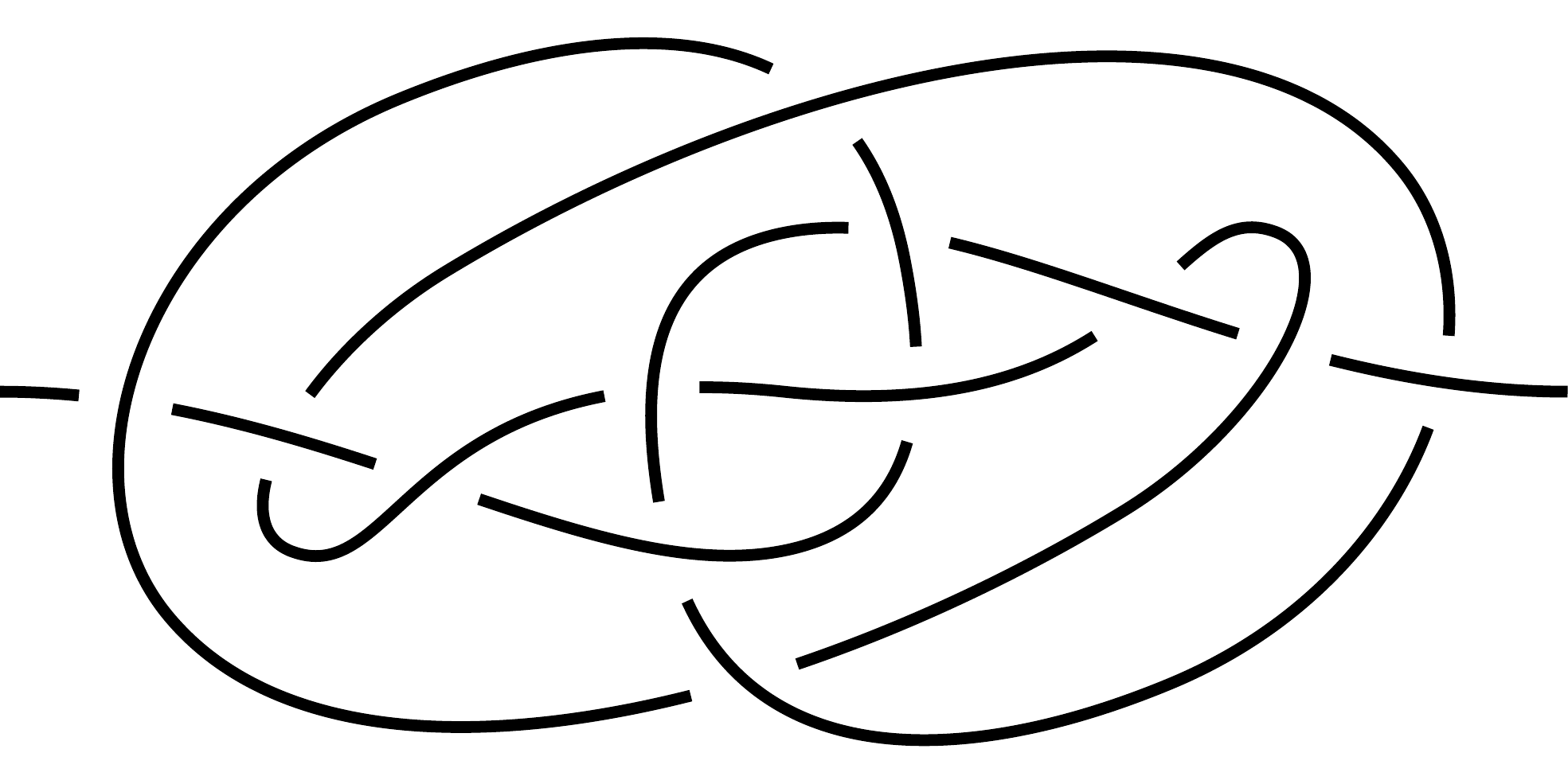}}}\\*
 &&\\*
 &&\\*
 &&\\*
 &&\\\hline\hline

$12a_{1254}$&$12a_{1260}$&$12a_{1267}$\\*
 \multirow{4}{*}{\scalebox{.2}{\includegraphics{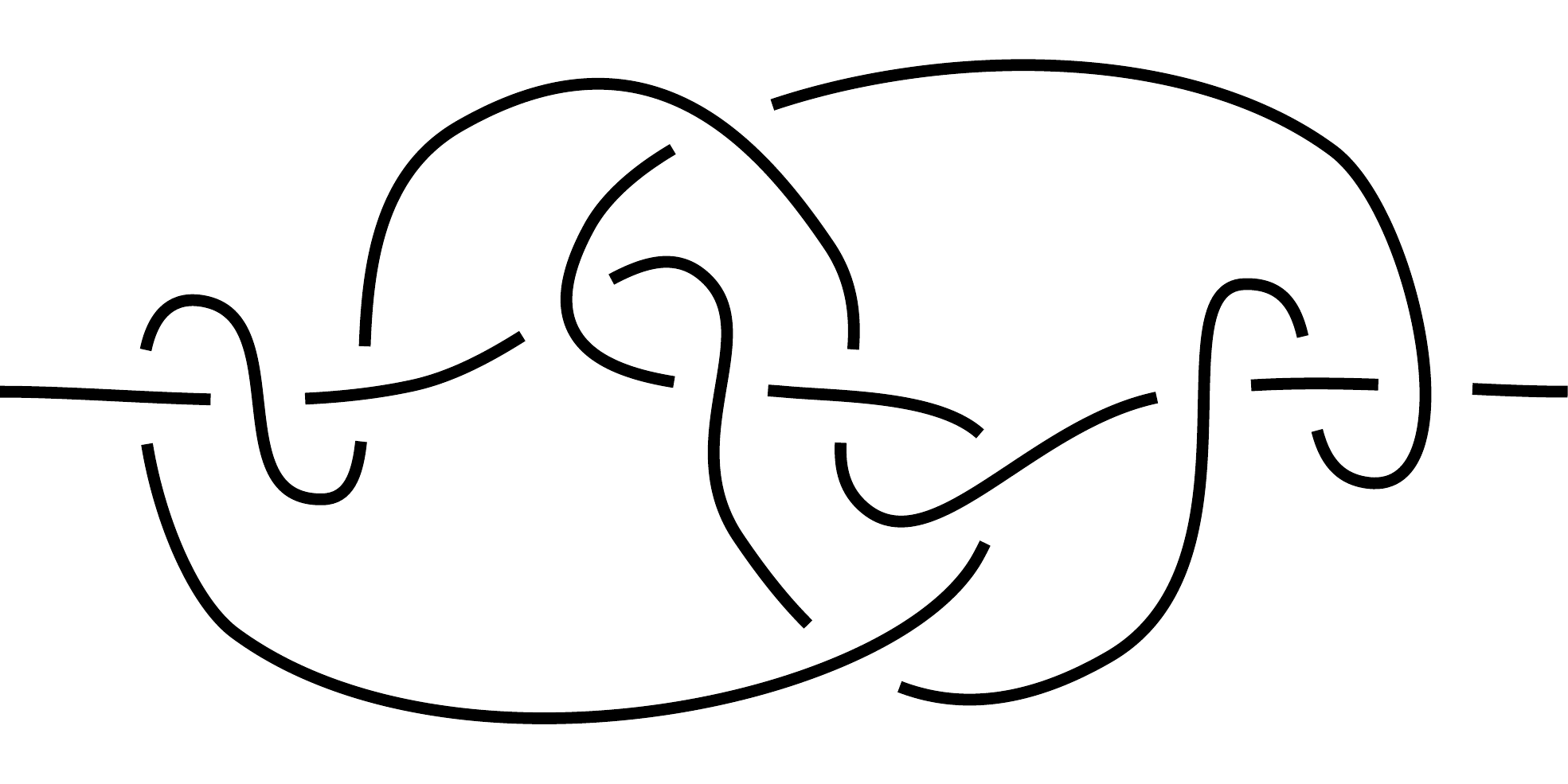}}}&\multirow{4}{*}{\scalebox{.2}{\includegraphics{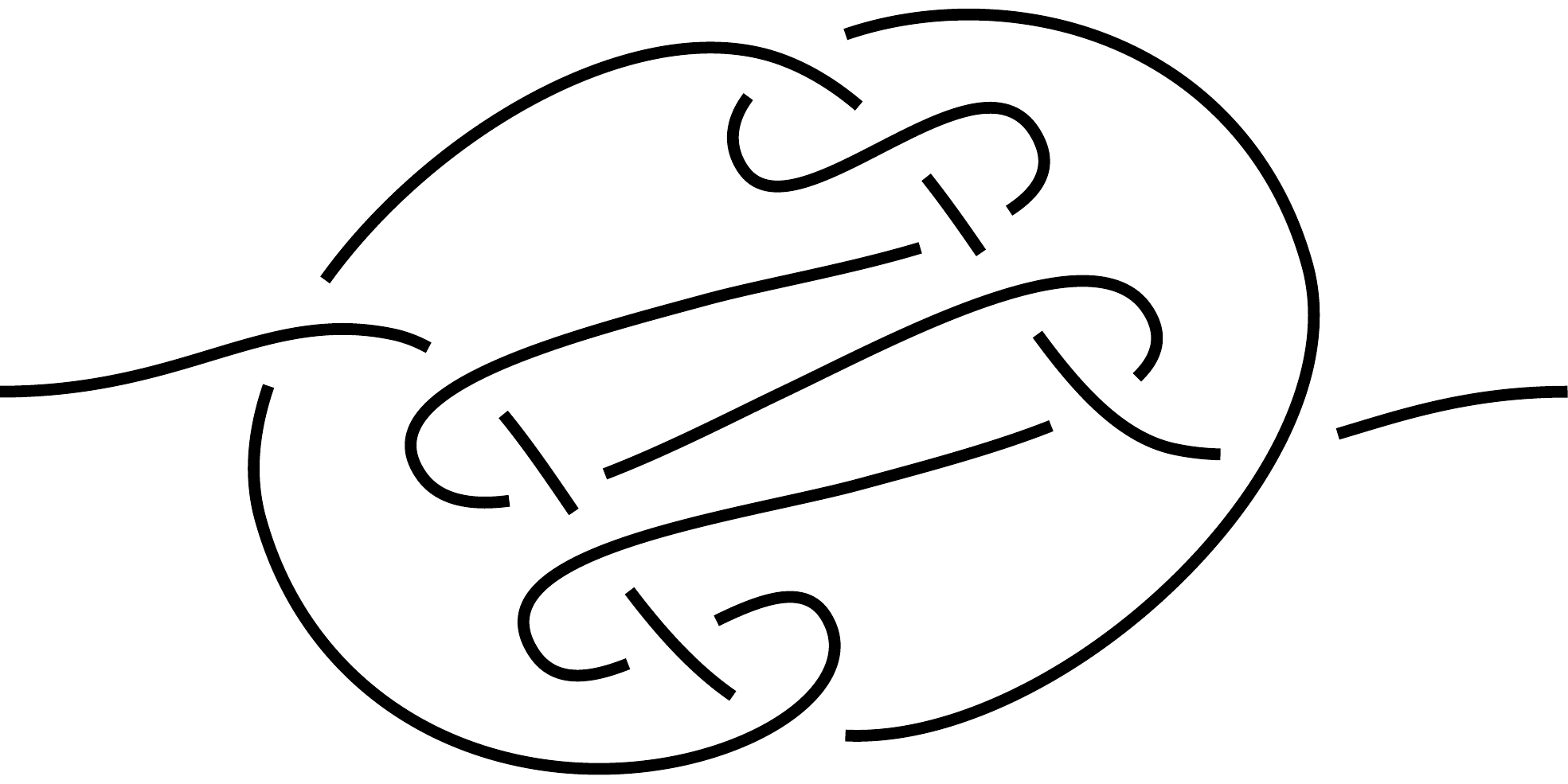}}}&
 \multirow{4}{*}{\scalebox{.2}{\includegraphics{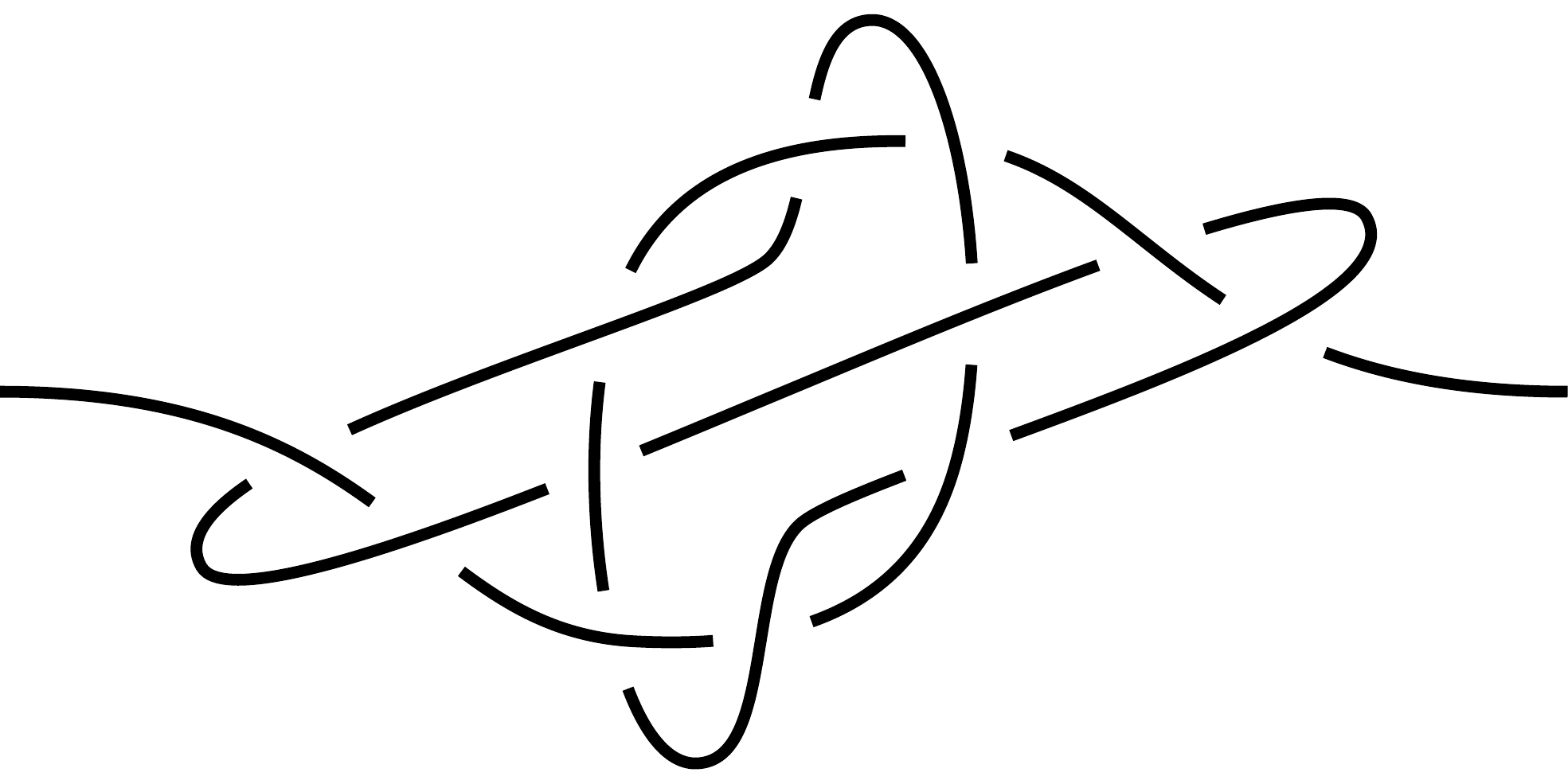}}}\\*
 &&\\*
 &&\\*
 &&\\*
 &&\\\hline\hline

$12a_{1269}$&$12a_{1273}$&$12a_{1275}$\\*
 \multirow{4}{*}{\scalebox{.2}{\includegraphics{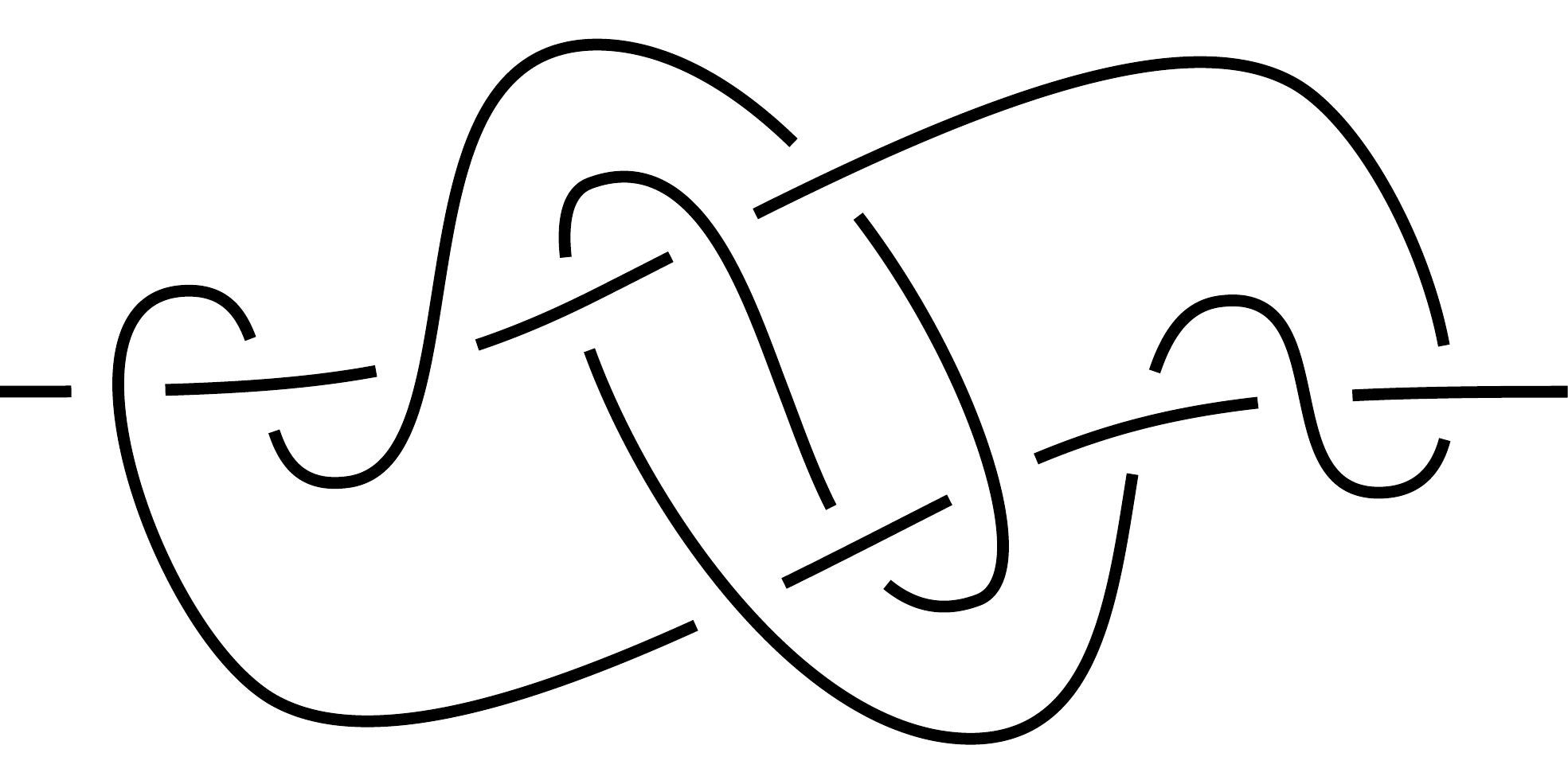}}}&\multirow{4}{*}{\scalebox{.2}{\includegraphics{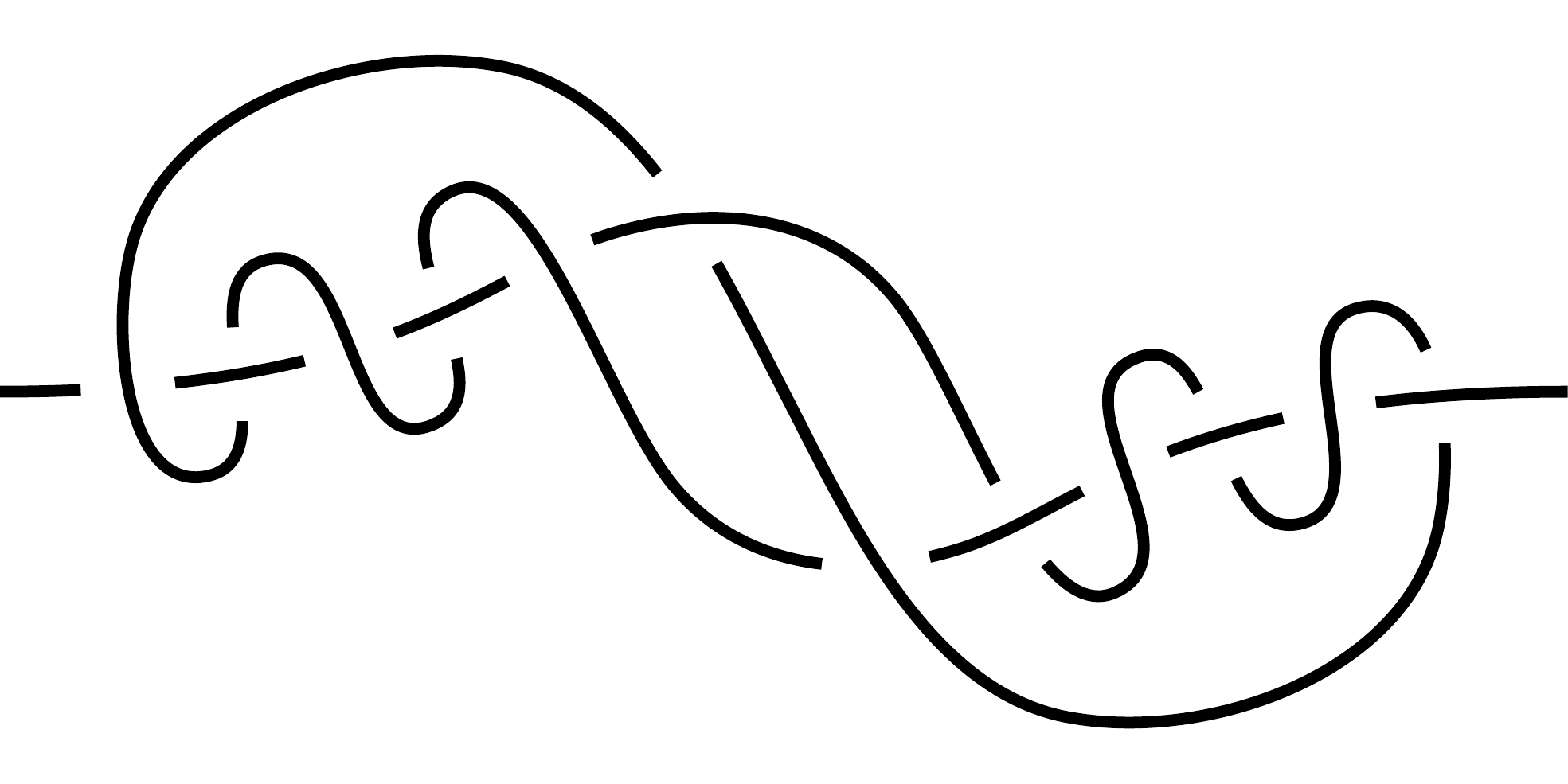}}}&
 \multirow{4}{*}{\scalebox{.2}{\includegraphics{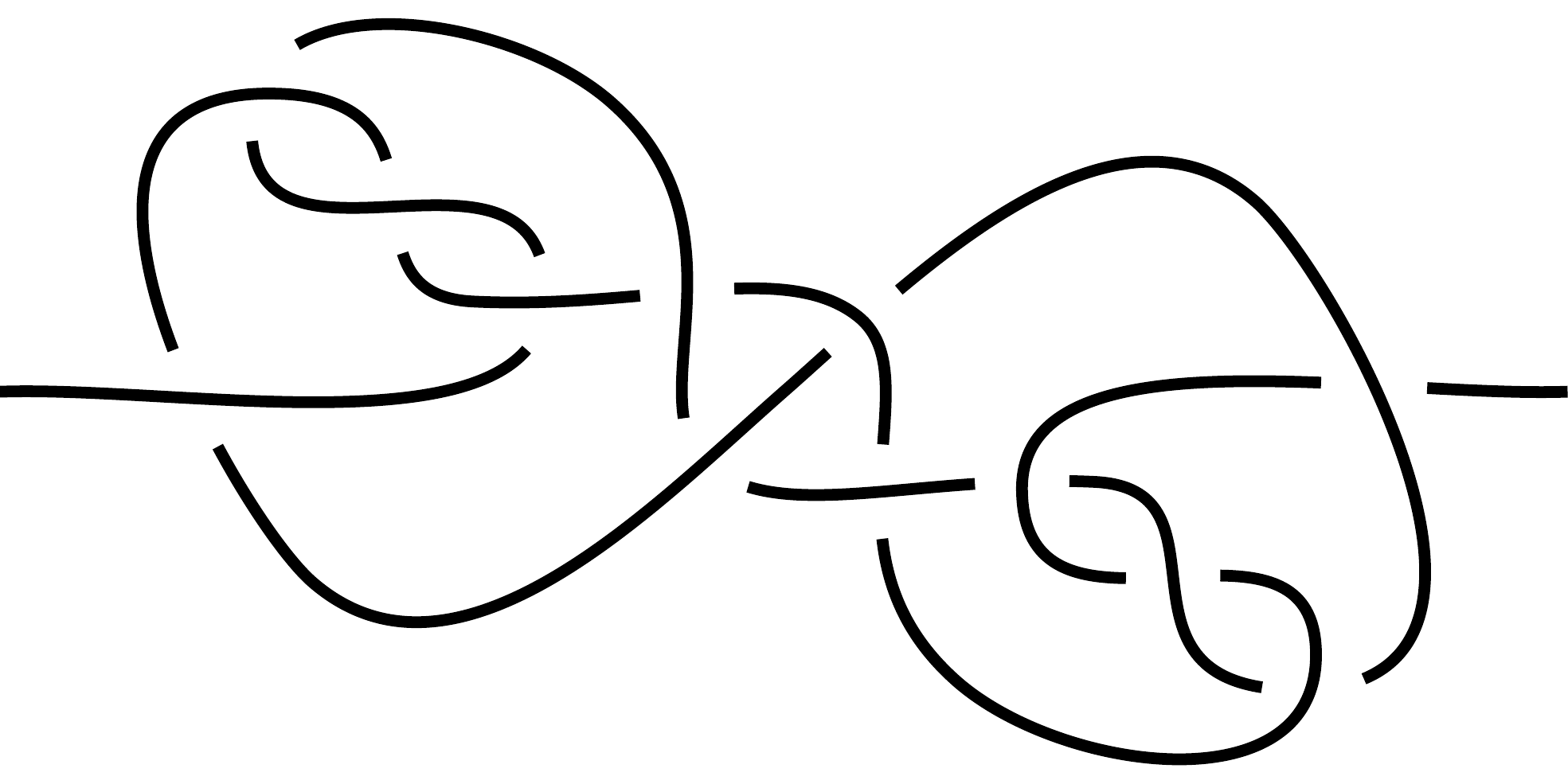}}}\\*
 &&\\*
 &&\\*
 &&\\*
 &&\\\hline\hline

$12a_{1280}$&$12a_{1281}$&$12a_{1287}$\\*
 \multirow{4}{*}{\scalebox{.2}{\includegraphics{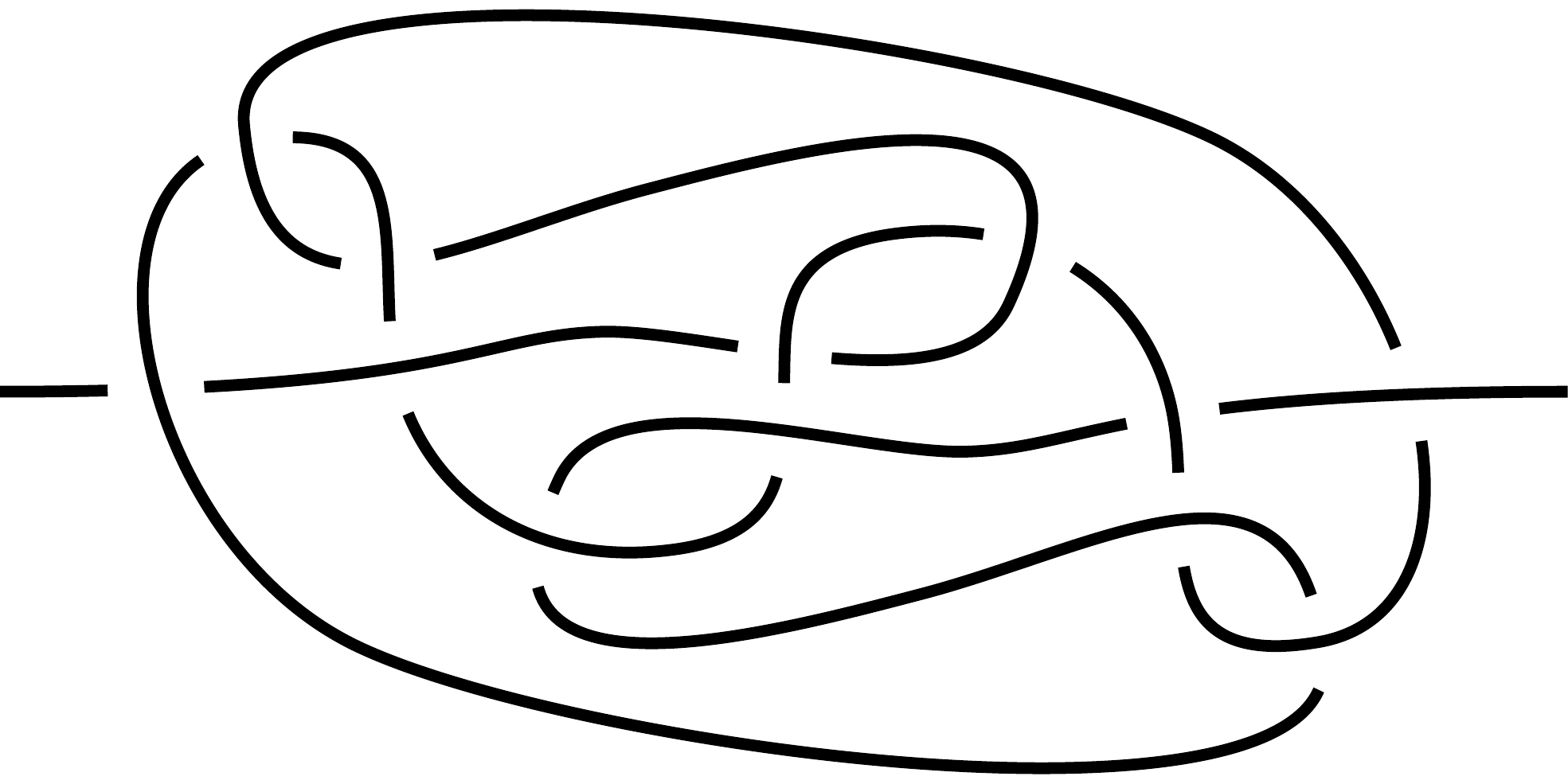}}}&\multirow{4}{*}{\scalebox{.2}{\includegraphics{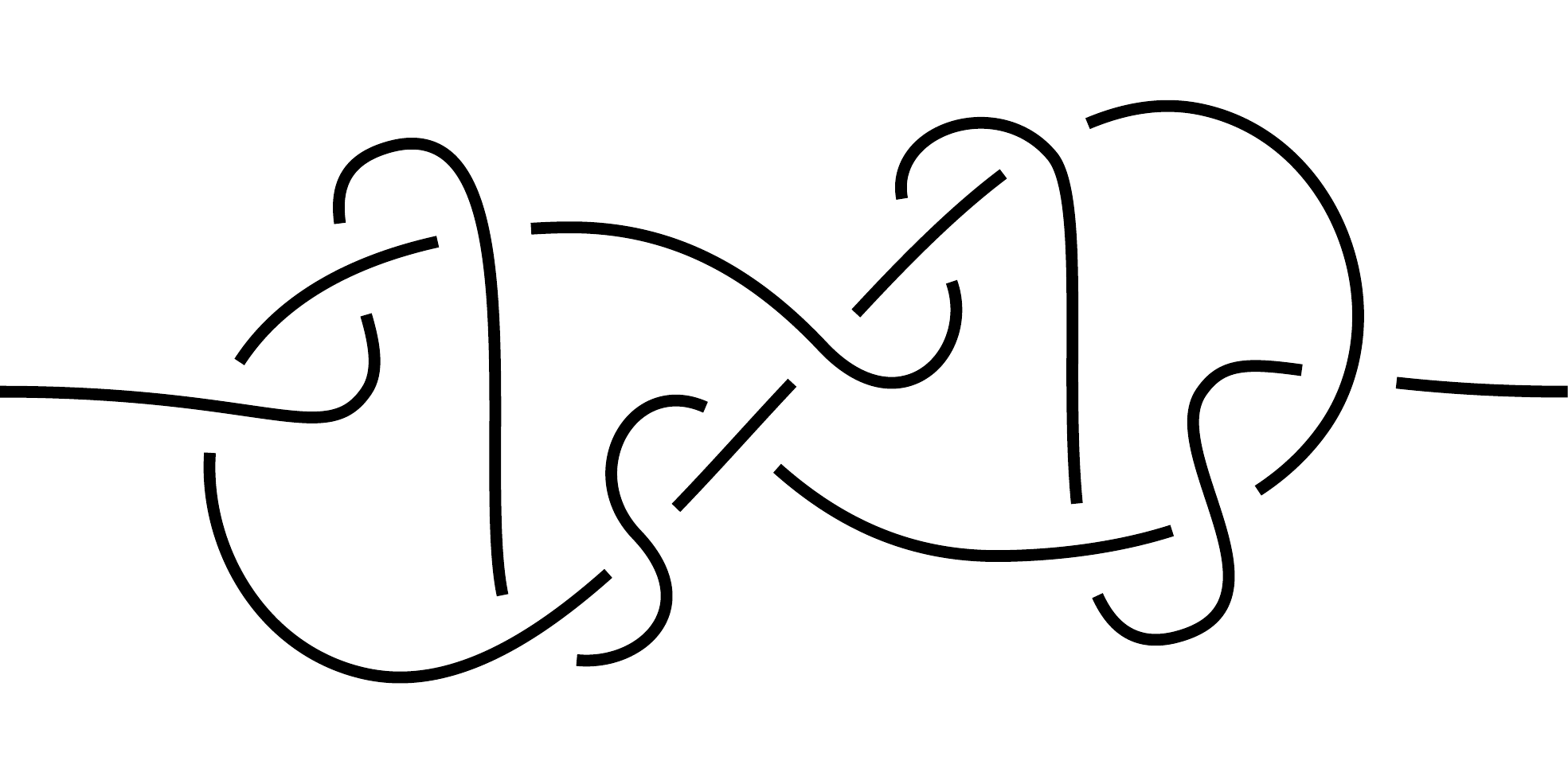}}}&
 \multirow{4}{*}{\scalebox{.2}{\includegraphics{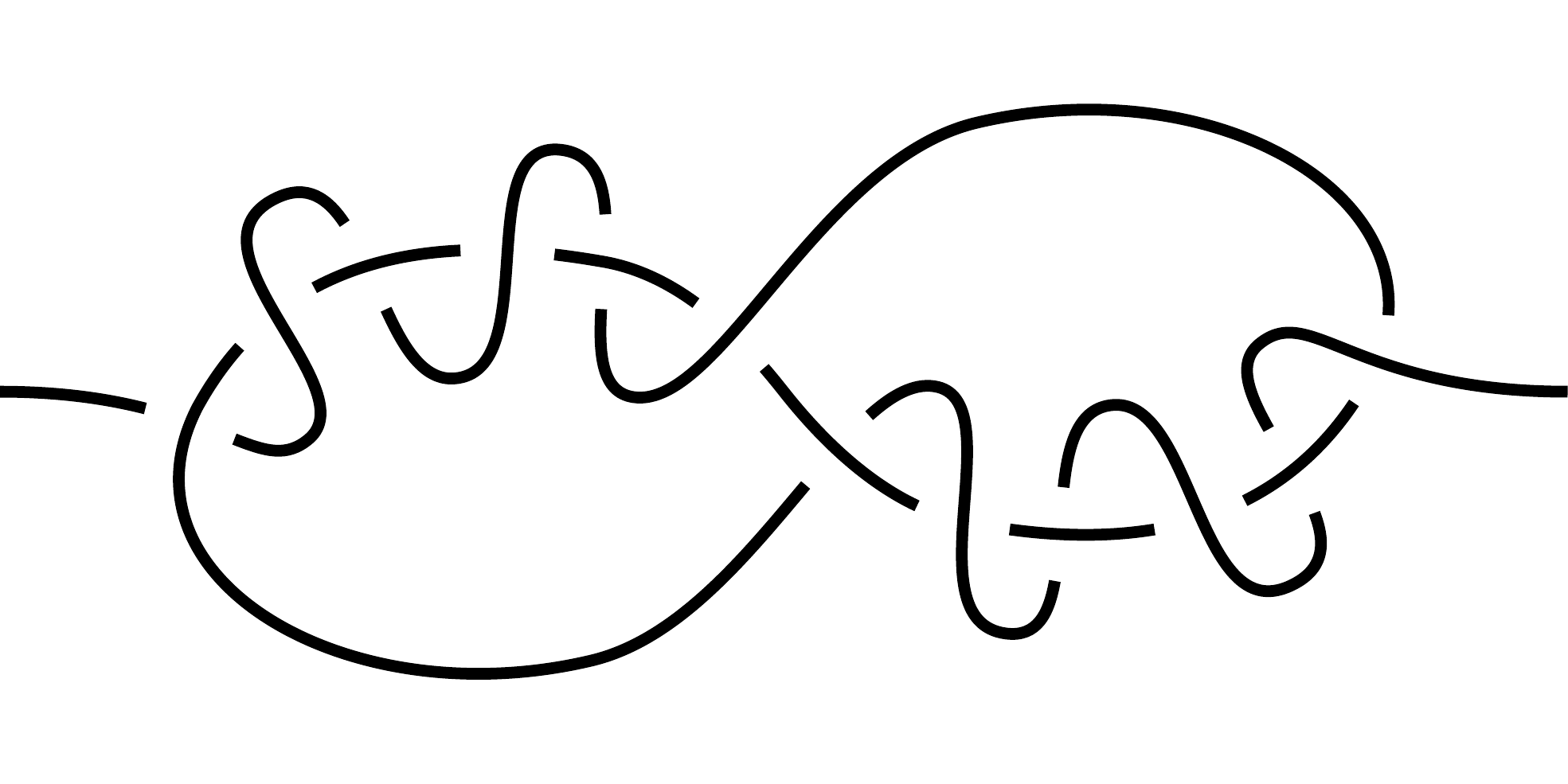}}}\\*
 &&\\*
 &&\\*
 &&\\*
 &&\\\hline\hline

$12a_{1288}$&$12n_{356}$&$12n_{462}$\\*
 \multirow{4}{*}{\scalebox{.2}{\includegraphics{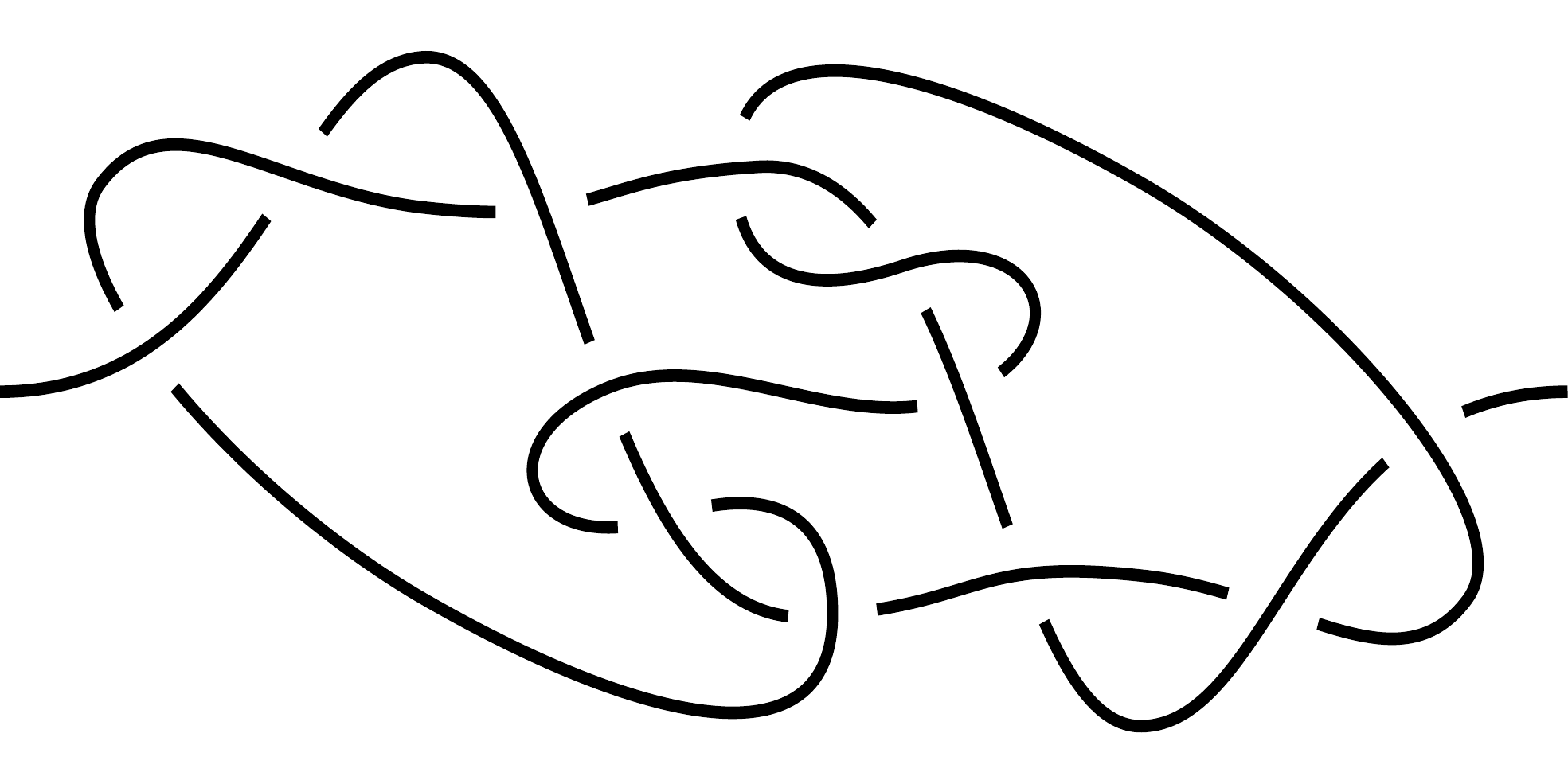}}}&\multirow{4}{*}{\scalebox{.2}{\includegraphics{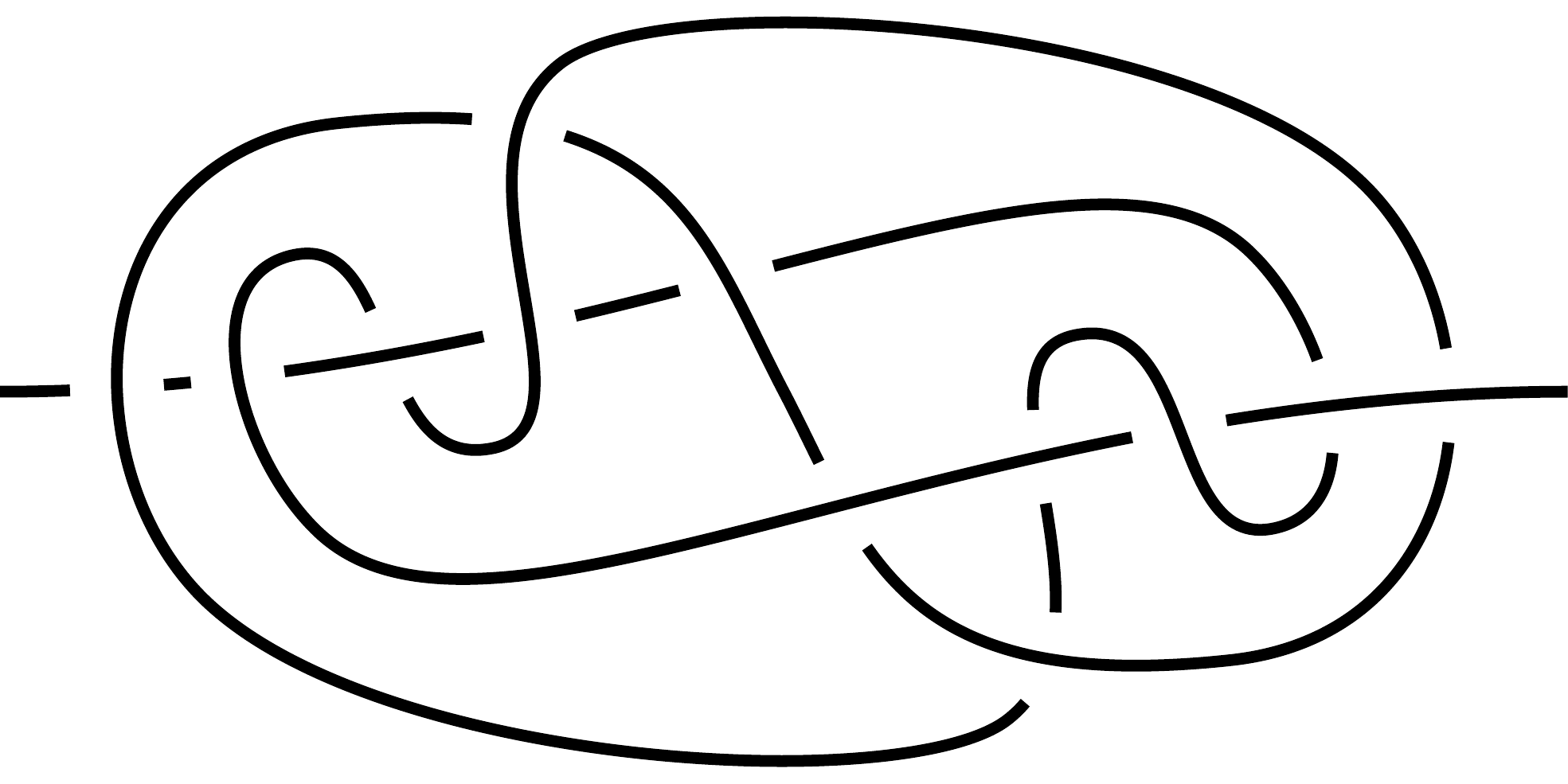}}}&
 \multirow{4}{*}{\scalebox{.2}{\includegraphics{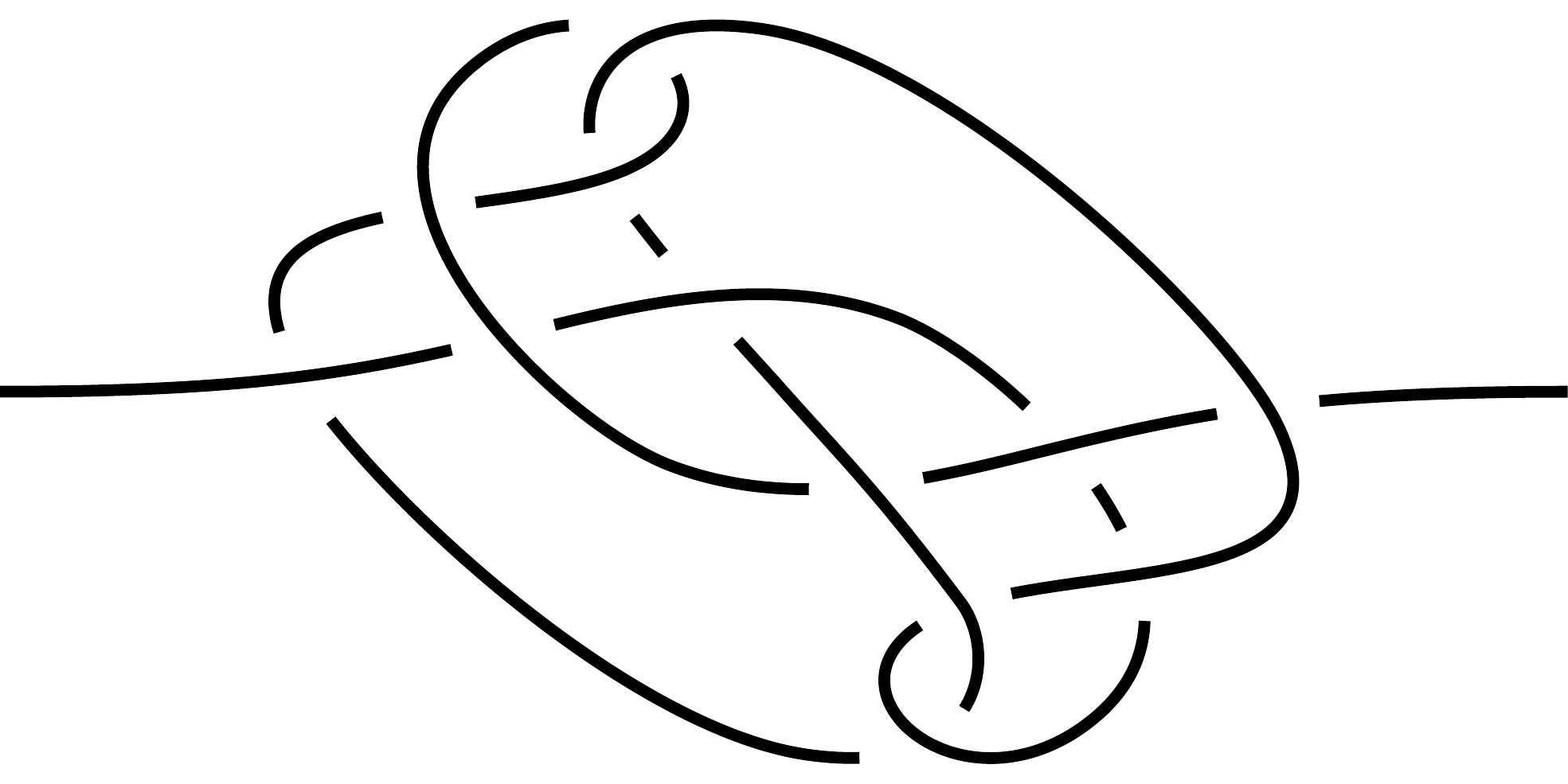}}}\\*
 &&\\*
 &&\\*
 &&\\*
 &&\\\hline\hline

$12n_{706}$&$12n_{873}$&\\*
 \multirow{4}{*}{\scalebox{.2}{\includegraphics{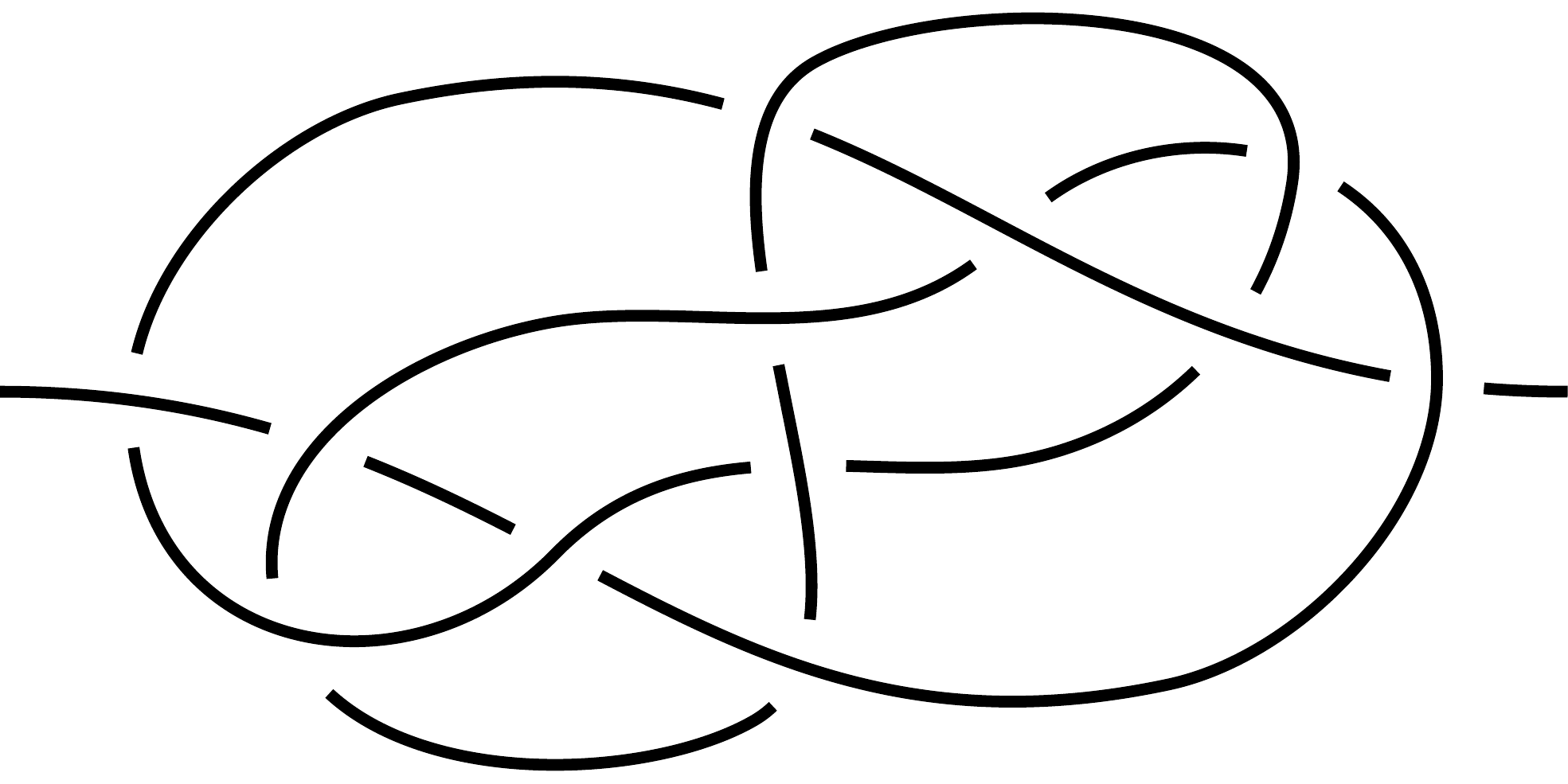}}}&\multirow{4}{*}{\scalebox{.2}{\includegraphics{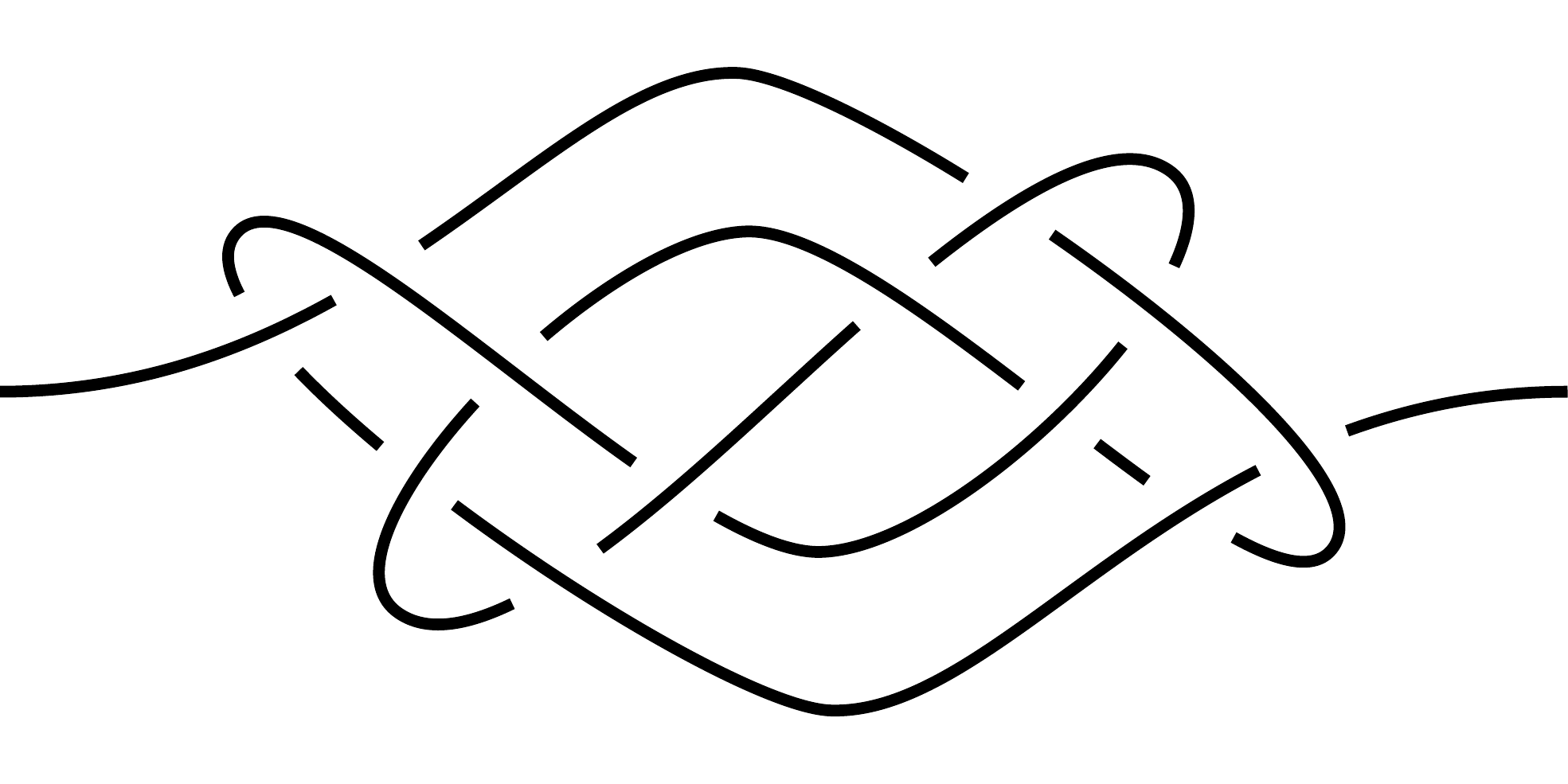}}}&\\*
 &&\\*
 &&\\*
 &&\\*
 &&\\\hline\hline

\end{longtable}

\bibliography{bibliography}
\bibliographystyle{alpha}
\end{document}